\documentclass[11pt]{amsart}
\usepackage[utf8]{inputenc}
\usepackage[T1]{fontenc}
\usepackage{bm,amsmath,amsthm,amssymb}
\usepackage{geometry}
\usepackage{graphicx}
\usepackage{tikz}
\usepackage{amsfonts}
\usepackage{array}
\usepackage{enumerate}
\usepackage{float}
\usepackage{lscape}
\usepackage[english]{babel}
\usepackage[all]{xy}
\usepackage{authblk}
\usepackage{hyperref}

\theoremstyle{plain}
\newtheorem{fact}{Fact}[section]
\newtheorem{theo}[fact]{Theorem}
\newtheorem{lem}[fact]{Lemma}
\newtheorem{defi}[fact]{Definition}

\newtheorem{prop}[fact]{Proposition}
\newtheorem{rmk}[fact]{Remark}
\newtheorem{coro}[fact]{Corollary}

\newtheorem{ex}[fact]{Example}
\newcommand{\cor}[1]{\textcolor{black}{#1}}
\newcommand{\coor}[1]{\textcolor{black}{#1}}

\newcommand{\rad}{\mathrm{rad}}

\title{A geometric realization of the $m$-cluster categories of type $\tilde{D_n}$}
\author{Lucie JACQUET-MALO}

\begin{document}

\maketitle

\begin{center}
\large{\sc{Lucie Jacquet-Malo}} \\
lucie.jacquet.malo@u-picardie.fr\\
\end{center}

\begin{abstract}
We show that a (non-full) subcategory of the $m$-cluster category of type $\tilde{D_n}$ is isomorphic to a category consisting of arcs in an $(n-2)m$-gon with two central $(m-1)$-gons inside of it. We show that the mutation of colored quivers and $m$-cluster-tilting objects is compatible with the flip of an $(m+2)$-angulation. In the final part of this paper, we detail an example of a quiver of type $\tilde{D_7}$.
\end{abstract}



\tableofcontents

\section*{Introduction}

Cluster algebras, defined by Fomin and Zelevinsky in \cite{FZ}, are commutative rings which have many connections with various fields of mathematics, such as Calabi-Yau algebras, integrable systems, Poisson geometry and quiver representations. Cluster categories were invented in order to get a better understanding of the link between quiver representations and cluster algebras. For a gentle introduction to cluster algebras, see \cite{F} and \cite{FZ2}.

\vspace{20pt}

Let $Q$ be an acyclic quiver and let $m$ be an integer. The cluster category ${\mathcal{C}}_Q$ associated with $Q$ was introduced by Buan, Marsh, Reineke, Reiten and Todorov in \cite{BMRRT} in order to categorify cluster algebras. The higher cluster category ${\mathcal{C}}^{(m)}_Q$ of $Q$ defined by Keller in \cite{Kel03} and Thomas in \cite{Tho} plays a role similar to that of cluster categories but with respect to the combinatorics of $m$-clusters, defined by Fomin and Reading in \cite{FR}. More precisely, there is a bijection between the $m$-clusters of a root system and the $m$-cluster-tilting objects of the $m$-cluster category. Wraalsen, and Zhou, Zhu in their respective papers \cite{W} and \cite{ZZ} also generalized many of the properties of cluster-tilting objects in cluster categories to the case of $m$-cluster-tilting objects in higher cluster categories. In particular, any $m$-rigid object $X$ having $n-1$ nonisomorphic indecomposable summands has exactly $m+1$ complements (i.e. nonisomorphic indecomposable objects $Y$ such that $X \bigoplus Y$ is an $m$-cluster-tilting object).

\vspace{20pt}


Acyclic cluster categories, as defined in \cite{BMRRT} are orbit categories of bounded derived categories of representations of hereditary algebras. It is thus remarkable that some of those categories can be described combinatorially. This fact is largely due the theory of Auslander and Reiten. Such a geometric description first appeared in the paper \cite{CCS} by Caldero, Chapoton and Schiffler, for a quiver $Q$ of type $A_n$, and for $m=1$. It was later generalized to all Dynkin types: in the article of Schiffler \cite{Sch} for type $D_n$ and in the article of Lamberti \cite{L}, in type $E$. Then Baur and Marsh generalized this method to higher cluster categories of type $A_n$ (in \cite{BM01}), and of type $D_n$ (in \cite{BM02}). They proved that the Auslander-Reiten quiver of the higher cluster category could be realized considering elementary moves of so-called $m$-arcs in a polygon for type $A$, and a punctured polygon for type $D$.

Similar geometric descriptions for Euclidean quivers are not so well-behaved. Indeed, the components of the Auslander-Reiten quiver are infinite so that the infinite radical of the $m$-cluster category does not vanish.

Unfortunately, drawing a quiver from an $(m+2)$-angulation did not give coherent results in terms of compatibility. To be precise, if $m>1$, then the Gabriel quiver of the $(m+2)$-angulation cannot determine the quiver of the flip of this $(m+2)$-angulation. Thus, in order to find back the new quiver corresponding to the flip, Buan and Thomas (cf \cite{BT}) introduced additionnal arrows, and built colored quivers, which means quivers where is associated a number with each arrow. This, with the new mutation of colored quivers, can be explicitly described from the mutation of $m$-cluster-tilting objects.

The case of $m$-cluster categories of type $\tilde{A_n}$ has been treated by Torkildsen in \cite{Tor}. Recently, Baur and Torkildsen in \cite{BauTor} have given a complete geometric realization of the cluster category of type $\tilde{A_n}$ including the geometric description of the infinite radical.

Our aim in this paper is to give a geometric description in type $\tilde{D_n}$. Let $P$ be an $(n-2)m$-gon with two $(m-1)$-gons inside of it. We define $m$-admissible arcs, $(m+2)$-angulations and their flip. We associate, in the spirit of Buan and Thomas in \cite{BT}, a colored quiver with an $(m+2)$-angulation of $P$. We prove in Theorem \ref{theo:corresp} that the flip of an $(m+2)$-angulation induces Buan-Thomas mutation of the associated colored quiver (see \cite{BT}). \cor{Moreover, we build a geometric category (consisting of arcs in $P$), which is equivalent to a (non-full) subcategory of the $m$-cluster category of type $\tilde{D}$ (see Theorem \ref{th:iso}) consisting of transjective and regular non homogeneous objects. The morphisms belonging to the infinite radical do not appear in the category.} Each $m$-rigid indecomposable object of ${\mathcal{C}}^{(m)}_{\tilde{D}}$ corresponds to some arc with no self-crossing in our geometric description.

\cor{In the companion paper \cite{JM}, we build a bijection between $(m+2)$-angulations and $m$-cluster-tilting objects of the $m$-cluster category of type $\tilde{D}_n$. We prove that, under this bijection, flips of $(m+2)$-angulations correspond to mutations of $m$-cluster-tilting objects. Our proofs apply to types $A$, $D$, $\tilde{A}$ as well.}

\vspace{20pt}

The paper is organized as follows.

In section $1$ we recall some important notions on $m$-cluster categories, rigid objects and colored quivers.

Let $P$ be an $(n-2)m$-gon with two $(m-1)$-gons inside of it. In section $2$, we define the notion of $m$-diagonals in $P$.

The definitions given in section $2$ lead us to build a new category of $m$-diagonals in section $3$.

In section $4$, we associate a colored quiver with an $(m+2)$-angulation of $P$. In Theorem \ref{theo:corresp}, we show the compatibility between flips of $(m+2)$-angulations and mutations of colored quivers.

Section $5$ is devoted to the study of the category of $m$-diagonals. We construct a bijection between arcs in $P$ and $m$-rigid indecomposable objects in the $m$-cluster category.

In section $6$ we focus on the indecomposable $m$-rigid objects in the $m$-cluster category ${\mathcal{C}}^{(m)}_Q$. We show that the bijection above induces a bijection between $m$-rigid objects and arcs which do not cross themselves. We give a bijection between $m$-cluster-tilting objects and $(m+2)$-angulations in \cite{JM}.

Finally, we give a detailed example illustrating some of our results, where $Q$ is a quiver of type $\tilde{D_7}$, and $m=2$. This corresponds to the case of a decagon with two inner boundary components inside of it.

\subsection*{Acknowledgements}
This article is part of my PhD thesis under the supervision of Yann Palu and Alexander Zimmermann. I would like to thank Yann Palu warmly for introducing me to the subject of cluster categories, and for his patience and kindness. \cor{Finally, I would like to thank the anonymous reviewer for his kind report and for his many advises that help improving the first version of this article.}

\section{Preliminaries}

Notations:

Throughout this paper, the letter $K$ denotes a field. We fix a positive integer $m$.


We denote by ${\mathrm{mod}}(KQ)$ the category of finitely generated right modules over the path algebra of a finite acyclic quiver $Q$. For an object $T$, we denote by $\mathrm{add}(T)$ the additive category spanned by $T$. For any further information about representation theory of associative algebras, see \cite{ASS}.

\subsection{Higher cluster categories}
In 2006, in order to categorify the notion of cluster algebras, Buan, Marsh, Reineke, Reiten and Todorov in \cite{BMRRT} defined the cluster category of an acyclic quiver in the following way:

If $Q$ is an acyclic quiver, let ${\mathcal{D}}^b(KQ)$ be the bounded derived category of the category of ${\mathrm{mod}}~KQ$. The category ${\mathcal{C}}_Q$ is the orbit category of the derived category under the functor $\tau^{-1}[1]$, where the letter $\tau$ stands for the Auslander-Reiten translation and the functor $[1]$ is the shift functor.

Cluster categories give a categorification of clusters in a cluster algebra in terms of tilting objects. To be precise, the cluster variables of the algebra are in $1-1$ correspondence with the indecomposable rigid objects in ${\mathcal{C}}_Q$, and the clusters are in $1-1$ correspondence with the basic cluster-tilting objects in ${\mathcal{C}}_Q$.

It is known from \cite{BMRRT} that ${\mathcal{C}}_Q$ is Krull-Schmidt. Keller in \cite{Kel03} has shown that it is a triangulated category. Since $\tau$ and $[1]$ become isomorphic in ${\mathcal{C}}_Q$, we have that ${\mathcal{C}}_Q$ is $2$-Calabi-Yau.

We can also define the higher cluster category
\[ {\mathcal{C}}_Q^m={\mathcal{D}}^b(KQ)/\tau^{-1}[m]. \]
Again, the higher cluster category is Krull-Schmidt, $(m+1)$-Calabi-Yau, and triangulated.

\begin{defi}
Let $T$ be an object in the category ${\mathcal{C}}^m_Q$. Then $T$ is $m$-rigid if
\[ {\mathrm{Ext}}^i(T,T)=0~\forall i \in \{1,\cdots,m \}. \]
\end{defi}

\begin{defi}\cite{KR}
Let $T$ be an object in the category ${\mathcal{C}}^m_Q$. Then $T$ is $m$-cluster-tilting when we have the following equivalence: \[ X \text{ is in } {\mathrm{add}}~T \iff {\mathrm{Ext}}^i_{\mathcal{C}_Q^{(m)}}(T,X)=0~\forall i \in \{ 1,\cdots,m \}. \]
\end{defi}

\begin{defi}\cite{ZZ}
\cor{Let $T$ be an object in the category ${\mathcal{C}}^m_Q$. Then $T$ is almost complete $m$-cluster-tilting if there exists an indecomposable object $X$ which do not belong to $\mathrm{add}~T$ such that $X\bigoplus T$ is an $m$-cluster-tilting object. The object $X$ is called a complement for $T$.}
\end{defi}

\begin{theo}\cite{ZZ}\cite{W}
\cor{Let $T$ be an object in the category ${\mathcal{C}}^m_Q$. If $T$ is an almost complete $m$-cluster-tilting object, then there exist exactly $m+1$ non isomorphic complements for $T$.}
\end{theo}

If $n$ corresponds to the number of vertices of $Q$, it is known from \cite{Z}, that $T$ is an $m$-cluster-tilting object if and only if $T$ has $n$ indecomposable direct summands (up to isomorphism) and is $m$-rigid. So, for $T=\bigoplus T_i$ a basic $m$-cluster-tilting object, where each $T_i$ is indecomposable, let us define the almost complete $m$-rigid object ${\overline{T}}=T/T_k$ (where $T_k$ is an indecomposable summand of $T$). There are, up to isomorphism, $m+1$ complements of the almost complete $m$-cluster-tilting object $\overline{T}$, denoted by $T_k^{(c)}$, for $c \in \{ 0,\cdots, m \}$. Iyama and Yoshino in \cite{IY} showed the following Theorem:

\begin{theo}\cite[Lemma 5.7]{IY}
There are $m+1$ exchange triangles:
\begin{equation}
\xymatrix{
T_k^{(c)} \ar^{f_k^{(c)}}[r] & B_k^{(c)} \ar^{g_k^{(c+1)}}[r] & T_k^{(c+1)} \ar^{h_k^{(c+1)}}[r] & T_k^{(c)}[1] }
\label{eq:10}
\end{equation} where $T_k^{(m+1)}=T_k^{(0)}=T_k$.
Here, the $B_k^{(c)}$ are in ${\mathrm{add}} \overline{T}$, the maps $f_k^{(c)}$ (respectively $g_k^{(c+1)}$) are minimal left (respectively right) ${\mathrm{add}} \overline{T}$-approximations, hence, not split monomorphism or split epimorphism.
\end{theo}

For a definition of left and right approximation, see \cite{AS}.

\begin{defi}
The left-mutation of $T$ at $T_k^{(c)}$, for $k \in \{1, \cdots, m\}$, is the $m$-cluster-tilting object
\[ \mu_{T_k^{(c)}}T=T/T_k^{(c)} \bigoplus T_k^{(c+1)} \]
\end{defi}

We will often write $\mu_k$ instead of $\mu_{T_k}$, whenever there is no possible confusion with $c$.

\subsection{Mutation of colored quivers}\label{sec:colq}
In this section, we let $T$ be an $m$-cluster-tilting object, and $T'$ be some $m$-cluster-tilting object which is obtained by mutation of $T$. Unfortunately, if $Q_T$ is the Gabriel quiver associated with $T$, there does not exist any quiver mutation $\mu_j$ such that $Q_{T'}=\mu_j(Q_T)$. Then, in order to overcome this problem, Buan and Thomas in \cite{BT} defined a graded version of $Q_T$, called the colored quiver associated with $T$. \cor{For a precise example of the lack of information given by the Gabriel quiver, see \cite{Buan}.}

\begin{defi}\cite{BT}
Given a positive integer $m$, a colored quiver consists of the data of a quiver $Q=(Q_0,Q_1,s,t)$ and of an application $c:Q_1 \to \{ 0,\cdots,m \}$ which associates with an arrow its color.
Let $q_{ij}^{(c)}$ be the number of arrows from $i$ to $j$ of color $c$. If there is an arrow from $i$ to $j$ of color $c$, then we write $i \xrightarrow{(c)} j$.
\end{defi}

We require our colored quivers to satisfy the following conditions:

\begin{enumerate}
\item $q_{ii}^{(c)}=0$ for all $c$.
\item monochromaticity: if $q_{ij}^{(c)} \neq 0$ then $q_{ij}^{(c')}=0$ for all $c' \neq c$.
\item symmetry: $q_{ij}^{(c)}=q_{ji}^{(m-c)}$.
\end{enumerate}

The operation we are about to define is an involution called the mutation of a colored quiver at a vertex.

\begin{defi}\cite{BT}
Let $Q$ be a colored quiver, and let $k$ be a vertex of $Q$. We define the new quiver $\mu_k(Q)$ with the same vertices, and the new number of arrows $\tilde{q}_{ij}^{(c)}$ given by:
\[ 
\tilde{q}_{ij}^{(c)} = \left\{
    \begin{array}{l}
        {q}_{ij}^{(c+1)} \mbox{ if } j = k \\
        {q}_{ij}^{(c-1)} \mbox{ if } i = k \\
        {\mathrm{max}} \{0,q_{ij}^{(c)}-\sum_{t \neq c}{q}_{ij}^{(t)} + ({q}_{ik}^{(c)}-{q}_{ik}^{(c-1)}){q}_{kj}^{(0)} + {q}_{ik}^{(m)}({q}_{kj}^{(c)}-{q}_{kj}^{(c+1)}) \} \mbox{ otherwise.}
    \end{array}
    \right.
\]
\end{defi}

They also showed in \cite{BT} that mutating a colored quiver in this way is equivalent to the following procedure:

\begin{enumerate}
\item For any $\xymatrix@1{i\ar[r]^{(c)} & k\ar[r]^{(0)} & j}$, and any integer $c$ in $\{0,\cdots,m\}$, draw an arrow $\xymatrix@1{i\ar[r]^{(c)} & j}$ and an arrow $\xymatrix@1{j\ar[r]^{(m-c)} & i}$.
\item If the condition of monochromaticity is not satisfied anymore from one vertex $i$ to one vertex $j$, then remove the same number of arrows of each color, in order to restore the condition.
\item For any arrow $\xymatrix@1{i \ar[r]^{(c)} & k}$, add $1$ to the color $c$, and for any arrow $\xymatrix@1{k \ar[r]^{(c)} & j}$, subtract $1$ to the color $c$.
\end{enumerate}


With each $m$-cluster-tilting object $T$ in the $m$-cluster category, we associate a corresponding colored quiver $Q_T$ such that:

\begin{enumerate}
\item The vertices of $Q_T$ are the integers from $1$ to $n$ where $n$ is the number of indecomposable summands of $T$.
\item The number $q_{ij}^{(c)}$ is the multiplicity of $T_j$ in $B_i^{(c)}$ in the exchange triangle (\ref{eq:10}).
\end{enumerate}

We note that the subquiver of $Q_T$ with same vertices but only $0$-colored arrows is the Gabriel quiver of the endomorphism algebra of $T$.

We now state the main Theorem about colored quivers and $m$-cluster-tilting objects, a proof of which can be found in \cite{BT}:

\begin{theo}\cite[Theorem 2.1]{BT}\label{th:mut}
Let $T=\bigoplus_{i=1}^{n} T_i$ and $T'=T/T_k \bigoplus T_k^{(1)}$ be $m$-cluster-tilting objects, where there is an exchange triangle $T_k \to B_k^{(0)} \to T_k^{(1)} \to T_k[1]$. Then
\[Q_{T'}=\mu_k(Q_T).\]
\end{theo}

In particular, the colored quiver $Q_{T'}$ only depends on the colored quiver $Q_T$.

\section{Geometric realization and flips}\label{sec:geo}

In this section we introduce the geometric realization inspired from \cite{CP01} and \cite{CP02}. We define the flip of an $(m+2)$-angulation, and build a colored quiver from an $(m+2)$-angulation. The main Theorem of the section, namely Theorem \ref{theo:corresp}, states the compatibility between the flip of an $(m+2)$-angulation and the mutation of the associated colored quiver.

Let $n \geq 4$ and $m \geq 2$ be integers. Let $P$ be an $(n-2)m$-gon with two central $(m-1)$-gons $L$ and $R$ inside of it (cf figure \ref{fig:figure1}). We set that $L$ (respectively $R$) is the leftmost (respectively the rightmost) $(m-1)$-gon in the interior of $P$. We replace each vertex of $L$ and $R$ by a disk, which we henceforth call a thick vertex. If $m=1$, then the surface considered is a disk with $n-2$ marked points on its border, with two thick vertices in its interior.

\cor{Considering that there is a realization of Dynkin type $D$ using tagged arcs, we could have used such arcs, with two punctures inside of the polygon in order to build a geometric realization of type $\tilde{D}$. However, this sort of realization is not compatible with the mutation of colored quivers associated with the $(m+2)$-angulations.}

\begin{figure}[!h]
\centering
\begin{tikzpicture}[scale=0.25]
\fill[fill=black,fill opacity=0.15] (0,4) -- (0,-4) -- (6.93,-8) -- (13.86,-4) -- (13.86,4) -- (6.93,8) -- cycle;
\draw [fill=black,fill opacity=1.0] (3.46,2) circle (0.5cm);
\draw [fill=black,fill opacity=1.0] (3.46,-2) circle (0.5cm);
\draw [fill=black,fill opacity=1.0] (10.39,2) circle (0.5cm);
\draw [fill=black,fill opacity=1.0] (10.39,-2) circle (0.5cm);
\draw (0,4)-- (0,-4);
\draw (0,-4)-- (6.93,-8);
\draw (6.93,-8)-- (13.86,-4);
\draw (13.86,-4)-- (13.86,4);
\draw (13.86,4)-- (6.93,8);
\draw (6.93,8)-- (0,4);
\draw [shift={(1.86,0)}] plot[domain=-0.89:0.89,variable=\t]({1*2.57*cos(\t r)+0*2.57*sin(\t r)},{0*2.57*cos(\t r)+1*2.57*sin(\t r)});
\draw [shift={(5.07,0)}] plot[domain=2.25:4.04,variable=\t]({1*2.57*cos(\t r)+0*2.57*sin(\t r)},{0*2.57*cos(\t r)+1*2.57*sin(\t r)});
\draw [shift={(12,0)}] plot[domain=2.25:4.04,variable=\t]({1*2.57*cos(\t r)+0*2.57*sin(\t r)},{0*2.57*cos(\t r)+1*2.57*sin(\t r)});
\draw [shift={(8.79,0)}] plot[domain=-0.89:0.89,variable=\t]({1*2.57*cos(\t r)+0*2.57*sin(\t r)},{0*2.57*cos(\t r)+1*2.57*sin(\t r)});
\end{tikzpicture}
\caption{The $(n-2)m$-gon with two digons. Here $m=3$ and $n=4$.}
\label{fig:figure1}
\end{figure}

\begin{defi}
Let us number the vertices of the polygon $P$ from $1$ to $(n-2)m$ clockwise. Then, an admissible arc between $i$ and $j$ is defined in the following way:
\begin{enumerate}
\item If $i \neq j$, then an admissible arc is an oriented path from $i$ to $j$, lying inside of $P$, which does not cross any of the two inner polygons, satisfying one of the following conditions:
\begin{itemize}
\item Either the arc crosses the space between both central polygons and cuts the figure into a $(km+1)$-gon and a $(k'm+1)$-gon, for some $k$ ($k'$ is entirely determined by $k$). This arc is of type $1$.
\item Either, the arc is homotopic to a boundary path, and cuts the figure into a $km$-gon with both central polygons inside of it and a $k'm+2$-gon (where $k'$ is still entirely determined by $k$). This arc is of type $2$.
\end{itemize}
\item If $i=j$, there are four types of admissible arcs :
\begin{itemize}
	\item A path ending in $i$, and the other end of the path tangent to one of the thick vertices placed around $L$.
	\item A path ending in $i$, and the other end of the path tangent to one of the thick vertices placed around $R$.
	\item A path starting and ending at $i$, going around $L$ or $R$, called a loop.
	\item A path starting and ending at $i$, going around $L$ or $R$, called a tagged loop.
\end{itemize}
\cor{All these arcs are of type 3.}
\item Any arc being tangent to two disks, one arising from $L$, and one from $R$ is admissible. \cor{This arc is of type 4.}
\end{enumerate}
\end{defi}

\cor{Note that any admissible arc can have some self-crossings. We will see at the end of the paper in Lemma \ref{lem:tri} that the admissible arcs without any self-crossing are in bijection with the rigid objects.}

\begin{defi}
Consider one moment an arc starting at an arbitrary vertex of $P$, and ending at a thick vertex $L$ or $R$. This arc is called to be left tangent at $L$ if it is $\mathcal{C}^\infty$, tangent to a thick vertex, and if there exists a neighborhood of this thick vertex such that this one is situated at the right of the vertex. We similarly define right tangency.
\end{defi}

\vspace{10pt}


Note that we only consider unoriented arcs, the order of $i$ and $j$ is not important. For convenience, we will nevertheless use the terminology "from $i$ to $j$".

\begin{defi}
\coor{We call a Dehn twist an application which, with an arc, associates another arc, obtained by the action of rotating $L$ (respectively $R$). It means that if we consider an arc $\alpha$ hung to $L$ (respectively $R$), applying a Dehn twist of $L$ (respectively $R$) makes $\alpha$ roll around $L$ (respectively $R$) only (see figure \ref{fig:roll}). We can define Dehn twist in both clockwise and counterclockwise directions.}
\end{defi}

We now define $m$-diagonals.

\begin{defi}\label{def:equiv}
Let $\alpha$ and $\beta$ be two admissible arcs of the same type. We say that these arcs are equivalent when:
\begin{itemize}
\item \cor{If $\alpha$ and $\beta$ are of type $1$, $2$ or $4$, then they are said to be equivalent if they are homotopic.
\item If $\alpha$ and $\beta$ are of type $3$ and hang to the same vertex and to the same inner polygon (say $L$ for instance, but $R$ is similar), then they are said to be equivalent if they are homotopic, or, if there exist a Dehn twist $t$ such that $\alpha=t(\beta)$. In this case, we add in the class of equivalence a tagged loop (respectively a loop) if the arcs is left tangent (respectively right tangent), drawn in figure \ref{fig:loop}, around $L$, ending at the same vertex as $\alpha$.}
\end{itemize}
\end{defi}

\begin{figure}[!h]
\centering
\begin{tikzpicture}[scale=0.5]
\fill[fill=black,fill opacity=0.1] (0,4) -- (-1.66,3.46) -- (-2.69,2.05) -- (-2.68,0.3) -- (-1.66,-1.11) -- (0,-1.65) -- (1.66,-1.11) -- (2.69,0.3) -- (2.69,2.05) -- (1.66,3.46) -- cycle;
\draw [fill=black,fill opacity=1.0] (-1.66,1.57) circle (0.2cm);
\draw [fill=black,fill opacity=1.0] (1.66,1.58) circle (0.2cm);
\draw (0,4)-- (-1.66,3.46);
\draw (-1.66,3.46)-- (-2.69,2.05);
\draw (-2.69,2.05)-- (-2.68,0.3);
\draw (-2.68,0.3)-- (-1.66,-1.11);
\draw (-1.66,-1.11)-- (0,-1.65);
\draw (0,-1.65)-- (1.66,-1.11);
\draw (1.66,-1.11)-- (2.69,0.3);
\draw (2.69,0.3)-- (2.69,2.05);
\draw (2.69,2.05)-- (1.66,3.46);
\draw (1.66,3.46)-- (0,4);
\draw(-1.66,1.17) circle (0.4cm);
\draw(1.66,1.18) circle (0.4cm);
\draw [shift={(0.3,0.15)}] plot[domain=0.96:3.44,variable=\t]({-0.77*2.41*cos(\t r)+0.64*1.46*sin(\t r)},{-0.64*2.41*cos(\t r)+-0.77*1.46*sin(\t r)});
\draw [shift={(0.99,-1.47)}] plot[domain=3.09:4.82,variable=\t]({-0.21*3.53*cos(\t r)+0.98*0.93*sin(\t r)},{-0.98*3.53*cos(\t r)+-0.21*0.93*sin(\t r)});
\end{tikzpicture}
\caption{Loop drawn at a vertex, around $L$. The same can be done for $R$}
\label{fig:loop}
\end{figure}

\begin{rmk}
\coor{The reader can wonder why we do not replace $L$ and $R$ by punctures, as Baur and Marsh did for case $D$ in \cite{BM02}. The problem is that, if $m\geq 2$, there cannot exists any definition of the an $(m+2)$-angulation which leads to a compatibility between flip of the $(m+2)$-angulation and the colored quiver mutation.}
\end{rmk}

\begin{defi}\label{def:diago}
An $m$-diagonal is an equivalence class of admissible arcs, up to homotopy.
\end{defi}

\coor{Whenever it is clear, we indifferently identify  the term of equivalence class and well-chosen representative. Moreover, there can be for the reader a confusion between the equivalence relation on arcs and the homotopy equivalence relation. Note that all the arcs are considered up to homotopy, and that the $m$-diagonals are classes under the equivalence relation defined in \ref{def:equiv}.}


%

\begin{rmk}
If $i=j$, then, there exist an infinity of classes of arcs tangents to the left or $L$. Indeed, the $m$-diagonals car roll around $L$ and $R$, and this leads to two different classes of arcs. See figure \ref{fig:rolla} for an illustration.
\end{rmk}

\begin{figure}[!h]
\centering
\begin{tikzpicture}[scale=0.5]
\fill[fill=black,fill opacity=0.1] (0,4) -- (-1.66,3.46) -- (-2.69,2.05) -- (-2.68,0.3) -- (-1.66,-1.11) -- (0,-1.65) -- (1.66,-1.11) -- (2.69,0.3) -- (2.69,2.05) -- (1.66,3.46) -- cycle;
\draw [fill=black,fill opacity=1.0] (-1.66,1.57) circle (0.2cm);
\draw [fill=black,fill opacity=1.0] (1.66,1.58) circle (0.2cm);
\draw (0,4)-- (-1.66,3.46);
\draw (-1.66,3.46)-- (-2.69,2.05);
\draw (-2.69,2.05)-- (-2.68,0.3);
\draw (-2.68,0.3)-- (-1.66,-1.11);
\draw (-1.66,-1.11)-- (0,-1.65);
\draw (0,-1.65)-- (1.66,-1.11);
\draw (1.66,-1.11)-- (2.69,0.3);
\draw (2.69,0.3)-- (2.69,2.05);
\draw (2.69,2.05)-- (1.66,3.46);
\draw (1.66,3.46)-- (0,4);
\draw(-1.66,1.17) circle (0.4cm);
\draw(1.66,1.18) circle (0.4cm);
\draw [shift={(-0.65,-0.55)}] plot[domain=1.52:3.46,variable=\t]({-0.79*3.42*cos(\t r)+0.62*1.27*sin(\t r)},{-0.62*3.42*cos(\t r)+-0.79*1.27*sin(\t r)});
\end{tikzpicture}
=
\begin{tikzpicture}[scale=0.5]
\fill[fill=black,fill opacity=0.1] (0,4) -- (-1.66,3.46) -- (-2.69,2.05) -- (-2.68,0.3) -- (-1.66,-1.11) -- (0,-1.65) -- (1.66,-1.11) -- (2.69,0.3) -- (2.69,2.05) -- (1.66,3.46) -- cycle;
\draw [fill=black,fill opacity=1.0] (-1.66,1.57) circle (0.2cm);
\draw [fill=black,fill opacity=1.0] (1.66,1.58) circle (0.2cm);
\draw (0,4)-- (-1.66,3.46);
\draw (-1.66,3.46)-- (-2.69,2.05);
\draw (-2.69,2.05)-- (-2.68,0.3);
\draw (-2.68,0.3)-- (-1.66,-1.11);
\draw (-1.66,-1.11)-- (0,-1.65);
\draw (0,-1.65)-- (1.66,-1.11);
\draw (1.66,-1.11)-- (2.69,0.3);
\draw (2.69,0.3)-- (2.69,2.05);
\draw (2.69,2.05)-- (1.66,3.46);
\draw (1.66,3.46)-- (0,4);
\draw(-1.66,1.17) circle (0.4cm);
\draw(1.66,1.18) circle (0.4cm);
\draw [shift={(0.3,0.15)}] plot[domain=0.96:3.44,variable=\t]({-0.77*2.41*cos(\t r)+0.64*1.46*sin(\t r)},{-0.64*2.41*cos(\t r)+-0.77*1.46*sin(\t r)});
\draw [shift={(0.99,-1.47)}] plot[domain=3.09:4.82,variable=\t]({-0.21*3.53*cos(\t r)+0.98*0.93*sin(\t r)},{-0.98*3.53*cos(\t r)+-0.21*0.93*sin(\t r)});
\end{tikzpicture}
$\neq$
\begin{tikzpicture}[scale=0.5]
\fill[fill=black,fill opacity=0.1] (0,4) -- (-1.66,3.46) -- (-2.69,2.05) -- (-2.68,0.3) -- (-1.66,-1.11) -- (0,-1.65) -- (1.66,-1.11) -- (2.69,0.3) -- (2.69,2.05) -- (1.66,3.46) -- cycle;
\draw [fill=black,fill opacity=1.0] (-1.66,1.57) circle (0.2cm);
\draw [fill=black,fill opacity=1.0] (1.66,1.58) circle (0.2cm);
\draw (0,4)-- (-1.66,3.46);
\draw (-1.66,3.46)-- (-2.69,2.05);
\draw (-2.69,2.05)-- (-2.68,0.3);
\draw (-2.68,0.3)-- (-1.66,-1.11);
\draw (-1.66,-1.11)-- (0,-1.65);
\draw (0,-1.65)-- (1.66,-1.11);
\draw (1.66,-1.11)-- (2.69,0.3);
\draw (2.69,0.3)-- (2.69,2.05);
\draw (2.69,2.05)-- (1.66,3.46);
\draw (1.66,3.46)-- (0,4);
\draw(-1.66,1.17) circle (0.4cm);
\draw(1.66,1.18) circle (0.4cm);
\draw [shift={(-0.2,0.55)}] plot[domain=1.31:5.14,variable=\t]({-0.9*2.87*cos(\t r)+0.43*2.15*sin(\t r)},{-0.43*2.87*cos(\t r)+-0.9*2.15*sin(\t r)});
\draw [shift={(-0.04,0.94)}] plot[domain=-0.48:3.92,variable=\t]({-0.99*2.43*cos(\t r)+-0.1*1.39*sin(\t r)},{0.1*2.43*cos(\t r)+-0.99*1.39*sin(\t r)});
\end{tikzpicture}
=
\begin{tikzpicture}[scale=0.5]
\fill[fill=black,fill opacity=0.1] (0,4) -- (-1.66,3.46) -- (-2.69,2.05) -- (-2.68,0.3) -- (-1.66,-1.11) -- (0,-1.65) -- (1.66,-1.11) -- (2.69,0.3) -- (2.69,2.05) -- (1.66,3.46) -- cycle;
\draw [fill=black,fill opacity=1.0] (-1.66,1.57) circle (0.2cm);
\draw [fill=black,fill opacity=1.0] (1.66,1.58) circle (0.2cm);
\draw (0,4)-- (-1.66,3.46);
\draw (-1.66,3.46)-- (-2.69,2.05);
\draw (-2.69,2.05)-- (-2.68,0.3);
\draw (-2.68,0.3)-- (-1.66,-1.11);
\draw (-1.66,-1.11)-- (0,-1.65);
\draw (0,-1.65)-- (1.66,-1.11);
\draw (1.66,-1.11)-- (2.69,0.3);
\draw (2.69,0.3)-- (2.69,2.05);
\draw (2.69,2.05)-- (1.66,3.46);
\draw (1.66,3.46)-- (0,4);
\draw(-1.66,1.17) circle (0.4cm);
\draw(1.66,1.18) circle (0.4cm);
\draw [shift={(-0.2,0.55)}] plot[domain=1.31:5.14,variable=\t]({-0.9*2.87*cos(\t r)+0.43*2.15*sin(\t r)},{-0.43*2.87*cos(\t r)+-0.9*2.15*sin(\t r)});
\draw [shift={(-0.2,1.02)}] plot[domain=-0.72:3.51,variable=\t]({-0.99*2.41*cos(\t r)+0.12*1.52*sin(\t r)},{-0.12*2.41*cos(\t r)+-0.99*1.52*sin(\t r)});
\draw [shift={(0.73,1.13)}] plot[domain=3.25:5.68,variable=\t]({-0.9*1.43*cos(\t r)+0.44*0.69*sin(\t r)},{-0.44*1.43*cos(\t r)+-0.9*0.69*sin(\t r)});
\draw [shift={(-1.22,1.32)}] plot[domain=-1.99:1.87,variable=\t]({-0.78*1.12*cos(\t r)+0.63*0.76*sin(\t r)},{-0.63*1.12*cos(\t r)+-0.78*0.76*sin(\t r)});
\draw [shift={(-0.58,0.28)}] plot[domain=1.46:4.65,variable=\t]({-0.89*3.09*cos(\t r)+0.45*2*sin(\t r)},{-0.45*3.09*cos(\t r)+-0.89*2*sin(\t r)});
\end{tikzpicture}
\caption{On the one hand, both figures on the left represent the same $m$-diagonal, whereas both figures on the right, considering that the top right is a untagged loop, represent the same $m$-diagonal.}
\label{fig:rolla}
\end{figure}

In figure \ref{fig:roll}, we can see an example of two arcs corresponding to the same $m$-diagonal, where we have applied a Dehn twist to one of the polygons.

\begin{figure}[!h]
\centering
\begin{tikzpicture}[scale=0.2]
\fill[fill=black,fill opacity=0.15] (0,4) -- (0,-4) -- (6.93,-8) -- (13.86,-4) -- (13.86,4) -- (6.93,8) -- cycle;
\draw [fill=black,fill opacity=1.0] (3.46,2) circle (0.5cm);
\draw [fill=black,fill opacity=1.0] (3.46,-2) circle (0.5cm);
\draw [fill=black,fill opacity=1.0] (10.39,2) circle (0.5cm);
\draw [fill=black,fill opacity=1.0] (10.39,-2) circle (0.5cm);
\draw (0,4)-- (0,-4);
\draw (0,-4)-- (6.93,-8);
\draw (6.93,-8)-- (13.86,-4);
\draw (13.86,-4)-- (13.86,4);
\draw (13.86,4)-- (6.93,8);
\draw (6.93,8)-- (0,4);
\draw [shift={(1.86,0)}] plot[domain=-0.89:0.89,variable=\t]({1*2.57*cos(\t r)+0*2.57*sin(\t r)},{0*2.57*cos(\t r)+1*2.57*sin(\t r)});
\draw [shift={(5.07,0)}] plot[domain=2.25:4.04,variable=\t]({1*2.57*cos(\t r)+0*2.57*sin(\t r)},{0*2.57*cos(\t r)+1*2.57*sin(\t r)});
\draw [shift={(12,0)}] plot[domain=2.25:4.04,variable=\t]({1*2.57*cos(\t r)+0*2.57*sin(\t r)},{0*2.57*cos(\t r)+1*2.57*sin(\t r)});
\draw [shift={(8.79,0)}] plot[domain=-0.89:0.89,variable=\t]({1*2.57*cos(\t r)+0*2.57*sin(\t r)},{0*2.57*cos(\t r)+1*2.57*sin(\t r)});
\draw [shift={(11.77,-0.5)},color=red]  plot[domain=3.31:4.14,variable=\t]({1*8.93*cos(\t r)+0*8.93*sin(\t r)},{0*8.93*cos(\t r)+1*8.93*sin(\t r)});
\end{tikzpicture}
\hspace{20pt}
\begin{tikzpicture}[scale=0.2]
\fill[fill=black,fill opacity=0.1] (0,4) -- (0,-4) -- (6.93,-8) -- (13.86,-4) -- (13.86,4) -- (6.93,8) -- cycle;
\draw [fill=black,fill opacity=1.0] (3.46,-2) circle (1cm);
\draw [fill=black,fill opacity=1.0] (10.39,2) circle (1cm);
\draw (0,4)-- (0,-4);
\draw (0,-4)-- (6.93,-8);
\draw (6.93,-8)-- (13.86,-4);
\draw (13.86,-4)-- (13.86,4);
\draw (13.86,4)-- (6.93,8);
\draw (6.93,8)-- (0,4);
\draw [shift={(6.93,-7.72)},color=red] plot[domain=1.04:2.1,variable=\t]({1*13.61*cos(\t r)+0*13.61*sin(\t r)},{0*13.61*cos(\t r)+1*13.61*sin(\t r)});
\end{tikzpicture}
\hspace{20pt}
\begin{tikzpicture}[scale=0.2]
\fill[fill=black,fill opacity=0.15] (0,4) -- (0,-4) -- (6.93,-8) -- (13.86,-4) -- (13.86,4) -- (6.93,8) -- cycle;
\draw [fill=black,fill opacity=1.0] (3.46,-1) circle (0.5cm);
\draw [fill=black,fill opacity=1.0] (10.39,-1) circle (0.5cm);
\draw (0,4)-- (0,-4);
\draw (0,-4)-- (6.93,-8);
\draw (6.93,-8)-- (13.86,-4);
\draw (13.86,-4)-- (13.86,4);
\draw (13.86,4)-- (6.93,8);
\draw (6.93,8)-- (0,4);
\draw(3.46,0) circle (1cm);
\draw(10.39,0) circle (1cm);
\draw [color=red](3.46,-1.5)-- (10.39,-1.5);
\begin{scriptsize}
\fill [color=black] (3.46,-1) circle (1.0pt);
\end{scriptsize}
\end{tikzpicture}
\caption{Some examples of arcs in the hexagon. For example, if $m \neq 1$, the second arc is not admissible because it cuts the polygon into a $3$-gon and a $5$-gon.}
\label{fig:arcs}
\end{figure}

\begin{figure}[!h]
	\centering
	\begin{tikzpicture}[scale=0.25]
	\fill[fill=black,fill opacity=0.15] (0,4) -- (0,-4) -- (6.93,-8) -- (13.86,-4) -- (13.86,4) -- (6.93,8) -- cycle;
	\draw [fill=black,fill opacity=1.0] (3.46,2) circle (0.4cm);
	\draw [fill=black,fill opacity=1.0] (3.46,-2) circle (0.4cm);
	\draw [fill=black,fill opacity=1.0] (10.39,2) circle (0.4cm);
	\draw [fill=black,fill opacity=1.0] (10.39,-2) circle (0.4cm);
	\draw (0,4)-- (0,-4);
	\draw (0,-4)-- (6.93,-8);
	\draw (6.93,-8)-- (13.86,-4);
	\draw (13.86,-4)-- (13.86,4);
	\draw (13.86,4)-- (6.93,8);
	\draw (6.93,8)-- (0,4);
	\draw [shift={(1.86,0)}] plot[domain=-0.89:0.89,variable=\t]({1*2.57*cos(\t r)+0*2.57*sin(\t r)},{0*2.57*cos(\t r)+1*2.57*sin(\t r)});
	\draw [shift={(5.07,0)}] plot[domain=2.25:4.04,variable=\t]({1*2.57*cos(\t r)+0*2.57*sin(\t r)},{0*2.57*cos(\t r)+1*2.57*sin(\t r)});
	\draw [shift={(12,0)}] plot[domain=2.25:4.04,variable=\t]({1*2.57*cos(\t r)+0*2.57*sin(\t r)},{0*2.57*cos(\t r)+1*2.57*sin(\t r)});
	\draw [shift={(8.79,0)}] plot[domain=-0.89:0.89,variable=\t]({1*2.57*cos(\t r)+0*2.57*sin(\t r)},{0*2.57*cos(\t r)+1*2.57*sin(\t r)});
	\draw [shift={(1.59,-3.73)}] plot[domain=-1.26:0.41,variable=\t]({0.94*11.55*cos(\t r)+-0.35*6.17*sin(\t r)},{0.35*11.55*cos(\t r)+0.94*6.17*sin(\t r)});
	\end{tikzpicture}
	\hspace{20pt}
	\begin{tikzpicture}[scale=0.25]
	\fill[fill=black,fill opacity=0.15] (0,4) -- (0,-4) -- (6.93,-8) -- (13.86,-4) -- (13.86,4) -- (6.93,8) -- cycle;
	\draw [fill=black,fill opacity=1.0] (3.46,2) circle (0.4cm);
	\draw [fill=black,fill opacity=1.0] (3.46,-2) circle (0.4cm);
	\draw [fill=black,fill opacity=1.0] (10.39,2) circle (0.4cm);
	\draw [fill=black,fill opacity=1.0] (10.39,-2) circle (0.4cm);
	\draw (0,4)-- (0,-4);
	\draw (0,-4)-- (6.93,-8);
	\draw (6.93,-8)-- (13.86,-4);
	\draw (13.86,-4)-- (13.86,4);
	\draw (13.86,4)-- (6.93,8);
	\draw (6.93,8)-- (0,4);
	\draw [shift={(1.86,0)}] plot[domain=-0.89:0.89,variable=\t]({1*2.57*cos(\t r)+0*2.57*sin(\t r)},{0*2.57*cos(\t r)+1*2.57*sin(\t r)});
	\draw [shift={(5.07,0)}] plot[domain=2.25:4.04,variable=\t]({1*2.57*cos(\t r)+0*2.57*sin(\t r)},{0*2.57*cos(\t r)+1*2.57*sin(\t r)});
	\draw [shift={(12,0)}] plot[domain=2.25:4.04,variable=\t]({1*2.57*cos(\t r)+0*2.57*sin(\t r)},{0*2.57*cos(\t r)+1*2.57*sin(\t r)});
	\draw [shift={(8.79,0)}] plot[domain=-0.89:0.89,variable=\t]({1*2.57*cos(\t r)+0*2.57*sin(\t r)},{0*2.57*cos(\t r)+1*2.57*sin(\t r)});
	\draw [shift={(8.18,-2.24)}] plot[domain=0.72:4.18,variable=\t]({-0.53*7.35*cos(\t r)+0.85*3.06*sin(\t r)},{-0.85*7.35*cos(\t r)+-0.53*3.06*sin(\t r)});
	\draw [shift={(9.47,0.18)}] plot[domain=-1.84:1.09,variable=\t]({-0.57*3.36*cos(\t r)+0.82*2.57*sin(\t r)},{-0.82*3.36*cos(\t r)+-0.57*2.57*sin(\t r)});
	\end{tikzpicture}
	\caption{These two arcs represent the same $m$-diagonal.}
	\label{fig:roll}
\end{figure}

\begin{defi}
We call by $t$-angle, a figure with $t$ sides, made of:
\begin{itemize}
\item Sides of $P$
\item Sides of $L$ or $R$
\item $m$-diagonals.
\end{itemize}
\end{defi}

\begin{defi}
\begin{itemize}
\item \cor{Two $m$-diagonals are said to be noncrossing if each of them are non self-crossing and if there exist representatives in their homotopy class which do not cross each other in the interior of the surface.
\item An $(m+2)$-angulation is a set of noncrossing $m$-diagonals, such that there exist representative admissible arcs which cut the figure into $(m+2)$-angles. Such a choice of representatives is called a good set of representatives.}
\end{itemize}
\end{defi}

\coor{For instance, in order to show that two $m$-diagonals (one of type $3$) do not cross, one can, if needed, replace by the loop the arc of type $3$, since the loop belongs to this equivalence class (see figure \ref{fig:oeil} for an example). If the arc is left tangent, we can replace it by a tagged loop. If it is right tangent, we replace it by the untaggeed loop.}

\begin{rmk}
\coor{Many articles used tagged and untagged arcs instead of left and right arcs. Indeed, if we consider isolated arcs (not belonging to an $(m+2)$-angulation), it is more comfortable to use tagged and untagged arcs. Nevertheless, as one of the aims of the article is to show compatibility between the flip of an $(m+2)$-angulation and the colored quiver mutation, we remain using left and right tangency since the result cannot be showed using tagged and untagged arcs.}
\end{rmk}


See figure \ref{fig:tridep} for the example of an $(m+2)$-angulation.

\begin{figure}[!h]
\centering
\begin{tikzpicture}[scale=0.25]
\fill[fill=black,fill opacity=0.15] (0,4) -- (0,-4) -- (6.93,-8) -- (13.86,-4) -- (13.86,4) -- (6.93,8) -- cycle;
\draw [fill=black,fill opacity=1.0] (3.46,2) circle (0.4cm);
\draw [fill=black,fill opacity=1.0] (3.46,-2) circle (0.4cm);
\draw [fill=black,fill opacity=1.0] (10.39,2) circle (0.4cm);
\draw [fill=black,fill opacity=1.0] (10.39,-2) circle (0.4cm);
\draw (0,4)-- (0,-4);
\draw (0,-4)-- (6.93,-8);
\draw (6.93,-8)-- (13.86,-4);
\draw (13.86,-4)-- (13.86,4);
\draw (13.86,4)-- (6.93,8);
\draw (6.93,8)-- (0,4);
\draw [shift={(1.86,0)}] plot[domain=-0.89:0.89,variable=\t]({1*2.57*cos(\t r)+0*2.57*sin(\t r)},{0*2.57*cos(\t r)+1*2.57*sin(\t r)});
\draw [shift={(5.07,0)}] plot[domain=2.25:4.04,variable=\t]({1*2.57*cos(\t r)+0*2.57*sin(\t r)},{0*2.57*cos(\t r)+1*2.57*sin(\t r)});
\draw [shift={(12,0)}] plot[domain=2.25:4.04,variable=\t]({1*2.57*cos(\t r)+0*2.57*sin(\t r)},{0*2.57*cos(\t r)+1*2.57*sin(\t r)});
\draw [shift={(8.79,0)}] plot[domain=-0.89:0.89,variable=\t]({1*2.57*cos(\t r)+0*2.57*sin(\t r)},{0*2.57*cos(\t r)+1*2.57*sin(\t r)});
\draw (6.93,8)-- (6.93,-8);
\draw [shift={(-11.62,-8.09)}] plot[domain=0:0.59,variable=\t]({1*18.55*cos(\t r)+0*18.55*sin(\t r)},{0*18.55*cos(\t r)+1*18.55*sin(\t r)});
\draw [shift={(8.18,-3.66)}] plot[domain=-0.47:1.77,variable=\t]({-0.8*8.37*cos(\t r)+-0.6*4.29*sin(\t r)},{0.6*8.37*cos(\t r)+-0.8*4.29*sin(\t r)});
\draw [shift={(5.68,-3.66)}] plot[domain=-0.47:1.77,variable=\t]({0.8*8.37*cos(\t r)+0.6*4.29*sin(\t r)},{0.6*8.37*cos(\t r)+-0.8*4.29*sin(\t r)});
\draw [shift={(25.48,-8.09)}] plot[domain=2.55:3.14,variable=\t]({1*18.55*cos(\t r)+0*18.55*sin(\t r)},{0*18.55*cos(\t r)+1*18.55*sin(\t r)});
\end{tikzpicture}
\hspace{20pt}
\begin{tikzpicture}[scale=0.7]
\fill[fill=black,fill opacity=0.1] (0,4) -- (-1.66,3.46) -- (-2.69,2.05) -- (-2.68,0.3) -- (-1.66,-1.11) -- (0,-1.65) -- (1.66,-1.11) -- (2.69,0.3) -- (2.69,2.05) -- (1.66,3.46) -- cycle;
\draw [fill=black,fill opacity=1.0] (1.66,0.74) circle (0.2cm);
\draw [fill=black,fill opacity=1.0] (-1.66,0.74) circle (0.2cm);
\draw (0,4)-- (-1.66,3.46);
\draw (-1.66,3.46)-- (-2.69,2.05);
\draw (-2.69,2.05)-- (-2.68,0.3);
\draw (-2.68,0.3)-- (-1.66,-1.11);
\draw (-1.66,-1.11)-- (0,-1.65);
\draw (0,-1.65)-- (1.66,-1.11);
\draw (1.66,-1.11)-- (2.69,0.3);
\draw (2.69,0.3)-- (2.69,2.05);
\draw (2.69,2.05)-- (1.66,3.46);
\draw (1.66,3.46)-- (0,4);
\draw [shift={(-8.03,-1.44)}] plot[domain=-0.03:0.66,variable=\t]({1*8.03*cos(\t r)+0*8.03*sin(\t r)},{0*8.03*cos(\t r)+1*8.03*sin(\t r)});
\draw [shift={(8.03,-1.43)}] plot[domain=2.49:3.17,variable=\t]({1*8.03*cos(\t r)+0*8.03*sin(\t r)},{0*8.03*cos(\t r)+1*8.03*sin(\t r)});
\draw [shift={(-0.75,-11.66)}] plot[domain=-0.82:0.59,variable=\t]({-0.07*14.49*cos(\t r)+-1*1.99*sin(\t r)},{1*14.49*cos(\t r)+-0.07*1.99*sin(\t r)});
\draw [shift={(1.88,0.59)}] plot[domain=-3.51:0.16,variable=\t]({-0.24*1.73*cos(\t r)+0.97*1.22*sin(\t r)},{-0.97*1.73*cos(\t r)+-0.24*1.22*sin(\t r)});
\draw(1.66,1.14) circle (0.4cm);
\draw(-1.66,1.14) circle (0.4cm);
\draw [shift={(3.75,0.06)}] plot[domain=2.83:3.65,variable=\t]({1*2.39*cos(\t r)+0*2.39*sin(\t r)},{0*2.39*cos(\t r)+1*2.39*sin(\t r)});
\draw [shift={(0.99,2.08)}] plot[domain=5.14:6.26,variable=\t]({1*1.69*cos(\t r)+0*1.69*sin(\t r)},{0*1.69*cos(\t r)+1*1.69*sin(\t r)});
\draw [shift={(3.95,0.08)}] plot[domain=2.87:3.62,variable=\t]({1*2.58*cos(\t r)+0*2.58*sin(\t r)},{0*2.58*cos(\t r)+1*2.58*sin(\t r)});
\draw [shift={(-4.08,-2.44)}] plot[domain=0.19:0.89,variable=\t]({1*4.16*cos(\t r)+0*4.16*sin(\t r)},{0*4.16*cos(\t r)+1*4.16*sin(\t r)});
\draw [shift={(-2.36,1.48)}] plot[domain=4.44:5.47,variable=\t]({1*1.22*cos(\t r)+0*1.22*sin(\t r)},{0*1.22*cos(\t r)+1*1.22*sin(\t r)});
\end{tikzpicture}
\caption{On the left, a collection of arcs that is not a good set of representatives since it does not cut the figure into $5$-angles. An example of a good set of representatives is shown on the right in the case $m=2$.}
\label{fig:tridep}
\end{figure}

\begin{figure}[!h]
\centering
\begin{tikzpicture}[scale=0.4]
\draw [fill=black,fill opacity=1.0] (2,2) circle (0.3cm);
\draw [fill=black,fill opacity=1.0] (2,0) circle (0.3cm);
\draw [shift={(1.43,1)}] plot[domain=-1.06:1.06,variable=\t]({1*1.15*cos(\t r)+0*1.15*sin(\t r)},{0*1.15*cos(\t r)+1*1.15*sin(\t r)});
\draw [shift={(2.57,1)}] plot[domain=2.09:4.2,variable=\t]({1*1.15*cos(\t r)+0*1.15*sin(\t r)},{0*1.15*cos(\t r)+1*1.15*sin(\t r)});
\draw [shift={(4.63,0.13)}] plot[domain=2.52:4.71,variable=\t]({1*3.53*cos(\t r)+0*3.53*sin(\t r)},{0*3.53*cos(\t r)+1*3.53*sin(\t r)});
\draw [shift={(-1.65,-2.76)}] plot[domain=-0.1:0.92,variable=\t]({1*6.3*cos(\t r)+0*6.3*sin(\t r)},{0*6.3*cos(\t r)+1*6.3*sin(\t r)});
\end{tikzpicture}
\hspace{20pt}
=
\begin{tikzpicture}[scale=0.4]
\draw [fill=black,fill opacity=1.0] (2,2) circle (0.3cm);
\draw [fill=black,fill opacity=1.0] (2,0) circle (0.3cm);
\draw [shift={(1.43,1)}] plot[domain=-1.06:1.06,variable=\t]({1*1.15*cos(\t r)+0*1.15*sin(\t r)},{0*1.15*cos(\t r)+1*1.15*sin(\t r)});
\draw [shift={(2.57,1)}] plot[domain=2.09:4.2,variable=\t]({1*1.15*cos(\t r)+0*1.15*sin(\t r)},{0*1.15*cos(\t r)+1*1.15*sin(\t r)});
\draw [shift={(3.11,0.07)}] plot[domain=-0.66:2.58,variable=\t]({-0.72*4.13*cos(\t r)+-0.69*2.74*sin(\t r)},{0.69*4.13*cos(\t r)+-0.72*2.74*sin(\t r)});
\draw [shift={(2.45,-1.07)}] plot[domain=4.11:6.2,variable=\t]({-0.16*4.65*cos(\t r)+-0.99*2.16*sin(\t r)},{0.99*4.65*cos(\t r)+-0.16*2.16*sin(\t r)});
\draw [shift={(-4.66,-4.06)}] plot[domain=0.07:0.74,variable=\t]({1*9.3*cos(\t r)+0*9.3*sin(\t r)},{0*9.3*cos(\t r)+1*9.3*sin(\t r)});
\end{tikzpicture}
\caption{The arc tangent to the left and the tagged loop belong to the same class of $m$-diagonals. The figure on the right is a $5$-angle (going twice through the same arc).}
\label{fig:oeil}
\end{figure}

\begin{lem}\label{lem:Dehntwist}
Let $\Delta$ be an $(m+2)$-angulation. Let $\{\alpha_1,\cdots,\alpha_r \}$ and $\{\beta_1,\cdots,\beta_r \}$ be two good sets of representatives for $\Delta$. Then there exists $\theta=(\theta_R,\theta_S)$ a Dehn twist on $L$ and $R$ such that, after exchanging some choices of tagged or untagged loops if necessary, for all $i$, $\theta (\beta_i)=\alpha_i$ up to homotopy. For any $i\in \{1,\cdots,r\}$, the arcs $\alpha_i$ and $\beta_i$ represent the same class of objects.
\end{lem}

\begin{rmk}\label{rmk:loop}
\cor{Sometimes, it is necessary to play with the representatives of the equivalence classes of $m$-diagonals. Here, for instance, to show the lemma, we are allowed to (and we will have to) replace an arc by a loop. Indeed, if there are two arcs of type $3$ ending at the same vertex, tangent on both sides of $L$ (respectively $R$), then replace the left tangent by a tagged loop. That does not change the $m$-diagonal, since both arcs belong to the same class of equivalence.}
\end{rmk}

\begin{proof}
We may assume that all the loops which appear in the good sets of representatives are tagged loops. We are only interested in the arcs attached to $L$. The case is similar for $R$.

Let $\alpha_i$ be an arc an extremity of which is a thick vertex of $L$. Let $\theta_i$ be a Dehn twist such that $\theta_i(\beta_i)=\alpha_i$. Let $\alpha_j$ be the neighbor of $\alpha_i$ according to $L$. Let $\theta_j$ such that $\theta_j(\beta_j)=\alpha_j$. Then the number of sides of $L$ between $\alpha_i$ and $\alpha_j$ is the same as the number of sides between $\beta_i$ and $\beta_j$. Indeed, the sets $\{\alpha_1,\cdots,\alpha_r \}$ and $\{\beta_1,\cdots,\beta_r \}$ are $(m+2)$-angulations, and the other extremity of $\alpha_i$ and $\beta_i$ (respectively $\alpha_j$ and $\beta_j$) is a vertex $a_i$ of $P$ (respectively $a_j$). Then to respect the notion of $(m+2)$-angles in an $(m+2)$-angulation, it means that if $k$ is the number of edges of $L$ between $\alpha_i$ and $\alpha_j$, $k$ is also the number of edges of $L$ between $\beta_i$ and $\beta_j$. Then $\theta_i=\theta_j$.

We can reproduce this process until we have treated all the arcs ending on $L$.
\end{proof}

We now introduce the twist which will be used in order to define the flip of an $(m+2)$-angulation at some $m$-diagonal.

\begin{rmk}
In a set of good representatives for an $(m+2)$-angulation, either there is a unique non-loop arc $\alpha$ of type $3$, incident with $L$ (respectively $R$), and a loop (tagged, if $\alpha$ is right tangent, untagged if $\alpha$ is left-tangent), or there are no loops.
\end{rmk}

\begin{defi}
Let $\Delta$ be an $(m+2)$-angulation and $\{\gamma_1,\cdots,\gamma_r\}$ a good set of representatives. Let $\alpha \in \{\gamma_1,\cdots,\gamma_r\}$, linking the vertices $a$ and $b$. If $\alpha$ is of type $3$, which means, linking a vertex $a$ to an inner polygon (say $L$, the issue is exactly the same for $R$) and surrounded by a loop (this loop can only appear at $a$, because there are no crossings), then replace this loop by the corresponding arc tangent to $L$ (left-tangent if it is tagged, right-tangent if it is untagged), and replace $\alpha$ by a loop.
\end{defi}

We are now in the good situation to define the twist of $\alpha$ in $\Delta$, whatever the type of $\alpha$.

\begin{defi}
Given a set of good representatives $\Delta$ of an $(m+2)$-angulation and $\alpha$ an $m$-diagonal in such a set, the twist of $\alpha$ is defined as follows:

\begin{itemize}
	\item Step 1: If $alpha$ is the unique $m$-diagonal in $\Delta$ incident with either $L$ or $R$ (in which case, we have an $(m+2)$-angle whose sides in $\Delta$ are $\alpha$ and a loop $\beta$), consider instead a different set of good representatives given by replacing $\alpha$ by a loop and $\beta$ by an arc incident with $L$ (or $R$), tagged if $\beta$ is left-tanget, untagged if not.
	\item Step 2: By orienting in the clockwise direction the two $(m+2)$-angles for which $\alpha$ is a side, let $\alpha_a$ and $\alpha_b$ be the sides of these $(m+2)$-angles which follow $\alpha$. The twist of $\alpha$ is then given by the concatenation $\alpha_a\alpha\alpha_b$.
\end{itemize}


If $m$ is even, then the side of the tangency does not change with the twist. If $m$ is odd, then change the side of the tangency of both arcs.

\end{defi}

\begin{rmk}
The question of parity is already present in the paper of Baur and Marsh on case $D$, in \cite[Remark 3.4]{BM02}. Remember that a left-tangent arc is associated with a tagged loop, and a right-tangent arc is associated with an untagged loop. Then in order to stay coherent with the tagged arcs in case $D$.
\end{rmk}

We note that, after applying the twist, the set of representatives is still a good one.


\begin{figure}[!h]
\centering
\begin{tikzpicture}[scale=0.8]
\draw (2.42,4.64)-- (3.84,3.26);
\draw (0.6,1.5)-- (2.3,-0.02);
\draw (2.42,4.64)-- (2.3,-0.02);
\draw (0.6,1.5)-- (3.84,3.26);
\begin{scriptsize}
\fill [color=black] (2.42,4.64) circle (1.5pt);
\draw[color=black] (2.58,4.9) node {$a$};
\fill [color=black] (2.3,-0.02) circle (1.5pt);
\draw[color=black] (2.46,0.24) node {$b$};
\draw[color=black] (2.5,3.84) node {$\alpha$};
\draw[color=black] (3.84,2.7) node {$\kappa\alpha$};
\end{scriptsize}
\end{tikzpicture}
\caption{Definition of a twist}
\label{fig:twist}
\end{figure}

\begin{defi}
Consider $\Delta$ an $(m+2)$-angulation. Choose a good set $\{ \gamma_1,\cdots,\gamma_r \}$ of representatives for $\Delta$. Let $\alpha$ be an arc in it. From Remark \ref{rmk:loop}, we may assume that $\alpha$ is a loop if necessary. The flip of the $(m+2)$-angulation at $\alpha$ is defined by $\mu_\alpha\Delta=\Delta \setminus \{ \alpha \} \cup \{ \alpha^* \}$ where $\alpha^*$ is given by $\kappa_{\Delta}(\alpha)$, the twist of $\alpha$.
\end{defi}

\begin{lem}
The definition of a flip is independent of the choice of good representatives.
\end{lem}

\begin{proof}
This result comes from the fact that the definition of a twist is independent of the choice of good representatives.

Here, we note that we choose a set of good representatives in which the arc corresponding to the $m$-diagonal $\alpha$ is a loop. As explained in Remark \ref{rmk:loop}, this is always possible.

Let $\{\alpha_1,\cdots,\alpha_r \}$ and $\{\beta_1,\cdots,\beta_r \}$ be two good sets of representatives for $\Delta$. There exists $\theta$ such that for all $i$, $\theta (\beta_i)=\alpha_i$, and $\theta$ commutes with the twist $\kappa_\Delta$.

In fact, applying the twist to $\alpha$ consists in "slipping" $\alpha$ clockwise. The Lemma \ref{lem:Dehntwist} says that "slipping along" a side of $L$ does not depend on the choice of representatives in a class under Dehn twist.
\end{proof}


\begin{rmk}
\begin{enumerate}
\item Note that the twist has an inverse, which consists in moving the arc counterclockwise. Then the flip is also invertible.
\item We note that the twist is not the same as the Dehn twist.
\item A flip does not change the number of $m$-diagonals in the $(m+2)$-angulation.
\end{enumerate}
\end{rmk}

\begin{figure}[!h]
\centering
\begin{tikzpicture}[scale=0.7]
\fill[fill=black,fill opacity=0.1] (0,4) -- (-1.66,3.46) -- (-2.69,2.05) -- (-2.68,0.3) -- (-1.66,-1.11) -- (0,-1.65) -- (1.66,-1.11) -- (2.69,0.3) -- (2.69,2.05) -- (1.66,3.46) -- cycle;
\draw [fill=black,fill opacity=1.0] (-1.66,1.57) circle (0.2cm);
\draw [fill=black,fill opacity=1.0] (1.66,1.58) circle (0.2cm);
\draw (0,4)-- (-1.66,3.46);
\draw (-1.66,3.46)-- (-2.69,2.05);
\draw (-2.69,2.05)-- (-2.68,0.3);
\draw (-2.68,0.3)-- (-1.66,-1.11);
\draw (-1.66,-1.11)-- (0,-1.65);
\draw (0,-1.65)-- (1.66,-1.11);
\draw (1.66,-1.11)-- (2.69,0.3);
\draw (2.69,0.3)-- (2.69,2.05);
\draw (2.69,2.05)-- (1.66,3.46);
\draw (1.66,3.46)-- (0,4);
\draw(-1.66,1.17) circle (0.4cm);
\draw(1.66,1.18) circle (0.4cm);
\draw [shift={(-7.95,-1.41)}] plot[domain=-0.03:0.66,variable=\t]({1*7.96*cos(\t r)+0*7.96*sin(\t r)},{0*7.96*cos(\t r)+1*7.96*sin(\t r)});
\draw [shift={(8.42,-1.56)},color=red]  plot[domain=2.5:3.15,variable=\t]({1*8.42*cos(\t r)+0*8.42*sin(\t r)},{0*8.42*cos(\t r)+1*8.42*sin(\t r)});
\draw [shift={(-1.34,-1.17)}] plot[domain=-1.77:1.07,variable=\t]({-0.15*3.42*cos(\t r)+-0.99*1.28*sin(\t r)},{0.99*3.42*cos(\t r)+-0.15*1.28*sin(\t r)});
\draw [shift={(1.34,-1.17)}] plot[domain=-1.77:1.07,variable=\t]({0.15*3.42*cos(\t r)+0.99*1.28*sin(\t r)},{0.99*3.42*cos(\t r)+-0.15*1.28*sin(\t r)});
\draw [shift={(1.47,-0.36)}] plot[domain=-0.32:1.37,variable=\t]({-0.91*3.99*cos(\t r)+-0.42*1.82*sin(\t r)},{0.42*3.99*cos(\t r)+-0.91*1.82*sin(\t r)});
\draw [shift={(-1.47,-0.36)}] plot[domain=-0.32:1.37,variable=\t]({0.9*3.99*cos(\t r)+0.43*1.82*sin(\t r)},{0.43*3.99*cos(\t r)+-0.9*1.82*sin(\t r)});
\draw [shift={(-0.37,0.1)}] plot[domain=2.86:5.31,variable=\t]({0.78*2.44*cos(\t r)+0.63*1.36*sin(\t r)},{-0.63*2.44*cos(\t r)+0.78*1.36*sin(\t r)});
\draw [shift={(-1.26,-1.78)}] plot[domain=-1.59:0.07,variable=\t]({-0.16*3.73*cos(\t r)+-0.99*1.26*sin(\t r)},{0.99*3.73*cos(\t r)+-0.16*1.26*sin(\t r)});
\draw [shift={(1.26,-1.78)}] plot[domain=-1.59:0.07,variable=\t]({0.16*3.73*cos(\t r)+0.99*1.26*sin(\t r)},{0.99*3.73*cos(\t r)+-0.16*1.26*sin(\t r)});
\draw [shift={(0.37,0.1)}] plot[domain=2.86:5.31,variable=\t]({-0.78*2.44*cos(\t r)+-0.63*1.36*sin(\t r)},{-0.63*2.44*cos(\t r)+0.78*1.36*sin(\t r)});
\end{tikzpicture}
$\to$
\begin{tikzpicture}[scale=0.7]
\fill[fill=black,fill opacity=0.1] (0,4) -- (-1.66,3.46) -- (-2.69,2.05) -- (-2.68,0.3) -- (-1.66,-1.11) -- (0,-1.65) -- (1.66,-1.11) -- (2.69,0.3) -- (2.69,2.05) -- (1.66,3.46) -- cycle;
\draw [fill=black,fill opacity=1.0] (-1.66,1.57) circle (0.2cm);
\draw [fill=black,fill opacity=1.0] (1.66,1.58) circle (0.2cm);
\draw (0,4)-- (-1.66,3.46);
\draw (-1.66,3.46)-- (-2.69,2.05);
\draw (-2.69,2.05)-- (-2.68,0.3);
\draw (-2.68,0.3)-- (-1.66,-1.11);
\draw (-1.66,-1.11)-- (0,-1.65);
\draw (0,-1.65)-- (1.66,-1.11);
\draw (1.66,-1.11)-- (2.69,0.3);
\draw (2.69,0.3)-- (2.69,2.05);
\draw (2.69,2.05)-- (1.66,3.46);
\draw (1.66,3.46)-- (0,4);
\draw(-1.66,1.17) circle (0.4cm);
\draw(1.66,1.18) circle (0.4cm);
\draw [shift={(-7.95,-1.41)}] plot[domain=-0.03:0.66,variable=\t]({1*7.96*cos(\t r)+0*7.96*sin(\t r)},{0*7.96*cos(\t r)+1*7.96*sin(\t r)});
\draw [shift={(-1.34,-1.17)}] plot[domain=-1.77:1.07,variable=\t]({-0.15*3.42*cos(\t r)+-0.99*1.28*sin(\t r)},{0.99*3.42*cos(\t r)+-0.15*1.28*sin(\t r)});
\draw [shift={(1.34,-1.17)}] plot[domain=-1.77:1.07,variable=\t]({0.15*3.42*cos(\t r)+0.99*1.28*sin(\t r)},{0.99*3.42*cos(\t r)+-0.15*1.28*sin(\t r)});
\draw [shift={(1.47,-0.36)}] plot[domain=-0.32:1.37,variable=\t]({-0.91*3.99*cos(\t r)+-0.42*1.82*sin(\t r)},{0.42*3.99*cos(\t r)+-0.91*1.82*sin(\t r)});
\draw [shift={(-1.47,-0.36)}] plot[domain=-0.32:1.37,variable=\t]({0.9*3.99*cos(\t r)+0.43*1.82*sin(\t r)},{0.43*3.99*cos(\t r)+-0.9*1.82*sin(\t r)});
\draw [shift={(-0.37,0.1)}] plot[domain=2.86:5.31,variable=\t]({0.78*2.44*cos(\t r)+0.63*1.36*sin(\t r)},{-0.63*2.44*cos(\t r)+0.78*1.36*sin(\t r)});
\draw [shift={(-1.26,-1.78)}] plot[domain=-1.59:0.07,variable=\t]({-0.16*3.73*cos(\t r)+-0.99*1.26*sin(\t r)},{0.99*3.73*cos(\t r)+-0.16*1.26*sin(\t r)});
\draw [shift={(1.26,-1.78)}] plot[domain=-1.59:0.07,variable=\t]({0.16*3.73*cos(\t r)+0.99*1.26*sin(\t r)},{0.99*3.73*cos(\t r)+-0.16*1.26*sin(\t r)});
\draw [shift={(0.37,0.1)}] plot[domain=2.86:5.31,variable=\t]({-0.78*2.44*cos(\t r)+-0.63*1.36*sin(\t r)},{-0.63*2.44*cos(\t r)+0.78*1.36*sin(\t r)});
\draw [shift={(-0.54,-0.5)},color=red]  plot[domain=0.67:1.85,variable=\t]({1*4.12*cos(\t r)+0*4.12*sin(\t r)},{0*4.12*cos(\t r)+1*4.12*sin(\t r)});
\end{tikzpicture}
\caption{First example of a flip}
\label{fig:flip}
\end{figure}

\begin{figure}[!h]
\centering
\begin{tikzpicture}[scale=0.25]
\fill[fill=black,fill opacity=0.15] (0,4) -- (0,-4) -- (6.93,-8) -- (13.86,-4) -- (13.86,4) -- (6.93,8) -- cycle;
\draw [fill=black,fill opacity=1.0] (3.46,2) circle (0.4cm);
\draw [fill=black,fill opacity=1.0] (3.46,-2) circle (0.4cm);
\draw [fill=black,fill opacity=1.0] (10.39,2) circle (0.4cm);
\draw [fill=black,fill opacity=1.0] (10.39,-2) circle (0.4cm);
\draw (0,4)-- (0,-4);
\draw (0,-4)-- (6.93,-8);
\draw (6.93,-8)-- (13.86,-4);
\draw (13.86,-4)-- (13.86,4);
\draw (13.86,4)-- (6.93,8);
\draw (6.93,8)-- (0,4);
\draw [shift={(1.86,0)}] plot[domain=-0.89:0.89,variable=\t]({1*2.57*cos(\t r)+0*2.57*sin(\t r)},{0*2.57*cos(\t r)+1*2.57*sin(\t r)});
\draw [shift={(5.07,0)}] plot[domain=2.25:4.04,variable=\t]({1*2.57*cos(\t r)+0*2.57*sin(\t r)},{0*2.57*cos(\t r)+1*2.57*sin(\t r)});
\draw [shift={(12,0)}] plot[domain=2.25:4.04,variable=\t]({1*2.57*cos(\t r)+0*2.57*sin(\t r)},{0*2.57*cos(\t r)+1*2.57*sin(\t r)});
\draw [shift={(8.79,0)}] plot[domain=-0.89:0.89,variable=\t]({1*2.57*cos(\t r)+0*2.57*sin(\t r)},{0*2.57*cos(\t r)+1*2.57*sin(\t r)});
\draw (6.93,8)-- (6.93,-8);
\draw [shift={(8.18,-3.66)}] plot[domain=-0.47:1.77,variable=\t]({-0.8*8.37*cos(\t r)+-0.6*4.29*sin(\t r)},{0.6*8.37*cos(\t r)+-0.8*4.29*sin(\t r)});
\draw [shift={(10.26,-4.74)}]  plot[domain=-2.03:0.34,variable=\t]({0.14*8.29*cos(\t r)+0.99*3.18*sin(\t r)},{0.99*8.29*cos(\t r)+-0.14*3.18*sin(\t r)});
\draw [shift={(-12.68,-5.94)}]  plot[domain=-0.17:0.82,variable=\t]({0.96*26.9*cos(\t r)+0.27*9.94*sin(\t r)},{0.27*26.9*cos(\t r)+-0.96*9.94*sin(\t r)});
\draw [shift={(25.48,-8.09)}] plot[domain=2.55:3.14,variable=\t]({1*18.55*cos(\t r)+0*18.55*sin(\t r)},{0*18.55*cos(\t r)+1*18.55*sin(\t r)});
\draw [shift={(-1.43,8.3)},color=red]  plot[domain=5.39:6.25,variable=\t]({1*8.36*cos(\t r)+0*8.36*sin(\t r)},{0*8.36*cos(\t r)+1*8.36*sin(\t r)});
\end{tikzpicture}
$\to$
\begin{tikzpicture}[scale=0.25]
\fill[fill=black,fill opacity=0.15] (0,4) -- (0,-4) -- (6.93,-8) -- (13.86,-4) -- (13.86,4) -- (6.93,8) -- cycle;
\draw [fill=black,fill opacity=1.0] (3.46,2) circle (0.4cm);
\draw [fill=black,fill opacity=1.0] (3.46,-2) circle (0.4cm);
\draw [fill=black,fill opacity=1.0] (10.39,2) circle (0.4cm);
\draw [fill=black,fill opacity=1.0] (10.39,-2) circle (0.4cm);
\draw (0,4)-- (0,-4);
\draw (0,-4)-- (6.93,-8);
\draw (6.93,-8)-- (13.86,-4);
\draw (13.86,-4)-- (13.86,4);
\draw (13.86,4)-- (6.93,8);
\draw (6.93,8)-- (0,4);
\draw [shift={(1.86,0)}] plot[domain=-0.89:0.89,variable=\t]({1*2.57*cos(\t r)+0*2.57*sin(\t r)},{0*2.57*cos(\t r)+1*2.57*sin(\t r)});
\draw [shift={(5.07,0)}] plot[domain=2.25:4.04,variable=\t]({1*2.57*cos(\t r)+0*2.57*sin(\t r)},{0*2.57*cos(\t r)+1*2.57*sin(\t r)});
\draw [shift={(12,0)}] plot[domain=2.25:4.04,variable=\t]({1*2.57*cos(\t r)+0*2.57*sin(\t r)},{0*2.57*cos(\t r)+1*2.57*sin(\t r)});
\draw [shift={(8.79,0)}] plot[domain=-0.89:0.89,variable=\t]({1*2.57*cos(\t r)+0*2.57*sin(\t r)},{0*2.57*cos(\t r)+1*2.57*sin(\t r)});
\draw (6.93,8)-- (6.93,-8);
\draw [shift={(7.79,-3.41)}] plot[domain=-0.38:1.88,variable=\t]({-0.76*7.65*cos(\t r)+-0.65*4.25*sin(\t r)},{0.65*7.65*cos(\t r)+-0.76*4.25*sin(\t r)});
\draw [shift={(10.26,-4.74)}]  plot[domain=-2.03:0.34,variable=\t]({0.14*8.29*cos(\t r)+0.99*3.18*sin(\t r)},{0.99*8.29*cos(\t r)+-0.14*3.18*sin(\t r)});
\draw [shift={(-12.68,-5.94)}]  plot[domain=-0.17:0.82,variable=\t]({0.96*26.9*cos(\t r)+0.27*9.94*sin(\t r)},{0.27*26.9*cos(\t r)+-0.96*9.94*sin(\t r)});
\draw [shift={(25.48,-8.09)}] plot[domain=2.55:3.14,variable=\t]({1*18.55*cos(\t r)+0*18.55*sin(\t r)},{0*18.55*cos(\t r)+1*18.55*sin(\t r)});
\draw [shift={(2.91,-8.25)},color=red]  plot[domain=1.6:3.86,variable=\t]({-0.03*12.22*cos(\t r)+1*4.01*sin(\t r)},{-1*12.22*cos(\t r)+-0.03*4.01*sin(\t r)});
\draw [shift={(24.63,-2.76)},color=red]  plot[domain=-0.1:0.8,variable=\t]({-0.99*24.47*cos(\t r)+-0.12*10.07*sin(\t r)},{0.12*24.47*cos(\t r)+-0.99*10.07*sin(\t r)});
\end{tikzpicture}
\caption{Example of a flip with a loop}
\label{fig:fliploop}
\end{figure}

\begin{figure}[!h]
	\centering
\begin{tikzpicture}[scale=0.2]

	\fill[line width=0.8pt,fill=black,fill opacity=0.15] (0,4) -- (0,-4) -- (6.92820323027551,-8) -- (13.856406460551021,-4) -- (13.856406460551023,4) -- (6.928203230275516,8) -- cycle;

	\draw [line width=0.8pt,fill=black,fill opacity=1] (3.4641016151377566,2) circle (0.4cm);

	\draw [line width=0.8pt,fill=black,fill opacity=1] (3.4641016151377553,-2) circle (0.4cm);

	\draw [line width=0.8pt,fill=black,fill opacity=1] (10.392304845413266,2) circle (0.4cm);

	\draw [line width=0.8pt,fill=black,fill opacity=1] (10.392304845413266,-2) circle (0.4cm);

	\draw [line width=0.8pt] (0,4)-- (0,-4);

	\draw [line width=0.8pt] (0,-4)-- (6.92820323027551,-8);

	\draw [line width=0.8pt] (6.92820323027551,-8)-- (13.856406460551021,-4);

	\draw [line width=0.8pt] (13.856406460551021,-4)-- (13.856406460551023,4);

	\draw [line width=0.8pt] (13.856406460551023,4)-- (6.928203230275516,8);

	\draw [line width=0.8pt] (6.928203230275516,8)-- (0,4);

	\draw [shift={(1.8568379886224946,0)},line width=0.8pt]  plot[domain=-0.893844784234993:0.8938447842349926,variable=\t]({1*2.565793515682622*cos(\t r)+0*2.565793515682622*sin(\t r)},{0*2.565793515682622*cos(\t r)+1*2.565793515682622*sin(\t r)});

	\draw [shift={(5.07136524165302,0)},line width=0.8pt]  plot[domain=2.247747869354801:4.035437437824785,variable=\t]({1*2.5657935156826235*cos(\t r)+0*2.5657935156826235*sin(\t r)},{0*2.5657935156826235*cos(\t r)+1*2.5657935156826235*sin(\t r)});

	\draw [shift={(11.999568471928534,0)},line width=0.8pt]  plot[domain=2.2477478693548:4.0354374378247835,variable=\t]({1*2.565793515682625*cos(\t r)+0*2.565793515682625*sin(\t r)},{0*2.565793515682625*cos(\t r)+1*2.565793515682625*sin(\t r)});

	\draw [shift={(8.785041218898007,0)},line width=0.8pt]  plot[domain=-0.8938447842349948:0.8938447842349915,variable=\t]({1*2.565793515682622*cos(\t r)+0*2.565793515682622*sin(\t r)},{0*2.565793515682622*cos(\t r)+1*2.565793515682622*sin(\t r)});

	\draw [line width=0.8pt] (6.928203230275516,8)-- (6.928203230275509,-8);

	\draw [shift={(-0.37771498252856145,-28.18004286379138)},line width=0.8pt]  plot[domain=2.2958653215886624:3.3623533989748404,variable=\t]({-0.11915351806269256*31.52474382946886*cos(\t r)+0.9928758427584405*6.479154038130095*sin(\t r)},{-0.9928758427584405*31.52474382946886*cos(\t r)+-0.11915351806269256*6.479154038130095*sin(\t r)});

	\draw [shift={(37.504080123626295,-7.09158114422124)},line width=0.8pt]  plot[domain=-0.24068533539115844:0.6646270440140595,variable=\t]({-0.9814400026731381*37.901671206601115*cos(\t r)+-0.1917694479132438*10.952422745890962*sin(\t r)},{0.1917694479132438*37.901671206601115*cos(\t r)+-0.9814400026731381*10.952422745890962*sin(\t r)});

	\draw [shift={(14.234121443079585,28.18004286379138)},line width=0.8pt]  plot[domain=2.2958653215886624:3.3623533989748404,variable=\t]({0.11915351806269256*31.52474382946886*cos(\t r)+-0.9928758427584405*6.479154038130095*sin(\t r)},{0.9928758427584405*31.52474382946886*cos(\t r)+0.11915351806269256*6.479154038130095*sin(\t r)});

	\draw [shift={(-23.64767366307527,7.09158114422124)},line width=0.8pt]  plot[domain=-0.24068533539115844:0.6646270440140595,variable=\t]({0.9814400026731381*37.901671206601115*cos(\t r)+0.1917694479132438*10.952422745890962*sin(\t r)},{-0.1917694479132438*37.901671206601115*cos(\t r)+0.9814400026731381*10.952422745890962*sin(\t r)});

	\draw [shift={(-19.90284480639907,-10.703366047769393)},line width=0.8pt,color=red]  plot[domain=0.10041626153753302:0.4922859138044798,variable=\t]({1*26.96689316058806*cos(\t r)+0*26.96689316058806*sin(\t r)},{0*26.96689316058806*cos(\t r)+1*26.96689316058806*sin(\t r)});

	\draw [shift={(33.75925126695005,10.703366047769382)},line width=0.8pt]  plot[domain=3.2420089151273257:3.633878567394273,variable=\t]({1*26.966893160588018*cos(\t r)+0*26.966893160588018*sin(\t r)},{0*26.966893160588018*cos(\t r)+1*26.966893160588018*sin(\t r)});

\end{tikzpicture}
Step 1
\begin{tikzpicture}[scale=0.2]

	\fill[line width=0.8pt,fill=black,fill opacity=0.15] (0,4) -- (0,-4) -- (6.92820323027551,-8) -- (13.856406460551021,-4) -- (13.856406460551023,4) -- (6.928203230275516,8) -- cycle;

	\draw [line width=0.8pt,fill=black,fill opacity=1] (3.4641016151377566,2) circle (0.4cm);

	\draw [line width=0.8pt,fill=black,fill opacity=1] (3.4641016151377553,-2) circle (0.4cm);

	\draw [line width=0.8pt,fill=black,fill opacity=1] (10.392304845413266,2) circle (0.4cm);

	\draw [line width=0.8pt,fill=black,fill opacity=1] (10.392304845413266,-2) circle (0.4cm);

	\draw [line width=0.8pt] (0,4)-- (0,-4);

	\draw [line width=0.8pt] (0,-4)-- (6.92820323027551,-8);

	\draw [line width=0.8pt] (6.92820323027551,-8)-- (13.856406460551021,-4);

	\draw [line width=0.8pt] (13.856406460551021,-4)-- (13.856406460551023,4);

	\draw [line width=0.8pt] (13.856406460551023,4)-- (6.928203230275516,8);

	\draw [line width=0.8pt] (6.928203230275516,8)-- (0,4);

	\draw [shift={(1.8568379886224946,0)},line width=0.8pt]  plot[domain=-0.893844784234993:0.8938447842349926,variable=\t]({1*2.565793515682622*cos(\t r)+0*2.565793515682622*sin(\t r)},{0*2.565793515682622*cos(\t r)+1*2.565793515682622*sin(\t r)});

	\draw [shift={(5.07136524165302,0)},line width=0.8pt]  plot[domain=2.247747869354801:4.035437437824785,variable=\t]({1*2.5657935156826235*cos(\t r)+0*2.5657935156826235*sin(\t r)},{0*2.5657935156826235*cos(\t r)+1*2.5657935156826235*sin(\t r)});

	\draw [shift={(11.999568471928534,0)},line width=0.8pt]  plot[domain=2.2477478693548:4.0354374378247835,variable=\t]({1*2.565793515682625*cos(\t r)+0*2.565793515682625*sin(\t r)},{0*2.565793515682625*cos(\t r)+1*2.565793515682625*sin(\t r)});

	\draw [shift={(8.785041218898007,0)},line width=0.8pt]  plot[domain=-0.8938447842349948:0.8938447842349915,variable=\t]({1*2.565793515682622*cos(\t r)+0*2.565793515682622*sin(\t r)},{0*2.565793515682622*cos(\t r)+1*2.565793515682622*sin(\t r)});

	\draw [line width=0.8pt] (6.928203230275516,8)-- (6.928203230275509,-8);

	\draw [shift={(-0.37771498252856145,-28.18004286379138)},line width=0.8pt,color=red]  plot[domain=2.2958653215886624:3.3623533989748404,variable=\t]({-0.11915351806269256*31.52474382946886*cos(\t r)+0.9928758427584405*6.479154038130095*sin(\t r)},{-0.9928758427584405*31.52474382946886*cos(\t r)+-0.11915351806269256*6.479154038130095*sin(\t r)});

	\draw [shift={(37.504080123626295,-7.09158114422124)},line width=0.8pt,color=red]  plot[domain=-0.24068533539115844:0.6646270440140595,variable=\t]({-0.9814400026731381*37.901671206601115*cos(\t r)+-0.1917694479132438*10.952422745890962*sin(\t r)},{0.1917694479132438*37.901671206601115*cos(\t r)+-0.9814400026731381*10.952422745890962*sin(\t r)});

	\draw [shift={(14.234121443079585,28.18004286379138)},line width=0.8pt]  plot[domain=2.2958653215886624:3.3623533989748404,variable=\t]({0.11915351806269256*31.52474382946886*cos(\t r)+-0.9928758427584405*6.479154038130095*sin(\t r)},{0.9928758427584405*31.52474382946886*cos(\t r)+0.11915351806269256*6.479154038130095*sin(\t r)});

	\draw [shift={(-23.64767366307527,7.09158114422124)},line width=0.8pt]  plot[domain=-0.24068533539115844:0.6646270440140595,variable=\t]({0.9814400026731381*37.901671206601115*cos(\t r)+0.1917694479132438*10.952422745890962*sin(\t r)},{-0.1917694479132438*37.901671206601115*cos(\t r)+0.9814400026731381*10.952422745890962*sin(\t r)});

	\draw [shift={(33.75925126695005,10.703366047769382)},line width=0.8pt]  plot[domain=3.2420089151273257:3.633878567394273,variable=\t]({1*26.966893160588018*cos(\t r)+0*26.966893160588018*sin(\t r)},{0*26.966893160588018*cos(\t r)+1*26.966893160588018*sin(\t r)});

	\draw [shift={(11.73276134772133,-0.6584689243902758)},line width=0.8pt]  plot[domain=3.3079930588264648:4.132902017012776,variable=\t]({1*8.773930546685326*cos(\t r)+0*8.773930546685326*sin(\t r)},{0*8.773930546685326*cos(\t r)+1*8.773930546685326*sin(\t r)});

\end{tikzpicture}
Step 2
\begin{tikzpicture}[scale=0.2]

	\fill[line width=0.8pt,fill=black,fill opacity=0.15] (0,4) -- (0,-4) -- (6.92820323027551,-8) -- (13.856406460551021,-4) -- (13.856406460551023,4) -- (6.928203230275516,8) -- cycle;

	\draw [line width=0.8pt,fill=black,fill opacity=1] (3.4641016151377566,2) circle (0.4cm);

	\draw [line width=0.8pt,fill=black,fill opacity=1] (3.4641016151377553,-2) circle (0.4cm);

	\draw [line width=0.8pt,fill=black,fill opacity=1] (10.392304845413266,2) circle (0.4cm);

	\draw [line width=0.8pt,fill=black,fill opacity=1] (10.392304845413266,-2) circle (0.4cm);

	\draw [line width=0.8pt] (0,4)-- (0,-4);

	\draw [line width=0.8pt] (0,-4)-- (6.92820323027551,-8);

	\draw [line width=0.8pt] (6.92820323027551,-8)-- (13.856406460551021,-4);

	\draw [line width=0.8pt] (13.856406460551021,-4)-- (13.856406460551023,4);

	\draw [line width=0.8pt] (13.856406460551023,4)-- (6.928203230275516,8);

	\draw [line width=0.8pt] (6.928203230275516,8)-- (0,4);

	\draw [shift={(1.8568379886224946,0)},line width=0.8pt]  plot[domain=-0.893844784234993:0.8938447842349926,variable=\t]({1*2.565793515682622*cos(\t r)+0*2.565793515682622*sin(\t r)},{0*2.565793515682622*cos(\t r)+1*2.565793515682622*sin(\t r)});

	\draw [shift={(5.07136524165302,0)},line width=0.8pt]  plot[domain=2.247747869354801:4.035437437824785,variable=\t]({1*2.5657935156826235*cos(\t r)+0*2.5657935156826235*sin(\t r)},{0*2.5657935156826235*cos(\t r)+1*2.5657935156826235*sin(\t r)});

	\draw [shift={(11.999568471928534,0)},line width=0.8pt]  plot[domain=2.2477478693548:4.0354374378247835,variable=\t]({1*2.565793515682625*cos(\t r)+0*2.565793515682625*sin(\t r)},{0*2.565793515682625*cos(\t r)+1*2.565793515682625*sin(\t r)});

	\draw [shift={(8.785041218898007,0)},line width=0.8pt]  plot[domain=-0.8938447842349948:0.8938447842349915,variable=\t]({1*2.565793515682622*cos(\t r)+0*2.565793515682622*sin(\t r)},{0*2.565793515682622*cos(\t r)+1*2.565793515682622*sin(\t r)});

	\draw [line width=0.8pt] (6.928203230275516,8)-- (6.928203230275509,-8);

	\draw [shift={(14.234121443079585,28.18004286379138)},line width=0.8pt]  plot[domain=2.2958653215886624:3.3623533989748404,variable=\t]({0.11915351806269256*31.52474382946886*cos(\t r)+-0.9928758427584405*6.479154038130095*sin(\t r)},{0.9928758427584405*31.52474382946886*cos(\t r)+0.11915351806269256*6.479154038130095*sin(\t r)});

	\draw [shift={(-23.64767366307527,7.09158114422124)},line width=0.8pt]  plot[domain=-0.24068533539115844:0.6646270440140595,variable=\t]({0.9814400026731381*37.901671206601115*cos(\t r)+0.1917694479132438*10.952422745890962*sin(\t r)},{-0.1917694479132438*37.901671206601115*cos(\t r)+0.9814400026731381*10.952422745890962*sin(\t r)});

	\draw [shift={(33.75925126695005,10.703366047769382)},line width=0.8pt]  plot[domain=3.2420089151273257:3.633878567394273,variable=\t]({1*26.966893160588018*cos(\t r)+0*26.966893160588018*sin(\t r)},{0*26.966893160588018*cos(\t r)+1*26.966893160588018*sin(\t r)});

	\draw [shift={(11.73276134772133,-0.6584689243902758)},line width=0.8pt]  plot[domain=3.3079930588264648:4.132902017012776,variable=\t]({1*8.773930546685326*cos(\t r)+0*8.773930546685326*sin(\t r)},{0*8.773930546685326*cos(\t r)+1*8.773930546685326*sin(\t r)});

	\draw [shift={(2.5668374809859555,-0.8985327002734569)},line width=0.8pt,color=red]  plot[domain=1.4236690846235973:5.80651777070688,variable=\t]({-0.4745796509628384*4.443433979113369*cos(\t r)+0.8802125623347977*1.71628974295355*sin(\t r)},{-0.8802125623347977*4.443433979113369*cos(\t r)+-0.4745796509628384*1.71628974295355*sin(\t r)});

\end{tikzpicture}
	\caption{Example of a flip in which step one in the definition of the twist is necessary}
	\label{fig:fliploop2}
\end{figure}

The twist is defined on the $m$-diagonals, which is a class of equivalence under Dehn twist, and the flip is defined on an $(m+2)$-angulation, which is a set of $m$-diagonal. Nonetheless, we will sometimes apply them on arcs, implicitly assuming that we have chosen a good set of representatives.

\begin{defi}\label{def:triinit}
	For any $m$ and $n \geq 4$, we define the initial $(m+2)$-angulation in the following way: From the vertex $1$ of $P$, we draw:
	\begin{enumerate}
		\item \cor{An arc going to $L$, tangent to one of its thick vertices (which one does not matter, thanks to the equivalence under Dehn twist).
			\item A tagged loop around $L$.
			\item An arc going to $L$, tangent to one of its thick vertices (which one does not matter, thanks to the equivalence under Dehn twist).
			\item A tagged loop around $S$.
			\item The arcs of type $1$ connect vertex $1$ wth vertex $km+1$, with $1 \leq k \leq n-3$.}
	\end{enumerate}
%

\end{defi}

It is obvious that the arcs drawn from the previous definition form a good set of representatives for the initial $(m+2)$-angulation. See figure \ref{fig:trianginit} for an example of initial $4$-angulation.

Note that we do not have to label the loops, since their labels are automatically deduced from the other arc ending in the corresponding inner polygon.

\begin{figure}[!h]
\centering
\begin{tikzpicture}[scale=0.7]
\fill[fill=black,fill opacity=0.1] (0,4) -- (-1.66,3.46) -- (-2.69,2.05) -- (-2.68,0.3) -- (-1.66,-1.11) -- (0,-1.65) -- (1.66,-1.11) -- (2.69,0.3) -- (2.69,2.05) -- (1.66,3.46) -- cycle;
\draw [fill=black,fill opacity=1.0] (-1.66,1.57) circle (0.2cm);
\draw [fill=black,fill opacity=1.0] (1.66,1.58) circle (0.2cm);
\draw (0,4)-- (-1.66,3.46);
\draw (-1.66,3.46)-- (-2.69,2.05);
\draw (-2.69,2.05)-- (-2.68,0.3);
\draw (-2.68,0.3)-- (-1.66,-1.11);
\draw (-1.66,-1.11)-- (0,-1.65);
\draw (0,-1.65)-- (1.66,-1.11);
\draw (1.66,-1.11)-- (2.69,0.3);
\draw (2.69,0.3)-- (2.69,2.05);
\draw (2.69,2.05)-- (1.66,3.46);
\draw (1.66,3.46)-- (0,4);
\draw(-1.66,1.17) circle (0.4cm);
\draw(1.66,1.18) circle (0.4cm);
\draw [shift={(1.95,-3.35)}] plot[domain=-0.11:1.3,variable=\t]({-0.07*6.83*cos(\t r)+-1*1.89*sin(\t r)},{1*6.83*cos(\t r)+-0.07*1.89*sin(\t r)});
\draw [shift={(1.44,-3.05)}] plot[domain=2.2:4.47,variable=\t]({-0.01*5.75*cos(\t r)+1*1.5*sin(\t r)},{-1*5.75*cos(\t r)+-0.01*1.5*sin(\t r)});
\draw [shift={(-0.65,-0.55)}] plot[domain=1.52:3.46,variable=\t]({-0.79*3.42*cos(\t r)+0.62*1.27*sin(\t r)},{-0.62*3.42*cos(\t r)+-0.79*1.27*sin(\t r)});
\draw [shift={(0.3,0.15)}] plot[domain=0.96:3.44,variable=\t]({-0.77*2.41*cos(\t r)+0.64*1.46*sin(\t r)},{-0.64*2.41*cos(\t r)+-0.77*1.46*sin(\t r)});
\draw [shift={(0.99,-1.47)}] plot[domain=3.09:4.82,variable=\t]({-0.21*3.53*cos(\t r)+0.98*0.93*sin(\t r)},{-0.98*3.53*cos(\t r)+-0.21*0.93*sin(\t r)});
\draw [shift={(-1.94,-3.36)}] plot[domain=-0.11:1.3,variable=\t]({0.07*6.83*cos(\t r)+1*1.89*sin(\t r)},{1*6.83*cos(\t r)+-0.07*1.89*sin(\t r)});	\draw [shift={(-1.43,-3.05)}] plot[domain=2.2:4.47,variable=\t]({0.01*5.75*cos(\t r)+-1*1.5*sin(\t r)},{-1*5.75*cos(\t r)+-0.01*1.5*sin(\t r)});
\draw [shift={(-0.98,-1.47)}] plot[domain=3.09:4.82,variable=\t]({0.22*3.53*cos(\t r)+-0.98*0.93*sin(\t r)},{-0.98*3.53*cos(\t r)+-0.22*0.93*sin(\t r)});
\draw [shift={(-0.29,0.15)}] plot[domain=0.96:3.44,variable=\t]({0.77*2.41*cos(\t r)+-0.64*1.46*sin(\t r)},{-0.64*2.41*cos(\t r)+-0.77*1.46*sin(\t r)});
\draw [shift={(-7.41,-2.97)}] plot[domain=0.18:0.67,variable=\t]({1*7.53*cos(\t r)+0*7.53*sin(\t r)},{0*7.53*cos(\t r)+1*7.53*sin(\t r)});
\end{tikzpicture}
\caption{The initial $4$-angulation for $m=2$ and $n=7$.}
\label{fig:trianginit}
\end{figure}

\begin{lem}
Every $(m+2)$-angulation can be obtained from the initial $(m+2)$-angulation by a finite sequence of flips.
\end{lem}

\begin{proof}
\cor{We know that there are at least five $m$-diagonals (as $n \geq 4$). We are going to proceed by reduction. The first aim is to hang to each inner polygon $R$ and $L$, exactly one $m$-diagonal and one loop to reduce to a well-known case. We treat the case of $L$, the case of $R$ being similar.}

\cor{If there is a loop, then, do nothing. Else, there cannot be less than two arcs ending at $L$. Indeed, we can argue that if there was only one arc ending at $L$ (or no arc at all), then the set of noncrossing $m$-diagonals is not maximal, since we can add a loop for example, which would not cross any of the present arcs.}

\cor{If there are more than two arcs ending at $L$, first remark that there is no loop (else, this set of $m$-diagonals would not give an $(m+2)$-angulation). Then note that there cannot be two arcs ending at the same vertex of $P$. Else, we could replace one of them by a loop, and this is impossible. Choose one thick vertex $L_i$ where one arc, named $\alpha$ ends. There exists (at least) one arc $\beta$ ending at $L_{i+k}$. Flip it $k$ times in order for it to be hung to $L_i$. Apply a finite sequence of flips to the eventual other arcs hung to $L_j$, $j \neq i$, so as they become of type $1$ or $2$ by slipping along $\alpha$ or $\beta$.}


We are then we are reduced to the case of figure \ref{fig:4arcs} (which is not completed by the other arcs of the $(m+2)$-angulation).



\begin{figure}[!h]
\centering
\begin{tikzpicture}[scale=0.2]
\fill[fill=black,fill opacity=0.15] (0,4) -- (0,-4) -- (6.93,-8) -- (13.86,-4) -- (13.86,4) -- (6.93,8) -- cycle;
\draw [fill=black,fill opacity=1.0] (3.46,2) circle (0.4cm);
\draw [fill=black,fill opacity=1.0] (3.46,-2) circle (0.4cm);
\draw [fill=black,fill opacity=1.0] (10.39,2) circle (0.4cm);
\draw [fill=black,fill opacity=1.0] (10.39,-2) circle (0.4cm);
\draw (0,4)-- (0,-4);
\draw (0,-4)-- (6.93,-8);
\draw (6.93,-8)-- (13.86,-4);
\draw (13.86,-4)-- (13.86,4);
\draw (13.86,4)-- (6.93,8);
\draw (6.93,8)-- (0,4);
\draw [shift={(1.86,0)}] plot[domain=-0.89:0.89,variable=\t]({1*2.57*cos(\t r)+0*2.57*sin(\t r)},{0*2.57*cos(\t r)+1*2.57*sin(\t r)});
\draw [shift={(5.07,0)}] plot[domain=2.25:4.04,variable=\t]({1*2.57*cos(\t r)+0*2.57*sin(\t r)},{0*2.57*cos(\t r)+1*2.57*sin(\t r)});
\draw [shift={(12,0)}] plot[domain=2.25:4.04,variable=\t]({1*2.57*cos(\t r)+0*2.57*sin(\t r)},{0*2.57*cos(\t r)+1*2.57*sin(\t r)});
\draw [shift={(8.79,0)}] plot[domain=-0.89:0.89,variable=\t]({1*2.57*cos(\t r)+0*2.57*sin(\t r)},{0*2.57*cos(\t r)+1*2.57*sin(\t r)});
\draw [shift={(-11.62,-8.09)}] plot[domain=0:0.59,variable=\t]({1*18.55*cos(\t r)+0*18.55*sin(\t r)},{0*18.55*cos(\t r)+1*18.55*sin(\t r)});
\draw [shift={(25.48,-8.09)}] plot[domain=2.55:3.14,variable=\t]({1*18.55*cos(\t r)+0*18.55*sin(\t r)},{0*18.55*cos(\t r)+1*18.55*sin(\t r)});
\draw [shift={(26.53,-5.94)}] plot[domain=-0.17:0.82,variable=\t]({-0.96*26.9*cos(\t r)+-0.27*9.94*sin(\t r)},{0.27*26.9*cos(\t r)+-0.96*9.94*sin(\t r)});
\draw [shift={(3.6,-4.74)}] plot[domain=-2.03:0.34,variable=\t]({-0.14*8.29*cos(\t r)+-0.99*3.18*sin(\t r)},{0.99*8.29*cos(\t r)+-0.14*3.18*sin(\t r)});
\draw [shift={(10.26,-4.74)}] plot[domain=-2.03:0.34,variable=\t]({0.14*8.29*cos(\t r)+0.99*3.18*sin(\t r)},{0.99*8.29*cos(\t r)+-0.14*3.18*sin(\t r)});
\draw [shift={(-12.68,-5.94)}] plot[domain=-0.17:0.82,variable=\t]({0.96*26.9*cos(\t r)+0.27*9.94*sin(\t r)},{0.27*26.9*cos(\t r)+-0.96*9.94*sin(\t r)});
\end{tikzpicture}
\caption{}
\label{fig:4arcs}
\end{figure}

The remaining arcs form an $(m+2)$-angulation of the $((n-2)m+2)$-gon defined by both loops and the boundary $(n-2)m$-gon. This is the geometric description of Dynkin type $A$, that we can study in \cite[Section 1]{BM01}. In that case, all $(m+2)$-angulations are related by a sequence of flips. This proves the result.

\end{proof}

\begin{coro}\label{lem:eq}
Any two $(m+2)$-angulations are related by a finite sequence of flips.
\end{coro}

This Corollary gives a different proof to the fact that
$m$-cluster-tilting mutation in the $m$-cluster category of type $\tilde{D}$ is transitive (assuming
$(m+2)$-angulations are in bijection with $m$-cluster-tilting objects, according to \cite{JM}).
The transitivity of cluster-tilting mutation in the hereditary case was proved in \cite{ZZ}.

\begin{coro}
Any $(m+2)$-angulation contains exactly $n+1$ $m$-diagonals. 	If it is given by a good set of representatives, then it has $n$ $(m+2)$-angles.
\end{coro}

%
%

We can also apply the "flip" on a maximal set $\tilde{\Delta}$ of non-crossing $m$-diagonals.
For $\alpha$ an $m$-diagonal, we can set $\kappa_{\tilde{\Delta}}(\alpha)$ exactly the same way we did for $(m+2)$-angulations.

\begin{prop}
A set of $m$-diagonals is an $(m+2)$-angulation if and only if it is a maximal set of noncrossing $m$-diagonals.
\end{prop}

\begin{proof}
On the one hand, a maximal set of noncrossing $m$-diagonals cuts the polygon into $(m+2)$-angles. Indeed, let us fix $\Delta$ a maximal set of noncrossing $m$-diagonals.

\begin{itemize}
\item If $\Delta$ contains two loops, then, these two loops cut the figure into a $(n-2)m+2$-gon. This reduces to case $A_{n-3}$ (since the polygon in the description in \cite{BM01} in type $A_p$ contains $pm+2$ sides). As the set is maximal, it cuts the figure into $(m+2)$-angles.
\item If $\Delta$ only contains one loop, then the loop cuts the figure into a $(n-2)m+1$-gon, with one thick vertex. This is the geometric realization of Baur and Marsh in \cite{BM02} for a category of Dynkin type $D_{n-1}$ (Baur and Marsh use the system of punctures instead of the inner polygons, but this does not change the proof). As the set is maximal, it cuts the figure into $(m+2)$-angles.

\item If $\Delta$ has no loops, then we use the following argument to create loops. There are several cases:
\begin{itemize}
\item There are two arcs ending to $L$ (for instance), and ending at the same vertex of $P$, then it suffices to replace one of them by a loop.
\item There are two arcs ending to $L$ (for instance), but ending at a different vertex of $P$. Say that $i$ ends in $a$ and $j$ ends in $b$. Then we can apply several times the twist to $j$ in order to hang it to the same vertex as $i$, $a$. It suffices then to replace $j$ by a loop.
\item There is one arc only ending to $R$. This would contradict the maximality of the number of $m$)diagonals, since we can add a loop.
\item \cor{If there are $k>3$ arcs ending to $R$, applying the twist several times to one of them make it slips along on arc, and become of other type. Then we are led to the case $k-1$, and we conclude by induction on $k$, the initialization being the previous case.}
\end{itemize}
\item If $\Delta$ contains an arc $i$ from $L$ to $R$, then any maximal set of noncrossing $m$-diagonals must contain two arcs with endpoints in $P$ and $L$, and similarly two arcs with endpoints in $P$ and $R$. Then we can reduce to the case where there arc two loops (one around $L$ and one around $R$).

If $\Delta$ contains an arc $i$ from $L$ to $R$, the maximality of $\Delta$ implies that there exist one arc from a vertex of $P$ to $L$ and one arc from $P$ to $R$ (indeed, if this was not the case, we have only $m$-diagonals linking different vertices of $P$, plus an arc from $L$ to $R$. We then can add an $m$-diagonal from a vertex of $P$ to $L$, which does not cross th others and is admissible, and this contradicts the maximality). Then, we apply the twist to $i$ as many times as necessary to hang one of its ends to the polygon $P$. Then, we come to the previous case.

Note that two arcs linking $L$ to $R$ are not compatible (as in \cite{FST}, for the case where $m=1$).
\end{itemize}

This finally shows that $\Delta$ cuts the figure into $(m+2)$-angles.

On the other hand, let $\Delta$ be a set of noncrossing $m$-diagonals, cutting the picture into $(m+2)$-angles. We first note that the $m$-diagonals do not cross. In addition, the set is maximal. Indeed, if it was not the case: Let us add a noncrossing arc $\alpha$ in an $(m+2)$-gon, then it cannot be an $m$-diagonal:

If $m=1$, then we cannot cut any triangle without being homotopic to one of the edges. Else, the only way to cut an $(m+2)$-angle by forming a $(km+1)$-gon is to form a triangle. This is impossible given the definition of an $m$-diagonal.
\end{proof}

\begin{coro}
A set of $m$-diagonals is an $(m+2)$-angulation if and only if it is a set of $n+1$ noncrossing $m$-diagonals.
\end{coro}

\begin{proof}
It suffices to show that if there exist $n+1$ noncrossing $m$-diagonals, then this number is maximal. Let $\Gamma$ be a set of good representatives of $n+1$ noncrossing $m$-diagonals. If there is an $m$-diagonal $\alpha \notin \Gamma$ which is not crossing all the $m$-diagonals of $\Gamma$, as there are $n$ $(m+2)$-angles, this $m$-diagonal cuts an $(m+2)$-angle, and this gives a contradiction.
\end{proof}

\begin{lem}
Any set of noncrossing $m$-diagonals, can be completed to an $(m+2)$-angulation.
\end{lem}

\begin{proof}
Suppose $\Lambda=\{\alpha_1,\cdots,\alpha_k \}$ is a set of noncrossing $m$-diagonals. If the set is maximal, then it is already an $(m+2)$-angulation. Else, we can add an new $m$-diagonal $\alpha_{k+1}$. We repeat the operation with $\Lambda \cup \{\alpha_{k+1} \}$ until the set is maximal. The process ends since $P$ has a finite number of sides, and once we have cut into $(m+2)$-angles, we cannot go further.
\end{proof}

\section{Construction of the category of $m$-diagonals}

We are now able to construct a new category, ${\mathcal{C}}$, equivalent to a subcategory of the higher cluster category of type $\tilde{D}$ by using the $m$-diagonals. More explicitly, we are going to draw an Auslander-Reiten quiver and we will see that it is isomorphic to a subquiver of the Auslander-Reiten quiver of ${\mathcal{C}}^m_Q$. We will also remark that the indecomposables of ${\mathcal{C}}$ correspond to the rigid indecomposables of the higher cluster category.

Before starting with the construction of the category we define what will be the generators the morphisms of ${\mathcal{C}}$: the elementary moves.

\subsection{Elementary moves}
Elementary moves are applications that send an $m$-diagonal to another one. They should not be confused with flips.

Let $Q$ be a quiver of type $\tilde{D_n}$. We want the elementary moves to correspond to the arrows in the Auslander-Reiten quiver of ${\mathcal{C}}^m_Q$.


We now recall the length of an $m$-diagonal of type $1$ from definition \ref{def:diago}.

\begin{defi}
	Let $a<b$ be two integers between $1$ and $(n-2)m$. The length of the boundary path from $a$ to $b$, denoted by $l(a,b)$ is the number of edges between $a$ and $b$ counted clockwise.
\end{defi}

\begin{defi}
	Let $P$ be the polygon considered in section \ref{sec:geo}. Let us number the vertices of $P$ clockwise. Let $\alpha$ be an $m$-diagonal from $i$ to $j$. The translation $\tau^{-1} \alpha$ of $\alpha$ is defined as follows:
	\begin{enumerate}
		\item If $i \neq j$ and $\alpha$ is of type $1$ or $2$, then $\tau^{-1} \alpha$ is the unique new arc obtained by by composing $\alpha$ with the boundary paths $B_{i~i+m}$ and $B_{j~j+m}$, linking $i+m$ and $j+m$, and of the same type.
		\item If $i=j$, the arc is of type $3$ then $\tau^{-1} \alpha$ is the new arc tangent to the same vertex of the inner polygon, ending at $i+m$. If $m$ is odd, then change the side of the tangency. To be clear, one can obtain $\tau^{-1} \alpha$ by moving the endpoint in the clockwise direction $m$ times.
		\item If $\alpha$ links both inner polygons, it means that $\alpha$ is of type $4$, from $L_k$ to $R_{k'}$ for instance, the new arc $\tau^{-1} \alpha$ links $L_{k-m}$ and $R_{k'+m}$ modulo $m-1$ (where indices are considered). The side of the tangency changes only if $m$ is odd. To be clear, one can obtain $\tau^{-1} \alpha$ by moving the endpoint in the clockwise direction $m$ times.
	\end{enumerate}
\end{defi}

For the last point, the new arc is uniquely defined, because it does not depends on the eventual Dehn twist of $L$ or $R$, but only of the relative position before applying $\tau^{-1}$.

\cor{This definition of a translation is the inverse of the Auslander-Reiten translation.}

Applying $\tau$ or $\tau^{-1}$ several times to an arc of type $1$ makes the arc "roll around" both central polygons.

From now and al throughought the paper, we cal by $B_{i,j}$ the boundary component linking $i$ to $j$ clockwise.

Here in figure \ref{fig:type1} and figure \ref{fig:1}, we can see examples of a translation:

\begin{figure}[!h]
	\centering
\begin{tikzpicture}[scale=0.25]
	\fill[line width=0.8pt,fill=black,fill opacity=0.15] (0,4) -- (0,-4) -- (6.92820323027551,-8) -- (13.856406460551021,-4) -- (13.856406460551023,4) -- (6.928203230275516,8) -- cycle;

	\draw [line width=0.8pt,fill=black,fill opacity=1] (3.4641016151377566,2) circle (0.4cm);

	\draw [line width=0.8pt,fill=black,fill opacity=1] (3.4641016151377553,-2) circle (0.4cm);

	\draw [line width=0.8pt,fill=black,fill opacity=1] (10.392304845413266,2) circle (0.4cm);

	\draw [line width=0.8pt,fill=black,fill opacity=1] (10.392304845413266,-2) circle (0.4cm);

	\draw [line width=0.8pt] (0,4)-- (0,-4);

	\draw [line width=0.8pt] (0,-4)-- (6.92820323027551,-8);

	\draw [line width=0.8pt] (6.92820323027551,-8)-- (13.856406460551021,-4);

	\draw [line width=0.8pt] (13.856406460551021,-4)-- (13.856406460551023,4);

	\draw [line width=0.8pt] (13.856406460551023,4)-- (6.928203230275516,8);

	\draw [line width=0.8pt] (6.928203230275516,8)-- (0,4);

	\draw [shift={(1.8568379886224946,0)},line width=0.8pt]  plot[domain=-0.893844784234993:0.8938447842349926,variable=\t]({1*2.565793515682622*cos(\t r)+0*2.565793515682622*sin(\t r)},{0*2.565793515682622*cos(\t r)+1*2.565793515682622*sin(\t r)});

	\draw [shift={(5.07136524165302,0)},line width=0.8pt]  plot[domain=2.247747869354801:4.035437437824785,variable=\t]({1*2.5657935156826235*cos(\t r)+0*2.5657935156826235*sin(\t r)},{0*2.5657935156826235*cos(\t r)+1*2.5657935156826235*sin(\t r)});

	\draw [shift={(11.999568471928534,0)},line width=0.8pt]  plot[domain=2.2477478693548:4.0354374378247835,variable=\t]({1*2.565793515682625*cos(\t r)+0*2.565793515682625*sin(\t r)},{0*2.565793515682625*cos(\t r)+1*2.565793515682625*sin(\t r)});

	\draw [shift={(8.785041218898007,0)},line width=0.8pt]  plot[domain=-0.8938447842349948:0.8938447842349915,variable=\t]({1*2.565793515682622*cos(\t r)+0*2.565793515682622*sin(\t r)},{0*2.565793515682622*cos(\t r)+1*2.565793515682622*sin(\t r)});

	\draw [shift={(12.576619119077439,-1.7495195332350346)},line width=0.8pt]  plot[domain=3.2924793999925908:5.229455306015604,variable=\t]({1*2.5889221638911213*cos(\t r)+0*2.5889221638911213*sin(\t r)},{0*2.5889221638911213*cos(\t r)+1*2.5889221638911213*sin(\t r)});

	\draw [shift={(12.576619119077439,-1.7495195332350346)},line width=0.8pt,color=red]  plot[domain=3.2924793999925908:5.229455306015604,variable=\t]({1*2.588922163891121*cos(\t r)+0*2.588922163891121*sin(\t r)},{0*2.588922163891121*cos(\t r)+1*2.588922163891121*sin(\t r)});

\end{tikzpicture}
	$\to$
\begin{tikzpicture}[scale=0.25]

	\fill[line width=0.8pt,fill=black,fill opacity=0.15] (0,4) -- (0,-4) -- (6.92820323027551,-8) -- (13.856406460551021,-4) -- (13.856406460551023,4) -- (6.928203230275516,8) -- cycle;

	\draw [line width=0.8pt,fill=black,fill opacity=1] (3.4641016151377566,2) circle (0.4cm);

	\draw [line width=0.8pt,fill=black,fill opacity=1] (3.4641016151377553,-2) circle (0.4cm);

	\draw [line width=0.8pt,fill=black,fill opacity=1] (10.392304845413266,2) circle (0.4cm);

	\draw [line width=0.8pt,fill=black,fill opacity=1] (10.392304845413266,-2) circle (0.4cm);

	\draw [line width=0.8pt] (0,4)-- (0,-4);

	\draw [line width=0.8pt] (0,-4)-- (6.92820323027551,-8);

	\draw [line width=0.8pt] (6.92820323027551,-8)-- (13.856406460551021,-4);

	\draw [line width=0.8pt] (13.856406460551021,-4)-- (13.856406460551023,4);

	\draw [line width=0.8pt] (13.856406460551023,4)-- (6.928203230275516,8);

	\draw [line width=0.8pt] (6.928203230275516,8)-- (0,4);

	\draw [shift={(1.8568379886224946,0)},line width=0.8pt]  plot[domain=-0.893844784234993:0.8938447842349926,variable=\t]({1*2.565793515682622*cos(\t r)+0*2.565793515682622*sin(\t r)},{0*2.565793515682622*cos(\t r)+1*2.565793515682622*sin(\t r)});

	\draw [shift={(5.07136524165302,0)},line width=0.8pt]  plot[domain=2.247747869354801:4.035437437824785,variable=\t]({1*2.5657935156826235*cos(\t r)+0*2.5657935156826235*sin(\t r)},{0*2.5657935156826235*cos(\t r)+1*2.5657935156826235*sin(\t r)});

	\draw [shift={(11.999568471928534,0)},line width=0.8pt]  plot[domain=2.2477478693548:4.0354374378247835,variable=\t]({1*2.565793515682625*cos(\t r)+0*2.565793515682625*sin(\t r)},{0*2.565793515682625*cos(\t r)+1*2.565793515682625*sin(\t r)});

	\draw [shift={(8.785041218898007,0)},line width=0.8pt]  plot[domain=-0.8938447842349948:0.8938447842349915,variable=\t]({1*2.565793515682622*cos(\t r)+0*2.565793515682622*sin(\t r)},{0*2.565793515682622*cos(\t r)+1*2.565793515682622*sin(\t r)});

	\draw [shift={(6.737716543589979,3.2639928053307075)},line width=0.8pt, color=red]  plot[domain=3.032787176890152:5.29602529986481,variable=\t]({1*6.7777968996106*cos(\t r)+0*6.7777968996106*sin(\t r)},{0*6.7777968996106*cos(\t r)+1*6.7777968996106*sin(\t r)});

\end{tikzpicture}
	\label{fig:1}
	\caption{Example of a translation}
\end{figure}

\begin{figure}[!h]
	\centering
	\begin{tikzpicture}[scale=0.25]
		\fill[fill=black,fill opacity=0.15] (0,4) -- (0,-4) -- (6.93,-8) -- (13.86,-4) -- (13.86,4) -- (6.93,8) -- cycle;
		\draw [fill=black,fill opacity=1.0] (3.46,2) circle (0.4cm);
		\draw [fill=black,fill opacity=1.0] (3.46,-2) circle (0.4cm);
		\draw [fill=black,fill opacity=1.0] (10.39,2) circle (0.4cm);
		\draw [fill=black,fill opacity=1.0] (10.39,-2) circle (0.4cm);
		\draw (0,4)-- (0,-4);
		\draw (0,-4)-- (6.93,-8);
		\draw (6.93,-8)-- (13.86,-4);
		\draw (13.86,-4)-- (13.86,4);
		\draw (13.86,4)-- (6.93,8);
		\draw (6.93,8)-- (0,4);
		\draw [shift={(1.86,0)}] plot[domain=-0.89:0.89,variable=\t]({1*2.57*cos(\t r)+0*2.57*sin(\t r)},{0*2.57*cos(\t r)+1*2.57*sin(\t r)});
		\draw [shift={(5.07,0)}] plot[domain=2.25:4.04,variable=\t]({1*2.57*cos(\t r)+0*2.57*sin(\t r)},{0*2.57*cos(\t r)+1*2.57*sin(\t r)});
		\draw [shift={(12,0)}] plot[domain=2.25:4.04,variable=\t]({1*2.57*cos(\t r)+0*2.57*sin(\t r)},{0*2.57*cos(\t r)+1*2.57*sin(\t r)});
		\draw [shift={(8.79,0)}] plot[domain=-0.89:0.89,variable=\t]({1*2.57*cos(\t r)+0*2.57*sin(\t r)},{0*2.57*cos(\t r)+1*2.57*sin(\t r)});
		\draw [shift={(5.22,-2.46)}] plot[domain=-1.4:2.54,variable=\t]({-0.55*6.77*cos(\t r)+-0.84*2.82*sin(\t r)},{0.84*6.77*cos(\t r)+-0.55*2.82*sin(\t r)});
		\draw [shift={(8.64,2.46)}] plot[domain=-1.4:2.54,variable=\t]({0.55*6.77*cos(\t r)+0.84*2.82*sin(\t r)},{-0.84*6.77*cos(\t r)+0.55*2.82*sin(\t r)});
	\end{tikzpicture}
	$\to$
	\begin{tikzpicture}[scale=0.25]
		\fill[fill=black,fill opacity=0.15] (0,4) -- (0,-4) -- (6.93,-8) -- (13.86,-4) -- (13.86,4) -- (6.93,8) -- cycle;
		\draw [fill=black,fill opacity=1.0] (3.46,2) circle (0.4cm);
		\draw [fill=black,fill opacity=1.0] (3.46,-2) circle (0.4cm);
		\draw [fill=black,fill opacity=1.0] (10.39,2) circle (0.4cm);
		\draw [fill=black,fill opacity=1.0] (10.39,-2) circle (0.4cm);
		\draw (0,4)-- (0,-4);
		\draw (0,-4)-- (6.93,-8);
		\draw (6.93,-8)-- (13.86,-4);
		\draw (13.86,-4)-- (13.86,4);
		\draw (13.86,4)-- (6.93,8);
		\draw (6.93,8)-- (0,4);
		\draw [shift={(1.86,0)}] plot[domain=-0.89:0.89,variable=\t]({1*2.57*cos(\t r)+0*2.57*sin(\t r)},{0*2.57*cos(\t r)+1*2.57*sin(\t r)});
		\draw [shift={(5.07,0)}] plot[domain=2.25:4.04,variable=\t]({1*2.57*cos(\t r)+0*2.57*sin(\t r)},{0*2.57*cos(\t r)+1*2.57*sin(\t r)});
		\draw [shift={(12,0)}] plot[domain=2.25:4.04,variable=\t]({1*2.57*cos(\t r)+0*2.57*sin(\t r)},{0*2.57*cos(\t r)+1*2.57*sin(\t r)});
		\draw [shift={(8.79,0)}] plot[domain=-0.89:0.89,variable=\t]({1*2.57*cos(\t r)+0*2.57*sin(\t r)},{0*2.57*cos(\t r)+1*2.57*sin(\t r)});
		\draw (6.93,8)-- (6.93,-8);
	\end{tikzpicture}
	\caption{Another example of a translation}
	\label{fig:type1}
\end{figure}

\begin{figure}[!h]
	\centering
\begin{tikzpicture}[scale=0.25]

	\fill[line width=0.8pt,fill=black,fill opacity=0.15] (0,4) -- (0,-4) -- (6.92820323027551,-8) -- (13.856406460551021,-4) -- (13.856406460551023,4) -- (6.928203230275516,8) -- cycle;

	\draw [line width=0.8pt,fill=black,fill opacity=1] (3.4641016151377566,2) circle (0.4cm);

	\draw [line width=0.8pt,fill=black,fill opacity=1] (10.392304845413266,2) circle (0.4cm);

	\fill[line width=0.8pt,fill=black,fill opacity=0.1] (3.4641016151377566,2) -- (2.4946890109468254,0.3067654035479556) -- (4.445779488136453,0.3138467597958013) -- cycle;

	\fill[line width=0.8pt,fill=black,fill opacity=0.15] (10.392304845413271,2) -- (11.3617174496042,0.30676540354795023) -- (9.41062697241457,0.3138467597957983) -- cycle;

	\draw [line width=0.8pt,fill=black,fill opacity=1] (2.4946890109468254,0.3067654035479556) circle (0.4018816135660477cm);

	\draw [line width=0.8pt,fill=black,fill opacity=1] (4.445779488136453,0.3138467597958013) circle (0.4182694886133407cm);

	\draw [line width=0.8pt,fill=black,fill opacity=1] (9.410626972414573,0.3138467597957983) circle (0.4182694886133407cm);

	\draw [line width=0.8pt,fill=black,fill opacity=1] (11.3617174496042,0.30676540354795023) circle (0.4018816135660477cm);

	\draw [line width=0.8pt] (0,4)-- (0,-4);

	\draw [line width=0.8pt] (0,-4)-- (6.92820323027551,-8);

	\draw [line width=0.8pt] (6.92820323027551,-8)-- (13.856406460551021,-4);

	\draw [line width=0.8pt] (13.856406460551021,-4)-- (13.856406460551023,4);

	\draw [line width=0.8pt] (13.856406460551023,4)-- (6.928203230275516,8);

	\draw [line width=0.8pt] (6.928203230275516,8)-- (0,4);




	\draw [shift={(4.212077115475866,-0.7391469794752008)},line width=0.8pt]  plot[domain=0.2170034915095111:1.797160748795029,variable=\t]({1*3.220513106249162*cos(\t r)+0*3.220513106249162*sin(\t r)},{0*3.220513106249162*cos(\t r)+1*3.220513106249162*sin(\t r)});

	\draw [shift={(9.53746156462467,0.37452751147670493)},line width=0.8pt]  plot[domain=3.332012513245802:6.209339886271012,variable=\t]({1*2.2205386151868196*cos(\t r)+0*2.2205386151868196*sin(\t r)},{0*2.2205386151868196*cos(\t r)+1*2.2205386151868196*sin(\t r)});

\end{tikzpicture}
	$\to$
\begin{tikzpicture}[scale=0.25]

	\fill[line width=0.8pt,fill=black,fill opacity=0.15] (0,4) -- (0,-4) -- (6.92820323027551,-8) -- (13.856406460551021,-4) -- (13.856406460551023,4) -- (6.928203230275516,8) -- cycle;

	\draw [line width=0.8pt,fill=black,fill opacity=1] (3.4641016151377566,2) circle (0.4cm);

	\draw [line width=0.8pt,fill=black,fill opacity=1] (10.392304845413266,2) circle (0.4cm);

	\fill[line width=2pt,fill=black,fill opacity=0.1] (3.4641016151377566,2) -- (2.4946890109468254,0.3067654035479556) -- (4.445779488136453,0.3138467597958013) -- cycle;

	\fill[line width=2pt,fill=black,fill opacity=0.15] (10.392304845413271,2) -- (11.3617174496042,0.30676540354795023) -- (9.41062697241457,0.3138467597957983) -- cycle;

	\draw [line width=0.8pt,fill=black,fill opacity=1] (2.4946890109468254,0.3067654035479556) circle (0.4018816135660477cm);

	\draw [line width=0.8pt,fill=black,fill opacity=1] (4.445779488136453,0.3138467597958013) circle (0.4182694886133407cm);

	\draw [line width=0.8pt,fill=black,fill opacity=1] (9.410626972414573,0.3138467597957983) circle (0.4182694886133407cm);

	\draw [line width=0.8pt,fill=black,fill opacity=1] (11.3617174496042,0.30676540354795023) circle (0.4018816135660477cm);

	\draw [line width=0.8pt] (0,4)-- (0,-4);

	\draw [line width=0.8pt] (0,-4)-- (6.92820323027551,-8);

	\draw [line width=0.8pt] (6.92820323027551,-8)-- (13.856406460551021,-4);

	\draw [line width=0.8pt] (13.856406460551021,-4)-- (13.856406460551023,4);

	\draw [line width=0.8pt] (13.856406460551023,4)-- (6.928203230275516,8);

	\draw [line width=0.8pt] (6.928203230275516,8)-- (0,4);



	\draw [shift={(13.042706632737085,22.030199016745826)},line width=0.8pt]  plot[domain=4.3328314592710875:4.436712769702422,variable=\t]({1*22.942222725390334*cos(\t r)+0*22.942222725390334*sin(\t r)},{0*22.942222725390334*cos(\t r)+1*22.942222725390334*sin(\t r)});

	\draw [shift={(8.210204109692489,3.9295621778492773)},line width=0.8pt]  plot[domain=4.371025733944578:5.054824368487186,variable=\t]({1*4.218743669646971*cos(\t r)+0*4.218743669646971*sin(\t r)},{0*4.218743669646971*cos(\t r)+1*4.218743669646971*sin(\t r)});

\end{tikzpicture}
	\caption{Last example of a translation, with $m=4$.}
	\label{fig:typem4}
\end{figure}

We now define elementary moves, which will correspond to arrows in the Auslander-Reiten quiver of ${\mathcal{C}}^m_{\tilde{D_n}}$ as shown in figure \ref{fig:ar1}.

Let us recall that the Auslander-Reiten quiver $\Gamma$ of ${\mathcal{C}}^m_{\tilde{D_n}}$ is defined as follows:
\begin{enumerate}
	\item The vertices of $\Gamma$ are the isomorphism classes of indecomposable objects in ${\mathcal{C}}^m_{\tilde{D_n}}$.
	\item Between two representants $X$ and $Y$, the arrows $X \to Y$ are in bijective correspondence with the vectors of a basis of $\mathrm{Irr}(X,Y)$, the irreducible morphisms from $X$ to $Y$.
\end{enumerate}


\begin{defi}
Let $\alpha$ be an $m$-diagonal. We define an elementary move from $\alpha$ to an arc $\beta$ by giving all possibilities given from the type of $\alpha$.
\begin{itemize}
	\item If $\alpha$ is of type $1$, an elementary move is given by fixing an endpoint $a$, and moving the other endpoint $b$ around the boundary component $m$ times in the clockwise direction. The resulting arc $\beta$ is either of type $1$ if $a\neq b+m$, of type $3$, if $a=b+m$.
	\item If $\alpha$ is of type $2$, an elementary move is given by fixing an endpoint $a$, and moving the other endpoint $b$ around the boundary component $m$ times in the clockwise direction. The resulting arc $\beta$ is also of type $2$.
	\item If $\alpha$ is of type $3$, the representative of its equivalence classgiven by a loop is chosen before performing the elementary move. Then, fixing an endpoint in $a$ (notice that the other endpoint $b$ is hung at the same vertex), move $b$, $m$ times, in order to find $\beta$, which is of type $1$.
	\item If $\alpha$ is of type $4$,  an elementary move is given by fixing an endpoint $a$ on $L$ (respectively $R$), and moving the other endpoint $b$, which is on $R$ (respectively $L$) to $L$ (respectively $R$), at vertex $a$. The resulting arc $\beta$ is of type $4$.
\end{itemize}
\end{defi}

Here is the tabular of elementary moves in the three first cases for a polygon with $(n-2)m$ sides (where e-m stands for "elementary move" and the $m$-diagonal $\alpha$ is black whereas $\beta$ is red). For sake of clarity, we have not drawn the disks on each thick vertex of the central polygons. The dotted lines hide $m$ edges.

The remaining cases in the definition will be treated in section $4$.

\vspace{20pt}

\begin{center}
	\begin{tabular}{|c|c|c|}
		\hline
		& type $1$ & type $2$\\
		\hline
		type $1$ & \begin{tikzpicture}[scale=0.35]
			\fill[dash pattern=on 3pt off 3pt,fill=black,fill opacity=0.1] (0,4) -- (-2.48,2.98) -- (-3.51,0.51) -- (-2.49,-1.97) -- (-0.02,-3.01) -- (2.46,-1.99) -- (3.49,0.49) -- (2.47,2.97) -- cycle;
			\fill[fill=black,fill opacity=1.0] (1.41,2.37) -- (1.21,2.14) -- (1.32,1.85) -- (1.61,1.8) -- (1.81,2.03) -- (1.71,2.32) -- cycle;
			\fill[fill=black,fill opacity=1.0] (-1.89,-0.91) -- (-1.66,-0.72) -- (-1.37,-0.82) -- (-1.32,-1.12) -- (-1.55,-1.31) -- (-1.84,-1.21) -- cycle;
			\draw (0,4)-- (-2.48,2.98);
			\draw [dash pattern=on 3pt off 3pt] (-2.48,2.98)-- (-3.51,0.51);
			\draw (-3.51,0.51)-- (-2.49,-1.97);
			\draw (-2.49,-1.97)-- (-0.02,-3.01);
			\draw (-0.02,-3.01)-- (2.46,-1.99);
			\draw (2.46,-1.99)-- (3.49,0.49);
			\draw (3.49,0.49)-- (2.47,2.97);
			\draw (2.47,2.97)-- (0,4);
			\draw [shift={(-4.35,-9.91)}] plot[domain=0.86:1.49,variable=\t]({1*10.45*cos(\t r)+0*10.45*sin(\t r)},{0*10.45*cos(\t r)+1*10.45*sin(\t r)});
			\draw [shift={(36.12,36.44)},color=red]  plot[domain=3.86:3.99,variable=\t]({1*51.09*cos(\t r)+0*51.09*sin(\t r)},{0*51.09*cos(\t r)+1*51.09*sin(\t r)});
			\draw (1.41,2.37)-- (1.21,2.14);
			\draw (1.21,2.14)-- (1.32,1.85);
			\draw (1.32,1.85)-- (1.61,1.8);
			\draw (1.61,1.8)-- (1.81,2.03);
			\draw (1.81,2.03)-- (1.71,2.32);
			\draw (1.71,2.32)-- (1.41,2.37);
		\end{tikzpicture} & \begin{tikzpicture}[scale=0.35]
			\fill[dash pattern=on 3pt off 3pt,fill=black,fill opacity=0.1] (0,4) -- (-2.48,2.98) -- (-3.51,0.51) -- (-2.49,-1.97) -- (-0.02,-3.01) -- (2.46,-1.99) -- (3.49,0.49) -- (2.47,2.97) -- cycle;
			\fill[fill=black,fill opacity=1.0] (1.41,2.37) -- (1.21,2.14) -- (1.32,1.85) -- (1.61,1.8) -- (1.81,2.03) -- (1.71,2.32) -- cycle;
			\fill[fill=black,fill opacity=1.0] (-1.89,-0.91) -- (-1.66,-0.72) -- (-1.37,-0.82) -- (-1.32,-1.12) -- (-1.55,-1.31) -- (-1.84,-1.21) -- cycle;
			\draw (0,4)-- (-2.48,2.98);
			\draw (-2.48,2.98)-- (-3.51,0.51);
			\draw (-3.51,0.51)-- (-2.49,-1.97);
			\draw (-2.49,-1.97)-- (-0.02,-3.01);
			\draw (-0.02,-3.01)-- (2.46,-1.99);
			\draw (2.46,-1.99)-- (3.49,0.49);
			\draw [dash pattern=on 3pt off 3pt] (3.49,0.49)-- (2.47,2.97);
			\draw (2.47,2.97)-- (0,4);
			\draw (1.41,2.37)-- (1.21,2.14);
			\draw (1.21,2.14)-- (1.32,1.85);
			\draw (1.32,1.85)-- (1.61,1.8);
			\draw (1.61,1.8)-- (1.81,2.03);
			\draw (1.81,2.03)-- (1.71,2.32);
			\draw (1.71,2.32)-- (1.41,2.37);
			\draw [shift={(7.63,1.41)}] plot[domain=-0.78:1,variable=\t]({-1*7.43*cos(\t r)+-0.09*1.54*sin(\t r)},{0.09*7.43*cos(\t r)+-1*1.54*sin(\t r)});
			\draw [shift={(5.65,6.36)},color=red]  plot[domain=3.95:4.36,variable=\t]({1*6.25*cos(\t r)+0*6.25*sin(\t r)},{0*6.25*cos(\t r)+1*6.25*sin(\t r)});
		\end{tikzpicture} \\
		\hline
		type $3$ & \begin{tikzpicture}[scale=0.35]
			\fill[dash pattern=on 3pt off 3pt,fill=black,fill opacity=0.1] (0,4) -- (-2.48,2.98) -- (-3.51,0.51) -- (-2.49,-1.97) -- (-0.02,-3.01) -- (2.46,-1.99) -- (3.49,0.49) -- (2.47,2.97) -- cycle;
			\fill[fill=black,fill opacity=1.0] (1.41,2.37) -- (1.21,2.14) -- (1.32,1.85) -- (1.61,1.8) -- (1.81,2.03) -- (1.71,2.32) -- cycle;
			\fill[fill=black,fill opacity=1.0] (-1.89,-0.91) -- (-1.66,-0.72) -- (-1.37,-0.82) -- (-1.32,-1.12) -- (-1.55,-1.31) -- (-1.84,-1.21) -- cycle;
			\draw (0,4)-- (-2.48,2.98);
			\draw (-2.48,2.98)-- (-3.51,0.51);
			\draw (-3.51,0.51)-- (-2.49,-1.97);
			\draw (-2.49,-1.97)-- (-0.02,-3.01);
			\draw (-0.02,-3.01)-- (2.46,-1.99);
			\draw (2.46,-1.99)-- (3.49,0.49);
			\draw [dash pattern=on 3pt off 3pt] (3.49,0.49)-- (2.47,2.97);
			\draw (2.47,2.97)-- (0,4);
			\draw (1.41,2.37)-- (1.21,2.14);
			\draw (1.21,2.14)-- (1.32,1.85);
			\draw (1.32,1.85)-- (1.61,1.8);
			\draw (1.61,1.8)-- (1.81,2.03);
			\draw (1.81,2.03)-- (1.71,2.32);
			\draw (1.71,2.32)-- (1.41,2.37);
			\draw [shift={(7.63,1.41)},color=red]  plot[domain=-0.78:1,variable=\t]({-1*7.43*cos(\t r)+-0.09*1.54*sin(\t r)},{0.09*7.43*cos(\t r)+-1*1.54*sin(\t r)});
			\draw [shift={(2.85,1.03)}] plot[domain=1.76:2.55,variable=\t]({1*1.98*cos(\t r)+0*1.98*sin(\t r)},{0*1.98*cos(\t r)+1*1.98*sin(\t r)});
		\end{tikzpicture} & no e-m \\
		\hline
	\end{tabular}
	\label{tab:tab}
\end{center}
in the $m$-cluster category
\vspace{10pt}

\subsection{The category of $m$-diagonals}
In this section, we will make use of the elementary moves defined above in order to generate the morphism of a category of $m$-diagonals. Let us just give a Lemma before starting.

\begin{lem}\label{lem:em}
	If $\alpha$ and $\beta$ are two $m$-diagonals, then there exists an elementary move $f:\alpha \to \beta$ if and only if there exists one $\overline{f}:\tau \beta \to \alpha$.
\end{lem}

\begin{proof}
	There are several cases to examine. We can always reduce to the case where $\alpha$ or $\beta$ are not loops. Note that if $\alpha$ is of type $1$ or $3$, there exists an elementary move from $\alpha$ to $\beta$ if and only if $\beta$ is also an arc of type $1$ or $3$.
	\begin{enumerate}
		\item Let us first study the case of arcs of type $1$ or $3$.
		\begin{enumerate}
			\item If $\beta$ is of type $3$.
			
			Then let $a$ be its only vertex. To be compatible with the definition of an elementary move, it is necessary for $\alpha$ to be of type $1$. Let $b$ be the other vertex of $\alpha$ (the first one being $a$, since both arcs have to share an oriented angle).
			
			By definition, $l(a,b)=m$. If $a'$ is the extremity of the $m$-diagonal $\tau \beta$, then $l(a',a)=m$. So $a'=b$ and $\tau \beta$ and $\alpha$ share an oriented angle. Thus there exists $\overline{f} : \tau \beta \to \alpha$
			
			\item If $\beta$ is of type $1$. There are two sub-cases:
			\begin{enumerate}
				\item If $\alpha$ is of type $3$.
				
				Let $a$ and $b$ be the vertices of $\beta$, where $a$ is also the one of $\alpha$.
				
				Then $l(a,b)=m$. So, if $c$ and $d$ are the vertices of $\tau \beta$ then $d=a$ (and $l(c,d)=m$). Thus there exists $\overline{f} : \tau \beta \to \alpha$.
				
				\item If $\alpha$ is of type $1$.
				
				Let $a$ and $b$ be the vertices of $\beta$ and $a$ and $c$ the ones of $\alpha$. We have $l(c,b)=m$.
				
				So, if $d$ and $e$ are the vertices of $\tau \beta$, we have $e=c$ and $l(d,a)=m$. Thus there exists $\overline{f} : \tau \beta \to \alpha$.
			\end{enumerate}
		\end{enumerate}
		\item Now, if $\alpha$ and $\beta$ are of type $2$ or $4$.
		\begin{enumerate}
			\item If $\alpha$ is homotopic to the boundary path (type $2$). Then $\beta$ must be also homotopic to the boundary path. Let $a$ and $b$ be the vertices of $\beta$, where $a$ is also the one of $\alpha$. Then, $l(c,b)=m$. So, if $d$ and $e$ are the vertices of $\tau \beta$ then $d=c$ (and $l(a,e)=m$). Thus there exists $\overline{f} : \tau \beta \to \alpha$.
			\item If $\alpha$ is from a thick vertex of $L$ to a thick vertex of $R$ (of type $4$), then $\beta$ is of the same type. Let $a$ and $b$ be the vertices of $\beta$, where $a$ is also the one of $\alpha$. Then, $l(c,b)=m$. So, if $d$ and $e$ are the vertices of $\tau \beta$ then $d=c$ (and $l(a,e)=m$). Thus there exists $\overline{f} : \tau \beta \to \alpha$
		\end{enumerate}
	\end{enumerate}
	
	The converse is similar.
\end{proof}


\begin{defi} \label{def:carquois}
	The quiver $\tilde{Q}$ is defined as follows:
	\begin{enumerate}
		\item The vertices are the $m$-diagonals.
		\item There is an arrow between two $m$-diagonals $\alpha$ and $\beta$ when there is an elementary move from $\alpha$ to $\beta$.
	\end{enumerate}
	The category ${\mathcal{C}}$ is defined to be the additive mesh category of $\tilde{Q}$.
\end{defi}

It means that the category ${\mathcal{C}}$ is the additive Krull-Schmidt category where:

\begin{enumerate}
	\item The indecomposable objects of ${\mathcal{C}}$ are the $m$-diagonals of definition \ref{def:diago}. So the objects of ${\mathcal{C}}$ are the finite direct sums of $m$-diagonals.
	\item The set of morphisms between two indecomposable objects $X$ and $Y$ is given by the $k$-vector space generated by the paths from $X$ to $Y$ in $\tilde{Q}$ by the subvector-space generated (elements in the ideal generated by) mesh relations.
	\item The composition is induced by the concatenation of paths in $\tilde{Q}$.
\end{enumerate}
Mesh relations: By the previous Lemma, we know that if we have $f:\beta \to \alpha$, then we also have $\overline{f}:\tau \alpha \to \beta$. Then the mesh relations are the relations:

\[ R_\alpha = \sum_{\beta,f:\beta \to \alpha} f\overline{f}. \]

We will see that ${\mathcal{C}}$ is equivalent to a sub-category of the higher cluster category ${\mathcal{C}}^{(m)}_{\tilde{D_n}}$.

Let us recall a definition.

\begin{defi}
	A stable translation quiver is a quiver $Q=(Q_0,Q_1)$ without any loops or multiple edges, together with a bijection $\tau:Q_0 \to Q_0$, such that for all vertices $x, y\in Q_0$, the number of arrows from $y$ to $x$ is equal to the number of arrows from $\tau(x)$ to $y$.
\end{defi}

\begin{coro}[Corollary of Lemma \ref{lem:em}]
	The quiver ($\tilde{Q}, \tau)$ is a stable translation quiver.
\end{coro}

\begin{proof}
	\cor{By definition, $\tau$ is defined on all vertices, and has a clear inverse.}
	
	Let $x$ be a vertex in $\tilde{Q}$. Given a vertex $y$, we have to show that if there is an arrow from $y$ to $x$, then there is an arrow from $\tau x$ to $y$. This is exactly the Lemma \ref{lem:em}.
	
	\cor{In addition, the quiver $\tilde{Q}$ does not have any loops (there cannot be any elementary move from $\alpha$ to $\alpha$) nor multiple arrows with the condition of the length in the definition of an elementary move.}
\end{proof}

Figure \ref{fig:ar1} illustrates a small part of the Auslander-Reiten quiver of ${\mathcal{C}}^{(m)}_Q$.

\begin{landscape}
	\begin{figure}
		\[ \scalebox{1.2}{
			\xymatrix {
				\cdots \ar[ddr] & & \begin{tikzpicture}[scale=0.1]
					\fill[fill=black,fill opacity=0.15] (0,4) -- (0,-4) -- (6.93,-8) -- (13.86,-4) -- (13.86,4) -- (6.93,8) -- cycle;
					\draw [fill=black,fill opacity=1.0] (3.46,2) circle (0.4cm);
					\draw [fill=black,fill opacity=1.0] (3.46,-2) circle (0.4cm);
					\draw [fill=black,fill opacity=1.0] (10.39,2) circle (0.4cm);
					\draw [fill=black,fill opacity=1.0] (10.39,-2) circle (0.4cm);
					\draw (0,4)-- (0,-4);
					\draw (0,-4)-- (6.93,-8);
					\draw (6.93,-8)-- (13.86,-4);
					\draw (13.86,-4)-- (13.86,4);
					\draw (13.86,4)-- (6.93,8);
					\draw (6.93,8)-- (0,4);
					\draw [shift={(1.86,0)}] plot[domain=-0.89:0.89,variable=\t]({1*2.57*cos(\t r)+0*2.57*sin(\t r)},{0*2.57*cos(\t r)+1*2.57*sin(\t r)});
					\draw [shift={(5.07,0)}] plot[domain=2.25:4.04,variable=\t]({1*2.57*cos(\t r)+0*2.57*sin(\t r)},{0*2.57*cos(\t r)+1*2.57*sin(\t r)});
					\draw [shift={(12,0)}] plot[domain=2.25:4.04,variable=\t]({1*2.57*cos(\t r)+0*2.57*sin(\t r)},{0*2.57*cos(\t r)+1*2.57*sin(\t r)});
					\draw [shift={(8.79,0)}] plot[domain=-0.89:0.89,variable=\t]({1*2.57*cos(\t r)+0*2.57*sin(\t r)},{0*2.57*cos(\t r)+1*2.57*sin(\t r)});
					\draw [shift={(25.48,-8.09)}] plot[domain=2.55:3.14,variable=\t]({1*18.55*cos(\t r)+0*18.55*sin(\t r)},{0*18.55*cos(\t r)+1*18.55*sin(\t r)});
				\end{tikzpicture} \ar[ddr] & & \begin{tikzpicture}[scale=0.1]
					\fill[fill=black,fill opacity=0.15] (0,4) -- (0,-4) -- (6.93,-8) -- (13.86,-4) -- (13.86,4) -- (6.93,8) -- cycle;
					\draw [fill=black,fill opacity=1.0] (3.46,2) circle (0.4cm);
					\draw [fill=black,fill opacity=1.0] (3.46,-2) circle (0.4cm);
					\draw [fill=black,fill opacity=1.0] (10.39,2) circle (0.4cm);
					\draw [fill=black,fill opacity=1.0] (10.39,-2) circle (0.4cm);
					\draw (0,4)-- (0,-4);
					\draw (0,-4)-- (6.93,-8);
					\draw (6.93,-8)-- (13.86,-4);
					\draw (13.86,-4)-- (13.86,4);
					\draw (13.86,4)-- (6.93,8);
					\draw (6.93,8)-- (0,4);
					\draw [shift={(1.86,0)}] plot[domain=-0.89:0.89,variable=\t]({1*2.57*cos(\t r)+0*2.57*sin(\t r)},{0*2.57*cos(\t r)+1*2.57*sin(\t r)});
					\draw [shift={(5.07,0)}] plot[domain=2.25:4.04,variable=\t]({1*2.57*cos(\t r)+0*2.57*sin(\t r)},{0*2.57*cos(\t r)+1*2.57*sin(\t r)});
					\draw [shift={(12,0)}] plot[domain=2.25:4.04,variable=\t]({1*2.57*cos(\t r)+0*2.57*sin(\t r)},{0*2.57*cos(\t r)+1*2.57*sin(\t r)});
					\draw [shift={(8.79,0)}] plot[domain=-0.89:0.89,variable=\t]({1*2.57*cos(\t r)+0*2.57*sin(\t r)},{0*2.57*cos(\t r)+1*2.57*sin(\t r)});
					\draw [shift={(5.98,1.98)}] plot[domain=-2.44:1.88,variable=\t]({-0.52*7.4*cos(\t r)+0.86*3.54*sin(\t r)},{-0.86*7.4*cos(\t r)+-0.52*3.54*sin(\t r)});
				\end{tikzpicture} \ar[ddr] & & \cdots\\
				\cdots \ar[dr] & & \begin{tikzpicture}[scale=0.1]
					\fill[fill=black,fill opacity=0.15] (0,4) -- (0,-4) -- (6.93,-8) -- (13.86,-4) -- (13.86,4) -- (6.93,8) -- cycle;
					\draw [fill=black,fill opacity=1.0] (3.46,2) circle (0.4cm);
					\draw [fill=black,fill opacity=1.0] (3.46,-2) circle (0.4cm);
					\draw [fill=black,fill opacity=1.0] (10.39,2) circle (0.4cm);
					\draw [fill=black,fill opacity=1.0] (10.39,-2) circle (0.4cm);
					\draw (0,4)-- (0,-4);
					\draw (0,-4)-- (6.93,-8);
					\draw (6.93,-8)-- (13.86,-4);
					\draw (13.86,-4)-- (13.86,4);
					\draw (13.86,4)-- (6.93,8);
					\draw (6.93,8)-- (0,4);
					\draw [shift={(1.86,0)}] plot[domain=-0.89:0.89,variable=\t]({1*2.57*cos(\t r)+0*2.57*sin(\t r)},{0*2.57*cos(\t r)+1*2.57*sin(\t r)});
					\draw [shift={(5.07,0)}] plot[domain=2.25:4.04,variable=\t]({1*2.57*cos(\t r)+0*2.57*sin(\t r)},{0*2.57*cos(\t r)+1*2.57*sin(\t r)});
					\draw [shift={(12,0)}] plot[domain=2.25:4.04,variable=\t]({1*2.57*cos(\t r)+0*2.57*sin(\t r)},{0*2.57*cos(\t r)+1*2.57*sin(\t r)});
					\draw [shift={(8.79,0)}] plot[domain=-0.89:0.89,variable=\t]({1*2.57*cos(\t r)+0*2.57*sin(\t r)},{0*2.57*cos(\t r)+1*2.57*sin(\t r)});
					\draw [shift={(5.68,-3.66)}] plot[domain=-0.47:1.77,variable=\t]({0.8*8.37*cos(\t r)+0.6*4.29*sin(\t r)},{0.6*8.37*cos(\t r)+-0.8*4.29*sin(\t r)});
				\end{tikzpicture} \ar[dr] & & \begin{tikzpicture}[scale=0.1]
					\fill[fill=black,fill opacity=0.15] (0,4) -- (0,-4) -- (6.93,-8) -- (13.86,-4) -- (13.86,4) -- (6.93,8) -- cycle;
					\draw [fill=black,fill opacity=1.0] (3.46,2) circle (0.4cm);
					\draw [fill=black,fill opacity=1.0] (3.46,-2) circle (0.4cm);
					\draw [fill=black,fill opacity=1.0] (10.39,2) circle (0.4cm);
					\draw [fill=black,fill opacity=1.0] (10.39,-2) circle (0.4cm);
					\draw (0,4)-- (0,-4);
					\draw (0,-4)-- (6.93,-8);
					\draw (6.93,-8)-- (13.86,-4);
					\draw (13.86,-4)-- (13.86,4);
					\draw (13.86,4)-- (6.93,8);
					\draw (6.93,8)-- (0,4);
					\draw [shift={(1.86,0)}] plot[domain=-0.89:0.89,variable=\t]({1*2.57*cos(\t r)+0*2.57*sin(\t r)},{0*2.57*cos(\t r)+1*2.57*sin(\t r)});
					\draw [shift={(5.07,0)}] plot[domain=2.25:4.04,variable=\t]({1*2.57*cos(\t r)+0*2.57*sin(\t r)},{0*2.57*cos(\t r)+1*2.57*sin(\t r)});
					\draw [shift={(12,0)}] plot[domain=2.25:4.04,variable=\t]({1*2.57*cos(\t r)+0*2.57*sin(\t r)},{0*2.57*cos(\t r)+1*2.57*sin(\t r)});
					\draw [shift={(8.79,0)}] plot[domain=-0.89:0.89,variable=\t]({1*2.57*cos(\t r)+0*2.57*sin(\t r)},{0*2.57*cos(\t r)+1*2.57*sin(\t r)});
					\draw [shift={(7.87,2.24)}] plot[domain=-2.01:1.41,variable=\t]({-0.71*8.01*cos(\t r)+0.7*5.25*sin(\t r)},{-0.7*8.01*cos(\t r)+-0.71*5.25*sin(\t r)});
				\end{tikzpicture} \ar[dr] & & \cdots\\
				& \begin{tikzpicture}[scale=0.1]
					\fill[fill=black,fill opacity=0.15] (0,4) -- (0,-4) -- (6.93,-8) -- (13.86,-4) -- (13.86,4) -- (6.93,8) -- cycle;
					\draw [fill=black,fill opacity=1.0] (3.46,2) circle (0.4cm);
					\draw [fill=black,fill opacity=1.0] (3.46,-2) circle (0.4cm);
					\draw [fill=black,fill opacity=1.0] (10.39,2) circle (0.4cm);
					\draw [fill=black,fill opacity=1.0] (10.39,-2) circle (0.4cm);
					\draw (0,4)-- (0,-4);
					\draw (0,-4)-- (6.93,-8);
					\draw (6.93,-8)-- (13.86,-4);
					\draw (13.86,-4)-- (13.86,4);
					\draw (13.86,4)-- (6.93,8);
					\draw (6.93,8)-- (0,4);
					\draw [shift={(1.86,0)}] plot[domain=-0.89:0.89,variable=\t]({1*2.57*cos(\t r)+0*2.57*sin(\t r)},{0*2.57*cos(\t r)+1*2.57*sin(\t r)});
					\draw [shift={(5.07,0)}] plot[domain=2.25:4.04,variable=\t]({1*2.57*cos(\t r)+0*2.57*sin(\t r)},{0*2.57*cos(\t r)+1*2.57*sin(\t r)});
					\draw [shift={(12,0)}] plot[domain=2.25:4.04,variable=\t]({1*2.57*cos(\t r)+0*2.57*sin(\t r)},{0*2.57*cos(\t r)+1*2.57*sin(\t r)});
					\draw [shift={(8.79,0)}] plot[domain=-0.89:0.89,variable=\t]({1*2.57*cos(\t r)+0*2.57*sin(\t r)},{0*2.57*cos(\t r)+1*2.57*sin(\t r)});
					\draw (6.93,8)-- (6.93,-8);
				\end{tikzpicture}
				\ar[uur] \ar[ur] \ar[dr] \ar[ddr] & & \begin{tikzpicture}[scale=0.1]
					\fill[fill=black,fill opacity=0.15] (0,4) -- (0,-4) -- (6.93,-8) -- (13.86,-4) -- (13.86,4) -- (6.93,8) -- cycle;
					\draw [fill=black,fill opacity=1.0] (3.46,2) circle (0.4cm);
					\draw [fill=black,fill opacity=1.0] (3.46,-2) circle (0.4cm);
					\draw [fill=black,fill opacity=1.0] (10.39,2) circle (0.4cm);
					\draw [fill=black,fill opacity=1.0] (10.39,-2) circle (0.4cm);
					\draw (0,4)-- (0,-4);
					\draw (0,-4)-- (6.93,-8);
					\draw (6.93,-8)-- (13.86,-4);
					\draw (13.86,-4)-- (13.86,4);
					\draw (13.86,4)-- (6.93,8);
					\draw (6.93,8)-- (0,4);
					\draw [shift={(1.86,0)}] plot[domain=-0.89:0.89,variable=\t]({1*2.57*cos(\t r)+0*2.57*sin(\t r)},{0*2.57*cos(\t r)+1*2.57*sin(\t r)});
					\draw [shift={(5.07,0)}] plot[domain=2.25:4.04,variable=\t]({1*2.57*cos(\t r)+0*2.57*sin(\t r)},{0*2.57*cos(\t r)+1*2.57*sin(\t r)});
					\draw [shift={(12,0)}] plot[domain=2.25:4.04,variable=\t]({1*2.57*cos(\t r)+0*2.57*sin(\t r)},{0*2.57*cos(\t r)+1*2.57*sin(\t r)});
					\draw [shift={(8.79,0)}] plot[domain=-0.89:0.89,variable=\t]({1*2.57*cos(\t r)+0*2.57*sin(\t r)},{0*2.57*cos(\t r)+1*2.57*sin(\t r)});
					\draw [shift={(8.71,-2.14)}] plot[domain=0.48:4.56,variable=\t]({-0.48*6.78*cos(\t r)+0.88*2.61*sin(\t r)},{-0.88*6.78*cos(\t r)+-0.48*2.61*sin(\t r)});
					\draw [shift={(5.15,2.14)}] plot[domain=0.48:4.56,variable=\t]({0.48*6.78*cos(\t r)+-0.88*2.61*sin(\t r)},{0.88*6.78*cos(\t r)+0.48*2.61*sin(\t r)});
				\end{tikzpicture} \ar[uur] \ar[ur] \ar[dr] \ar[ddr] & & \begin{tikzpicture}[scale=0.1]
					\fill[fill=black,fill opacity=0.15] (0,4) -- (0,-4) -- (6.93,-8) -- (13.86,-4) -- (13.86,4) -- (6.93,8) -- cycle;
					\draw [fill=black,fill opacity=1.0] (3.46,2) circle (0.4cm);
					\draw [fill=black,fill opacity=1.0] (3.46,-2) circle (0.4cm);
					\draw [fill=black,fill opacity=1.0] (10.39,2) circle (0.4cm);
					\draw [fill=black,fill opacity=1.0] (10.39,-2) circle (0.4cm);
					\draw (0,4)-- (0,-4);
					\draw (0,-4)-- (6.93,-8);
					\draw (6.93,-8)-- (13.86,-4);
					\draw (13.86,-4)-- (13.86,4);
					\draw (13.86,4)-- (6.93,8);
					\draw (6.93,8)-- (0,4);
					\draw [shift={(1.86,0)}] plot[domain=-0.89:0.89,variable=\t]({1*2.57*cos(\t r)+0*2.57*sin(\t r)},{0*2.57*cos(\t r)+1*2.57*sin(\t r)});
					\draw [shift={(5.07,0)}] plot[domain=2.25:4.04,variable=\t]({1*2.57*cos(\t r)+0*2.57*sin(\t r)},{0*2.57*cos(\t r)+1*2.57*sin(\t r)});
					\draw [shift={(12,0)}] plot[domain=2.25:4.04,variable=\t]({1*2.57*cos(\t r)+0*2.57*sin(\t r)},{0*2.57*cos(\t r)+1*2.57*sin(\t r)});
					\draw [shift={(8.79,0)}] plot[domain=-0.89:0.89,variable=\t]({1*2.57*cos(\t r)+0*2.57*sin(\t r)},{0*2.57*cos(\t r)+1*2.57*sin(\t r)});
					\draw [shift={(5.43,-2.31)}] plot[domain=1.26:4.51,variable=\t]({-0.75*8.75*cos(\t r)+0.66*5.52*sin(\t r)},{-0.66*8.75*cos(\t r)+-0.75*5.52*sin(\t r)});
					\draw [shift={(4.35,0.43)}] plot[domain=-2.04:1.85,variable=\t]({-0.51*3.43*cos(\t r)+0.86*2.54*sin(\t r)},{-0.86*3.43*cos(\t r)+-0.51*2.54*sin(\t r)});
					\draw [shift={(9.51,-0.43)}] plot[domain=-2.04:1.85,variable=\t]({0.51*3.43*cos(\t r)+-0.86*2.54*sin(\t r)},{0.86*3.43*cos(\t r)+0.51*2.54*sin(\t r)});
					\draw [shift={(8.43,2.31)}] plot[domain=1.26:4.51,variable=\t]({0.75*8.75*cos(\t r)+-0.66*5.52*sin(\t r)},{0.66*8.75*cos(\t r)+0.75*5.52*sin(\t r)});
				\end{tikzpicture} \ar[uur] \ar[ur] \ar[dr] \ar[ddr] & \\
				\cdots \ar[ur] & & \begin{tikzpicture}[scale=0.1]
					\fill[fill=black,fill opacity=0.15] (0,4) -- (0,-4) -- (6.93,-8) -- (13.86,-4) -- (13.86,4) -- (6.93,8) -- cycle;
					\draw [fill=black,fill opacity=1.0] (3.46,2) circle (0.4cm);
					\draw [fill=black,fill opacity=1.0] (3.46,-2) circle (0.4cm);
					\draw [fill=black,fill opacity=1.0] (10.39,2) circle (0.4cm);
					\draw [fill=black,fill opacity=1.0] (10.39,-2) circle (0.4cm);
					\draw (0,4)-- (0,-4);
					\draw (0,-4)-- (6.93,-8);
					\draw (6.93,-8)-- (13.86,-4);
					\draw (13.86,-4)-- (13.86,4);
					\draw (13.86,4)-- (6.93,8);
					\draw (6.93,8)-- (0,4);
					\draw [shift={(1.86,0)}] plot[domain=-0.89:0.89,variable=\t]({1*2.57*cos(\t r)+0*2.57*sin(\t r)},{0*2.57*cos(\t r)+1*2.57*sin(\t r)});
					\draw [shift={(5.07,0)}] plot[domain=2.25:4.04,variable=\t]({1*2.57*cos(\t r)+0*2.57*sin(\t r)},{0*2.57*cos(\t r)+1*2.57*sin(\t r)});
					\draw [shift={(12,0)}] plot[domain=2.25:4.04,variable=\t]({1*2.57*cos(\t r)+0*2.57*sin(\t r)},{0*2.57*cos(\t r)+1*2.57*sin(\t r)});
					\draw [shift={(8.79,0)}] plot[domain=-0.89:0.89,variable=\t]({1*2.57*cos(\t r)+0*2.57*sin(\t r)},{0*2.57*cos(\t r)+1*2.57*sin(\t r)});
					\draw [shift={(8.18,3.66)}] plot[domain=-0.47:1.77,variable=\t]({-0.8*8.37*cos(\t r)+-0.6*4.29*sin(\t r)},{-0.6*8.37*cos(\t r)+0.8*4.29*sin(\t r)});
				\end{tikzpicture} \ar[ur] & & \begin{tikzpicture}[scale=0.1]
					\fill[fill=black,fill opacity=0.15] (0,4) -- (0,-4) -- (6.93,-8) -- (13.86,-4) -- (13.86,4) -- (6.93,8) -- cycle;
					\draw [fill=black,fill opacity=1.0] (3.46,2) circle (0.4cm);
					\draw [fill=black,fill opacity=1.0] (3.46,-2) circle (0.4cm);
					\draw [fill=black,fill opacity=1.0] (10.39,2) circle (0.4cm);
					\draw [fill=black,fill opacity=1.0] (10.39,-2) circle (0.4cm);
					\draw (0,4)-- (0,-4);
					\draw (0,-4)-- (6.93,-8);
					\draw (6.93,-8)-- (13.86,-4);
					\draw (13.86,-4)-- (13.86,4);
					\draw (13.86,4)-- (6.93,8);
					\draw (6.93,8)-- (0,4);
					\draw [shift={(1.86,0)}] plot[domain=-0.89:0.89,variable=\t]({1*2.57*cos(\t r)+0*2.57*sin(\t r)},{0*2.57*cos(\t r)+1*2.57*sin(\t r)});
					\draw [shift={(5.07,0)}] plot[domain=2.25:4.04,variable=\t]({1*2.57*cos(\t r)+0*2.57*sin(\t r)},{0*2.57*cos(\t r)+1*2.57*sin(\t r)});
					\draw [shift={(12,0)}] plot[domain=2.25:4.04,variable=\t]({1*2.57*cos(\t r)+0*2.57*sin(\t r)},{0*2.57*cos(\t r)+1*2.57*sin(\t r)});
					\draw [shift={(8.79,0)}] plot[domain=-0.89:0.89,variable=\t]({1*2.57*cos(\t r)+0*2.57*sin(\t r)},{0*2.57*cos(\t r)+1*2.57*sin(\t r)});
					\draw [shift={(5.99,-2.24)}] plot[domain=-2.01:1.41,variable=\t]({0.71*8.01*cos(\t r)+-0.7*5.25*sin(\t r)},{0.7*8.01*cos(\t r)+0.71*5.25*sin(\t r)});
				\end{tikzpicture} \ar[ur] & & \cdots\\
				\cdots \ar[uur] & & \begin{tikzpicture}[scale=0.1]
					\fill[fill=black,fill opacity=0.15] (0,4) -- (0,-4) -- (6.93,-8) -- (13.86,-4) -- (13.86,4) -- (6.93,8) -- cycle;
					\draw [fill=black,fill opacity=1.0] (3.46,2) circle (0.4cm);
					\draw [fill=black,fill opacity=1.0] (3.46,-2) circle (0.4cm);
					\draw [fill=black,fill opacity=1.0] (10.39,2) circle (0.4cm);
					\draw [fill=black,fill opacity=1.0] (10.39,-2) circle (0.4cm);
					\draw (0,4)-- (0,-4);
					\draw (0,-4)-- (6.93,-8);
					\draw (6.93,-8)-- (13.86,-4);
					\draw (13.86,-4)-- (13.86,4);
					\draw (13.86,4)-- (6.93,8);
					\draw (6.93,8)-- (0,4);
					\draw [shift={(1.86,0)}] plot[domain=-0.89:0.89,variable=\t]({1*2.57*cos(\t r)+0*2.57*sin(\t r)},{0*2.57*cos(\t r)+1*2.57*sin(\t r)});
					\draw [shift={(5.07,0)}] plot[domain=2.25:4.04,variable=\t]({1*2.57*cos(\t r)+0*2.57*sin(\t r)},{0*2.57*cos(\t r)+1*2.57*sin(\t r)});
					\draw [shift={(12,0)}] plot[domain=2.25:4.04,variable=\t]({1*2.57*cos(\t r)+0*2.57*sin(\t r)},{0*2.57*cos(\t r)+1*2.57*sin(\t r)});
					\draw [shift={(8.79,0)}] plot[domain=-0.89:0.89,variable=\t]({1*2.57*cos(\t r)+0*2.57*sin(\t r)},{0*2.57*cos(\t r)+1*2.57*sin(\t r)});
					\draw [shift={(-11.62,8.09)}] plot[domain=5.69:6.28,variable=\t]({1*18.55*cos(\t r)+0*18.55*sin(\t r)},{0*18.55*cos(\t r)+1*18.55*sin(\t r)});
				\end{tikzpicture} \ar[uur] & & \begin{tikzpicture}[scale=0.1]
					\fill[fill=black,fill opacity=0.15] (0,4) -- (0,-4) -- (6.93,-8) -- (13.86,-4) -- (13.86,4) -- (6.93,8) -- cycle;
					\draw [fill=black,fill opacity=1.0] (3.46,2) circle (0.4cm);
					\draw [fill=black,fill opacity=1.0] (3.46,-2) circle (0.4cm);
					\draw [fill=black,fill opacity=1.0] (10.39,2) circle (0.4cm);
					\draw [fill=black,fill opacity=1.0] (10.39,-2) circle (0.4cm);
					\draw (0,4)-- (0,-4);
					\draw (0,-4)-- (6.93,-8);
					\draw (6.93,-8)-- (13.86,-4);
					\draw (13.86,-4)-- (13.86,4);
					\draw (13.86,4)-- (6.93,8);
					\draw (6.93,8)-- (0,4);
					\draw [shift={(1.86,0)}] plot[domain=-0.89:0.89,variable=\t]({1*2.57*cos(\t r)+0*2.57*sin(\t r)},{0*2.57*cos(\t r)+1*2.57*sin(\t r)});
					\draw [shift={(5.07,0)}] plot[domain=2.25:4.04,variable=\t]({1*2.57*cos(\t r)+0*2.57*sin(\t r)},{0*2.57*cos(\t r)+1*2.57*sin(\t r)});
					\draw [shift={(12,0)}] plot[domain=2.25:4.04,variable=\t]({1*2.57*cos(\t r)+0*2.57*sin(\t r)},{0*2.57*cos(\t r)+1*2.57*sin(\t r)});
					\draw [shift={(8.79,0)}] plot[domain=-0.89:0.89,variable=\t]({1*2.57*cos(\t r)+0*2.57*sin(\t r)},{0*2.57*cos(\t r)+1*2.57*sin(\t r)});
					\draw [shift={(7.74,-2.14)}] plot[domain=0.65:5,variable=\t]({-0.49*6.93*cos(\t r)+0.87*3.55*sin(\t r)},{-0.87*6.93*cos(\t r)+-0.49*3.55*sin(\t r)});
				\end{tikzpicture} \ar[uur] & & \cdots
		}}
		\]
		\caption{A subquiver of the Auslander-Reiten quiver of $\tilde{Q}$ with vertices $m$-diagonals, for $m=3$ and $n=4$}
		\label{fig:ar1}
	\end{figure}
\end{landscape}

We can find on a slice of the Auslander-Reiten quiver of $\mathcal{C}$ the following initial $(m+2)$-angulation. In fact, each arc of the initial $(m+2)$-angulation corresponds to the image in the $m$-cluster category of a preprojective indecomposable representation.

\begin{rmk}
	We note that, on the previous pictures, we have only drawn part of the first transjective component of the Auslander-Reiten quiver of $\mathcal{C}$, but there are $m$ copies of it, where we can find the $m$-diagonals shifted $1, \cdots, m-1$ times (see definition \ref{def:shift}). Indeed, between the tubes, there are $m$ copies of the preinjective and preprojective components. On figure \ref{fig:preproj} we can see how the preprojective, regular and preinjective components of a Auslander-Reiten quiver of type $\tilde{D_n}$ are set. We can see the first shift of the preprojective objects, and also the first shift of the regular objects.
\end{rmk}

\begin{figure}[!h]
	\begin{tikzpicture}[scale=0.2]
		\draw (-4,2)-- (-4,-2);
		\draw [dotted] (0.54,2)-- (3.3,2);
		\draw (-4,-2)-- (0.54,-2);
		\draw [dotted] (0.54,-2)-- (3.3,-2);
		\draw (9.64,7.49)-- (9.64,-5.78);
		\draw (9.64,-5.78)-- (11.91,-5.78);
		\draw (15.17,-5.78)-- (17.44,-5.78);
		\draw (9.64,-5.78)-- (9.64,7.49);
		\draw (11.91,-5.78)-- (11.91,7.49);
		\draw (15.17,-5.78)-- (15.17,7.49);
		\draw (17.44,-5.78)-- (17.44,7.49);
		\draw (-4,2)-- (0.54,2);
		\draw (-3.74,-7.62) node[anchor=north west] {preprojective};
		\draw (14.58,-7.93) node[anchor=north west] {regular};
		\draw (28.59,-8.15) node[anchor=north west] {preinj};
		\draw (20.7,-5.78)-- (20.7,7.49);
		\draw (22.97,-5.78)-- (22.97,7.49);
		\draw (22.97,-5.78)-- (20.7,-5.78);
		\draw (36.61,2)-- (32.07,2);
		\draw [dotted] (32.07,2)-- (29.31,2);
		\draw (36.61,2)-- (36.61,-2);
		\draw (36.61,-2)-- (32.07,-2);
		\draw [dotted] (32.07,-2)-- (29.31,-2);
		\draw (36.61,2)-- (41.15,2);
		\draw [dotted] (41.15,2)-- (43.91,2);
		\draw [dotted] (41.15,-2)-- (43.91,-2);
		\draw (36.61,-2)-- (41.15,-2);
		\draw (50.25,-5.78)-- (50.25,7.49);
		\draw (50.25,-5.78)-- (52.52,-5.78);
		\draw (52.52,-5.78)-- (52.52,7.49);
		\draw (55.78,-5.78)-- (55.78,7.49);
		\draw (58.05,-5.78)-- (58.05,7.49);
		\draw (58.05,-5.78)-- (55.78,-5.78);
		\draw (61.3,-5.78)-- (61.3,7.49);
		\draw (63.57,-5.78)-- (61.3,-5.78);
		\draw (63.57,7.49)-- (63.57,-5.78);
		\draw [dash pattern=on 9pt off 9pt] (69.33,0)-- (73.55,0);
		\draw (37.44,-8.15) node[anchor=north west] {first shift of preproj};
		\draw (55.16,-8.15) node[anchor=north west] {first shift of regular};
	\end{tikzpicture}
	\caption{}
	\label{fig:preproj}
\end{figure}

\subsection{Mesh relations}
It is possible to make use of the $m$-diagonals so as to easily picture mesh relations. There are several types of mesh relations that can thus be described depending on the number of paths from $\tau \alpha$ to $\alpha$.

The first type of relation is of the following form:

\[ \xymatrix@1{
	\begin{tikzpicture}[scale=0.3]
		\fill[fill=black,fill opacity=0.1] (0,4) -- (-1.66,3.46) -- (-2.69,2.05) -- (-2.68,0.3) -- (-1.66,-1.11) -- (0,-1.65) -- (1.66,-1.11) -- (2.69,0.3) -- (2.69,2.05) -- (1.66,3.46) -- cycle;
		\draw [fill=black,fill opacity=1.0] (1.66,1.61) circle (0.2cm);
		\draw [fill=black,fill opacity=1.0] (-1.66,1.61) circle (0.2cm);
		\draw (0,4)-- (-1.66,3.46);
		\draw (-1.66,3.46)-- (-2.69,2.05);
		\draw (-2.69,2.05)-- (-2.68,0.3);
		\draw (-2.68,0.3)-- (-1.66,-1.11);
		\draw (-1.66,-1.11)-- (0,-1.65);
		\draw (0,-1.65)-- (1.66,-1.11);
		\draw (1.66,-1.11)-- (2.69,0.3);
		\draw (2.69,0.3)-- (2.69,2.05);
		\draw (2.69,2.05)-- (1.66,3.46);
		\draw (1.66,3.46)-- (0,4);
		\draw(1.67,1.18) circle (0.44cm);
		\draw(-1.67,1.17) circle (0.44cm);
		\draw [shift={(-2.52,-0.42)}] plot[domain=-0.43:1.16,variable=\t]({0.96*5.16*cos(\t r)+0.3*2.09*sin(\t r)},{0.3*5.16*cos(\t r)+-0.96*2.09*sin(\t r)});
	\end{tikzpicture} \ar[dr] & & \begin{tikzpicture}[scale=0.3]
		\fill[fill=black,fill opacity=0.1] (0,4) -- (-1.66,3.46) -- (-2.69,2.05) -- (-2.68,0.3) -- (-1.66,-1.11) -- (0,-1.65) -- (1.66,-1.11) -- (2.69,0.3) -- (2.69,2.05) -- (1.66,3.46) -- cycle;
		\draw [fill=black,fill opacity=1.0] (1.66,1.61) circle (0.2cm);
		\draw [fill=black,fill opacity=1.0] (-1.66,1.61) circle (0.2cm);
		\draw (0,4)-- (-1.66,3.46);
		\draw (-1.66,3.46)-- (-2.69,2.05);
		\draw (-2.69,2.05)-- (-2.68,0.3);
		\draw (-2.68,0.3)-- (-1.66,-1.11);
		\draw (-1.66,-1.11)-- (0,-1.65);
		\draw (0,-1.65)-- (1.66,-1.11);
		\draw (1.66,-1.11)-- (2.69,0.3);
		\draw (2.69,0.3)-- (2.69,2.05);
		\draw (2.69,2.05)-- (1.66,3.46);
		\draw (1.66,3.46)-- (0,4);
		\draw(1.67,1.18) circle (0.44cm);
		\draw(-1.67,1.17) circle (0.44cm);
		\draw [shift={(-0.61,1.16)}] plot[domain=0.92:3.89,variable=\t]({-0.98*3.04*cos(\t r)+-0.21*1.6*sin(\t r)},{0.21*3.04*cos(\t r)+-0.98*1.6*sin(\t r)});
	\end{tikzpicture} \\
	& \begin{tikzpicture}[scale=0.3]
		\fill[fill=black,fill opacity=0.1] (0,4) -- (-1.66,3.46) -- (-2.69,2.05) -- (-2.68,0.3) -- (-1.66,-1.11) -- (0,-1.65) -- (1.66,-1.11) -- (2.69,0.3) -- (2.69,2.05) -- (1.66,3.46) -- cycle;
		\draw [fill=black,fill opacity=1.0] (1.66,1.61) circle (0.2cm);
		\draw [fill=black,fill opacity=1.0] (-1.66,1.61) circle (0.2cm);
		\draw (0,4)-- (-1.66,3.46);
		\draw (-1.66,3.46)-- (-2.69,2.05);
		\draw (-2.69,2.05)-- (-2.68,0.3);
		\draw (-2.68,0.3)-- (-1.66,-1.11);
		\draw (-1.66,-1.11)-- (0,-1.65);
		\draw (0,-1.65)-- (1.66,-1.11);
		\draw (1.66,-1.11)-- (2.69,0.3);
		\draw (2.69,0.3)-- (2.69,2.05);
		\draw (2.69,2.05)-- (1.66,3.46);
		\draw (1.66,3.46)-- (0,4);
		\draw(1.67,1.18) circle (0.44cm);
		\draw(-1.67,1.17) circle (0.44cm);
		\draw [shift={(0.51,0.16)}] plot[domain=0.83:4.74,variable=\t]({-0.67*2.51*cos(\t r)+0.75*1.12*sin(\t r)},{-0.75*2.51*cos(\t r)+-0.67*1.12*sin(\t r)});
		\draw [shift={(-1.94,2.31)}] plot[domain=4.36:5.54,variable=\t]({1*2.14*cos(\t r)+0*2.14*sin(\t r)},{0*2.14*cos(\t r)+1*2.14*sin(\t r)});
	\end{tikzpicture} \ar[ur] &
} \]

It is easy to see that, if we call by $a$ and $b$ the extremities of both arcs of type different from $1$, then the arc of type $1$ links $a$ to $b$. This is a general case.

The second one is of this shape:

\[\xymatrix{
	& \begin{tikzpicture}[scale=0.3]
		\fill[fill=black,fill opacity=0.1] (0,4) -- (-1.66,3.46) -- (-2.69,2.05) -- (-2.68,0.3) -- (-1.66,-1.11) -- (0,-1.65) -- (1.66,-1.11) -- (2.69,0.3) -- (2.69,2.05) -- (1.66,3.46) -- cycle;
		\draw [fill=black,fill opacity=1.0] (1.66,1.61) circle (0.2cm);
		\draw [fill=black,fill opacity=1.0] (-1.66,1.61) circle (0.2cm);
		\draw (0,4)-- (-1.66,3.46);
		\draw (-1.66,3.46)-- (-2.69,2.05);
		\draw (-2.69,2.05)-- (-2.68,0.3);
		\draw (-2.68,0.3)-- (-1.66,-1.11);
		\draw (-1.66,-1.11)-- (0,-1.65);
		\draw (0,-1.65)-- (1.66,-1.11);
		\draw (1.66,-1.11)-- (2.69,0.3);
		\draw (2.69,0.3)-- (2.69,2.05);
		\draw (2.69,2.05)-- (1.66,3.46);
		\draw (1.66,3.46)-- (0,4);
		\draw [shift={(8.03,-1.43)}] plot[domain=2.49:3.17,variable=\t]({1*8.03*cos(\t r)+0*8.03*sin(\t r)},{0*8.03*cos(\t r)+1*8.03*sin(\t r)});
		\draw(1.67,1.18) circle (0.44cm);
		\draw(-1.67,1.17) circle (0.44cm);
	\end{tikzpicture} \ar[dr] & \\
	\begin{tikzpicture}[scale=0.3]
		\fill[fill=black,fill opacity=0.1] (0,4) -- (-1.66,3.46) -- (-2.69,2.05) -- (-2.68,0.3) -- (-1.66,-1.11) -- (0,-1.65) -- (1.66,-1.11) -- (2.69,0.3) -- (2.69,2.05) -- (1.66,3.46) -- cycle;
		\draw [fill=black,fill opacity=1.0] (1.66,1.61) circle (0.2cm);
		\draw [fill=black,fill opacity=1.0] (-1.66,1.61) circle (0.2cm);
		\draw (0,4)-- (-1.66,3.46);
		\draw (-1.66,3.46)-- (-2.69,2.05);
		\draw (-2.69,2.05)-- (-2.68,0.3);
		\draw (-2.68,0.3)-- (-1.66,-1.11);
		\draw (-1.66,-1.11)-- (0,-1.65);
		\draw (0,-1.65)-- (1.66,-1.11);
		\draw (1.66,-1.11)-- (2.69,0.3);
		\draw (2.69,0.3)-- (2.69,2.05);
		\draw (2.69,2.05)-- (1.66,3.46);
		\draw (1.66,3.46)-- (0,4);
		\draw [shift={(-8.03,-1.44)}] plot[domain=-0.03:0.66,variable=\t]({1*8.03*cos(\t r)+0*8.03*sin(\t r)},{0*8.03*cos(\t r)+1*8.03*sin(\t r)});
		\draw(1.67,1.18) circle (0.44cm);
		\draw(-1.67,1.17) circle (0.44cm);
	\end{tikzpicture} \ar[ur] \ar[dr] & & \begin{tikzpicture}[scale=0.3]
		\fill[fill=black,fill opacity=0.1] (0,4) -- (-1.66,3.46) -- (-2.69,2.05) -- (-2.68,0.3) -- (-1.66,-1.11) -- (0,-1.65) -- (1.66,-1.11) -- (2.69,0.3) -- (2.69,2.05) -- (1.66,3.46) -- cycle;
		\draw [fill=black,fill opacity=1.0] (1.66,1.61) circle (0.2cm);
		\draw [fill=black,fill opacity=1.0] (-1.66,1.61) circle (0.2cm);
		\draw (0,4)-- (-1.66,3.46);
		\draw (-1.66,3.46)-- (-2.69,2.05);
		\draw (-2.69,2.05)-- (-2.68,0.3);
		\draw (-2.68,0.3)-- (-1.66,-1.11);
		\draw (-1.66,-1.11)-- (0,-1.65);
		\draw (0,-1.65)-- (1.66,-1.11);
		\draw (1.66,-1.11)-- (2.69,0.3);
		\draw (2.69,0.3)-- (2.69,2.05);
		\draw (2.69,2.05)-- (1.66,3.46);
		\draw (1.66,3.46)-- (0,4);
		\draw(1.67,1.18) circle (0.44cm);
		\draw(-1.67,1.17) circle (0.44cm);
		\draw [shift={(-2.18,1.68)}] plot[domain=4.36:5.86,variable=\t]({1*1.47*cos(\t r)+0*1.47*sin(\t r)},{0*1.47*cos(\t r)+1*1.47*sin(\t r)});
		\draw [shift={(3.22,-0.67)}] plot[domain=1.93:2.74,variable=\t]({1*4.42*cos(\t r)+0*4.42*sin(\t r)},{0*4.42*cos(\t r)+1*4.42*sin(\t r)});
	\end{tikzpicture} \\
	& \begin{tikzpicture}[scale=0.3]
		\fill[fill=black,fill opacity=0.1] (0,4) -- (-1.66,3.46) -- (-2.69,2.05) -- (-2.68,0.3) -- (-1.66,-1.11) -- (0,-1.65) -- (1.66,-1.11) -- (2.69,0.3) -- (2.69,2.05) -- (1.66,3.46) -- cycle;
		\draw [fill=black,fill opacity=1.0] (1.66,1.61) circle (0.2cm);
		\draw [fill=black,fill opacity=1.0] (-1.66,1.61) circle (0.2cm);
		\draw (0,4)-- (-1.66,3.46);
		\draw (-1.66,3.46)-- (-2.69,2.05);
		\draw (-2.69,2.05)-- (-2.68,0.3);
		\draw (-2.68,0.3)-- (-1.66,-1.11);
		\draw (-1.66,-1.11)-- (0,-1.65);
		\draw (0,-1.65)-- (1.66,-1.11);
		\draw (1.66,-1.11)-- (2.69,0.3);
		\draw (2.69,0.3)-- (2.69,2.05);
		\draw (2.69,2.05)-- (1.66,3.46);
		\draw (1.66,3.46)-- (0,4);
		\draw(1.67,1.18) circle (0.44cm);
		\draw(-1.67,1.17) circle (0.44cm);
		\draw [shift={(-2.18,1.68)}] plot[domain=4.36:5.86,variable=\t]({1*1.47*cos(\t r)+0*1.47*sin(\t r)},{0*1.47*cos(\t r)+1*1.47*sin(\t r)});
		\draw [shift={(-2.59,1.81)}] plot[domain=-0.4:1.06,variable=\t]({1*1.9*cos(\t r)+0*1.9*sin(\t r)},{0*1.9*cos(\t r)+1*1.9*sin(\t r)});
	\end{tikzpicture} \ar[ur]}\]

The third one is of this shape:

\[ \scalebox{1.0}{
	\xymatrix{
		& \begin{tikzpicture}[scale=0.3]
			\fill[fill=black,fill opacity=0.1] (0,4) -- (-1.66,3.46) -- (-2.69,2.05) -- (-2.68,0.3) -- (-1.66,-1.11) -- (0,-1.65) -- (1.66,-1.11) -- (2.69,0.3) -- (2.69,2.05) -- (1.66,3.46) -- cycle;
			\draw [fill=black,fill opacity=1.0] (1.66,1.61) circle (0.2cm);
			\draw [fill=black,fill opacity=1.0] (-1.66,1.61) circle (0.2cm);
			\draw (0,4)-- (-1.66,3.46);
			\draw (-1.66,3.46)-- (-2.69,2.05);
			\draw (-2.69,2.05)-- (-2.68,0.3);
			\draw (-2.68,0.3)-- (-1.66,-1.11);
			\draw (-1.66,-1.11)-- (0,-1.65);
			\draw (0,-1.65)-- (1.66,-1.11);
			\draw (1.66,-1.11)-- (2.69,0.3);
			\draw (2.69,0.3)-- (2.69,2.05);
			\draw (2.69,2.05)-- (1.66,3.46);
			\draw (1.66,3.46)-- (0,4);
			\draw(1.67,1.18) circle (0.44cm);
			\draw(-1.67,1.17) circle (0.44cm);
			\draw [shift={(-2.52,-0.42)}] plot[domain=-0.43:1.16,variable=\t]({0.96*5.16*cos(\t r)+0.3*2.09*sin(\t r)},{0.3*5.16*cos(\t r)+-0.96*2.09*sin(\t r)});
		\end{tikzpicture} \ar[dr] & \\
		\begin{tikzpicture}[scale=0.3]
			\fill[fill=black,fill opacity=0.1] (0,4) -- (-1.66,3.46) -- (-2.69,2.05) -- (-2.68,0.3) -- (-1.66,-1.11) -- (0,-1.65) -- (1.66,-1.11) -- (2.69,0.3) -- (2.69,2.05) -- (1.66,3.46) -- cycle;
			\draw [fill=black,fill opacity=1.0] (1.66,1.61) circle (0.2cm);
			\draw [fill=black,fill opacity=1.0] (-1.66,1.61) circle (0.2cm);
			\draw (0,4)-- (-1.66,3.46);
			\draw (-1.66,3.46)-- (-2.69,2.05);
			\draw (-2.69,2.05)-- (-2.68,0.3);
			\draw (-2.68,0.3)-- (-1.66,-1.11);
			\draw (-1.66,-1.11)-- (0,-1.65);
			\draw (0,-1.65)-- (1.66,-1.11);
			\draw (1.66,-1.11)-- (2.69,0.3);
			\draw (2.69,0.3)-- (2.69,2.05);
			\draw (2.69,2.05)-- (1.66,3.46);
			\draw (1.66,3.46)-- (0,4);
			\draw(1.67,1.18) circle (0.44cm);
			\draw(-1.67,1.17) circle (0.44cm);
			\draw [shift={(1.35,-0.73)}] plot[domain=-1.27:1.84,variable=\t]({-0.04*3.28*cos(\t r)+-1*1.44*sin(\t r)},{1*3.28*cos(\t r)+-0.04*1.44*sin(\t r)});
		\end{tikzpicture} \ar[dr] \ar[ur] \ar[r] & \begin{tikzpicture}[scale=0.3]
			\fill[fill=black,fill opacity=0.1] (0,4) -- (-1.66,3.46) -- (-2.69,2.05) -- (-2.68,0.3) -- (-1.66,-1.11) -- (0,-1.65) -- (1.66,-1.11) -- (2.69,0.3) -- (2.69,2.05) -- (1.66,3.46) -- cycle;
			\draw [fill=black,fill opacity=1.0] (1.66,1.61) circle (0.2cm);
			\draw [fill=black,fill opacity=1.0] (-1.66,1.61) circle (0.2cm);
			\draw (0,4)-- (-1.66,3.46);
			\draw (-1.66,3.46)-- (-2.69,2.05);
			\draw (-2.69,2.05)-- (-2.68,0.3);
			\draw (-2.68,0.3)-- (-1.66,-1.11);
			\draw (-1.66,-1.11)-- (0,-1.65);
			\draw (0,-1.65)-- (1.66,-1.11);
			\draw (1.66,-1.11)-- (2.69,0.3);
			\draw (2.69,0.3)-- (2.69,2.05);
			\draw (2.69,2.05)-- (1.66,3.46);
			\draw (1.66,3.46)-- (0,4);
			\draw(1.67,1.18) circle (0.44cm);
			\draw(-1.67,1.17) circle (0.44cm);
			\draw [shift={(3.31,-1.08)}] plot[domain=2.13:3.31,variable=\t]({1*3.35*cos(\t r)+0*3.35*sin(\t r)},{0*3.35*cos(\t r)+1*3.35*sin(\t r)});
		\end{tikzpicture} \ar[r] & \begin{tikzpicture}[scale=0.3]
			\fill[fill=black,fill opacity=0.1] (0,4) -- (-1.66,3.46) -- (-2.69,2.05) -- (-2.68,0.3) -- (-1.66,-1.11) -- (0,-1.65) -- (1.66,-1.11) -- (2.69,0.3) -- (2.69,2.05) -- (1.66,3.46) -- cycle;
			\draw [fill=black,fill opacity=1.0] (1.66,1.61) circle (0.2cm);
			\draw [fill=black,fill opacity=1.0] (-1.66,1.61) circle (0.2cm);
			\draw (0,4)-- (-1.66,3.46);
			\draw (-1.66,3.46)-- (-2.69,2.05);
			\draw (-2.69,2.05)-- (-2.68,0.3);
			\draw (-2.68,0.3)-- (-1.66,-1.11);
			\draw (-1.66,-1.11)-- (0,-1.65);
			\draw (0,-1.65)-- (1.66,-1.11);
			\draw (1.66,-1.11)-- (2.69,0.3);
			\draw (2.69,0.3)-- (2.69,2.05);
			\draw (2.69,2.05)-- (1.66,3.46);
			\draw (1.66,3.46)-- (0,4);
			\draw(1.67,1.18) circle (0.44cm);
			\draw(-1.67,1.17) circle (0.44cm);
			\draw [shift={(0.51,0.16)}] plot[domain=0.83:4.74,variable=\t]({-0.67*2.51*cos(\t r)+0.75*1.12*sin(\t r)},{-0.75*2.51*cos(\t r)+-0.67*1.12*sin(\t r)});
			\draw [shift={(-1.94,2.31)}] plot[domain=4.36:5.54,variable=\t]({1*2.14*cos(\t r)+0*2.14*sin(\t r)},{0*2.14*cos(\t r)+1*2.14*sin(\t r)});
		\end{tikzpicture} \\
		& \begin{tikzpicture}[scale=0.3]
			\fill[fill=black,fill opacity=0.1] (0,4) -- (-1.66,3.46) -- (-2.69,2.05) -- (-2.68,0.3) -- (-1.66,-1.11) -- (0,-1.65) -- (1.66,-1.11) -- (2.69,0.3) -- (2.69,2.05) -- (1.66,3.46) -- cycle;
			\draw [fill=black,fill opacity=1.0] (1.66,1.61) circle (0.2cm);
			\draw [fill=black,fill opacity=1.0] (-1.66,1.61) circle (0.2cm);
			\draw (0,4)-- (-1.66,3.46);
			\draw (-1.66,3.46)-- (-2.69,2.05);
			\draw (-2.69,2.05)-- (-2.68,0.3);
			\draw (-2.68,0.3)-- (-1.66,-1.11);
			\draw (-1.66,-1.11)-- (0,-1.65);
			\draw (0,-1.65)-- (1.66,-1.11);
			\draw (1.66,-1.11)-- (2.69,0.3);
			\draw (2.69,0.3)-- (2.69,2.05);
			\draw (2.69,2.05)-- (1.66,3.46);
			\draw (1.66,3.46)-- (0,4);
			\draw(1.67,1.18) circle (0.44cm);
			\draw(-1.67,1.17) circle (0.44cm);
			\draw [shift={(-1.85,1.13)}] plot[domain=3.92:6.01,variable=\t]({1*1.18*cos(\t r)+0*1.18*sin(\t r)},{0*1.18*cos(\t r)+1*1.18*sin(\t r)});
			\draw [shift={(0.96,0.34)}] plot[domain=-0.02:2.87,variable=\t]({1*1.73*cos(\t r)+0*1.73*sin(\t r)},{0*1.73*cos(\t r)+1*1.73*sin(\t r)});
		\end{tikzpicture} \ar[ur] & \\
}}
\]

If $\alpha$ is the $m$-diagonal at the left from $a$ to $c$, then the $m$-diagonal at the right is $\tau \alpha$ from $c$ to $b$. They are of same length $m$.

The upper two pictures in the middle column (or the lower two if we are situated at the bottom of the Auslander-Reiten quiver) on the figures \ref{fig:ar1} and \ref{fig:ar72} are arcs tangent to one side of a central polygon, and the other extremity $b$ is at the vertex shared by both $\alpha$ and $\beta$ (as $l(a,c)=m$ and $\beta=\tau \alpha$, they share one vertex). Finally the lowest (or uppermost) picture is an $m$-diagonal of type 1, and $l(a,b)=2m$. The bases of the tube will be described in the next section.

\subsection{The case of regular modules}

\cor{Now that we have set the category of $m$-diagonals and that we visualize what happens in the transjective component of the Auslander-Reiten quiver, we are interested in the case of non-homogeneous tubes, and their correspondences in terms of arcs. We are going to focus on arcs of type $2$ and $4$. Indeed, these remaining diagonals have the particularity of being cyclic. If we apply $\tau$ several times to them, they come to their first position at a rank $r$. This cannot be the case of the types of $m$-diagonals defined in \ref{def:diago} (because they wrap around the central polygons).}

In figure \ref{fig:tub}, we can see the other types of diagonals. Note that we only have drawn the first ones, but there are also their successive images under $\tau$ (of which there are a finite number because they are cyclic).

\begin{figure}[!h]
	\centering
	\begin{tikzpicture}[scale=0.7]
		\fill[fill=black,fill opacity=0.1] (0,4) -- (-1.66,3.46) -- (-2.69,2.05) -- (-2.68,0.3) -- (-1.66,-1.11) -- (0,-1.65) -- (1.66,-1.11) -- (2.69,0.3) -- (2.69,2.05) -- (1.66,3.46) -- cycle;
		\draw [fill=black,fill opacity=1.0] (1.66,1.61) circle (0.2cm);
		\draw [fill=black,fill opacity=1.0] (-1.66,1.61) circle (0.2cm);
		\draw (0,4)-- (-1.66,3.46);
		\draw (-1.66,3.46)-- (-2.69,2.05);
		\draw (-2.69,2.05)-- (-2.68,0.3);
		\draw (-2.68,0.3)-- (-1.66,-1.11);
		\draw (-1.66,-1.11)-- (0,-1.65);
		\draw (0,-1.65)-- (1.66,-1.11);
		\draw (1.66,-1.11)-- (2.69,0.3);
		\draw (2.69,0.3)-- (2.69,2.05);
		\draw (2.69,2.05)-- (1.66,3.46);
		\draw (1.66,3.46)-- (0,4);
		\draw(1.67,1.18) circle (0.44cm);
		\draw(-1.67,1.17) circle (0.44cm);
		\draw [shift={(-0.68,-0.91)}] plot[domain=0.72:1.79,variable=\t]({1*4.48*cos(\t r)+0*4.48*sin(\t r)},{0*4.48*cos(\t r)+1*4.48*sin(\t r)});
	\end{tikzpicture}
	\hspace{20pt}
	\begin{tikzpicture}[scale=0.7]
		\fill[fill=black,fill opacity=0.1] (0,4) -- (-1.66,3.46) -- (-2.69,2.05) -- (-2.68,0.3) -- (-1.66,-1.11) -- (0,-1.65) -- (1.66,-1.11) -- (2.69,0.3) -- (2.69,2.05) -- (1.66,3.46) -- cycle;
		\draw [fill=black,fill opacity=1.0] (1.66,1.61) circle (0.2cm);
		\draw [fill=black,fill opacity=1.0] (-1.66,1.61) circle (0.2cm);
		\draw (0,4)-- (-1.66,3.46);
		\draw (-1.66,3.46)-- (-2.69,2.05);
		\draw (-2.69,2.05)-- (-2.68,0.3);
		\draw (-2.68,0.3)-- (-1.66,-1.11);
		\draw (-1.66,-1.11)-- (0,-1.65);
		\draw (0,-1.65)-- (1.66,-1.11);
		\draw (1.66,-1.11)-- (2.69,0.3);
		\draw (2.69,0.3)-- (2.69,2.05);
		\draw (2.69,2.05)-- (1.66,3.46);
		\draw (1.66,3.46)-- (0,4);
		\draw(1.67,1.18) circle (0.44cm);
		\draw(-1.67,1.17) circle (0.44cm);
		\draw [shift={(-2.61,-3.65)}] plot[domain=0.83:1.4,variable=\t]({1*5.54*cos(\t r)+0*5.54*sin(\t r)},{0*5.54*cos(\t r)+1*5.54*sin(\t r)});
		\draw [shift={(1.65,1.03)}] plot[domain=-2.27:1.59,variable=\t]({1*0.78*cos(\t r)+0*0.78*sin(\t r)},{0*0.78*cos(\t r)+1*0.78*sin(\t r)});
	\end{tikzpicture}
	\hspace{20pt}
	\begin{tikzpicture}[scale=0.7]
		\fill[fill=black,fill opacity=0.1] (0,4) -- (-1.66,3.46) -- (-2.69,2.05) -- (-2.68,0.3) -- (-1.66,-1.11) -- (0,-1.65) -- (1.66,-1.11) -- (2.69,0.3) -- (2.69,2.05) -- (1.66,3.46) -- cycle;
		\draw [fill=black,fill opacity=1.0] (1.66,1.61) circle (0.2cm);
		\draw [fill=black,fill opacity=1.0] (-1.66,1.61) circle (0.2cm);
		\draw (0,4)-- (-1.66,3.46);
		\draw (-1.66,3.46)-- (-2.69,2.05);
		\draw (-2.69,2.05)-- (-2.68,0.3);
		\draw (-2.68,0.3)-- (-1.66,-1.11);
		\draw (-1.66,-1.11)-- (0,-1.65);
		\draw (0,-1.65)-- (1.66,-1.11);
		\draw (1.66,-1.11)-- (2.69,0.3);
		\draw (2.69,0.3)-- (2.69,2.05);
		\draw (2.69,2.05)-- (1.66,3.46);
		\draw (1.66,3.46)-- (0,4);
		\draw(1.67,1.18) circle (0.44cm);
		\draw(-1.67,1.17) circle (0.44cm);
		\draw (-1.65,1.81)-- (1.68,1.81);
	\end{tikzpicture}
	\caption{The three different types of arcs in the tube, type $2$, and the last two of type $4$, for $m=2$ and $n=7$}
	\label{fig:tub}
\end{figure}

\begin{defi}
	Let $d \in \{ 1, \cdots, m \}$. We denote by $T^d_2$ (respectively $T'^d_2$) the connected components in the quiver composed by the $d$-th shift of $\tau^i(\alpha)$, $i \in \mathbb{N}$ for the $m$-diagonal of type $4$ at the center of the figure \ref{fig:tub} (respectively at the right of the figure). The arrows correspond to the elementary moves we can draw.
	
	We denote by $T^d_{n-2}$ the connected components in the quiver composed by the $d$-th shift of $\tau^i(\alpha)$, $i \in \mathbb{N}$, where $\alpha$ is the $m$-diagonal of type $2$ ending at $1$.
\end{defi}

These diagonals have to figure in the Auslander-Reiten quiver of $Q$, and we can see them in the non-homogeneous tubes (it means in the regular part of the quiver).

In case $\tilde{D_n}$, as we can see in \cite{WCB}, there are three types of non-homogeneous tubes, two of period $2$, and one of period $n-2$.

We have to notice, that we have $m$ copies of each tube, it means we have $3m$ non-homogeneous tubes. The first tube contains the first picture of figure \ref{fig:tub} and all its images under $\tau$. There are $m$ copies of this tube, which correspond to the successive shifts of the arcs. This is the same process for the second and third tube.

\[
\xymatrix@1{
	& & \vdots & & \\
	& \cdots \ar[dr] & & \cdots \ar[dr] & \\
	d^t_2 \ar[ur] \ar@/_5pc/[rrrr]_{\text{objects identified}} & & \tau^{-1} d^t_2 \ar[ur] & & \tau ^{-2} d^t_2 \ar@/^5pc/[llll] \\
}
\]
\vspace{20pt}

It is known in \cite{WCB} that in a tube of size $r$, only the first $r-1$ layers contain rigid objects. Graphically this corresponds to noncrossing arcs. It means that in $\tilde{Q}$, for the example of the only tube of size $n-2$, the $(n-1)$ lowest lines are made of noncrossing arcs. Then the arc crosses itself, and this does not correspond anymore to an $m$-rigid object in the higher cluster category, and this is the same for the shifted arcs in the successive copies of this first tube. Then with an arc situated on the $r-1$ first lines of the tube, we associate the $m$-rigid object which actually takes place in the Auslander-Reiten quiver of ${\mathcal{C}}^{(m)}_{\tilde{D_n}}$.

In order to make sure that an arc corresponds to a unique $m$-rigid object in the tube, we have to choose a convention. Let $\Delta$ be the initial $(m+2)$-angulation. Let $\beta$ be the red $m$-diagonal in figure \ref{fig:tube}. Then the vertex of $Q$ associated with $\beta$ is $n-1$ from the isomorphism of Theorem \ref{th:iso}. Let $\alpha=\mu_\Delta(\beta)$. We only need to associate an $m$-rigid object in order to find all the objects of the tube of size $n-2$. We choose to associate with $\alpha$ the simple object $\tau^{-1}S_{n-1}$ at vertex $n-1$.

To be precise, we set at the base of the first tube of size $n-2$ the flip of the arc corresponding to the preprojective at the first slice of the Auslander-Reiten quiver. This arc is thus associated with the simple regular module which corresponds to the mutation of the preprojective. This means that we have mutated the object $P_{n-1}$ in the $(m+2)$-angulation containing the arcs corresponding to the sum of the indecomposable projective objects.

For the tubes of size $2$, we set the only arcs linking both central polygons without self-crossing (see figure \ref{fig:tub}).

\begin{figure}[!h]
	\centering
	\begin{tikzpicture}[scale=0.6]
		\fill[fill=black,fill opacity=0.1] (0,4) -- (-1.66,3.46) -- (-2.69,2.05) -- (-2.68,0.3) -- (-1.66,-1.11) -- (0,-1.65) -- (1.66,-1.11) -- (2.69,0.3) -- (2.69,2.05) -- (1.66,3.46) -- cycle;
		\draw [fill=black,fill opacity=1.0] (1.66,1.61) circle (0.2cm);
		\draw [fill=black,fill opacity=1.0] (-1.66,1.61) circle (0.2cm);
		\draw (0,4)-- (-1.66,3.46);
		\draw (-1.66,3.46)-- (-2.69,2.05);
		\draw (-2.69,2.05)-- (-2.68,0.3);
		\draw (-2.68,0.3)-- (-1.66,-1.11);
		\draw (-1.66,-1.11)-- (0,-1.65);
		\draw (0,-1.65)-- (1.66,-1.11);
		\draw (1.66,-1.11)-- (2.69,0.3);
		\draw (2.69,0.3)-- (2.69,2.05);
		\draw (2.69,2.05)-- (1.66,3.46);
		\draw (1.66,3.46)-- (0,4);
		\draw [shift={(-8.03,-1.44)}] plot[domain=-0.03:0.66,variable=\t]({1*8.03*cos(\t r)+0*8.03*sin(\t r)},{0*8.03*cos(\t r)+1*8.03*sin(\t r)});
		\draw [shift={(-0.75,-11.66)}] plot[domain=-0.82:0.59,variable=\t]({-0.07*14.49*cos(\t r)+-1*1.99*sin(\t r)},{1*14.49*cos(\t r)+-0.07*1.99*sin(\t r)});
		\draw [shift={(8.03,-1.43)}] plot[domain=2.49:3.17,variable=\t]({1*8.03*cos(\t r)+0*8.03*sin(\t r)},{0*8.03*cos(\t r)+1*8.03*sin(\t r)});
		\draw [shift={(0.76,-11.66)},color=red]  plot[domain=-0.82:0.59,variable=\t]({0.07*14.49*cos(\t r)+1*1.99*sin(\t r)},{1*14.49*cos(\t r)+-0.07*1.99*sin(\t r)});
		\draw(1.67,1.18) circle (0.44cm);
		\draw(-1.67,1.17) circle (0.44cm);
		\draw [shift={(-7.33,-2.91)}] plot[domain=0.17:0.66,variable=\t]({1*7.44*cos(\t r)+0*7.44*sin(\t r)},{0*7.44*cos(\t r)+1*7.44*sin(\t r)});
		\draw [shift={(7.33,-2.9)}] plot[domain=2.48:2.97,variable=\t]({1*7.44*cos(\t r)+0*7.44*sin(\t r)},{0*7.44*cos(\t r)+1*7.44*sin(\t r)});
		\draw [shift={(2.52,-0.41)}] plot[domain=-0.43:1.16,variable=\t]({-0.96*5.16*cos(\t r)+-0.29*2.09*sin(\t r)},{0.29*5.16*cos(\t r)+-0.96*2.09*sin(\t r)});
		\draw [shift={(-2.52,-0.42)}] plot[domain=-0.43:1.16,variable=\t]({0.96*5.16*cos(\t r)+0.3*2.09*sin(\t r)},{0.3*5.16*cos(\t r)+-0.96*2.09*sin(\t r)});
		\begin{scriptsize}
			\draw[color=black] (-0.37,1.28) node {$1$};
			\draw[color=black] (-1.48,2.56) node {$6$};
			\draw[color=black] (0.48,1.19) node {$2$};	\draw[color=black] (1.62,2.85) node {$3$};
			\draw[color=black] (-0.7,0.23) node {$7$};
			\draw[color=black] (0.62,0.23) node {$4$};
			\draw[color=black] (-2.12,0.21) node {$8$};
			\draw[color=black] (2.41,0.43) node {$5$};
		\end{scriptsize}
	\end{tikzpicture}
	\hspace{10pt}
	$\xrightarrow{\text{flip at 3}}$
	\hspace{10pt}
	\begin{tikzpicture}[scale=0.6]
		\fill[fill=black,fill opacity=0.1] (0,4) -- (-1.66,3.46) -- (-2.69,2.05) -- (-2.68,0.3) -- (-1.66,-1.11) -- (0,-1.65) -- (1.66,-1.11) -- (2.69,0.3) -- (2.69,2.05) -- (1.66,3.46) -- cycle;
		\draw [fill=black,fill opacity=1.0] (1.66,1.61) circle (0.2cm);
		\draw [fill=black,fill opacity=1.0] (-1.66,1.61) circle (0.2cm);
		\draw (0,4)-- (-1.66,3.46);
		\draw (-1.66,3.46)-- (-2.69,2.05);
		\draw (-2.69,2.05)-- (-2.68,0.3);
		\draw (-2.68,0.3)-- (-1.66,-1.11);
		\draw (-1.66,-1.11)-- (0,-1.65);
		\draw (0,-1.65)-- (1.66,-1.11);
		\draw (1.66,-1.11)-- (2.69,0.3);
		\draw (2.69,0.3)-- (2.69,2.05);
		\draw (2.69,2.05)-- (1.66,3.46);
		\draw (1.66,3.46)-- (0,4);
		\draw [shift={(-8.03,-1.44)}] plot[domain=-0.03:0.66,variable=\t]({1*8.03*cos(\t r)+0*8.03*sin(\t r)},{0*8.03*cos(\t r)+1*8.03*sin(\t r)});
		\draw [shift={(-0.75,-11.66)}] plot[domain=-0.82:0.59,variable=\t]({-0.07*14.49*cos(\t r)+-1*1.99*sin(\t r)},{1*14.49*cos(\t r)+-0.07*1.99*sin(\t r)});
		\draw [shift={(8.03,-1.43)}] plot[domain=2.49:3.17,variable=\t]({1*8.03*cos(\t r)+0*8.03*sin(\t r)},{0*8.03*cos(\t r)+1*8.03*sin(\t r)});
		\draw(1.67,1.18) circle (0.44cm);
		\draw(-1.67,1.17) circle (0.44cm);
		\draw [shift={(0.6,-0.09)}] plot[domain=-0.36:1.61,variable=\t]({-0.9*3.18*cos(\t r)+-0.43*1.66*sin(\t r)},{0.43*3.18*cos(\t r)+-0.9*1.66*sin(\t r)});
		\draw [shift={(-0.6,-0.09)}] plot[domain=-0.36:1.61,variable=\t]({0.9*3.18*cos(\t r)+0.43*1.66*sin(\t r)},{0.43*3.18*cos(\t r)+-0.9*1.66*sin(\t r)});
		\draw [shift={(-1.22,1.18)},color=red]  plot[domain=-0.67:0.67,variable=\t]({1*3.67*cos(\t r)+0*3.67*sin(\t r)},{0*3.67*cos(\t r)+1*3.67*sin(\t r)});
		\draw [shift={(-8.55,-3.43)}] plot[domain=0.21:0.64,variable=\t]({1*8.73*cos(\t r)+0*8.73*sin(\t r)},{0*8.73*cos(\t r)+1*8.73*sin(\t r)});
		\draw [shift={(8.55,-3.42)}] plot[domain=2.51:2.94,variable=\t]({1*8.73*cos(\t r)+0*8.73*sin(\t r)},{0*8.73*cos(\t r)+1*8.73*sin(\t r)});
		\begin{scriptsize}
			\draw[color=black] (-0.37,1.28) node {$1$};
			\draw[color=black] (-1.48,2.56) node {$6$};
			\draw[color=black] (0.48,1.19) node {$2$};
			\draw[color=black] (-1.92,-0.06) node {$8$};
			\draw[color=black] (2.07,0.27) node {$5$};
			\draw[color=black] (2.56,1.31) node {$3$};
			\draw[color=black] (-0.76,0.28) node {$7$};
			\draw[color=black] (0.8,0.44) node {$4$};
		\end{scriptsize}
	\end{tikzpicture}
	\caption{We flip $\beta$ in order to define $\alpha$ in the base of the first tube}
	\label{fig:tube}
\end{figure}


With all the definitions of $m$-diagonals, and $(m+2)$-angulations on the polygon, we are now able to associate a quiver with an $(m+2)$-angulation. We will see that the flip of an $(m+2)$-angulation is compatible with quiver mutation at the same vertex.

\section{Colored quivers and $(m+2)$-angulations}
First, we indicate how to associate a quiver with an $(m+2)$-angulation and then we will see how to draw a colored quiver from an $(m+2)$-angulation.

\begin{defi}\label{def:quiver}
Let $\Delta$ be an $(m+2)$-angulation. We choose a good set of representatives. We define the quiver $Q_\Delta$ associated with $\Delta$ in the following way:

\begin{enumerate}
\item The vertices of $Q_\Delta$ are in bijective correspondence with the $m$-diagonals of $\Delta$.
\item Between two vertices $i$ and $j$ (corresponding to two $m$-diagonals in $\Delta$), we draw an arrow from $i$ to $j$ when both diagonals share an oriented angle. We require $j$ to be consecutive to $i$ clockwise.
\item \cor{If either $i$ is a loop and $j$ an arc inside of it, then forget $i$ (respectively $j$) to draw the incident arrows of $j$ (respectively $i$). Do not forget to draw an arrow from $i$ to $j$.}
\end{enumerate}
\end{defi}

\begin{rmk}
Even if we have chosen a good set of representatives, the definition is independent of the choice of such set.
\end{rmk}

\begin{rmk}
From there, we call $m$-diagonals and vertices by the same name. It should be clear from the context whether we talk about polygons or quivers.

This implies that $Q_\Delta$ from Lemma \ref{lem:Dehntwist} is well-defined, because it is independent of the choice of a good system of representatives. Indeed, the Dehn twist does not affect the order of the $m$-diagonals. More precisely, if we apply $\theta_R$ and $\theta_S$ two Dehn twists to the polygons $L$ and $R$, let us recall $i'$ and $j'$ the images if $i$ and $j$ by $\theta_S \circ \theta_R$. The end of $i'$ (respectively $j'$) hung to $P$ si the same as the one of $i$ (respectively $j$). Then the order is the same, and there is an arrow if $'i$ is consecutive to $j'$ (which is the same as $i$ being consecutive to $j$).
\end{rmk}

For example in case $m=1$, if we take the following triangulation $\Delta$ (figure \ref{fig:triang}):

\begin{figure}[!h]
\centering
\begin{tikzpicture}[scale=0.7]
\fill[fill=black,fill opacity=0.1] (0,4) -- (-4,0) -- (0,-4) -- (4,0) -- cycle;
\draw [fill=black,fill opacity=1.0] (1,1) circle (0.5cm);
\draw [fill=black,fill opacity=1.0] (-1,-1) circle (0.5cm);
\draw (0,4)-- (-4,0);
\draw (-4,0)-- (0,-4);
\draw (0,-4)-- (4,0);
\draw (4,0)-- (0,4);
\draw (-4,0)-- (4,0);
\draw [shift={(-3.53,-11.12)}] plot[domain=0.98:1.22,variable=\t]({1*13.43*cos(\t r)+0*13.43*sin(\t r)},{0*13.43*cos(\t r)+1*13.43*sin(\t r)});
\draw [shift={(4.65,13.41)}] plot[domain=4.42:4.66,variable=\t]({1*13.43*cos(\t r)+0*13.43*sin(\t r)},{0*13.43*cos(\t r)+1*13.43*sin(\t r)});
\draw [shift={(3.53,11.12)}] plot[domain=4.12:4.36,variable=\t]({1*13.43*cos(\t r)+0*13.43*sin(\t r)},{0*13.43*cos(\t r)+1*13.43*sin(\t r)});
\draw [shift={(-4.65,-13.41)}] plot[domain=1.28:1.52,variable=\t]({1*13.43*cos(\t r)+0*13.43*sin(\t r)},{0*13.43*cos(\t r)+1*13.43*sin(\t r)});
\draw [shift={(-16.21,-7.04)}] plot[domain=2.72:3.84,variable=\t]({-0.95*17.95*cos(\t r)+0.3*4.77*sin(\t r)},{-0.3*17.95*cos(\t r)+-0.95*4.77*sin(\t r)});
\draw [shift={(16.21,7.04)}] plot[domain=2.72:3.84,variable=\t]({0.95*17.95*cos(\t r)+-0.3*4.77*sin(\t r)},{0.3*17.95*cos(\t r)+0.95*4.77*sin(\t r)});
\begin{scriptsize}
\draw[color=black] (2.64,0.88) node {$7$};
\draw[color=black] (1.8,0.34) node {$6$};
\draw[color=black] (0,0) node {$1$};
\draw[color=black] (-2.24,-0.96) node {$4$};
\draw[color=black] (-1.7,-0.3) node {$3$};
\draw[color=black] (0.62,-1.12) node {$2$};
\draw[color=black] (-0.46,1.08) node {$5$};
\end{scriptsize}
\end{tikzpicture}
\caption{A triangulation $\Delta$. For sake of clarity we have replaced loops $6$ and $4$ by arcs tangent to $L$ and $R$, but the loops actually make part of the good set of representatives.}
\label{fig:triang}
\end{figure}

Now, following the rule, we associate the following quiver $Q_\Delta$ with this triangulation:

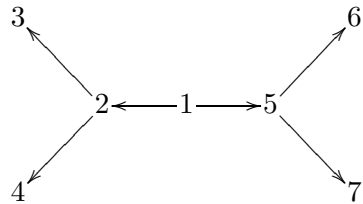
\[ \xymatrix@1{
3 & & & & 6 \\
& 2\ar[ul]\ar[dl] & 1\ar[l]\ar[r] & 5\ar[ur]\ar[dr] & \\
4 & & & & 7} \]

Let us now mutate this quiver at vertex $2$. We obtain this new quiver $Q'$:

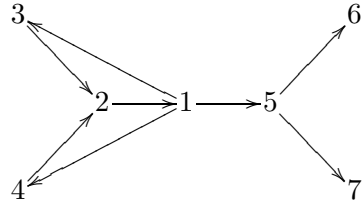
\[ \xymatrix@1{
3\ar[dr] & & & & 6 \\
& 2\ar[r] & 1\ar[r]\ar[ull]\ar[dll] & 5\ar[ur]\ar[dr] & \\
4\ar[ur] & & & & 7} \]

We remark that if we flip the previous triangulation at the $m$-diagonal $2$, we obtain the new triangulation $\Delta'$ of figure \ref{fig:triang2}.

\begin{figure}[!h]
\centering
\begin{tikzpicture}[scale=0.7]
\fill[fill=black,fill opacity=0.1] (0,4) -- (-4,0) -- (0,-4) -- (4,0) -- cycle;
\draw [fill=black,fill opacity=1.0] (1,1) circle (0.5cm);
\draw [fill=black,fill opacity=1.0] (-1,-1) circle (0.5cm);
\draw (0,4)-- (-4,0);
\draw (-4,0)-- (0,-4);
\draw (0,-4)-- (4,0);
\draw (4,0)-- (0,4);
\draw (-4,0)-- (4,0);
\draw [shift={(-3.53,-11.12)}] plot[domain=0.98:1.22,variable=\t]({1*13.43*cos(\t r)+0*13.43*sin(\t r)},{0*13.43*cos(\t r)+1*13.43*sin(\t r)});
\draw [shift={(4.65,13.41)}] plot[domain=4.42:4.66,variable=\t]({1*13.43*cos(\t r)+0*13.43*sin(\t r)},{0*13.43*cos(\t r)+1*13.43*sin(\t r)});
\draw [shift={(3.53,11.12)}] plot[domain=4.12:4.36,variable=\t]({1*13.43*cos(\t r)+0*13.43*sin(\t r)},{0*13.43*cos(\t r)+1*13.43*sin(\t r)});
\draw [shift={(-4.65,-13.41)}] plot[domain=1.28:1.52,variable=\t]({1*13.43*cos(\t r)+0*13.43*sin(\t r)},{0*13.43*cos(\t r)+1*13.43*sin(\t r)});
\draw [shift={(16.21,7.04)}] plot[domain=2.72:3.84,variable=\t]({0.95*17.95*cos(\t r)+-0.3*4.77*sin(\t r)},{0.3*17.95*cos(\t r)+0.95*4.77*sin(\t r)});
\draw [shift={(-0.2,33)}] plot[domain=2.79:3.48,variable=\t]({-0.01*35.07*cos(\t r)+-1*11.83*sin(\t r)},{1*35.07*cos(\t r)+-0.01*11.83*sin(\t r)});
\begin{scriptsize}
\draw[color=black] (2.64,0.88) node {$7$};
\draw[color=black] (1.8,0.34) node {$6$};
\draw[color=black] (-2.24,-0.96) node {$4$};
\draw[color=black] (-1.7,-0.3) node {$3$};
\draw[color=black] (0,0) node {$1$};
\draw[color=black] (-0.46,1.08) node {$5$};
\draw[color=black] (0.08,-2.02) node {$2$};
\end{scriptsize}
\end{tikzpicture}
\caption{The triangulation $\Delta'$}
\label{fig:triang2}
\end{figure}

It is noticeable that the quiver $Q_\Delta'$ arising from the new triangulation corresponds to the mutation at $2$ of the quiver $Q_\Delta$. This is a result showed in \cite{FST} for $m=1$.

\begin{prop}[\cite{FST}, Proposition $4.8$]
Let $\Delta$ be any triangulation. Let $Q_\Delta$ be the quiver associated with $\Delta$ as before. If $\Delta_i$ is the new triangulation flipped at $i$ from $\Delta$, then the quiver $Q_{\Delta_i}$ associated with $\Delta_i$ corresponds to the mutation at vertex $i$ of $Q_\Delta$.
\end{prop}

We will later exhibit a proof for the case of colored quivers of mutation-type $\tilde{D}$.	 We now describe how to associate a colored quiver with an $(m+2)$-angulation.

\begin{lem}
If $m=1$, then the $m$-diagonals of $P$ are in bijection with the tagged arcs in \cite{FST} for the twice punctured $(n-2)m$-gon and the obvious bijection respects the noncrossing conditions.
\end{lem}

\begin{proof}
Let $\mathcal{D}$ be the set of all $1$-diagonals in our setup. Let $\mathcal{A}$ be the set of all tagged arcs in the setup of \cite{FST}. Define a bijection $\mathcal{D} \to \mathcal{A}$ as follows :

\begin{itemize}
\item If $\alpha \in \mathcal{D}$ links two vertices of $P$, then send it to the same arc in $\mathcal{A}$.\item If $\alpha \in \mathcal{D}$ is a left/right tangent to a thick vertex, send it to the corresponding plain/notched arc in $\mathcal{A}$.
\end{itemize}

\end{proof}

\begin{defi}\label{def:colquiver}
Let $\Delta$ be an $(m+2)$-angulation, where we have taken a good set of representatives (here again, the definition does not depends ont such a choice). We define the colored quiver $Q_\Delta$ associated with $\Delta$ in the following way:

\begin{enumerate}
\item The vertices of $Q_\Delta$ are in bijection with the $m$-diagonals of $\Delta$.
\item If $i$ and $j$ form two sides of the same polygon, then we draw an arrow from $i$ to $j$. The color of the corresponding arrow is the number of edges between both $m$-diagonals, counted clockwise from $i$.
\item \cor{If either $i$ is a loop and $j$ an arc inside of it, then forget $i$ (respectively $j$) to draw the incident arrows of $j$ (respectively $i$). In this case, do not forget to draw an arrow from $i$ to $j$ of color $0$, and an arrow from $j$ to $i$ of color $m$.}
\end{enumerate}
\end{defi}

\begin{ex}
For example, the colored quiver associated with figure \ref{fig:9cot} is the following one:

\[\scalebox{1.5}{\xymatrix@1{
& & 2 \ar@<1ex>^{\scriptscriptstyle{(2)}}[drr] \ar^{\scriptscriptstyle{(2)}}[dll] \ar@<1ex>^{\scriptscriptstyle{(1)}}[d] & & \\
6 \ar@<1ex>^{\scriptscriptstyle{(1)}}[urr] \ar^{\scriptscriptstyle{(3)}}[rr] \ar^{\scriptscriptstyle{(0)}}[d] & & 1 \ar^{\scriptscriptstyle{(2)}}[u] \ar@<1ex>^{\scriptscriptstyle{(0)}}[ll] \ar@<1ex>^{\scriptscriptstyle{(0)}}[rr] \ar@<1ex>^{\scriptscriptstyle{(2)}}[drr] & & 3 \ar^{\scriptscriptstyle{(1)}}[ull] \ar^{\scriptscriptstyle{(3)}}[ll] \ar@<1ex>^{\scriptscriptstyle{(1)}}[d] \\
5 \ar^{\scriptscriptstyle{(1)}}[rrrr] \ar@<1ex>^{\scriptscriptstyle{(3)}}[u] & & & & 4 \ar^{\scriptscriptstyle{(1)}}[ull] \ar^{\scriptscriptstyle{(2)}}[u] \ar@<1ex>^{\scriptscriptstyle{(2)}}[llll]
}}\]

\begin{figure}[h!]
	\centering
\begin{tikzpicture}[scale=0.5]

	\fill[line width=0.8pt,fill=black,fill opacity=0.1] (-6.21,-0.44) -- (-3.31,-0.44) -- (-1.0884711149549648,1.4240840680909628) -- (-0.584891399720866,4.2800265518263645) -- (-2.0348913997208653,6.791500222801237) -- (-4.76,7.7833586384456765) -- (-7.485108600279132,6.79150022280124) -- (-8.935108600279133,4.280026551826368) -- (-8.431528885045037,1.4240840680909654) -- cycle;

	\draw [line width=0.8pt,fill=black,fill opacity=1] (-6.91,4.56) circle (0.37947331922020555cm);

	\draw [line width=0.8pt,fill=black,fill opacity=1] (-2.61,2.527684516418403) circle (0.37947331922020555cm);

	\draw [line width=0.8pt,fill=black,fill opacity=1] (-6.911301849727237,2.530443566892921) circle (0.37947331922020555cm);

	\draw [line width=0.8pt,fill=black,fill opacity=1] (-2.6086981502727635,4.557240949525481) circle (0.37947331922020555cm);

	\draw [line width=0.8pt] (-6.21,-0.44)-- (-3.31,-0.44);

	\draw [line width=0.8pt] (-3.31,-0.44)-- (-1.0884711149549648,1.4240840680909628);

	\draw [line width=0.8pt] (-1.0884711149549648,1.4240840680909628)-- (-0.584891399720866,4.2800265518263645);

	\draw [line width=0.8pt] (-0.584891399720866,4.2800265518263645)-- (-2.0348913997208653,6.791500222801237);

	\draw [line width=0.8pt] (-2.0348913997208653,6.791500222801237)-- (-4.76,7.7833586384456765);

	\draw [line width=0.8pt] (-4.76,7.7833586384456765)-- (-7.485108600279132,6.79150022280124);

	\draw [line width=0.8pt] (-7.485108600279132,6.79150022280124)-- (-8.935108600279133,4.280026551826368);

	\draw [line width=0.8pt] (-8.935108600279133,4.280026551826368)-- (-8.431528885045037,1.4240840680909654);

	\draw [line width=0.8pt] (-8.431528885045037,1.4240840680909654)-- (-6.21,-0.44);

	\draw [shift={(-3.561643138173719,3.5430735776599684)},line width=0.8pt]  plot[domain=-0.8177939835105823:0.8165110927986511,variable=\t]({1*1.391631994502193*cos(\t r)+0*1.391631994502193*sin(\t r)},{0*1.391631994502193*cos(\t r)+1*1.391631994502193*sin(\t r)});

	\draw [shift={(-1.6570550120990424,3.541851888283916)},line width=0.8pt]  plot[domain=2.323798670079213:3.958103746388443,variable=\t]({1*1.3916319945021955*cos(\t r)+0*1.3916319945021955*sin(\t r)},{0*1.3916319945021955*cos(\t r)+1*1.3916319945021955*sin(\t r)});

	\draw [shift={(-5.958356861826281,3.544610938758434)},line width=0.8pt]  plot[domain=2.3237986700792113:3.958103746388444,variable=\t]({1*1.391631994502193*cos(\t r)+0*1.391631994502193*sin(\t r)},{0*1.391631994502193*cos(\t r)+1*1.391631994502193*sin(\t r)});

	\draw [shift={(-7.862944987900958,3.5458326281344883)},line width=0.8pt]  plot[domain=-0.8177939835105814:0.8165110927986498,variable=\t]({1*1.3916319945021958*cos(\t r)+0*1.3916319945021958*sin(\t r)},{0*1.3916319945021958*cos(\t r)+1*1.3916319945021958*sin(\t r)});

	\draw [shift={(0.4761944745201511,6.482505804133501)},line width=0.8pt]  plot[domain=2.898087678214627:4.053847948892344,variable=\t]({1*5.395363812713038*cos(\t r)+0*5.395363812713038*sin(\t r)},{0*5.395363812713038*cos(\t r)+1*5.395363812713038*sin(\t r)});

	\draw [shift={(-2.3059571682978515,2.4328101719759787)},line width=0.8pt]  plot[domain=0.8207367629426887:1.7029286948460944,variable=\t]({1*2.5247328044280355*cos(\t r)+0*2.5247328044280355*sin(\t r)},{0*2.5247328044280355*cos(\t r)+1*2.5247328044280355*sin(\t r)});

	\draw [shift={(-3.151511877279275,3.430560873615431)},line width=0.8pt]  plot[domain=-3.211903802703982:1.2023393894049665,variable=\t]({1*1.6124722363591635*cos(\t r)+0*1.6124722363591635*sin(\t r)},{0*1.6124722363591635*cos(\t r)+1*1.6124722363591635*sin(\t r)});

	\draw [shift={(-6.368488122720725,3.657123642802972)},line width=0.8pt]  plot[domain=-0.07031114911418879:4.34393204299476,variable=\t]({1*1.6124722363591637*cos(\t r)+0*1.6124722363591637*sin(\t r)},{0*1.6124722363591637*cos(\t r)+1*1.6124722363591637*sin(\t r)});

	\draw [shift={(-2.199944693337082,5.7936075443250505)},line width=0.8pt]  plot[domain=3.372297945858585:4.9614758836687995,variable=\t]({1*4.508670416519972*cos(\t r)+0*4.508670416519972*sin(\t r)},{0*4.508670416519972*cos(\t r)+1*4.508670416519972*sin(\t r)});

	\draw [shift={(-9.065646439997627,0.8202139431321874)},line width=0.8pt]  plot[domain=-0.4156005268739831:0.6432973447075572,variable=\t]({1*3.1213547976377014*cos(\t r)+0*3.1213547976377014*sin(\t r)},{0*3.1213547976377014*cos(\t r)+1*3.1213547976377014*sin(\t r)});

	\draw [shift={(7.6251125720296535,18.963863033928604)},line width=0.8pt]  plot[domain=2.5910333724042474:3.3074272829938387,variable=\t]({0.6835490082746702*22.544611322992676*cos(\t r)+-0.7299046193076975*5.1773707160877676*sin(\t r)},{0.7299046193076975*22.544611322992676*cos(\t r)+0.6835490082746702*5.1773707160877676*sin(\t r)});

	\draw (-5.1,4.54) node[anchor=north west] {1};

	\draw (-4.57,5.66) node[anchor=north west] {2};

	\draw (-1.96,5.5) node[anchor=north west] {3};

	\draw (-5.01,2.04) node[anchor=north west] {4};

	\draw (-5.83,1.52) node[anchor=north west] {5};

	\draw (-8.8,3.84) node[anchor=north west] {6};

	\begin{scriptsize}

		\draw [fill=black] (-6.21,-0.44) circle (2.5pt);

		\draw [fill=black] (-3.31,-0.44) circle (2.5pt);

		\draw [fill=black] (-1.0884711149549648,1.4240840680909628) circle (2.5pt);

		\draw [fill=black] (-0.584891399720866,4.2800265518263645) circle (2.5pt);

		\draw [fill=black] (-2.0348913997208653,6.791500222801237) circle (2.5pt);

		\draw [fill=black] (-4.76,7.7833586384456765) circle (2.5pt);

		\draw [fill=black] (-7.485108600279132,6.79150022280124) circle (2.5pt);

		\draw [fill=black] (-8.935108600279133,4.280026551826368) circle (2.5pt);

		\draw [fill=black] (-8.431528885045037,1.4240840680909654) circle (2.5pt);

	\end{scriptsize}

\end{tikzpicture}
	\caption{Example of the associated quiver of this $(m+2)$-angulation, for $n=5$ and $m=3$.}
	\label{fig:9cot}
\end{figure}

\end{ex}

We note that if we only draw the arrows of color $0$, then we find the classical quiver associated with the $(m+2)$-angulation as in Definition \ref{def:quiver}.

\begin{prop}
There is an equivalent definition: the vertices are similarly defined, and for $i$ and $j$ two vertices, and $c$ an integer,
\[ 
q_{ij}^{(c)} = \left\{
    \begin{array}{ll}
        1 & \mbox{if } \kappa_\Delta^c(i) \mbox{ and } j \mbox{ share an oriented angle} \\
        0 & \mbox{otherwise.}
    \end{array}
    \right.
\]
\end{prop}

\begin{proof}
We only have to show that the arrows are the same. If $i$ and $j$ form two sides of the polygon, with a color $c$, it means that if we apply the twist to $i$, then there will be $c-1$ edges between $ \kappa_\Delta(i)$ and $j$. Then if we apply the twist $c$ times, there will be no edge between $ \kappa_\Delta^c(i)$ and $j$, and they will share an oriented angle.

On the other hand, if $ \kappa_\Delta^c(i)$ and $j$ share an oriented angle, it suffices to apply the inverse of the twist $c$ times to make sure that $i$ and $j$ form two sides of a polygon, and that there are $c$ edges between $i$ and $j$.
\end{proof}

\begin{figure}[!h]
\centering
\begin{tikzpicture}[scale=0.2]
\fill[fill=black,fill opacity=0.15] (0,4) -- (0,-4) -- (6.93,-8) -- (13.86,-4) -- (13.86,4) -- (6.93,8) -- cycle;
\draw [fill=black,fill opacity=1.0] (3.46,2) circle (0.4cm);
\draw [fill=black,fill opacity=1.0] (3.46,-2) circle (0.4cm);
\draw [fill=black,fill opacity=1.0] (10.39,2) circle (0.4cm);
\draw [fill=black,fill opacity=1.0] (10.39,-2) circle (0.4cm);
\draw (0,4)-- (0,-4);
\draw (0,-4)-- (6.93,-8);
\draw (6.93,-8)-- (13.86,-4);
\draw (13.86,-4)-- (13.86,4);
\draw (13.86,4)-- (6.93,8);
\draw (6.93,8)-- (0,4);
\draw [shift={(1.86,0)}] plot[domain=-0.89:0.89,variable=\t]({1*2.57*cos(\t r)+0*2.57*sin(\t r)},{0*2.57*cos(\t r)+1*2.57*sin(\t r)});
\draw [shift={(5.07,0)}] plot[domain=2.25:4.04,variable=\t]({1*2.57*cos(\t r)+0*2.57*sin(\t r)},{0*2.57*cos(\t r)+1*2.57*sin(\t r)});
\draw [shift={(12,0)}] plot[domain=2.25:4.04,variable=\t]({1*2.57*cos(\t r)+0*2.57*sin(\t r)},{0*2.57*cos(\t r)+1*2.57*sin(\t r)});
\draw [shift={(8.79,0)}] plot[domain=-0.89:0.89,variable=\t]({1*2.57*cos(\t r)+0*2.57*sin(\t r)},{0*2.57*cos(\t r)+1*2.57*sin(\t r)});
\draw (6.93,8)-- (6.93,-8);
\draw [shift={(22.8,-8.35)}] plot[domain=-0.37:0.89,variable=\t]({-0.93*23.56*cos(\t r)+-0.38*7.31*sin(\t r)},{0.38*23.56*cos(\t r)+-0.93*7.31*sin(\t r)});
\draw [shift={(-0.38,-28.18)}] plot[domain=2.3:3.36,variable=\t]({-0.12*31.52*cos(\t r)+0.99*6.48*sin(\t r)},{-0.99*31.52*cos(\t r)+-0.12*6.48*sin(\t r)});
\draw [shift={(37.5,-7.09)}] plot[domain=-0.24:0.66,variable=\t]({-0.98*37.9*cos(\t r)+-0.19*10.95*sin(\t r)},{0.19*37.9*cos(\t r)+-0.98*10.95*sin(\t r)});
\draw [shift={(14.23,28.18)},color=red]  plot[domain=2.3:3.36,variable=\t]({0.12*31.52*cos(\t r)+-0.99*6.48*sin(\t r)},{0.99*31.52*cos(\t r)+0.12*6.48*sin(\t r)});
\draw [shift={(-23.65,7.09)},color=red]  plot[domain=-0.24:0.66,variable=\t]({0.98*37.9*cos(\t r)+0.19*10.95*sin(\t r)},{-0.19*37.9*cos(\t r)+0.98*10.95*sin(\t r)});
\draw [shift={(34.18,10.82)}] plot[domain=3.24:3.64,variable=\t]({1*27.41*cos(\t r)+0*27.41*sin(\t r)},{0*27.41*cos(\t r)+1*27.41*sin(\t r)});
\begin{scriptsize}
\draw[color=black] (6.78,0.23) node {$1$};
\draw[color=black] (6.03,-1.83) node {$4$};
\draw[color=black] (1.69,-3.06) node {$5$};
\draw[color=black] (7.96,2.79) node {$2$};
\draw[color=black] (12.35,2.82) node {$3$};
\end{scriptsize}
\end{tikzpicture}
\hspace{20pt}
$\xrightarrow{\text{flip at 3}}$
\hspace{20pt}
\begin{tikzpicture}[scale=0.2]
\fill[fill=black,fill opacity=0.15] (0,4) -- (0,-4) -- (6.93,-8) -- (13.86,-4) -- (13.86,4) -- (6.93,8) -- cycle;
\draw [fill=black,fill opacity=1.0] (3.46,2) circle (0.4cm);
\draw [fill=black,fill opacity=1.0] (3.46,-2) circle (0.4cm);
\draw [fill=black,fill opacity=1.0] (10.39,2) circle (0.4cm);
\draw [fill=black,fill opacity=1.0] (10.39,-2) circle (0.4cm);
\draw (0,4)-- (0,-4);
\draw (0,-4)-- (6.93,-8);
\draw (6.93,-8)-- (13.86,-4);
\draw (13.86,-4)-- (13.86,4);
\draw (13.86,4)-- (6.93,8);
\draw (6.93,8)-- (0,4);
\draw [shift={(1.86,0)}] plot[domain=-0.89:0.89,variable=\t]({1*2.57*cos(\t r)+0*2.57*sin(\t r)},{0*2.57*cos(\t r)+1*2.57*sin(\t r)});
\draw [shift={(5.07,0)}] plot[domain=2.25:4.04,variable=\t]({1*2.57*cos(\t r)+0*2.57*sin(\t r)},{0*2.57*cos(\t r)+1*2.57*sin(\t r)});
\draw [shift={(12,0)}] plot[domain=2.25:4.04,variable=\t]({1*2.57*cos(\t r)+0*2.57*sin(\t r)},{0*2.57*cos(\t r)+1*2.57*sin(\t r)});
\draw [shift={(8.79,0)}] plot[domain=-0.89:0.89,variable=\t]({1*2.57*cos(\t r)+0*2.57*sin(\t r)},{0*2.57*cos(\t r)+1*2.57*sin(\t r)});
\draw (6.93,8)-- (6.93,-8);
\draw [shift={(22.8,-8.35)}] plot[domain=-0.37:0.89,variable=\t]({-0.93*23.56*cos(\t r)+-0.38*7.31*sin(\t r)},{0.38*23.56*cos(\t r)+-0.93*7.31*sin(\t r)});
\draw [shift={(-0.38,-28.18)}] plot[domain=2.3:3.36,variable=\t]({-0.12*31.52*cos(\t r)+0.99*6.48*sin(\t r)},{-0.99*31.52*cos(\t r)+-0.12*6.48*sin(\t r)});
\draw [shift={(37.5,-7.09)}] plot[domain=-0.24:0.66,variable=\t]({-0.98*37.9*cos(\t r)+-0.19*10.95*sin(\t r)},{0.19*37.9*cos(\t r)+-0.98*10.95*sin(\t r)});
\draw [shift={(34.18,10.82)}] plot[domain=3.24:3.64,variable=\t]({1*27.41*cos(\t r)+0*27.41*sin(\t r)},{0*27.41*cos(\t r)+1*27.41*sin(\t r)});
\draw [shift={(11.34,-1.74)},color=red]  plot[domain=3.58:5.93,variable=\t]({1*0.61*cos(\t r)+0*0.61*sin(\t r)},{0*0.61*cos(\t r)+1*0.61*sin(\t r)});
\draw [shift={(-32.93,16)},color=red]  plot[domain=5.9:6.03,variable=\t]({1*48.31*cos(\t r)+0*48.31*sin(\t r)},{0*48.31*cos(\t r)+1*48.31*sin(\t r)});
\begin{scriptsize}
\draw[color=black] (6.79,0.22) node {$1$};
\draw[color=black] (6.04,-1.87) node {$4$};
\draw[color=black] (1.68,-2.99) node {$5$};
\draw[color=black] (7.96,2.78) node {$2$};
\draw[color=black] (13.16,0) node {$3$};
\end{scriptsize}
\end{tikzpicture}
\end{figure}
\begin{figure}[!h]
\centering
$ \xymatrix@1{
4 \ar^{(3)}[dr] & & 2 \ar^{(3)}[dl]\ar@<1ex>^{(3)}[dd] \\
& 1 \ar@<1ex>^{(0)}[ul] \ar@<1ex>^{(0)}[ur] \ar@<1ex>^{(0)}[dr] \ar@<1ex>^{(0)}[dl] & \\
5 \ar^{(3)}[ur] & & 3 \ar^{(3)}[ul] \ar^{(0)}[uu] } $
\hspace{10pt}
$\xrightarrow{\text{mutation at 3}}$
\hspace{10pt}
$ \xymatrix@1{
4 \ar^{(3)}[dr] & & 2 \ar^{(3)}[dl]\ar@<1ex>^{(0)}[dd] \\
& 1 \ar@<1ex>^{(0)}[ul] \ar@<1ex>^{(0)}[ur] \ar@<1ex>^{(1)}[dr] \ar@<1ex>^{(0)}[dl] & \\
5 \ar^{(3)}[ur] & & 3 \ar^{(2)}[ul]\ar^{(3)}[uu] } $
\caption{Quiver mutation is compatible with the flip of an $(m+2)$-angulation, for $m=3$.}
\label{fig:compat}
\end{figure}

\begin{figure}[!h]
\centering
\begin{tikzpicture}[scale=0.6]
\fill[fill=black,fill opacity=0.1] (0,4) -- (-1.66,3.46) -- (-2.69,2.05) -- (-2.68,0.3) -- (-1.66,-1.11) -- (0,-1.65) -- (1.66,-1.11) -- (2.69,0.3) -- (2.69,2.05) -- (1.66,3.46) -- cycle;
\draw [fill=black,fill opacity=1.0] (1.66,1.61) circle (0.2cm);
\draw [fill=black,fill opacity=1.0] (-1.66,1.61) circle (0.2cm);
\draw (0,4)-- (-1.66,3.46);
\draw (-1.66,3.46)-- (-2.69,2.05);
\draw (-2.69,2.05)-- (-2.68,0.3);
\draw (-2.68,0.3)-- (-1.66,-1.11);
\draw (-1.66,-1.11)-- (0,-1.65);
\draw (0,-1.65)-- (1.66,-1.11);
\draw (1.66,-1.11)-- (2.69,0.3);
\draw (2.69,0.3)-- (2.69,2.05);
\draw (2.69,2.05)-- (1.66,3.46);
\draw (1.66,3.46)-- (0,4);
\draw [shift={(-8.03,-1.44)}] plot[domain=-0.03:0.66,variable=\t]({1*8.03*cos(\t r)+0*8.03*sin(\t r)},{0*8.03*cos(\t r)+1*8.03*sin(\t r)});
\draw [shift={(-0.75,-11.66)}] plot[domain=-0.82:0.59,variable=\t]({-0.07*14.49*cos(\t r)+-1*1.99*sin(\t r)},{1*14.49*cos(\t r)+-0.07*1.99*sin(\t r)});
\draw [shift={(8.03,-1.43)}] plot[domain=2.49:3.17,variable=\t]({1*8.03*cos(\t r)+0*8.03*sin(\t r)},{0*8.03*cos(\t r)+1*8.03*sin(\t r)});
\draw [shift={(0.76,-11.66)},color=red]  plot[domain=-0.82:0.59,variable=\t]({0.07*14.49*cos(\t r)+1*1.99*sin(\t r)},{1*14.49*cos(\t r)+-0.07*1.99*sin(\t r)});
\draw(1.67,1.18) circle (0.44cm);
\draw(-1.67,1.17) circle (0.44cm);
\draw [shift={(-7.33,-2.91)}] plot[domain=0.17:0.66,variable=\t]({1*7.44*cos(\t r)+0*7.44*sin(\t r)},{0*7.44*cos(\t r)+1*7.44*sin(\t r)});
\draw [shift={(7.33,-2.9)}] plot[domain=2.48:2.97,variable=\t]({1*7.44*cos(\t r)+0*7.44*sin(\t r)},{0*7.44*cos(\t r)+1*7.44*sin(\t r)});
\draw [shift={(2.52,-0.41)}] plot[domain=-0.43:1.16,variable=\t]({-0.96*5.16*cos(\t r)+-0.29*2.09*sin(\t r)},{0.29*5.16*cos(\t r)+-0.96*2.09*sin(\t r)});
\draw [shift={(-2.52,-0.42)}] plot[domain=-0.43:1.16,variable=\t]({0.96*5.16*cos(\t r)+0.3*2.09*sin(\t r)},{0.3*5.16*cos(\t r)+-0.96*2.09*sin(\t r)});
\begin{scriptsize}
\draw[color=black] (-0.37,1.28) node {$1$};
\draw[color=black] (-1.48,2.56) node {$6$};
\draw[color=black] (0.48,1.19) node {$2$};
\draw[color=black] (1.62,2.85) node {$3$};
\draw[color=black] (-0.7,0.23) node {$7$};
\draw[color=black] (0.62,0.23) node {$4$};
\draw[color=black] (-2.12,0.21) node {$8$};
\draw[color=black] (2.41,0.43) node {$5$};
\end{scriptsize}
\end{tikzpicture}
\hspace{10pt}
$\xrightarrow{\text{flip at 3}}$
\hspace{10pt}
\begin{tikzpicture}[scale=0.6]
\fill[fill=black,fill opacity=0.1] (0,4) -- (-1.66,3.46) -- (-2.69,2.05) -- (-2.68,0.3) -- (-1.66,-1.11) -- (0,-1.65) -- (1.66,-1.11) -- (2.69,0.3) -- (2.69,2.05) -- (1.66,3.46) -- cycle;
\draw [fill=black,fill opacity=1.0] (1.66,1.61) circle (0.2cm);
\draw [fill=black,fill opacity=1.0] (-1.66,1.61) circle (0.2cm);
\draw (0,4)-- (-1.66,3.46);
\draw (-1.66,3.46)-- (-2.69,2.05);
\draw (-2.69,2.05)-- (-2.68,0.3);
\draw (-2.68,0.3)-- (-1.66,-1.11);
\draw (-1.66,-1.11)-- (0,-1.65);
\draw (0,-1.65)-- (1.66,-1.11);
\draw (1.66,-1.11)-- (2.69,0.3);
\draw (2.69,0.3)-- (2.69,2.05);
\draw (2.69,2.05)-- (1.66,3.46);
\draw (1.66,3.46)-- (0,4);
\draw [shift={(-8.03,-1.44)}] plot[domain=-0.03:0.66,variable=\t]({1*8.03*cos(\t r)+0*8.03*sin(\t r)},{0*8.03*cos(\t r)+1*8.03*sin(\t r)});
\draw [shift={(-0.75,-11.66)}] plot[domain=-0.82:0.59,variable=\t]({-0.07*14.49*cos(\t r)+-1*1.99*sin(\t r)},{1*14.49*cos(\t r)+-0.07*1.99*sin(\t r)});
\draw [shift={(8.03,-1.43)}] plot[domain=2.49:3.17,variable=\t]({1*8.03*cos(\t r)+0*8.03*sin(\t r)},{0*8.03*cos(\t r)+1*8.03*sin(\t r)});
\draw(1.67,1.18) circle (0.44cm);
\draw(-1.67,1.17) circle (0.44cm);
\draw [shift={(0.6,-0.09)}] plot[domain=-0.36:1.61,variable=\t]({-0.9*3.18*cos(\t r)+-0.43*1.66*sin(\t r)},{0.43*3.18*cos(\t r)+-0.9*1.66*sin(\t r)});
\draw [shift={(-0.6,-0.09)}] plot[domain=-0.36:1.61,variable=\t]({0.9*3.18*cos(\t r)+0.43*1.66*sin(\t r)},{0.43*3.18*cos(\t r)+-0.9*1.66*sin(\t r)});
\draw [shift={(-1.22,1.18)},color=red]  plot[domain=-0.67:0.67,variable=\t]({1*3.67*cos(\t r)+0*3.67*sin(\t r)},{0*3.67*cos(\t r)+1*3.67*sin(\t r)});
\draw [shift={(-8.55,-3.43)}] plot[domain=0.21:0.64,variable=\t]({1*8.73*cos(\t r)+0*8.73*sin(\t r)},{0*8.73*cos(\t r)+1*8.73*sin(\t r)});
\draw [shift={(8.55,-3.42)}] plot[domain=2.51:2.94,variable=\t]({1*8.73*cos(\t r)+0*8.73*sin(\t r)},{0*8.73*cos(\t r)+1*8.73*sin(\t r)});
\begin{scriptsize}
\draw[color=black] (-0.37,1.28) node {$1$};
\draw[color=black] (-1.48,2.56) node {$6$};
\draw[color=black] (0.48,1.19) node {$2$};
\draw[color=black] (-1.92,-0.06) node {$8$};
\draw[color=black] (2.07,0.27) node {$5$};
\draw[color=black] (2.56,1.31) node {$3$};
\draw[color=black] (-0.76,0.28) node {$7$};
\draw[color=black] (0.8,0.44) node {$4$};
\end{scriptsize}
\end{tikzpicture}
\end{figure}
\begin{figure}[!h]
\centering
$\scalebox{0.9}{ \xymatrix@1{
7 \ar@<1ex>^{(2)}[dr] & & & & & 4 \ar@<1ex>^{(0)}[dl] \\
& 6 \ar@<1ex>^{(0)}[ul] \ar@<1ex>^{(0)}[dl] \ar@<1ex>^{(2)}[r] & 1 \ar@<1ex>^{(0)}[l] \ar@<1ex>^{(2)}[r] & 2 \ar@<1ex>^{(0)}[l] \ar@<1ex>^{(2)}[r] & 3 \ar@<1ex>^{(0)}[l] \ar@<1ex>^{(2)}[ur] \ar@<1ex>^{(2)}[dr] & \\
8 \ar@<1ex>^{(2)}[ur] & & & & & 5 \ar@<1ex>^{(0)}[ul] }} $
\hspace{10pt}
$\xrightarrow{\text{mutation at 3}}$
\hspace{10pt}
$ \scalebox{0.9}{\xymatrix@1{
7 \ar@<1ex>^{(2)}[dr] & & & & & 4 \ar@<1ex>^{(1)}[dl] \ar@<-2ex>_{\scriptscriptstyle{(0)}}[dll] \\
& 6 \ar@<1ex>^{(0)}[ul] \ar@<1ex>^{(0)}[dl] \ar@<1ex>^{(2)}[r] & 1 \ar@<1ex>^{(0)}[l] \ar@<1ex>^{(2)}[r] & 2 \ar@<1ex>^{(0)}[l] \ar@<1ex>^{\scriptscriptstyle{(0)}}[r] \ar@<1ex>_{\scriptscriptstyle{(2)}}[urr] \ar@<-1ex>^{\scriptscriptstyle{(2)}}[drr] & 3 \ar@<1ex>^{\scriptscriptstyle{(2)}}[l] \ar[ur] \ar^{(1)}[dr] & \\
8 \ar@<1ex>^{(2)}[ur] & & & & & 5 \ar@<1ex>[ul] \ar@<2ex>^{\scriptscriptstyle{(0)}}[ull] }} $
\caption{Quiver mutation is compatible with the flip of an $(m+2)$-angulation, for $m=2$. In order to keep a clear figure, we have not replaced the arcs 8 and 5 by loops.}
\label{fig:compat2}
\end{figure}

\begin{lem}
The quiver fulfills the conditions asked for colored quivers in section \ref{sec:colq}. It has no loops, is monochrome and symmetric.
\end{lem}

\begin{proof}
By definition, the quiver contains no loops (i.e. no arrows from $i$ to $i$).

If there is an arrow from $i$ to $j$ of color $c$, it means that $i$ and $j$ share two sides of an $(m+2)$-angle. If we count from $i$ to $j$, there are $c$ edges between them. But, as we are considering $(m+2)$-angles, it means that from $j$ to $i$ there are $m-c$ edges. So there is an arrow from $j$ to $i$ of color $m-c$. Then the symmetry is respected.

Now it remains to show that the quiver is monochrome. We divide this proof into three cases.

\cor{First case: $m=1$. The possible colors are either $0$ or $1$. If we have $\xymatrix{i \ar[r]^{(1)} & j}$ and $\xymatrix{i \ar[r]^{(0)} & j}$. Then, we have an oriented angle from $i$ to $j$, and, with the symmetry of the figure, an arrow $\xymatrix{j \ar[r]^{(0)} & i}$, then an oriented angle from $j$ to $i$. This leads to a digon and not a triangle, as it must be the case with $m=1$. This leads to a contradiction, and there is only one possible color for arrows: the quiver is monochrome.}

\cor{If $i$ is a loop starting in $a$, around $L$ for instance, and $j$ linking $a$ to $L$, then there is no arrow between $i$ and $j$ by definition. If $j$ is another $m$-diagonal, then the situation is the same as above: two arrows of different color incident to $i$ would create a digon.}

\vspace{10pt}

Second case: $m > 2$. If we have $\xymatrix{i \ar[r]^{(c)} & j}$ and $\xymatrix{i \ar[r]^{(c')} & j}$, it means that $i$ and $j$ share two $(m+2)$-gons, one with $c$ edges between $i$ and $j$, and the other one with $c'$ edges. So we are in the type of situation of figure \ref{fig:oreille}).

\begin{figure}[!h]
\centering
\begin{tikzpicture}[scale=1]
\fill[fill=black,fill opacity=0.1] (0,4) -- (-1.66,3.46) -- (-2.69,2.05) -- (-2.68,0.3) -- (-1.66,-1.11) -- (0,-1.65) -- (1.66,-1.11) -- (2.69,0.3) -- (2.69,2.05) -- (1.66,3.46) -- cycle;
\fill[fill=black,fill opacity=1.0] (-1.94,1.52) -- (-1.94,1.26) -- (-1.71,1.13) -- (-1.49,1.26) -- (-1.49,1.52) -- (-1.71,1.65) -- cycle;
\fill[fill=black,fill opacity=1.0] (1.94,1.52) -- (1.94,1.26) -- (1.72,1.13) -- (1.49,1.26) -- (1.49,1.52) -- (1.72,1.65) -- cycle;
\draw (0,4)-- (-1.66,3.46);
\draw [dotted] (-1.66,3.46)-- (-2.69,2.05);
\draw (-2.69,2.05)-- (-2.68,0.3);
\draw [dotted] (-2.68,0.3)-- (-1.66,-1.11);
\draw (-1.66,-1.11)-- (0,-1.65);
\draw [dotted] (0,-1.65)-- (1.66,-1.11);
\draw (1.66,-1.11)-- (2.69,0.3);
\draw [dotted] (2.69,0.3)-- (2.69,2.05);
\draw (2.69,2.05)-- (1.66,3.46);
\draw [dotted] (1.66,3.46)-- (0,4);
\draw (-1.94,1.52)-- (-1.94,1.26);
\draw (-1.94,1.26)-- (-1.71,1.13);
\draw (-1.71,1.13)-- (-1.49,1.26);
\draw (-1.49,1.26)-- (-1.49,1.52);
\draw (-1.49,1.52)-- (-1.71,1.65);
\draw (-1.71,1.65)-- (-1.94,1.52);
\draw (-2.68,0.3)-- (-1.71,1.13);
\draw (-1.49,1.26)-- (1.52,1.51);
\draw (1.94,1.27)-- (2.69,0.3);
\draw [shift={(2.22,-0.44)},dash pattern=on 1pt off 1pt on 3pt off 4pt,color=red]  plot[domain=2.03:2.99,variable=\t]({1*4.96*cos(\t r)+0*4.96*sin(\t r)},{0*4.96*cos(\t r)+1*4.96*sin(\t r)});
\draw [shift={(-1.43,0.13)},dash pattern=on 1pt off 1pt on 3pt off 4pt,color=red]  plot[domain=0.04:1.22,variable=\t]({1*4.13*cos(\t r)+0*4.13*sin(\t r)},{0*4.13*cos(\t r)+1*4.13*sin(\t r)});
\draw [shift={(12.07,17.8)},dash pattern=on 1pt off 1pt on 3pt off 4pt,color=blue]  plot[domain=4.01:4.16,variable=\t]({1*22.89*cos(\t r)+0*22.89*sin(\t r)},{0*22.89*cos(\t r)+1*22.89*sin(\t r)});
\draw [shift={(-1.85,3.73)},dash pattern=on 1pt off 1pt on 3pt off 4pt,color=blue]  plot[domain=5.04:5.64,variable=\t]({1*5.69*cos(\t r)+0*5.69*sin(\t r)},{0*5.69*cos(\t r)+1*5.69*sin(\t r)});
\draw (-2.46,4.2) node[anchor=north west] {P1};
\draw (-2.05,-1.53) node[anchor=north west] {P2};
\draw [shift={(0,1.06)}] plot[domain=-0.24:3.36,variable=\t]({1*3.46*cos(\t r)+0*3.46*sin(\t r)},{0*3.46*cos(\t r)+1*3.46*sin(\t r)});
\draw [shift={(0,1.05)}] plot[domain=3.39:6.02,variable=\t]({1*3.04*cos(\t r)+0*3.04*sin(\t r)},{0*3.04*cos(\t r)+1*3.04*sin(\t r)});
\draw [shift={(-1.72,1.43)}] plot[domain=0.02:4.23,variable=\t]({1*0.37*cos(\t r)+0*0.37*sin(\t r)},{0*0.37*cos(\t r)+1*0.37*sin(\t r)});
\draw (-1.56,1.97) node[anchor=north west] {R1};
\draw [shift={(1.74,1.5)}] plot[domain=-0.6:2.66,variable=\t]({1*0.32*cos(\t r)+0*0.32*sin(\t r)},{0*0.32*cos(\t r)+1*0.32*sin(\t r)});
\draw (1.81,2.22) node[anchor=north west] {S1};
\draw (-1.63,1.02)-- (-1.36,1.12);
\draw (-1.52,1.02) node[anchor=north west] {R2};
\draw [shift={(1.71,1.36)}] plot[domain=3.26:5.28,variable=\t]({1*0.35*cos(\t r)+0*0.35*sin(\t r)},{0*0.35*cos(\t r)+1*0.35*sin(\t r)});
\draw (1.46,1.01) node[anchor=north west] {S2};
\draw (-0.87,3.31) node[anchor=north west] {Ti};
\draw (-1.06,-0.66) node[anchor=north west] {Ui};
\begin{scriptsize}
\draw[color=black] (-2.13,1.02) node {$i$};
\draw[color=black] (2.44,0.43) node {$j$};
\end{scriptsize}
\end{tikzpicture}
\caption{Pathological case. The number $R_k$ (respectively $S_k$) counts the number of thick vertices on each side or $R$ (respectively $S$). The numbers $T_l$ and $U_l$ counts the $m$-diagonals which are homotopic to the boundary path. Finally, $P_m$ denotes the number of sides from $i$ to $j$ (and from $j$ to $i$).}
\label{fig:oreille}
\end{figure}

Let $T_i$ (respectively $U_i$) be the vertices along the dotted edges cutting on the side of $P_1$ (respectively $P_2$) the figure into two parts, one of type $A$ and one of type $D$. On the figure, we can visualize $T_i$ in red and $U_i$ in blue. Let $R_1$ and $R_2$ (respectively $S_1$ and $S_2$) be the number of vertices as shown in figure \ref{fig:oreille} in $R$ (respectively in $S$).

Then, by making a computation on the total number of vertices, we are led to a contradiction.

Indeed, we have the following equations:

\begin{equation}
\left\{
    \begin{array}{l}
        \sum T_i+R_1+S_1=m-1 \\
        \sum U_i+R_2+S_2=m-1 \\
        R_1+R_2=m-1 \\
        S_1+S_2=m-1 \\
        \sum T_i + \sum U_i=n-2 \\
    \end{array}
    \right.
\end{equation}

Then we have that

\[ \sum T_i +R_1+S_1 + \sum U_i+R_2 +S_2= 2(m-1). \]

So

\[ \sum T_i+\sum U_i=0 \]

This leads to $\sum T_i=\sum U_i=0$ and this is impossible.

\vspace{10pt}

Third case: $m=2$. Consider three different possible sub-cases. 

First sub-case: if we have $\xymatrix{i \ar[r]^{(0)} & j}$ and $\xymatrix{i \ar[r]^{(1)} & j}$ it means that $i$ and $j$ belong to a triangle, which is impossible since $m=2$.

Second sub-case: if we have $\xymatrix{i \ar[r]^{(0)} & j}$ and $\xymatrix{i \ar[r]^{(2)} & j}$, it means that we have $\xymatrix{i \ar[r]^{(0)} & j}$ and $\xymatrix{j \ar[r]^{(0)} & i}$ which is also impossible.

Third sub-case: if we have $\xymatrix{i \ar[r]^{(1)} & j}$ and $\xymatrix{i \ar[r]^{(2)} & j}$, this reduces to the first case, since we have $\xymatrix{j \ar[r]^{(1)} & i}$ and $\xymatrix{j \ar[r]^{(0)} & i}$.

We thus have shown that there is no way to have two arrows from $i$ to $j$ of different colors This show the Lemma.
\end{proof}

\begin{theo}\label{theo:corresp}
Let $\Delta$ be any $(m+2)$-angulation. Let $Q_\Delta$ be the colored quiver associated with the $(m+2)$-angulation $\Delta$. If $\Delta_k$ is the new $(m+2)$-angulation flipped at $k$ from $\Delta$, then the colored quiver $Q_{\Delta_k}$ associated with $\Delta_k$ is the mutation at vertex $k$ of the colored quiver $Q_\Delta$ (see figures \ref{fig:compat} and \ref{fig:compat2} for an illustration).
\end{theo}

\begin{proof}
Let $\Delta$ be an $(m+2)$-angulation, let $Q_\Delta$ be the associated colored quiver, and let $k$ be a vertex of $Q$.

We want to show that $Q_{\mu_k(\Delta)}=\mu_k(Q_\Delta)$. Let us call by ${\mathcal{T}}$ the new $(m+2)$-angulation obtained from $\Delta$ by flipping the arc $k$. There is an evident bijection between the vertices of $Q_\Delta$ and $Q_T$. Let $i$ and $j$ be two vertices of $Q_{\mu_k(\Delta)}$ (and of $\mu_k(Q_\Delta$)).

Let $\tilde{q}_{ij}^{(c)}$ (respectively $\overline{q}_{ij}^{(c)}$) be the number of arrows of color $c$ from $i$ to $j$ in $Q_{\mu_k(\Delta)}$ (respectively $\mu_k(Q_\Delta)$). Let us show that $\tilde{q}_{ij}^{(c)}=\overline{q}_{ij}^{(c)}$.

First case: If $k=i$, then $\overline{q}_{ij}^{(c)}=q_{ij}^{(c-1)}$ and $\kappa_{\mathcal{T}}^{c}(i)=\kappa_\Delta^{c+1}(i)$. Consequently $\tilde{q}_{ij}^{(c)}=q_{ij}^{(c-1)}$.

Second case: If $k=j$, then $\overline{q}_{ij}^{(c)}=q_{ij}^{(c+1)}$ and $\kappa_{\mathcal{T}}^{c}(i)=\kappa_\Delta^{c-1}(i)$. Consequently $\tilde{q}_{ij}^{(c)}=q_{ij}^{(c+1)}$.

Third case (see figure \ref{fig:case33}): Assume that, in $Q_\Delta$, we have $\xymatrix@1{i\ar[r]^{(c)} & k\ar[r]^{(0)} & j}$ and $i$ and $j$ are not two sides of the same $(m+2)$-angle. Then $\overline{q}_{ij}^{(c)}=q_{ij}^{(c)}+1$. We notice that we cannot have two arrows from $k$ to $j$ of color $0$. We are in a situation of the following type:

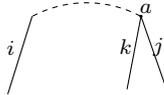
\begin{center}
\begin{tikzpicture}[scale=0.4]
\draw (1.26,2.8)-- (0.46,0.22);
\draw [shift={(3.05,-0.51)},dash pattern=on 2pt off 2pt]  plot[domain=1.08:2.07,variable=\t]({1*3.76*cos(\t r)+0*3.76*sin(\t r)},{0*3.76*cos(\t r)+1*3.76*sin(\t r)});
\draw (4.84,2.8)-- (4.38,0.38);
\draw (4.84,2.8)-- (5.64,0.5);
\begin{scriptsize}
\draw[color=black] (0.56,1.76) node {$i$};
\fill [color=black] (4.84,2.8) circle (1.5pt);
\draw[color=black] (5,3.06) node {$a$};
\draw[color=black] (4.34,1.8) node {$k$};
\draw[color=black] (5.42,1.78) node {$j$};
\end{scriptsize}
\end{tikzpicture}
\end{center}

Let then $a$ be the common vertex of $k$ and $j$. Let $b$ be the other vertex of $j$ (where $a$ or $b$ can be central polygons). Then $\kappa_\Delta(k)$ and $j$ have $b$ as a common vertex. As, in $\Delta$, $\kappa_\Delta^c(i)$ and $k$ have one common vertex $a$, then $\kappa_{\mathcal{T}}^{c}(i)$ and $j$ have one common vertex. As a consequence, $\tilde{q}_{ij}^{(c)}=q_{ij}^{(c)}+1$. There is no loss because there are not any arrows from $i$ to $j$ of color different from $c$. We note that there can be two arrows of different colors appearing from $j$ to $k$, but as this is forbidden, we remove them, and this confirms that there are no two such arcs in $\mu_k(\Delta)$. This case is symmetric to the case where $i$ and $k$ share a common vertex.

Fourth case: Assume that we are in a situation of the following type:

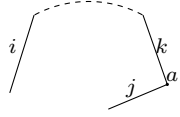
\begin{center}
\begin{tikzpicture}[scale=0.4]
\draw (1.26,2.8)-- (0.46,0.22);
\draw [shift={(3.05,-0.51)},dash pattern=on 2pt off 2pt]  plot[domain=1.08:2.07,variable=\t]({1*3.76*cos(\t r)+0*3.76*sin(\t r)},{0*3.76*cos(\t r)+1*3.76*sin(\t r)});
\draw (4.84,2.8)-- (5.64,0.5);
\draw (5.64,0.5)-- (3.7,-0.32);
\begin{scriptsize}
\draw[color=black] (0.56,1.76) node {$i$};
\draw[color=black] (5.46,1.78) node {$k$};
\fill [color=black] (5.64,0.5) circle (1.5pt);
\draw[color=black] (5.8,0.76) node {$a$};
\draw[color=black] (4.46,0.38) node {$j$};
\end{scriptsize}
\end{tikzpicture}
\end{center}

It means that we have a quiver:
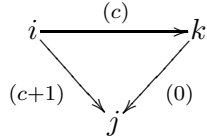
\[ \xymatrix@1{
i\ar[rr]^{(c)} \ar[dr]_{(c+1)} & & k \ar[dl]^{(0)} \\
& j
} \]

Then $\overline{q}_{ij}^{(c)}=q_{ij}^{(c)}-1=0$, because the arrows get erased.
Moreover, under the notations of the third case, $\kappa_\Delta(k)$ and $j$ have $b$ as a common vertex, and $i$ and $j$ do not share a common polygon anymore. Then $\tilde{q}_{ij}^{(c)}=q_{ij}^{(c)}-1$.
\end{proof}

\begin{ex}
Let us draw an example in figure \ref{fig:case33}.
	
\begin{figure}[!h]
\centering
\begin{tikzpicture}[scale=0.7]
\fill[fill=black,fill opacity=0.1] (0,4) -- (-1.66,3.46) -- (-2.69,2.05) -- (-2.68,0.3) -- (-1.66,-1.11) -- (0,-1.65) -- (1.66,-1.11) -- (2.69,0.3) -- (2.69,2.05) -- (1.66,3.46) -- cycle;
\draw [fill=black,fill opacity=1.0] (-1.66,1.57) circle (0.2cm);
\draw [fill=black,fill opacity=1.0] (1.66,1.58) circle (0.2cm);
\draw [dash pattern=on 1pt off 1pt on 3pt off 4pt](0,4)-- (-1.66,3.46);
\draw (-1.66,3.46)-- (-2.69,2.05);
\draw (-2.69,2.05)-- (-2.68,0.3);
\draw (-2.68,0.3)-- (-1.66,-1.11);
\draw [dash pattern=on 1pt off 1pt on 3pt off 4pt](-1.66,-1.11)-- (0,-1.65);
\draw [dash pattern=on 1pt off 1pt on 3pt off 4pt](0,-1.65)-- (1.66,-1.11);
\draw (1.66,-1.11)-- (2.69,0.3);
\draw (2.69,0.3)-- (2.69,2.05);
\draw (2.69,2.05)-- (1.66,3.46);
\draw [dash pattern=on 1pt off 1pt on 3pt off 4pt](1.66,3.46)-- (0,4);
\draw(-1.66,1.17) circle (0.4cm);
\draw(1.66,1.18) circle (0.4cm);
\draw [shift={(3.19,3.49)}] plot[domain=2.98:4.56,variable=\t]({1*3.23*cos(\t r)+0*3.23*sin(\t r)},{0*3.23*cos(\t r)+1*3.23*sin(\t r)});
\draw [shift={(19.33,1.18)}] plot[domain=3:3.29,variable=\t]({1*19.53*cos(\t r)+0*19.53*sin(\t r)},{0*19.53*cos(\t r)+1*19.53*sin(\t r)});
\draw [shift={(-9.63,111.22)}] plot[domain=3.03:3.41,variable=\t]({-0.07*111.42*cos(\t r)+-1*7.82*sin(\t r)},{1*111.42*cos(\t r)+-0.07*7.82*sin(\t r)});
\draw (-0.74,1.61) node[anchor=north west] {i};
\draw (-0.06,0.73) node[anchor=north west] {k};
\draw (0.45,2.1) node[anchor=north west] {j};
\end{tikzpicture}
$\to$
\begin{tikzpicture}[scale=0.7]
\fill[fill=black,fill opacity=0.1] (0,4) -- (-1.66,3.46) -- (-2.69,2.05) -- (-2.68,0.3) -- (-1.66,-1.11) -- (0,-1.65) -- (1.66,-1.11) -- (2.69,0.3) -- (2.69,2.05) -- (1.66,3.46) -- cycle;
\draw [fill=black,fill opacity=1.0] (-1.66,1.57) circle (0.2cm);
\draw [fill=black,fill opacity=1.0] (1.66,1.58) circle (0.2cm);
\draw [dash pattern=on 1pt off 1pt on 3pt off 4pt](0,4)-- (-1.66,3.46);
\draw (-1.66,3.46)-- (-2.69,2.05);
\draw (-2.69,2.05)-- (-2.68,0.3);
\draw (-2.68,0.3)-- (-1.66,-1.11);
\draw [dash pattern=on 1pt off 1pt on 3pt off 4pt](-1.66,-1.11)-- (0,-1.65);
\draw [dash pattern=on 1pt off 1pt on 3pt off 4pt](0,-1.65)-- (1.66,-1.11);
\draw (1.66,-1.11)-- (2.69,0.3);
\draw (2.69,0.3)-- (2.69,2.05);
\draw (2.69,2.05)-- (1.66,3.46);
\draw [dash pattern=on 1pt off 1pt on 3pt off 4pt](1.66,3.46)-- (0,4);
\draw(-1.66,1.17) circle (0.4cm);
\draw(1.66,1.18) circle (0.4cm);
\draw [shift={(3.19,3.49)}] plot[domain=2.98:4.56,variable=\t]({1*3.23*cos(\t r)+0*3.23*sin(\t r)},{0*3.23*cos(\t r)+1*3.23*sin(\t r)});
\draw [shift={(-9.63,111.22)}] plot[domain=3.03:3.41,variable=\t]({-0.07*111.42*cos(\t r)+-1*7.82*sin(\t r)},{1*111.42*cos(\t r)+-0.07*7.82*sin(\t r)});
\draw (-0.74,1.61) node[anchor=north west] {i};
\draw (0.64,-0.7) node[anchor=north west] {k};
\draw (0.45,2.1) node[anchor=north west] {j};
\draw [shift={(-0.56,2.9)}] plot[domain=4.45:5.61,variable=\t]({1*4.16*cos(\t r)+0*4.16*sin(\t r)},{0*4.16*cos(\t r)+1*4.16*sin(\t r)});
\end{tikzpicture}
\caption{The mutation at $k$ in the third case}
\label{fig:case33}
\end{figure}
	
We have $\xymatrix@1{i\ar[r]^{(0)} & k\ar[r]^{(0)} & j}$. When we mutate $k$, we obtain the figure in the right. The quiver becomes 
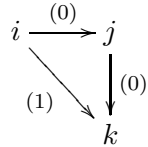
\[\xymatrix{i \ar^{(0)}[r] \ar_{(1)}[dr] & j \ar^{(0)}[d] \\ & k}\] (where we do not note the inverse arrow for sake of clarity).
	
Now let us mutate at vertex $i$. The $4$-angulation becomes:
	
\begin{figure}[!h]
\centering
\begin{tikzpicture}[scale=0.7]
\fill[fill=black,fill opacity=0.1] (0,4) -- (-1.66,3.46) -- (-2.69,2.05) -- (-2.68,0.3) -- (-1.66,-1.11) -- (0,-1.65) -- (1.66,-1.11) -- (2.69,0.3) -- (2.69,2.05) -- (1.66,3.46) -- cycle;
\draw [fill=black,fill opacity=1.0] (-1.66,1.57) circle (0.2cm);
\draw [fill=black,fill opacity=1.0] (1.66,1.58) circle (0.2cm);
\draw [dash pattern=on 1pt off 1pt on 3pt off 4pt](0,4)-- (-1.66,3.46);
\draw (-1.66,3.46)-- (-2.69,2.05);
\draw (-2.69,2.05)-- (-2.68,0.3);
\draw (-2.68,0.3)-- (-1.66,-1.11);
\draw [dash pattern=on 1pt off 1pt on 3pt off 4pt](-1.66,-1.11)-- (0,-1.65);
\draw [dash pattern=on 1pt off 1pt on 3pt off 4pt](0,-1.65)-- (1.66,-1.11);
\draw (1.66,-1.11)-- (2.69,0.3);
\draw (2.69,0.3)-- (2.69,2.05);
\draw (2.69,2.05)-- (1.66,3.46);
\draw [dash pattern=on 1pt off 1pt on 3pt off 4pt](1.66,3.46)-- (0,4);
\draw(-1.66,1.17) circle (0.4cm);
\draw(1.66,1.18) circle (0.4cm);
\draw [shift={(3.19,3.49)}] plot[domain=2.98:4.56,variable=\t]({1*3.23*cos(\t r)+0*3.23*sin(\t r)},{0*3.23*cos(\t r)+1*3.23*sin(\t r)});
\draw (-0.62,-0.03) node[anchor=north west] {i};
\draw (0.64,-0.7) node[anchor=north west] {k};
\draw (0.45,2.1) node[anchor=north west] {j};
\draw [shift={(-0.56,2.9)}] plot[domain=4.45:5.61,variable=\t]({1*4.16*cos(\t r)+0*4.16*sin(\t r)},{0*4.16*cos(\t r)+1*4.16*sin(\t r)});
\draw [shift={(0.52,2.78)}] plot[domain=3.37:5.43,variable=\t]({1*3.29*cos(\t r)+0*3.29*sin(\t r)},{0*3.29*cos(\t r)+1*3.29*sin(\t r)});
\end{tikzpicture}
\end{figure}
	
If we mutate the colored quivers, since two arrows of different colors are drawn from $j$ to $k$ (one of color $(0)$ and one of color $(1)$), they are removed and the quiver becomes $j \to i \to k$ of color $(0)$. This exactly corresponds to the quiver of the $4$-angulation.
\end{ex}

\section{Equivalence of categories}
In this section, assuming that $Q$ is still a quiver of type $\tilde{D_n}$, we build an additive functor $F$ from the category ${\mathcal{C}}$ to the $m$-cluster category.

\begin{defi}\label{def:shift}
Let $P$ be the polygon considered in section \ref{sec:geo}. Let us number the vertices of $P$ clockwise. Let $\alpha$ be an $m$-diagonal from $i$ to $j$. The shift $\alpha[1]$ of $\alpha$ is defined as follows:
\begin{enumerate}
\item If $i \neq j$ and $\alpha$ is of type $1$ or $2$, then $\alpha[1]$ is the unique new arc obtained by composing $\alpha$ with the boundary paths $B_{i~i+1}$ and $B_{j~j+1}$, linking $i+1$ and $j+1$, and of the same type.
\item If $i=j$, $\alpha$ is of type $3$, then $\alpha[1]$ is the new arc tangent to the same vertex of the inner polygon, ending at $i+1$. If $m$ is odd, then change the side of the tangency. (see figure \ref{fig:shift1}).
\item If $\alpha$ is of type $4$, which means it links both inner polygons, from $L_k$ to $R_{k'}$ for instance, the new arc $\alpha[1]$ links $L_{k-1}$ and $R_{k'+1}$. The side of the tangency changes only if $m$ is odd (see figure \ref{fig:shift2}).
\end{enumerate}
\end{defi}

\begin{figure}[!h]
\centering
\begin{tikzpicture}[scale=0.25]
\fill[fill=black,fill opacity=0.15] (0.,4.) -- (0.,-4.) -- (6.92820323028,-8.) -- (13.8564064606,-4.) -- (13.8564064606,4.) -- (6.92820323028,8.) -- cycle;
\draw [fill=black,fill opacity=1.0] (3.46410161514,2.) circle (0.4cm);
\draw [fill=black,fill opacity=1.0] (3.46410161514,-2.) circle (0.4cm);
\draw [fill=black,fill opacity=1.0] (10.3923048454,2.) circle (0.4cm);
\draw [fill=black,fill opacity=1.0] (10.3923048454,-2.) circle (0.4cm);
\draw (0.,4.)-- (0.,-4.);
\draw (0.,-4.)-- (6.92820323028,-8.);
\draw (6.92820323028,-8.)-- (13.8564064606,-4.);
\draw (13.8564064606,-4.)-- (13.8564064606,4.);
\draw (13.8564064606,4.)-- (6.92820323028,8.);
\draw (6.92820323028,8.)-- (0.,4.);
\draw [shift={(1.85683798862,0.)}] plot[domain=-0.893844784235:0.893844784235,variable=\t]({1.*2.56579351568*cos(\t r)+0.*2.56579351568*sin(\t r)},{0.*2.56579351568*cos(\t r)+1.*2.56579351568*sin(\t r)});
\draw [shift={(5.07136524165,0.)}] plot[domain=2.24774786935:4.03543743782,variable=\t]({1.*2.56579351568*cos(\t r)+0.*2.56579351568*sin(\t r)},{0.*2.56579351568*cos(\t r)+1.*2.56579351568*sin(\t r)});
\draw [shift={(11.9995684719,0.)}] plot[domain=2.24774786935:4.03543743782,variable=\t]({1.*2.56579351568*cos(\t r)+0.*2.56579351568*sin(\t r)},{0.*2.56579351568*cos(\t r)+1.*2.56579351568*sin(\t r)});
\draw [shift={(8.7850412189,0.)}] plot[domain=-0.893844784235:0.893844784235,variable=\t]({1.*2.56579351568*cos(\t r)+0.*2.56579351568*sin(\t r)},{0.*2.56579351568*cos(\t r)+1.*2.56579351568*sin(\t r)});
\draw [shift={(-8.39982219049,-6.22540112081)}] plot[domain=0.414238148445:0.748105935798,variable=\t]({1.*20.9119678736*cos(\t r)+0.*20.9119678736*sin(\t r)},{0.*20.9119678736*cos(\t r)+1.*20.9119678736*sin(\t r)});
\end{tikzpicture}
$\to$
\begin{tikzpicture}[scale=0.25]
\fill[fill=black,fill opacity=0.15] (0.,4.) -- (0.,-4.) -- (6.92820323028,-8.) -- (13.8564064606,-4.) -- (13.8564064606,4.) -- (6.92820323028,8.) -- cycle;
\draw [fill=black,fill opacity=1.0] (3.46410161514,2.) circle (0.4cm);
\draw [fill=black,fill opacity=1.0] (3.46410161514,-2.) circle (0.4cm);
\draw [fill=black,fill opacity=1.0] (10.3923048454,2.) circle (0.4cm);
\draw [fill=black,fill opacity=1.0] (10.3923048454,-2.) circle (0.4cm);
\draw (0.,4.)-- (0.,-4.);
\draw (0.,-4.)-- (6.92820323028,-8.);
\draw (6.92820323028,-8.)-- (13.8564064606,-4.);
\draw (13.8564064606,-4.)-- (13.8564064606,4.);
\draw (13.8564064606,4.)-- (6.92820323028,8.);
\draw (6.92820323028,8.)-- (0.,4.);
\draw [shift={(1.85683798862,0.)}] plot[domain=-0.893844784235:0.893844784235,variable=\t]({1.*2.56579351568*cos(\t r)+0.*2.56579351568*sin(\t r)},{0.*2.56579351568*cos(\t r)+1.*2.56579351568*sin(\t r)});
\draw [shift={(5.07136524165,0.)}] plot[domain=2.24774786935:4.03543743782,variable=\t]({1.*2.56579351568*cos(\t r)+0.*2.56579351568*sin(\t r)},{0.*2.56579351568*cos(\t r)+1.*2.56579351568*sin(\t r)});
\draw [shift={(11.9995684719,0.)}] plot[domain=2.24774786935:4.03543743782,variable=\t]({1.*2.56579351568*cos(\t r)+0.*2.56579351568*sin(\t r)},{0.*2.56579351568*cos(\t r)+1.*2.56579351568*sin(\t r)});
\draw [shift={(8.7850412189,0.)}] plot[domain=-0.893844784235:0.893844784235,variable=\t]({1.*2.56579351568*cos(\t r)+0.*2.56579351568*sin(\t r)},{0.*2.56579351568*cos(\t r)+1.*2.56579351568*sin(\t r)});
\draw [shift={(13.6041993016,-0.403940802752)}] plot[domain=1.51359029435:2.49336089239,variable=\t]({1.*4.41115665616*cos(\t r)+0.*4.41115665616*sin(\t r)},{0.*4.41115665616*cos(\t r)+1.*4.41115665616*sin(\t r)});
\end{tikzpicture}
\label{fig:shift1}
\caption{The shift for $m$ odd and $\alpha$ of type $3$}
\end{figure}

\begin{figure}[!h]
\centering
\begin{tikzpicture}[scale=0.25]
\fill[fill=black,fill opacity=0.15] (0.,4.) -- (0.,-4.) -- (6.92820323028,-8.) -- (13.8564064606,-4.) -- (13.8564064606,4.) -- (6.92820323028,8.) -- cycle;
\draw [fill=black,fill opacity=1.0] (3.46410161514,2.) circle (0.4cm);
\draw [fill=black,fill opacity=1.0] (3.46410161514,-2.) circle (0.4cm);
\draw [fill=black,fill opacity=1.0] (10.3923048454,2.) circle (0.4cm);
\draw [fill=black,fill opacity=1.0] (10.3923048454,-2.) circle (0.4cm);
\draw (0.,4.)-- (0.,-4.);
\draw (0.,-4.)-- (6.92820323028,-8.);
\draw (6.92820323028,-8.)-- (13.8564064606,-4.);
\draw (13.8564064606,-4.)-- (13.8564064606,4.);
\draw (13.8564064606,4.)-- (6.92820323028,8.);
\draw (6.92820323028,8.)-- (0.,4.);
\draw [shift={(1.85683798862,0.)}] plot[domain=-0.893844784235:0.893844784235,variable=\t]({1.*2.56579351568*cos(\t r)+0.*2.56579351568*sin(\t r)},{0.*2.56579351568*cos(\t r)+1.*2.56579351568*sin(\t r)});
\draw [shift={(5.07136524165,0.)}] plot[domain=2.24774786935:4.03543743782,variable=\t]({1.*2.56579351568*cos(\t r)+0.*2.56579351568*sin(\t r)},{0.*2.56579351568*cos(\t r)+1.*2.56579351568*sin(\t r)});
\draw [shift={(11.9995684719,0.)}] plot[domain=2.24774786935:4.03543743782,variable=\t]({1.*2.56579351568*cos(\t r)+0.*2.56579351568*sin(\t r)},{0.*2.56579351568*cos(\t r)+1.*2.56579351568*sin(\t r)});
\draw [shift={(8.7850412189,0.)}] plot[domain=-0.893844784235:0.893844784235,variable=\t]({1.*2.56579351568*cos(\t r)+0.*2.56579351568*sin(\t r)},{0.*2.56579351568*cos(\t r)+1.*2.56579351568*sin(\t r)});
\draw [shift={(11.077585728,-2.29029671889)}] plot[domain=1.71504065064:2.63724523296,variable=\t]({1.*4.739497249*cos(\t r)+0.*4.739497249*sin(\t r)},{0.*4.739497249*cos(\t r)+1.*4.739497249*sin(\t r)});
\draw [shift={(2.77882073252,2.29029671889)}] plot[domain=4.85663330423:5.77883788655,variable=\t]({1.*4.739497249*cos(\t r)+0.*4.739497249*sin(\t r)},{0.*4.739497249*cos(\t r)+1.*4.739497249*sin(\t r)});
\end{tikzpicture}
$\to$
\begin{tikzpicture}[scale=0.25]
\fill[fill=black,fill opacity=0.15] (0.,4.) -- (0.,-4.) -- (6.92820323028,-8.) -- (13.8564064606,-4.) -- (13.8564064606,4.) -- (6.92820323028,8.) -- cycle;
\draw [fill=black,fill opacity=1.0] (3.46410161514,2.) circle (0.4cm);
\draw [fill=black,fill opacity=1.0] (3.46410161514,-2.) circle (0.4cm);
\draw [fill=black,fill opacity=1.0] (10.3923048454,2.) circle (0.4cm);
\draw [fill=black,fill opacity=1.0] (10.3923048454,-2.) circle (0.4cm);
\draw (0.,4.)-- (0.,-4.);
\draw (0.,-4.)-- (6.92820323028,-8.);
\draw (6.92820323028,-8.)-- (13.8564064606,-4.);
\draw (13.8564064606,-4.)-- (13.8564064606,4.);
\draw (13.8564064606,4.)-- (6.92820323028,8.);
\draw (6.92820323028,8.)-- (0.,4.);
\draw [shift={(1.85683798862,0.)}] plot[domain=-0.893844784235:0.893844784235,variable=\t]({1.*2.56579351568*cos(\t r)+0.*2.56579351568*sin(\t r)},{0.*2.56579351568*cos(\t r)+1.*2.56579351568*sin(\t r)});
\draw [shift={(5.07136524165,0.)}] plot[domain=2.24774786935:4.03543743782,variable=\t]({1.*2.56579351568*cos(\t r)+0.*2.56579351568*sin(\t r)},{0.*2.56579351568*cos(\t r)+1.*2.56579351568*sin(\t r)});
\draw [shift={(11.9995684719,0.)}] plot[domain=2.24774786935:4.03543743782,variable=\t]({1.*2.56579351568*cos(\t r)+0.*2.56579351568*sin(\t r)},{0.*2.56579351568*cos(\t r)+1.*2.56579351568*sin(\t r)});
\draw [shift={(8.7850412189,0.)}] plot[domain=-0.893844784235:0.893844784235,variable=\t]({1.*2.56579351568*cos(\t r)+0.*2.56579351568*sin(\t r)},{0.*2.56579351568*cos(\t r)+1.*2.56579351568*sin(\t r)});
\draw [shift={(12.6114252192,4.50603240067)}] plot[domain=3.81196831547:4.40190091165,variable=\t]({1.*7.25281601667*cos(\t r)+0.*7.25281601667*sin(\t r)},{0.*7.25281601667*cos(\t r)+1.*7.25281601667*sin(\t r)});
\draw [shift={(1.24498124132,-4.50603240067)}] plot[domain=0.670375661881:1.26030825806,variable=\t]({1.*7.25281601667*cos(\t r)+0.*7.25281601667*sin(\t r)},{0.*7.25281601667*cos(\t r)+1.*7.25281601667*sin(\t r)});
\end{tikzpicture}
\label{fig:shift2}
\caption{The shift for $m$ odd and $\alpha$ of type $4$}
\end{figure}

\begin{defi}
If $d$ is an integer in $\{ 1, \cdots, m \}$, then we define $S^d$ as the connected component in the quiver of Definition \ref{def:carquois} containing the $d$-th shift of the initial $(m+2)$-angulation, the first component being the one corresponding to the initial $(m+2)$-angulation, and the following ones corresponding to the initial $(m+2)$-angulation where $[1],[2],\cdots$ has been applied.

Let ${\mathcal{S}}^d$ be the component in the Auslander-Reiten quiver of ${\mathcal{C}}^{(m)}_{\tilde{D_n}}$, containing objects of type $\tau ^s P[d]$ where $P$ is a projective indecomposable object and $s$ is an integer.
\end{defi}

\begin{theo}\label{th:iso}
Let $d \in \{ 1, \cdots, m \}$. Then we have an isomorphism between both components:

\[ S^d \simeq {\mathcal{S}}^d \]
\end{theo}

\begin{proof}
For $n \geq 6$, we fix an orientation for our quivers. The orientation we choose is the one of the quiver corresponding to the initial $m$-angulation:

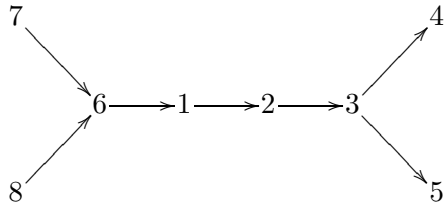
\[ \xymatrix@1{
n \ar[dr] & & & & & n-3 \\
& n-1 \ar[r] & 1 \ar[r] & 2 \ar@{.>}[r] & n-4 \ar[ur]\ar[dr] & \\
n+1 \ar[ur] & & & & & n-2} \]

Let $\Delta^0_1$ be the initial $(m+2)$-angulation in the first copy of the quiver $\tilde{Q}$. It corresponds to the $m$-cluster-tilting object $T= \bigoplus P_i$ at the initial slice of the Auslander-Reiten quiver of $Q$. With the $m$-diagonal $\alpha_i$ (at vertex $i$), we associate the projective module $P_i=F(\alpha_i)$. Then we associate with the elementary moves of $m$-diagonals, the irreducible morphisms in the cluster category. Indeed, according to the tabular of section \ref{tab:tab}, there is an elementary move from $\alpha$ to $\beta$ when both share an oriented angle. As it has been told in Definition \ref{def:quiver}, in the quiver it means that there is an arrow from $\alpha$ to $\beta$. So there is an irreducible morphism from $F(\alpha)$ to $F(\beta)$.

Now that we have treated the case of the initial slice in the first copy, still calling $\alpha_i$ the $m$-diagonal at vertex $i$ in this slice, we note that every $m$-diagonal in the component $S^d$ in the quiver $\tilde{Q}$ is of the form $\tau ^t \alpha_i [d]$, for some $t \in {\mathbb{Z}}$ and $i \in \{ 1,\cdots,n+1 \}$. So it is natural to define

\[ F(\tau ^t \alpha_i [d])= \tau ^t (P_i[d])  \forall t. \]

Therefore, there is a bijection between the arcs of $S^1$ and the indecomposables of the transjective components in the Auslander-Reiten quiver of ${\mathcal{C}}^{(m)}_{\tilde{D_n}}$ containing the image of the indecomposable projective objects.

Indeed, there are two things to show: First, for any arc $\alpha$, there is a unique way to write it $\tau ^t \alpha_i[d]$. If we had $\tau ^t \alpha_i[d]=\tau ^s \alpha_j[d]$ in a component $d$, then $\tau^t \alpha_i=\tau ^s \alpha_j$, so $i=j$ because of the type of the diagonal (if it is tangent one or another side of the central polygon, or if it is of type $1$, with length $l$, etc...), and if $t \neq s$ it would mean that there is the same $m$-diagonal at two different vertices in the quiver $\tilde{Q}$, which is impossible. Second, if $X$ is an indecomposable in the $d$-th component of $\tilde{Q}$, it can be written $\tau ^t P_i[d]$, and this has a unique antecedent by $F$, which is $\tau ^t \alpha_i[d]$.

As a matter of conclusion, $F$ is bijective on the objects. Let us show that it is also a bijection on the elementary moves. Suppose that $f : \alpha \to \beta$ is an elementary move. Then there exist $s \in {\mathbb{N}}, i \in \{1, \cdots, n+1 \}$, $d \in \{1, \cdots, m \}$ such that $\tau ^s \alpha = \alpha_i[d]$ and $\tau ^s \beta = \beta^*[d]$ (i.e $\alpha=\tau ^{-s} \alpha_i [d]$ and $\beta=\tau ^{-s} \beta^* [d]$, where $\beta^*=\alpha_j$ when there is an elementary move $\alpha_i \to \alpha_j$ or $\beta^*=\tau^{-1}\alpha_j$ where there is an elementary move $\alpha_j \to \alpha_i$. In the first case, there is an irreducible morphism from $P_i$ to $P_j$ thus from $F(\alpha_i)$ to $F(\beta^*)$ and from $F(\alpha)=\tau^{-s}F(\alpha_i)[d]$ to $F(\beta)=\tau^{-s}F(\beta^*)[d]$. In the second case, there is an irreducible morphism from $P_j$ to $P_i$ thus from $P_i$ to $\tau^{-1}P_j$ and from $F(\tau ^s \alpha)=F(\alpha_i[d]) \to F(\beta_i[d])=F(\tau ^s \beta)$, thus an arrow $F(\alpha) \to F(\beta)$.

This proves the statement.
\end{proof}

\begin{defi}
\cor{Let $d \in \{ 1, \cdots, m \}$. We recall that $T^{d}_2$ (respectively $T'^{d}_2$) is the connected components in the quiver composed by the $d$-th shift of the two types of $m$-diagonals of type $4$. Let $T^d_{n-2}$ be the connected components in the quiver composed by the $d$-th shift of the $m$-diagonals of type $4$.}

Let ${\mathcal{T}}^d_2$ and ${\mathcal{T'}}^d_2$ be the $d$-th shift of the tubes of size $2$ in the Auslander-Reiten quiver of ${\mathcal{C}}^{(m)}_{\tilde{D_n}}$. Let ${\mathcal{T}}^d_{n-2}$ be the $d$-th shift of the tube of size $n-2$.
\end{defi}

\begin{theo}
\cor{Under the notations of the previous definition, we have the isomorphism between the pairs of components \[T^d_2 \simeq {\mathcal{T}}^d_2 \text{ ; } T'^d_2 \simeq {\mathcal{T'}}^d_2 \text{ and } T^d_{n-2} \simeq {\mathcal{T}}^d_{n-2}\]}
\end{theo}

\begin{proof}
We only state the proof for the bijection $T^{n-2} \simeq {\mathcal{T}}^{n-2}$, the others are similar.

\cor{We first note that the components of the quiver are of the same size and shape. Indeed, ${\mathcal{T}}^{n-2}$ is a non-homogeneous tube of size $n-2$. Moreover, the tube $T^{n-2}$ is also of size $n-2$ (because the $m$-diagonals are cyclic of order $n-2$). We know from \cite{JM} that the simple module at the base of the first tube corresponds to the arc of type $2$ and of size $m$ starting at vertex $1$. Then there is a 1-1 correspondence between the elementary moves and the irreducible morphisms.}
\end{proof}

\section{About non-self-crossing arcs and $m$-rigid objects}

We know by Theorem \ref{th:iso} that an arc in the geometric realization can be interpreted as an object in the $m$-cluster-category. We use the following notation: if $\alpha$ is an arc in the category ${\mathcal{C}}$ let $X_\alpha$ be the object associated in the $m$-cluster category.

\begin{lem}\label{lem:tri}
Let $A$ be the set of equivalence class of $m$-diagonals with no selfcrossing. The application \[\begin{array}{ccccc}
 & A & \to & \{m\text{-rigid indecomposable objects}\} \\
 & \alpha & \mapsto & X_\alpha \\
\end{array}\] is a bijection.
\end{lem}

\begin{proof}
With each arc $\alpha$ which does not cross itself we associate by Theorem \ref{th:iso} an $m$-rigid object $X_\alpha$. Then $\forall i \in \{1,\cdots,m \}~{\mathrm{Ext}}^i(X_\alpha,X_\alpha)=0$.

Indeed, first, if $\alpha$ belongs to the preprojective part of the Auslander-Reiten quiver, it belongs in particular to a slice of it which forms a quiver of type $\tilde{D_n}$. If we add all the arcs forming this quiver, we obtain an $(m+2)$-angulation. In the Auslander-Reiten quiver of ${\mathcal{C}}^m_Q$, by Theorem \ref{th:iso}, this slice corresponds to an $m$-cluster-tilting object, made of the sum of all $m$-rigid objects at the vertices of the slice. Then, with $\alpha$ is associated an $m$-rigid object $X_\alpha$, and as $\alpha$ is in this part of the Auslander-Reiten quiver of $Q$, it does not cross itself.

Second, if $\alpha$ belongs to a tube, then if it does not cross itself, it is at the base of it, and by the previous section, this corresponds to an $m$-rigid object (because in the case of a tube of size $r$, $X_\alpha$ is situated at one of the $r-1$ first lines). Conversely, if $X_\alpha$ is $m$-rigid, then it is at the base (in the first $r-1$ lines) of the tube of size $r$, it means that it corresponds to an arc of the first lines, which does not cross itself.

We have thus shown that $\alpha$ is an arc which does not cross itself if and only if $\forall i \in \{1,\cdots,m \}, {\mathrm{Ext}}^i(X_\alpha,X_\alpha)=0$.
\end{proof}

\begin{rmk}
We can extend this bijection to the arcs in the non-homogeneous tubes including arcs with a self-crossing. Indeed, the arcs at the base of the tube are the arcs of lenght $m$. If we number the levels in the tube, by starting at $0$ for the arcs of length $m$, then the arcs at level $i$ are the arcs of length $im$. This can be easily visualized in figure \ref{fig:tubu}.
\end{rmk}

In \cite[Theorem 4.5]{JM}, it is shown that there is a bijection between $(m+2)$-angulations and $m$-cluster-tilting objects.

\section{An example}

Let us resume this paper with a complete study of an Euclidean quiver of type $\tilde{D_n}$, and its $m$-cluster category, taking $n=7$ and $m=2$. So we study the following category \[ {\mathcal{C}}^2_{\tilde{D_7}}={\mathcal{D}}^b(K \tilde{D_7})/\tau^{-1}[2], \] where $\tilde{D_7}$ is the following quiver:

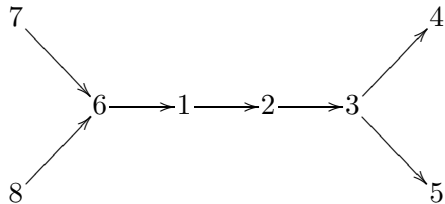
\[ \xymatrix@1{
7\ar[dr] & & & & & 4 \\
& 6\ar[r] & 1\ar[r] & 2\ar[r] & 3\ar[ur]\ar[dr] & \\
8\ar[ur] & & & & & 5} \]

\subsection{Quadrangulations of the decagon and colored quiver}
The geometric realization shown in figure \ref{fig:decagon} for this quiver is a decagon with two squares inside of it.

\begin{figure}[!h]
\centering
\begin{tikzpicture}[scale=0.8]
\fill[fill=black,fill opacity=0.1] (0,4) -- (-1.66,3.46) -- (-2.69,2.05) -- (-2.68,0.3) -- (-1.66,-1.11) -- (0,-1.65) -- (1.66,-1.11) -- (2.69,0.3) -- (2.69,2.05) -- (1.66,3.46) -- cycle;
\draw [fill=black,fill opacity=1.0] (1.66,1.61) circle (0.2cm);
\draw [fill=black,fill opacity=1.0] (-1.66,1.61) circle (0.2cm);
\draw (0,4)-- (-1.66,3.46);
\draw (-1.66,3.46)-- (-2.69,2.05);
\draw (-2.69,2.05)-- (-2.68,0.3);
\draw (-2.68,0.3)-- (-1.66,-1.11);
\draw (-1.66,-1.11)-- (0,-1.65);
\draw (0,-1.65)-- (1.66,-1.11);
\draw (1.66,-1.11)-- (2.69,0.3);
\draw (2.69,0.3)-- (2.69,2.05);
\draw (2.69,2.05)-- (1.66,3.46);
\draw (1.66,3.46)-- (0,4);
\draw(1.67,1.18) circle (0.44cm);
\draw(-1.67,1.17) circle (0.44cm);
\end{tikzpicture}
\caption{The decagon with two monogons}
\label{fig:decagon}
\end{figure}

We are now drawing eight $2$-diagonals in order to cut the decagon in quadrangles. This is in figure \ref{fig:triinit2} the initial $(m+2)$-angulation from which we will build later the translation quiver isomorphic to a sub-quiver of the Auslander-Reiten quiver of ${\mathcal{C}}^{(m)}_{\tilde{D_7}}$.

\begin{figure}[!h]
\centering
\begin{tikzpicture}[scale=0.8]
\fill[fill=black,fill opacity=0.1] (0,4) -- (-1.66,3.46) -- (-2.69,2.05) -- (-2.68,0.3) -- (-1.66,-1.11) -- (0,-1.65) -- (1.66,-1.11) -- (2.69,0.3) -- (2.69,2.05) -- (1.66,3.46) -- cycle;
\draw [fill=black,fill opacity=1.0] (1.66,1.61) circle (0.2cm);
\draw [fill=black,fill opacity=1.0] (-1.66,1.61) circle (0.2cm);
\draw (0,4)-- (-1.66,3.46);
\draw (-1.66,3.46)-- (-2.69,2.05);
\draw (-2.69,2.05)-- (-2.68,0.3);
\draw (-2.68,0.3)-- (-1.66,-1.11);
\draw (-1.66,-1.11)-- (0,-1.65);
\draw (0,-1.65)-- (1.66,-1.11);
\draw (1.66,-1.11)-- (2.69,0.3);
\draw (2.69,0.3)-- (2.69,2.05);
\draw (2.69,2.05)-- (1.66,3.46);
\draw (1.66,3.46)-- (0,4);
\draw [shift={(-8.03,-1.44)}] plot[domain=-0.03:0.66,variable=\t]({1*8.03*cos(\t r)+0*8.03*sin(\t r)},{0*8.03*cos(\t r)+1*8.03*sin(\t r)});
\draw [shift={(-0.75,-11.66)}] plot[domain=-0.82:0.59,variable=\t]({-0.07*14.49*cos(\t r)+-1*1.99*sin(\t r)},{1*14.49*cos(\t r)+-0.07*1.99*sin(\t r)});
\draw [shift={(8.03,-1.43)}] plot[domain=2.49:3.17,variable=\t]({1*8.03*cos(\t r)+0*8.03*sin(\t r)},{0*8.03*cos(\t r)+1*8.03*sin(\t r)});
\draw [shift={(0.76,-11.66)}] plot[domain=-0.82:0.59,variable=\t]({0.07*14.49*cos(\t r)+1*1.99*sin(\t r)},{1*14.49*cos(\t r)+-0.07*1.99*sin(\t r)});
\draw(1.67,1.18) circle (0.44cm);
\draw(-1.67,1.17) circle (0.44cm);
\draw [shift={(-7.33,-2.91)}] plot[domain=0.17:0.66,variable=\t]({1*7.44*cos(\t r)+0*7.44*sin(\t r)},{0*7.44*cos(\t r)+1*7.44*sin(\t r)});
\draw [shift={(7.33,-2.9)}] plot[domain=2.48:2.97,variable=\t]({1*7.44*cos(\t r)+0*7.44*sin(\t r)},{0*7.44*cos(\t r)+1*7.44*sin(\t r)});
\draw [shift={(2.52,-0.41)}] plot[domain=-0.43:1.16,variable=\t]({-0.96*5.16*cos(\t r)+-0.29*2.09*sin(\t r)},{0.29*5.16*cos(\t r)+-0.96*2.09*sin(\t r)});
\draw [shift={(-2.52,-0.42)}] plot[domain=-0.43:1.16,variable=\t]({0.96*5.16*cos(\t r)+0.3*2.09*sin(\t r)},{0.3*5.16*cos(\t r)+-0.96*2.09*sin(\t r)});
\begin{scriptsize}
\draw[color=black] (-0.39,1.39) node {$1$};
\draw[color=black] (-1.35,2.39) node {$6$};
\draw[color=black] (0.37,1.38) node {$2$};
\draw[color=black] (1.53,2.62) node {$3$};
\draw[color=black] (-0.65,0.18) node {$7$};
\draw[color=black] (0.68,0.28) node {$4$};
\draw[color=black] (-2.19,0.27) node {$8$};
\draw[color=black] (2.25,0.39) node {$5$};
\end{scriptsize}
\end{tikzpicture}
\caption{The initial quadrangulation for $n=7$ and $m=2$}
\label{fig:triinit2}
\end{figure}

If we follow the rules of Definition \ref{def:quiver}, we associate with this quadrangulation the quiver at the beginning of this section.

Moreover, if we associate the colored quiver with this quadrangulation according to Definition \ref{def:colquiver}, we obtain this one.

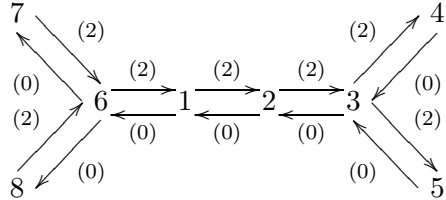
\[ \scalebox{1.0}{ \xymatrix@1{
7 \ar@<1ex>^{(2)}[dr] & & & & & 4 \ar@<1ex>^{(0)}[dl] \\
& 6 \ar@<1ex>^{(0)}[ul] \ar@<1ex>^{(0)}[dl] \ar@<1ex>^{(2)}[r] & 1 \ar@<1ex>^{(0)}[l] \ar@<1ex>^{(2)}[r] & 2 \ar@<1ex>^{(0)}[l] \ar@<1ex>^{(2)}[r] & 3 \ar@<1ex>^{(0)}[l] \ar@<1ex>^{(2)}[ur] \ar@<1ex>^{(2)}[dr] & \\
8 \ar@<1ex>^{(2)}[ur] & & & & & 5 \ar@<1ex>^{(0)}[ul] }} \]

Now we know that mutating a colored quiver corresponds to flipping an arc in an $(m+2)$-angulation, we can flip in any way the quadrangulation, draw the corresponding  colored quiver, and note that the colored quivers are related by mutation at the corresponding vertex.

\subsection{$2$-diagonals and the preprojective component of the Auslander-Reiten quiver of $\tilde{D_7}$}
From the initial quadrangulation, we extract the arcs and put them in the first component of the Auslander-Reiten quiver of $\tilde{D_7}$.

\begin{figure}[!h]
\[ \scalebox{0.5}{
 \xymatrix{
 & & & & & \begin{tikzpicture}[scale=0.5]
 \fill[fill=black,fill opacity=0.1] (0,4) -- (-1.66,3.46) -- (-2.69,2.05) -- (-2.68,0.3) -- (-1.66,-1.11) -- (0,-1.65) -- (1.66,-1.11) -- (2.69,0.3) -- (2.69,2.05) -- (1.66,3.46) -- cycle;
 \draw [fill=black,fill opacity=1.0] (1.66,1.61) circle (0.2cm);
 \draw [fill=black,fill opacity=1.0] (-1.66,1.61) circle (0.2cm);
 \draw (0,4)-- (-1.66,3.46);
 \draw (-1.66,3.46)-- (-2.69,2.05);
 \draw (-2.69,2.05)-- (-2.68,0.3);
 \draw (-2.68,0.3)-- (-1.66,-1.11);
 \draw (-1.66,-1.11)-- (0,-1.65);
 \draw (0,-1.65)-- (1.66,-1.11);
 \draw (1.66,-1.11)-- (2.69,0.3);
 \draw (2.69,0.3)-- (2.69,2.05);
 \draw (2.69,2.05)-- (1.66,3.46);
 \draw (1.66,3.46)-- (0,4);
 \draw(1.67,1.18) circle (0.44cm);
 \draw(-1.67,1.17) circle (0.44cm);
 \draw [shift={(-2.52,-0.42)}] plot[domain=-0.43:1.16,variable=\t]({0.96*5.16*cos(\t r)+0.3*2.09*sin(\t r)},{0.3*5.16*cos(\t r)+-0.96*2.09*sin(\t r)});
 \end{tikzpicture} \\
 & & & & \begin{tikzpicture}[scale=0.5]
 \fill[fill=black,fill opacity=0.1] (0,4) -- (-1.66,3.46) -- (-2.69,2.05) -- (-2.68,0.3) -- (-1.66,-1.11) -- (0,-1.65) -- (1.66,-1.11) -- (2.69,0.3) -- (2.69,2.05) -- (1.66,3.46) -- cycle;
 \draw [fill=black,fill opacity=1.0] (1.66,1.61) circle (0.2cm);
 \draw [fill=black,fill opacity=1.0] (-1.66,1.61) circle (0.2cm);
 \draw (0,4)-- (-1.66,3.46);
 \draw (-1.66,3.46)-- (-2.69,2.05);
 \draw (-2.69,2.05)-- (-2.68,0.3);
 \draw (-2.68,0.3)-- (-1.66,-1.11);
 \draw (-1.66,-1.11)-- (0,-1.65);
 \draw (0,-1.65)-- (1.66,-1.11);
 \draw (1.66,-1.11)-- (2.69,0.3);
 \draw (2.69,0.3)-- (2.69,2.05);
 \draw (2.69,2.05)-- (1.66,3.46);
 \draw (1.66,3.46)-- (0,4);
 \draw [shift={(0.76,-11.66)}] plot[domain=-0.82:0.59,variable=\t]({0.07*14.49*cos(\t r)+1*1.99*sin(\t r)},{1*14.49*cos(\t r)+-0.07*1.99*sin(\t r)});
 \draw(1.67,1.18) circle (0.44cm);
 \draw(-1.67,1.17) circle (0.44cm);
 \end{tikzpicture} \ar[ur] \ar[r] & \begin{tikzpicture}[scale=0.5]
 \fill[fill=black,fill opacity=0.1] (0,4) -- (-1.66,3.46) -- (-2.69,2.05) -- (-2.68,0.3) -- (-1.66,-1.11) -- (0,-1.65) -- (1.66,-1.11) -- (2.69,0.3) -- (2.69,2.05) -- (1.66,3.46) -- cycle;
 \draw [fill=black,fill opacity=1.0] (1.66,1.61) circle (0.2cm);
 \draw [fill=black,fill opacity=1.0] (-1.66,1.61) circle (0.2cm);
 \draw (0,4)-- (-1.66,3.46);
 \draw (-1.66,3.46)-- (-2.69,2.05);
 \draw (-2.69,2.05)-- (-2.68,0.3);
 \draw (-2.68,0.3)-- (-1.66,-1.11);
 \draw (-1.66,-1.11)-- (0,-1.65);
 \draw (0,-1.65)-- (1.66,-1.11);
 \draw (1.66,-1.11)-- (2.69,0.3);
 \draw (2.69,0.3)-- (2.69,2.05);
 \draw (2.69,2.05)-- (1.66,3.46);
 \draw (1.66,3.46)-- (0,4);
 \draw(1.67,1.18) circle (0.44cm);
 \draw(-1.67,1.17) circle (0.44cm);
 \draw [shift={(7.33,-2.9)}] plot[domain=2.48:2.97,variable=\t]({1*7.44*cos(\t r)+0*7.44*sin(\t r)},{0*7.44*cos(\t r)+1*7.44*sin(\t r)});
 \end{tikzpicture} \\
 & & & \begin{tikzpicture}[scale=0.5]
 \fill[fill=black,fill opacity=0.1] (0,4) -- (-1.66,3.46) -- (-2.69,2.05) -- (-2.68,0.3) -- (-1.66,-1.11) -- (0,-1.65) -- (1.66,-1.11) -- (2.69,0.3) -- (2.69,2.05) -- (1.66,3.46) -- cycle;
 \draw [fill=black,fill opacity=1.0] (1.66,1.61) circle (0.2cm);
 \draw [fill=black,fill opacity=1.0] (-1.66,1.61) circle (0.2cm);
 \draw (0,4)-- (-1.66,3.46);
 \draw (-1.66,3.46)-- (-2.69,2.05);
 \draw (-2.69,2.05)-- (-2.68,0.3);
 \draw (-2.68,0.3)-- (-1.66,-1.11);
 \draw (-1.66,-1.11)-- (0,-1.65);
 \draw (0,-1.65)-- (1.66,-1.11);
 \draw (1.66,-1.11)-- (2.69,0.3);
 \draw (2.69,0.3)-- (2.69,2.05);
 \draw (2.69,2.05)-- (1.66,3.46);
 \draw (1.66,3.46)-- (0,4);
 \draw [shift={(8.03,-1.43)}] plot[domain=2.49:3.17,variable=\t]({1*8.03*cos(\t r)+0*8.03*sin(\t r)},{0*8.03*cos(\t r)+1*8.03*sin(\t r)});
 \draw(1.67,1.18) circle (0.44cm);
 \draw(-1.67,1.17) circle (0.44cm);
 \end{tikzpicture} \ar[ur] & & \\
 & & \begin{tikzpicture}[scale=0.5]
 \fill[fill=black,fill opacity=0.1] (0,4) -- (-1.66,3.46) -- (-2.69,2.05) -- (-2.68,0.3) -- (-1.66,-1.11) -- (0,-1.65) -- (1.66,-1.11) -- (2.69,0.3) -- (2.69,2.05) -- (1.66,3.46) -- cycle;
 \draw [fill=black,fill opacity=1.0] (1.66,1.61) circle (0.2cm);
 \draw [fill=black,fill opacity=1.0] (-1.66,1.61) circle (0.2cm);
 \draw (0,4)-- (-1.66,3.46);
 \draw (-1.66,3.46)-- (-2.69,2.05);
 \draw (-2.69,2.05)-- (-2.68,0.3);
 \draw (-2.68,0.3)-- (-1.66,-1.11);
 \draw (-1.66,-1.11)-- (0,-1.65);
 \draw (0,-1.65)-- (1.66,-1.11);
 \draw (1.66,-1.11)-- (2.69,0.3);
 \draw (2.69,0.3)-- (2.69,2.05);
 \draw (2.69,2.05)-- (1.66,3.46);
 \draw (1.66,3.46)-- (0,4);
 \draw [shift={(-8.03,-1.44)}] plot[domain=-0.03:0.66,variable=\t]({1*8.03*cos(\t r)+0*8.03*sin(\t r)},{0*8.03*cos(\t r)+1*8.03*sin(\t r)});
 \draw(1.67,1.18) circle (0.44cm);
 \draw(-1.67,1.17) circle (0.44cm);
 \end{tikzpicture} \ar[ur] & & & \\
 \begin{tikzpicture}[scale=0.5]
 \fill[fill=black,fill opacity=0.1] (0,4) -- (-1.66,3.46) -- (-2.69,2.05) -- (-2.68,0.3) -- (-1.66,-1.11) -- (0,-1.65) -- (1.66,-1.11) -- (2.69,0.3) -- (2.69,2.05) -- (1.66,3.46) -- cycle;
 \draw [fill=black,fill opacity=1.0] (1.66,1.61) circle (0.2cm);
 \draw [fill=black,fill opacity=1.0] (-1.66,1.61) circle (0.2cm);
 \draw (0,4)-- (-1.66,3.46);
 \draw (-1.66,3.46)-- (-2.69,2.05);
 \draw (-2.69,2.05)-- (-2.68,0.3);
 \draw (-2.68,0.3)-- (-1.66,-1.11);
 \draw (-1.66,-1.11)-- (0,-1.65);
 \draw (0,-1.65)-- (1.66,-1.11);
 \draw (1.66,-1.11)-- (2.69,0.3);
 \draw (2.69,0.3)-- (2.69,2.05);
 \draw (2.69,2.05)-- (1.66,3.46);
 \draw (1.66,3.46)-- (0,4);
 \draw(1.67,1.18) circle (0.44cm);
 \draw(-1.67,1.17) circle (0.44cm);
 \draw [shift={(2.52,-0.41)}] plot[domain=-0.43:1.16,variable=\t]({-0.96*5.16*cos(\t r)+-0.29*2.09*sin(\t r)},{0.29*5.16*cos(\t r)+-0.96*2.09*sin(\t r)});
 \end{tikzpicture} \ar[r] & \begin{tikzpicture}[scale=0.5]
 \fill[fill=black,fill opacity=0.1] (0,4) -- (-1.66,3.46) -- (-2.69,2.05) -- (-2.68,0.3) -- (-1.66,-1.11) -- (0,-1.65) -- (1.66,-1.11) -- (2.69,0.3) -- (2.69,2.05) -- (1.66,3.46) -- cycle;
 \draw [fill=black,fill opacity=1.0] (1.66,1.61) circle (0.2cm);
 \draw [fill=black,fill opacity=1.0] (-1.66,1.61) circle (0.2cm);
 \draw (0,4)-- (-1.66,3.46);
 \draw (-1.66,3.46)-- (-2.69,2.05);
 \draw (-2.69,2.05)-- (-2.68,0.3);
 \draw (-2.68,0.3)-- (-1.66,-1.11);
 \draw (-1.66,-1.11)-- (0,-1.65);
 \draw (0,-1.65)-- (1.66,-1.11);
 \draw (1.66,-1.11)-- (2.69,0.3);
 \draw (2.69,0.3)-- (2.69,2.05);
 \draw (2.69,2.05)-- (1.66,3.46);
 \draw (1.66,3.46)-- (0,4);
 \draw [shift={(-0.75,-11.66)}] plot[domain=-0.82:0.59,variable=\t]({-0.07*14.49*cos(\t r)+-1*1.99*sin(\t r)},{1*14.49*cos(\t r)+-0.07*1.99*sin(\t r)});
 \draw(1.67,1.18) circle (0.44cm);
 \draw(-1.67,1.17) circle (0.44cm);
 \end{tikzpicture} \ar[ur] & & & & \\
 \begin{tikzpicture}[scale=0.5]
 \fill[fill=black,fill opacity=0.1] (0,4) -- (-1.66,3.46) -- (-2.69,2.05) -- (-2.68,0.3) -- (-1.66,-1.11) -- (0,-1.65) -- (1.66,-1.11) -- (2.69,0.3) -- (2.69,2.05) -- (1.66,3.46) -- cycle;
 \draw [fill=black,fill opacity=1.0] (1.66,1.61) circle (0.2cm);
 \draw [fill=black,fill opacity=1.0] (-1.66,1.61) circle (0.2cm);
 \draw (0,4)-- (-1.66,3.46);
 \draw (-1.66,3.46)-- (-2.69,2.05);
 \draw (-2.69,2.05)-- (-2.68,0.3);
 \draw (-2.68,0.3)-- (-1.66,-1.11);
 \draw (-1.66,-1.11)-- (0,-1.65);
 \draw (0,-1.65)-- (1.66,-1.11);
 \draw (1.66,-1.11)-- (2.69,0.3);
 \draw (2.69,0.3)-- (2.69,2.05);
 \draw (2.69,2.05)-- (1.66,3.46);
 \draw (1.66,3.46)-- (0,4);
 \draw(1.67,1.18) circle (0.44cm);
 \draw(-1.67,1.17) circle (0.44cm);
 \draw [shift={(-7.33,-2.91)}] plot[domain=0.17:0.66,variable=\t]({1*7.44*cos(\t r)+0*7.44*sin(\t r)},{0*7.44*cos(\t r)+1*7.44*sin(\t r)});
 \end{tikzpicture} \ar[ur] & & & & & \\
 }}
\]
 \label{fig:ar2}
\end{figure}

As the preprojective component of the Auslander-Reiten quiver is of the following form (where we have associated the vertices with the types of the arcs we are working with), we just have to apply $\tau$ and the shift functor to the arcs to be able to complete the preprojective component and obtain the result in figure \ref{fig:ar72}.

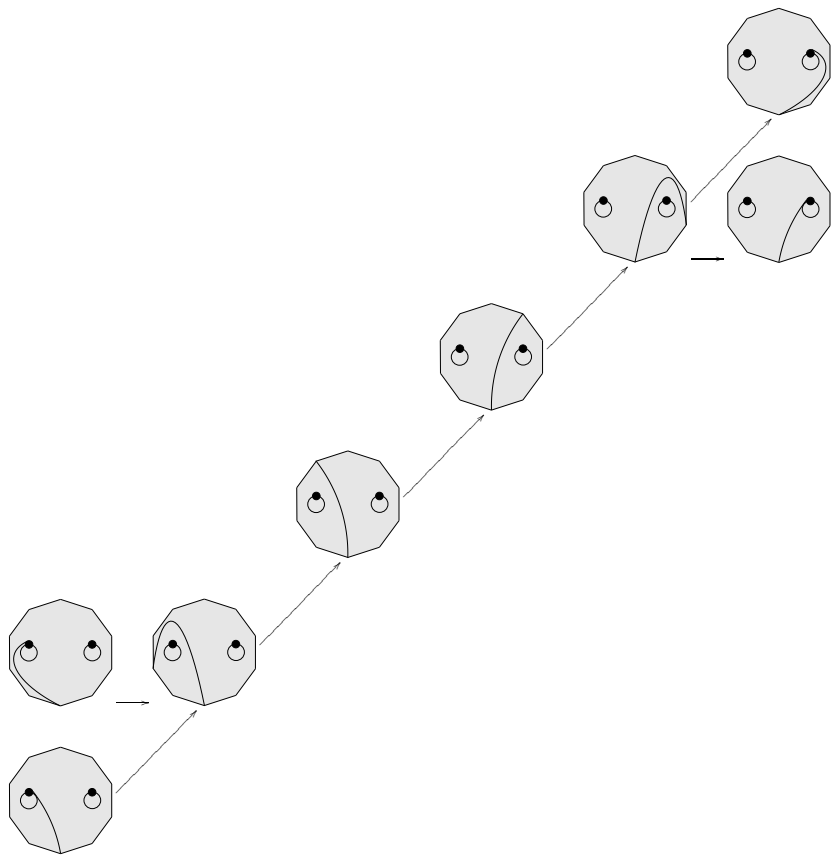
\[ \scalebox{.9}{
 \xymatrix{
 & & & & & d_7 \ar[dr] & & \tau^{-1}d_7 \ar[dr] & & \tau^{-2}d_7 \\
 & & & & d_6 \ar[dr] \ar[ur] \ar[r] & d_8 \ar[r] & \tau^{-1}d_6 \ar[dr] \ar[ur] \ar[r] & \tau^{-1}d_8 \ar[r] & \tau^{-2}d_6 \ar[ur] \ar[r] & \tau^{-2}d_8 \\
 & & & d_1 \ar[dr] \ar[ur] & & \tau^{-1}d_1 \ar[dr] \ar[ur] & & \tau^{-2}d_1 \ar[ur] & & \\
 & & d_2 \ar[dr] \ar[ur] & & \tau^{-1}d_2 \ar[dr] \ar[ur] & & \tau^{-2}d_2 \ar[ur] & & & \\
 d_4 \ar[r] & d_3 \ar[ur] \ar[dr] \ar[r] & \tau^{-1}d_4 \ar[r] & \tau^{-1}d_3 \ar[dr] \ar[ur] \ar[r] & \tau^{-2}d_4 \ar[r] & \tau^{-2}d_3 \ar[ur] \ar[r] & \tau^{-3}d_4 & & & \\
 d_5 \ar[ur] & & \tau^{-1}d_5 \ar[ur] & & \tau^{-2}d_5 \ar[ur] & & & & & \\
 }}
\]

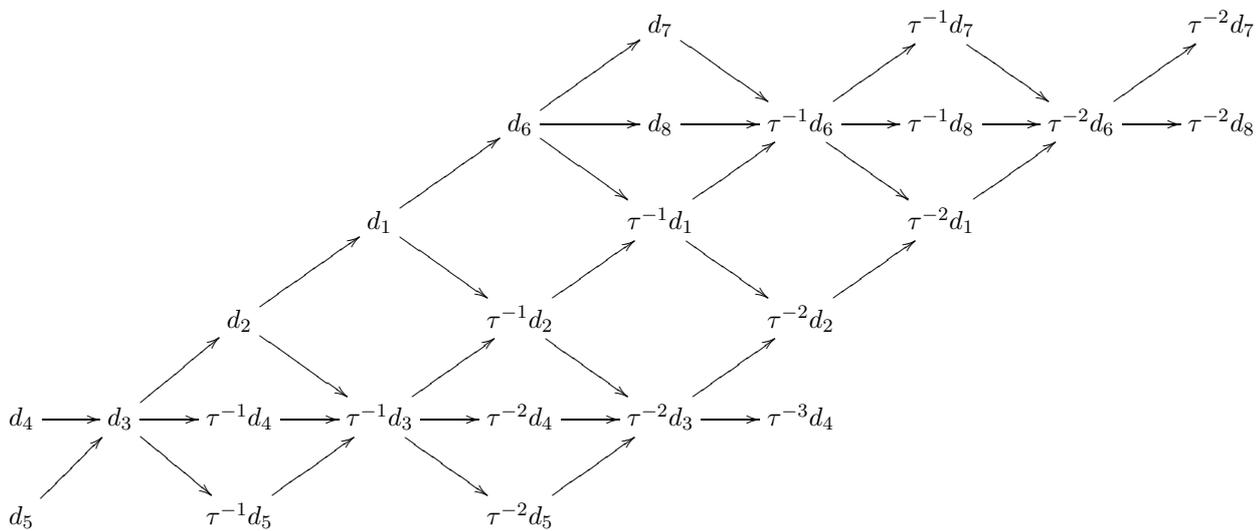
\begin{figure}
\[ \scalebox{0.7}{
 \xymatrix{
 & & & & & \begin{tikzpicture}[scale=0.3]
 \fill[fill=black,fill opacity=0.1] (0,4) -- (-1.66,3.46) -- (-2.69,2.05) -- (-2.68,0.3) -- (-1.66,-1.11) -- (0,-1.65) -- (1.66,-1.11) -- (2.69,0.3) -- (2.69,2.05) -- (1.66,3.46) -- cycle;
 \draw [fill=black,fill opacity=1.0] (1.66,1.61) circle (0.2cm);
 \draw [fill=black,fill opacity=1.0] (-1.66,1.61) circle (0.2cm);
 \draw (0,4)-- (-1.66,3.46);
 \draw (-1.66,3.46)-- (-2.69,2.05);
 \draw (-2.69,2.05)-- (-2.68,0.3);
 \draw (-2.68,0.3)-- (-1.66,-1.11);
 \draw (-1.66,-1.11)-- (0,-1.65);
 \draw (0,-1.65)-- (1.66,-1.11);
 \draw (1.66,-1.11)-- (2.69,0.3);
 \draw (2.69,0.3)-- (2.69,2.05);
 \draw (2.69,2.05)-- (1.66,3.46);
 \draw (1.66,3.46)-- (0,4);
 \draw(1.67,1.18) circle (0.44cm);
 \draw(-1.67,1.17) circle (0.44cm);
 \draw [shift={(-2.52,-0.42)}] plot[domain=-0.43:1.16,variable=\t]({0.96*5.16*cos(\t r)+0.3*2.09*sin(\t r)},{0.3*5.16*cos(\t r)+-0.96*2.09*sin(\t r)});
 \end{tikzpicture} \ar[dr] & & \begin{tikzpicture}[scale=0.3]
 \fill[fill=black,fill opacity=0.1] (0,4) -- (-1.66,3.46) -- (-2.69,2.05) -- (-2.68,0.3) -- (-1.66,-1.11) -- (0,-1.65) -- (1.66,-1.11) -- (2.69,0.3) -- (2.69,2.05) -- (1.66,3.46) -- cycle;
 \draw [fill=black,fill opacity=1.0] (1.66,1.61) circle (0.2cm);
 \draw [fill=black,fill opacity=1.0] (-1.66,1.61) circle (0.2cm);
 \draw (0,4)-- (-1.66,3.46);
 \draw (-1.66,3.46)-- (-2.69,2.05);
 \draw (-2.69,2.05)-- (-2.68,0.3);
 \draw (-2.68,0.3)-- (-1.66,-1.11);
 \draw (-1.66,-1.11)-- (0,-1.65);
 \draw (0,-1.65)-- (1.66,-1.11);
 \draw (1.66,-1.11)-- (2.69,0.3);
 \draw (2.69,0.3)-- (2.69,2.05);
 \draw (2.69,2.05)-- (1.66,3.46);
 \draw (1.66,3.46)-- (0,4);
 \draw(1.67,1.18) circle (0.44cm);
 \draw(-1.67,1.17) circle (0.44cm);
 \draw [shift={(-0.13,0.88)}] plot[domain=0.29:3.91,variable=\t]({-0.99*2.71*cos(\t r)+0.11*1.01*sin(\t r)},{-0.11*2.71*cos(\t r)+-0.99*1.01*sin(\t r)});
 \end{tikzpicture} \ar[dr] & \\
 & & & & \begin{tikzpicture}[scale=0.3]
 \fill[fill=black,fill opacity=0.1] (0,4) -- (-1.66,3.46) -- (-2.69,2.05) -- (-2.68,0.3) -- (-1.66,-1.11) -- (0,-1.65) -- (1.66,-1.11) -- (2.69,0.3) -- (2.69,2.05) -- (1.66,3.46) -- cycle;
 \draw [fill=black,fill opacity=1.0] (1.66,1.61) circle (0.2cm);
 \draw [fill=black,fill opacity=1.0] (-1.66,1.61) circle (0.2cm);
 \draw (0,4)-- (-1.66,3.46);
 \draw (-1.66,3.46)-- (-2.69,2.05);
 \draw (-2.69,2.05)-- (-2.68,0.3);
 \draw (-2.68,0.3)-- (-1.66,-1.11);
 \draw (-1.66,-1.11)-- (0,-1.65);
 \draw (0,-1.65)-- (1.66,-1.11);
 \draw (1.66,-1.11)-- (2.69,0.3);
 \draw (2.69,0.3)-- (2.69,2.05);Conclusion partielle
 \draw (2.69,2.05)-- (1.66,3.46);
 \draw (1.66,3.46)-- (0,4);
 \draw [shift={(0.76,-11.66)}] plot[domain=-0.82:0.59,variable=\t]({0.07*14.49*cos(\t r)+1*1.99*sin(\t r)},{1*14.49*cos(\t r)+-0.07*1.99*sin(\t r)});
 \draw(1.67,1.18) circle (0.44cm);
 \draw(-1.67,1.17) circle (0.44cm);
 \end{tikzpicture} \ar[dr] \ar[ur] \ar[r] & \begin{tikzpicture}[scale=0.3]
 \fill[fill=black,fill opacity=0.1] (0,4) -- (-1.66,3.46) -- (-2.69,2.05) -- (-2.68,0.3) -- (-1.66,-1.11) -- (0,-1.65) -- (1.66,-1.11) -- (2.69,0.3) -- (2.69,2.05) -- (1.66,3.46) -- cycle;
 \draw [fill=black,fill opacity=1.0] (1.66,1.61) circle (0.2cm);
 \draw [fill=black,fill opacity=1.0] (-1.66,1.61) circle (0.2cm);
 \draw (0,4)-- (-1.66,3.46);
 \draw (-1.66,3.46)-- (-2.69,2.05);
 \draw (-2.69,2.05)-- (-2.68,0.3);
 \draw (-2.68,0.3)-- (-1.66,-1.11);
 \draw (-1.66,-1.11)-- (0,-1.65);
 \draw (0,-1.65)-- (1.66,-1.11);
 \draw (1.66,-1.11)-- (2.69,0.3);
 \draw (2.69,0.3)-- (2.69,2.05);
 \draw (2.69,2.05)-- (1.66,3.46);
 \draw (1.66,3.46)-- (0,4);
 \draw(1.67,1.18) circle (0.44cm);
 \draw(-1.67,1.17) circle (0.44cm);
 \draw [shift={(7.33,-2.9)}] plot[domain=2.48:2.97,variable=\t]({1*7.44*cos(\t r)+0*7.44*sin(\t r)},{0*7.44*cos(\t r)+1*7.44*sin(\t r)});
 \end{tikzpicture} \ar[r] & \begin{tikzpicture}[scale=0.3]
 \fill[fill=black,fill opacity=0.1] (0,4) -- (-1.66,3.46) -- (-2.69,2.05) -- (-2.68,0.3) -- (-1.66,-1.11) -- (0,-1.65) -- (1.66,-1.11) -- (2.69,0.3) -- (2.69,2.05) -- (1.66,3.46) -- cycle;
 \draw [fill=black,fill opacity=1.0] (1.66,1.61) circle (0.2cm);
 \draw [fill=black,fill opacity=1.0] (-1.66,1.61) circle (0.2cm);
 \draw (0,4)-- (-1.66,3.46);
 \draw (-1.66,3.46)-- (-2.69,2.05);
 \draw (-2.69,2.05)-- (-2.68,0.3);
 \draw (-2.68,0.3)-- (-1.66,-1.11);
 \draw (-1.66,-1.11)-- (0,-1.65);
 \draw (0,-1.65)-- (1.66,-1.11);
 \draw (1.66,-1.11)-- (2.69,0.3);
 \draw (2.69,0.3)-- (2.69,2.05);
 \draw (2.69,2.05)-- (1.66,3.46);
 \draw (1.66,3.46)-- (0,4);
 \draw(1.67,1.18) circle (0.44cm);
 \draw(-1.67,1.17) circle (0.44cm);
 \draw [shift={(-1.64,2.62)}] plot[domain=4.29:5.6,variable=\t]({1*2.55*cos(\t r)+0*2.55*sin(\t r)},{0*2.55*cos(\t r)+1*2.55*sin(\t r)});
 \draw [shift={(0.92,0.29)}] plot[domain=0.43:4.6,variable=\t]({-0.58*2.32*cos(\t r)+0.81*0.9*sin(\t r)},{-0.81*2.32*cos(\t r)+-0.58*0.9*sin(\t r)});
 \end{tikzpicture} \ar[dr] \ar[ur] \ar[r] & \begin{tikzpicture}[scale=0.3]
 \fill[fill=black,fill opacity=0.1] (0,4) -- (-1.66,3.46) -- (-2.69,2.05) -- (-2.68,0.3) -- (-1.66,-1.11) -- (0,-1.65) -- (1.66,-1.11) -- (2.69,0.3) -- (2.69,2.05) -- (1.66,3.46) -- cycle;
 \draw [fill=black,fill opacity=1.0] (1.66,1.61) circle (0.2cm);
 \draw [fill=black,fill opacity=1.0] (-1.66,1.61) circle (0.2cm);
 \draw (0,4)-- (-1.66,3.46);
 \draw (-1.66,3.46)-- (-2.69,2.05);
 \draw (-2.69,2.05)-- (-2.68,0.3);
 \draw (-2.68,0.3)-- (-1.66,-1.11);
 \draw (-1.66,-1.11)-- (0,-1.65);
 \draw (0,-1.65)-- (1.66,-1.11);
 \draw (1.66,-1.11)-- (2.69,0.3);
 \draw (2.69,0.3)-- (2.69,2.05);
 \draw (2.69,2.05)-- (1.66,3.46);
 \draw (1.66,3.46)-- (0,4);
 \draw(1.67,1.18) circle (0.44cm);
 \draw(-1.67,1.17) circle (0.44cm);
 \draw [shift={(-1.64,2.62)}] plot[domain=4.29:5.6,variable=\t]({1*2.55*cos(\t r)+0*2.55*sin(\t r)},{0*2.55*cos(\t r)+1*2.55*sin(\t r)});
 \draw [shift={(2.14,-0.5)}] plot[domain=1.76:2.44,variable=\t]({1*2.35*cos(\t r)+0*2.35*sin(\t r)},{0*2.35*cos(\t r)+1*2.35*sin(\t r)});
 \end{tikzpicture} \ar[r] & \cdots \\
 & \cdots & & \begin{tikzpicture}[scale=0.3]
 \fill[fill=black,fill opacity=0.1] (0,4) -- (-1.66,3.46) -- (-2.69,2.05) -- (-2.68,0.3) -- (-1.66,-1.11) -- (0,-1.65) -- (1.66,-1.11) -- (2.69,0.3) -- (2.69,2.05) -- (1.66,3.46) -- cycle;
 \draw [fill=black,fill opacity=1.0] (1.66,1.61) circle (0.2cm);
 \draw [fill=black,fill opacity=1.0] (-1.66,1.61) circle (0.2cm);
 \draw (0,4)-- (-1.66,3.46);
 \draw (-1.66,3.46)-- (-2.69,2.05);
 \draw (-2.69,2.05)-- (-2.68,0.3);
 \draw (-2.68,0.3)-- (-1.66,-1.11);
 \draw (-1.66,-1.11)-- (0,-1.65);
 \draw (0,-1.65)-- (1.66,-1.11);
 \draw (1.66,-1.11)-- (2.69,0.3);
 \draw (2.69,0.3)-- (2.69,2.05);
 \draw (2.69,2.05)-- (1.66,3.46);
 \draw (1.66,3.46)-- (0,4);
 \draw [shift={(8.03,-1.43)}] plot[domain=2.49:3.17,variable=\t]({1*8.03*cos(\t r)+0*8.03*sin(\t r)},{0*8.03*cos(\t r)+1*8.03*sin(\t r)});
 \draw(1.67,1.18) circle (0.44cm);
 \draw(-1.67,1.17) circle (0.44cm);
 \end{tikzpicture} \ar[dr] \ar[ur] & & \begin{tikzpicture}[scale=0.3]
 \fill[fill=black,fill opacity=0.1] (0,4) -- (-1.66,3.46) -- (-2.69,2.05) -- (-2.68,0.3) -- (-1.66,-1.11) -- (0,-1.65) -- (1.66,-1.11) -- (2.69,0.3) -- (2.69,2.05) -- (1.66,3.46) -- cycle;
 \draw [fill=black,fill opacity=1.0] (1.66,1.61) circle (0.2cm);
 \draw [fill=black,fill opacity=1.0] (-1.66,1.61) circle (0.2cm);
 \draw (0,4)-- (-1.66,3.46);
 \draw (-1.66,3.46)-- (-2.69,2.05);
 \draw (-2.69,2.05)-- (-2.68,0.3);
 \draw (-2.68,0.3)-- (-1.66,-1.11);
 \draw (-1.66,-1.11)-- (0,-1.65);
 \draw (0,-1.65)-- (1.66,-1.11);
 \draw (1.66,-1.11)-- (2.69,0.3);
 \draw (2.69,0.3)-- (2.69,2.05);
 \draw (2.69,2.05)-- (1.66,3.46);
 \draw (1.66,3.46)-- (0,4);
 \draw(1.67,1.18) circle (0.44cm);
 \draw(-1.67,1.17) circle (0.44cm);
 \draw [shift={(-2.18,1.68)}] plot[domain=4.36:5.86,variable=\t]({1*1.47*cos(\t r)+0*1.47*sin(\t r)},{0*1.47*cos(\t r)+1*1.47*sin(\t r)});
 \draw [shift={(0.84,0.32)}] plot[domain=-0.01:2.72,variable=\t]({1*1.84*cos(\t r)+0*1.84*sin(\t r)},{0*1.84*cos(\t r)+1*1.84*sin(\t r)});
 \end{tikzpicture} \ar[dr] \ar[ur] & & \cdots \ar[ur] & \\
 & & \begin{tikzpicture}[scale=0.3]
 \fill[fill=black,fill opacity=0.1] (0,4) -- (-1.66,3.46) -- (-2.69,2.05) -- (-2.68,0.3) -- (-1.66,-1.11) -- (0,-1.65) -- (1.66,-1.11) -- (2.69,0.3) -- (2.69,2.05) -- (1.66,3.46) -- cycle;
 \draw [fill=black,fill opacity=1.0] (1.66,1.61) circle (0.2cm);
 \draw [fill=black,fill opacity=1.0] (-1.66,1.61) circle (0.2cm);
 \draw (0,4)-- (-1.66,3.46);
 \draw (-1.66,3.46)-- (-2.69,2.05);
 \draw (-2.69,2.05)-- (-2.68,0.3);
 \draw (-2.68,0.3)-- (-1.66,-1.11);
 \draw (-1.66,-1.11)-- (0,-1.65);
 \draw (0,-1.65)-- (1.66,-1.11);
 \draw (1.66,-1.11)-- (2.69,0.3);
 \draw (2.69,0.3)-- (2.69,2.05);
 \draw (2.69,2.05)-- (1.66,3.46);
 \draw (1.66,3.46)-- (0,4);
 \draw [shift={(-8.03,-1.44)}] plot[domain=-0.03:0.66,variable=\t]({1*8.03*cos(\t r)+0*8.03*sin(\t r)},{0*8.03*cos(\t r)+1*8.03*sin(\t r)});
 \draw(1.67,1.18) circle (0.44cm);
 \draw(-1.67,1.17) circle (0.44cm);
 \end{tikzpicture} \ar[dr] \ar[ur] & & \begin{tikzpicture}[scale=0.3]
 \fill[fill=black,fill opacity=0.1] (0,4) -- (-1.66,3.46) -- (-2.69,2.05) -- (-2.68,0.3) -- (-1.66,-1.11) -- (0,-1.65) -- (1.66,-1.11) -- (2.69,0.3) -- (2.69,2.05) -- (1.66,3.46) -- cycle;
 \draw [fill=black,fill opacity=1.0] (1.66,1.61) circle (0.2cm);
 \draw [fill=black,fill opacity=1.0] (-1.66,1.61) circle (0.2cm);
 \draw (0,4)-- (-1.66,3.46);
 \draw (-1.66,3.46)-- (-2.69,2.05);
 \draw (-2.69,2.05)-- (-2.68,0.3);
 \draw (-2.68,0.3)-- (-1.66,-1.11);
 \draw (-1.66,-1.11)-- (0,-1.65);
 \draw (0,-1.65)-- (1.66,-1.11);
 \draw (1.66,-1.11)-- (2.69,0.3);
 \draw (2.69,0.3)-- (2.69,2.05);
 \draw (2.69,2.05)-- (1.66,3.46);
 \draw (1.66,3.46)-- (0,4);
 \draw(1.67,1.18) circle (0.44cm);
 \draw(-1.67,1.17) circle (0.44cm);
 \draw [shift={(-2.18,1.68)}] plot[domain=4.36:5.86,variable=\t]({1*1.47*cos(\t r)+0*1.47*sin(\t r)},{0*1.47*cos(\t r)+1*1.47*sin(\t r)});
 \draw [shift={(3.22,-0.67)}] plot[domain=1.93:2.74,variable=\t]({1*4.42*cos(\t r)+0*4.42*sin(\t r)},{0*4.42*cos(\t r)+1*4.42*sin(\t r)});
 \end{tikzpicture} \ar[dr] \ar[ur] & & \cdots \ar[ur] & & \\
 \begin{tikzpicture}[scale=0.3]
 \fill[fill=black,fill opacity=0.1] (0,4) -- (-1.66,3.46) -- (-2.69,2.05) -- (-2.68,0.3) -- (-1.66,-1.11) -- (0,-1.65) -- (1.66,-1.11) -- (2.69,0.3) -- (2.69,2.05) -- (1.66,3.46) -- cycle;
 \draw [fill=black,fill opacity=1.0] (1.66,1.61) circle (0.2cm);
 \draw [fill=black,fill opacity=1.0] (-1.66,1.61) circle (0.2cm);
 \draw (0,4)-- (-1.66,3.46);
 \draw (-1.66,3.46)-- (-2.69,2.05);
 \draw (-2.69,2.05)-- (-2.68,0.3);
 \draw (-2.68,0.3)-- (-1.66,-1.11);
 \draw (-1.66,-1.11)-- (0,-1.65);
 \draw (0,-1.65)-- (1.66,-1.11);
 \draw (1.66,-1.11)-- (2.69,0.3);
 \draw (2.69,0.3)-- (2.69,2.05);
 \draw (2.69,2.05)-- (1.66,3.46);
 \draw (1.66,3.46)-- (0,4);
 \draw(1.67,1.18) circle (0.44cm);
 \draw(-1.67,1.17) circle (0.44cm);
 \draw [shift={(2.52,-0.41)}] plot[domain=-0.43:1.16,variable=\t]({-0.96*5.16*cos(\t r)+-0.29*2.09*sin(\t r)},{0.29*5.16*cos(\t r)+-0.96*2.09*sin(\t r)});
 \end{tikzpicture} \ar[r] & \begin{tikzpicture}[scale=0.3]
 \fill[fill=black,fill opacity=0.1] (0,4) -- (-1.66,3.46) -- (-2.69,2.05) -- (-2.68,0.3) -- (-1.66,-1.11) -- (0,-1.65) -- (1.66,-1.11) -- (2.69,0.3) -- (2.69,2.05) -- (1.66,3.46) -- cycle;
 \draw [fill=black,fill opacity=1.0] (1.66,1.61) circle (0.2cm);
 \draw [fill=black,fill opacity=1.0] (-1.66,1.61) circle (0.2cm);
 \draw (0,4)-- (-1.66,3.46);
 \draw (-1.66,3.46)-- (-2.69,2.05);
 \draw (-2.69,2.05)-- (-2.68,0.3);
 \draw (-2.68,0.3)-- (-1.66,-1.11);
 \draw (-1.66,-1.11)-- (0,-1.65);
 \draw (0,-1.65)-- (1.66,-1.11);
 \draw (1.66,-1.11)-- (2.69,0.3);
 \draw (2.69,0.3)-- (2.69,2.05);
 \draw (2.69,2.05)-- (1.66,3.46);
 \draw (1.66,3.46)-- (0,4);
 \draw [shift={(-0.75,-11.66)}] plot[domain=-0.82:0.59,variable=\t]({-0.07*14.49*cos(\t r)+-1*1.99*sin(\t r)},{1*14.49*cos(\t r)+-0.07*1.99*sin(\t r)});
 \draw(1.67,1.18) circle (0.44cm);
 \draw(-1.67,1.17) circle (0.44cm);
 \end{tikzpicture} \ar[ur] \ar[dr] \ar[r] & \begin{tikzpicture}[scale=0.3]
 \fill[fill=black,fill opacity=0.1] (0,4) -- (-1.66,3.46) -- (-2.69,2.05) -- (-2.68,0.3) -- (-1.66,-1.11) -- (0,-1.65) -- (1.66,-1.11) -- (2.69,0.3) -- (2.69,2.05) -- (1.66,3.46) -- cycle;
 \draw [fill=black,fill opacity=1.0] (1.66,1.61) circle (0.2cm);
 \draw [fill=black,fill opacity=1.0] (-1.66,1.61) circle (0.2cm);
 \draw (0,4)-- (-1.66,3.46);
 \draw (-1.66,3.46)-- (-2.69,2.05);
 \draw (-2.69,2.05)-- (-2.68,0.3);
 \draw (-2.68,0.3)-- (-1.66,-1.11);
 \draw (-1.66,-1.11)-- (0,-1.65);
 \draw (0,-1.65)-- (1.66,-1.11);
 \draw (1.66,-1.11)-- (2.69,0.3);
 \draw (2.69,0.3)-- (2.69,2.05);
 \draw (2.69,2.05)-- (1.66,3.46);
 \draw (1.66,3.46)-- (0,4);
 \draw(1.67,1.18) circle (0.44cm);
 \draw(-1.67,1.17) circle (0.44cm);
 \draw [shift={(-1.23,0.43)}] plot[domain=1.9:3.23,variable=\t]({1*1.46*cos(\t r)+0*1.46*sin(\t r)},{0*1.46*cos(\t r)+1*1.46*sin(\t r)});
 \end{tikzpicture} \ar[r] & \begin{tikzpicture}[scale=0.3]
 \fill[fill=black,fill opacity=0.1] (0,4) -- (-1.66,3.46) -- (-2.69,2.05) -- (-2.68,0.3) -- (-1.66,-1.11) -- (0,-1.65) -- (1.66,-1.11) -- (2.69,0.3) -- (2.69,2.05) -- (1.66,3.46) -- cycle;
 \draw [fill=black,fill opacity=1.0] (1.66,1.61) circle (0.2cm);
 \draw [fill=black,fill opacity=1.0] (-1.66,1.61) circle (0.2cm);
 \draw (0,4)-- (-1.66,3.46);
 \draw (-1.66,3.46)-- (-2.69,2.05);
 \draw (-2.69,2.05)-- (-2.68,0.3);
 \draw (-2.68,0.3)-- (-1.66,-1.11);
 \draw (-1.66,-1.11)-- (0,-1.65);
 \draw (0,-1.65)-- (1.66,-1.11);
 \draw (1.66,-1.11)-- (2.69,0.3);
 \draw (2.69,0.3)-- (2.69,2.05);
 \draw (2.69,2.05)-- (1.66,3.46);
 \draw (1.66,3.46)-- (0,4);
 \draw(1.67,1.18) circle (0.44cm);
 \draw(-1.67,1.17) circle (0.44cm);
 \draw [shift={(-2.18,1.68)}] plot[domain=4.36:5.86,variable=\t]({1*1.47*cos(\t r)+0*1.47*sin(\t r)},{0*1.47*cos(\t r)+1*1.47*sin(\t r)});
 \draw [shift={(-2.59,1.81)}] plot[domain=-0.4:1.06,variable=\t]({1*1.9*cos(\t r)+0*1.9*sin(\t r)},{0*1.9*cos(\t r)+1*1.9*sin(\t r)});
 \end{tikzpicture} \ar[dr] \ar[ur] \ar[r] & \begin{tikzpicture}[scale=0.3]
 \fill[fill=black,fill opacity=0.1] (0,4) -- (-1.66,3.46) -- (-2.69,2.05) -- (-2.68,0.3) -- (-1.66,-1.11) -- (0,-1.65) -- (1.66,-1.11) -- (2.69,0.3) -- (2.69,2.05) -- (1.66,3.46) -- cycle;
 \draw [fill=black,fill opacity=1.0] (1.66,1.61) circle (0.2cm);
 \draw [fill=black,fill opacity=1.0] (-1.66,1.61) circle (0.2cm);
 \draw (0,4)-- (-1.66,3.46);
 \draw (-1.66,3.46)-- (-2.69,2.05);
 \draw (-2.69,2.05)-- (-2.68,0.3);
 \draw (-2.68,0.3)-- (-1.66,-1.11);
 \draw (-1.66,-1.11)-- (0,-1.65);
 \draw (0,-1.65)-- (1.66,-1.11);
 \draw (1.66,-1.11)-- (2.69,0.3);
 \draw (2.69,0.3)-- (2.69,2.05);
 \draw (2.69,2.05)-- (1.66,3.46);
 \draw (1.66,3.46)-- (0,4);
 \draw(1.67,1.18) circle (0.44cm);
 \draw(-1.67,1.17) circle (0.44cm);
 \draw [shift={(-4.15,2.25)}] plot[domain=-0.25:0.45,variable=\t]({1*2.77*cos(\t r)+0*2.77*sin(\t r)},{0*2.77*cos(\t r)+1*2.77*sin(\t r)});
 \end{tikzpicture} \ar[r] & \cdots \ar[ur] \ar[r] & \cdots & & \\
 \begin{tikzpicture}[scale=0.3]
 \fill[fill=black,fill opacity=0.1] (0,4) -- (-1.66,3.46) -- (-2.69,2.05) -- (-2.68,0.3) -- (-1.66,-1.11) -- (0,-1.65) -- (1.66,-1.11) -- (2.69,0.3) -- (2.69,2.05) -- (1.66,3.46) -- cycle;
 \draw [fill=black,fill opacity=1.0] (1.66,1.61) circle (0.2cm);
 \draw [fill=black,fill opacity=1.0] (-1.66,1.61) circle (0.2cm);
 \draw (0,4)-- (-1.66,3.46);
 \draw (-1.66,3.46)-- (-2.69,2.05);
 \draw (-2.69,2.05)-- (-2.68,0.3);
 \draw (-2.68,0.3)-- (-1.66,-1.11);
 \draw (-1.66,-1.11)-- (0,-1.65);
 \draw (0,-1.65)-- (1.66,-1.11);
 \draw (1.66,-1.11)-- (2.69,0.3);
 \draw (2.69,0.3)-- (2.69,2.05);
 \draw (2.69,2.05)-- (1.66,3.46);
 \draw (1.66,3.46)-- (0,4);
 \draw(1.67,1.18) circle (0.44cm);
 \draw(-1.67,1.17) circle (0.44cm);
 \draw [shift={(-7.33,-2.91)}] plot[domain=0.17:0.66,variable=\t]({1*7.44*cos(\t r)+0*7.44*sin(\t r)},{0*7.44*cos(\t r)+1*7.44*sin(\t r)});
 \end{tikzpicture} \ar[ur] & & \begin{tikzpicture}[scale=0.3]
 \fill[fill=black,fill opacity=0.1] (0,4) -- (-1.66,3.46) -- (-2.69,2.05) -- (-2.68,0.3) -- (-1.66,-1.11) -- (0,-1.65) -- (1.66,-1.11) -- (2.69,0.3) -- (2.69,2.05) -- (1.66,3.46) -- cycle;
 \draw [fill=black,fill opacity=1.0] (1.66,1.61) circle (0.2cm);
 \draw [fill=black,fill opacity=1.0] (-1.66,1.61) circle (0.2cm);
 \draw (0,4)-- (-1.66,3.46);
 \draw (-1.66,3.46)-- (-2.69,2.05);
 \draw (-2.69,2.05)-- (-2.68,0.3);
 \draw (-2.68,0.3)-- (-1.66,-1.11);
 \draw (-1.66,-1.11)-- (0,-1.65);
 \draw (0,-1.65)-- (1.66,-1.11);
 \draw (1.66,-1.11)-- (2.69,0.3);
 \draw (2.69,0.3)-- (2.69,2.05);
 \draw (2.69,2.05)-- (1.66,3.46);
 \draw (1.66,3.46)-- (0,4);
 \draw(1.67,1.18) circle (0.44cm);
 \draw(-1.67,1.17) circle (0.44cm);
 \draw [shift={(-1.82,0.82)}] plot[domain=-2.61:1.42,variable=\t]({1*1.01*cos(\t r)+0*1.01*sin(\t r)},{0*1.01*cos(\t r)+1*1.01*sin(\t r)});
 \end{tikzpicture} \ar[ur] & & \cdots \ar[ur] & & & & \\
 }}
\]
 \caption{Arcs on the Auslander-Reiten quiver of $\tilde{D_7}$}
 \label{fig:ar72}
\end{figure}

Note that there are two copies of this component, and the other one is about the same but all the arcs have been shifted once clockwise.

\cor{The preprojective component is infinite, and as we apply the Auslander-Reiten translation several times, the arcs wrap around the inner polygons.}

\subsection{$2$-diagonals in the non-homogeneous tubes of the Auslander-Reiten quiver of $\tilde{D_7 }$}
Concerning the regular modules, there are three types of non-homogeneous tubes: The first type, containing $m$ tubes of size $5$ and the two other types of tubes (each one containing $m$ copies of it) of size $2$. The one of size $5$ corresponds to the following cyclic arc (in figure \ref{fig:arc2}) and its images under the shift.

\begin{figure}[!h]
\centering
\begin{tikzpicture}[scale=0.6]
\fill[fill=black,fill opacity=0.1] (0,4) -- (-1.66,3.46) -- (-2.69,2.05) -- (-2.68,0.3) -- (-1.66,-1.11) -- (0,-1.65) -- (1.66,-1.11) -- (2.69,0.3) -- (2.69,2.05) -- (1.66,3.46) -- cycle;
\draw [fill=black,fill opacity=1.0] (1.66,1.61) circle (0.2cm);
\draw [fill=black,fill opacity=1.0] (-1.66,1.61) circle (0.2cm);
\draw (0,4)-- (-1.66,3.46);
\draw (-1.66,3.46)-- (-2.69,2.05);
\draw (-2.69,2.05)-- (-2.68,0.3);
\draw (-2.68,0.3)-- (-1.66,-1.11);
\draw (-1.66,-1.11)-- (0,-1.65);
\draw (0,-1.65)-- (1.66,-1.11);
\draw (1.66,-1.11)-- (2.69,0.3);
\draw (2.69,0.3)-- (2.69,2.05);
\draw (2.69,2.05)-- (1.66,3.46);
\draw (1.66,3.46)-- (0,4);
\draw(1.67,1.18) circle (0.44cm);
\draw(-1.67,1.17) circle (0.44cm);
\draw [shift={(-0.56,-0.56)}] plot[domain=0.68:1.84,variable=\t]({1*4.16*cos(\t r)+0*4.16*sin(\t r)},{0*4.16*cos(\t r)+1*4.16*sin(\t r)});
\end{tikzpicture}
\caption{Arc of order $5$}
\label{fig:arc2}
\end{figure}

Indeed, if we apply $\tau$ five times, the arc returns to the origin, so it is cyclic of order $5$. The two remaining arcs corresponding to the tubes of order $2$ are the ones of figure \ref{fig:arc22}, plus their successive images under the shift.

\begin{figure}[!h]
\centering
\begin{tikzpicture}[scale=0.6]
\fill[fill=black,fill opacity=0.1] (0,4) -- (-1.66,3.46) -- (-2.69,2.05) -- (-2.68,0.3) -- (-1.66,-1.11) -- (0,-1.65) -- (1.66,-1.11) -- (2.69,0.3) -- (2.69,2.05) -- (1.66,3.46) -- cycle;
\draw [fill=black,fill opacity=1.0] (1.66,1.61) circle (0.2cm);
\draw [fill=black,fill opacity=1.0] (-1.66,1.61) circle (0.2cm);
\draw (0,4)-- (-1.66,3.46);
\draw (-1.66,3.46)-- (-2.69,2.05);
\draw (-2.69,2.05)-- (-2.68,0.3);
\draw (-2.68,0.3)-- (-1.66,-1.11);
\draw (-1.66,-1.11)-- (0,-1.65);
\draw (0,-1.65)-- (1.66,-1.11);
\draw (1.66,-1.11)-- (2.69,0.3);
\draw (2.69,0.3)-- (2.69,2.05);
\draw (2.69,2.05)-- (1.66,3.46);
\draw (1.66,3.46)-- (0,4);
\draw(1.67,1.18) circle (0.44cm);
\draw(-1.67,1.17) circle (0.44cm);
\draw [shift={(-1.66,1.06)}] plot[domain=1.61:4.66,variable=\t]({1*0.75*cos(\t r)+0*0.75*sin(\t r)},{0*0.75*cos(\t r)+1*0.75*sin(\t r)});
\draw [shift={(-1.39,3.98)}] plot[domain=4.63:5.64,variable=\t]({1*3.68*cos(\t r)+0*3.68*sin(\t r)},{0*3.68*cos(\t r)+1*3.68*sin(\t r)});
\end{tikzpicture}
\hspace{20pt}
\begin{tikzpicture}[scale=0.6]
\fill[fill=black,fill opacity=0.1] (0,4) -- (-1.66,3.46) -- (-2.69,2.05) -- (-2.68,0.3) -- (-1.66,-1.11) -- (0,-1.65) -- (1.66,-1.11) -- (2.69,0.3) -- (2.69,2.05) -- (1.66,3.46) -- cycle;
\draw [fill=black,fill opacity=1.0] (1.66,1.61) circle (0.2cm);
\draw [fill=black,fill opacity=1.0] (-1.66,1.61) circle (0.2cm);
\draw (0,4)-- (-1.66,3.46);
\draw (-1.66,3.46)-- (-2.69,2.05);
\draw (-2.69,2.05)-- (-2.68,0.3);
\draw (-2.68,0.3)-- (-1.66,-1.11);
\draw (-1.66,-1.11)-- (0,-1.65);
\draw (0,-1.65)-- (1.66,-1.11);
\draw (1.66,-1.11)-- (2.69,0.3);
\draw (2.69,0.3)-- (2.69,2.05);
\draw (2.69,2.05)-- (1.66,3.46);
\draw (1.66,3.46)-- (0,4);
\draw(1.67,1.18) circle (0.44cm);
\draw(-1.67,1.17) circle (0.44cm);
\draw (-1.63,1.81)-- (1.67,1.81);
\end{tikzpicture}
\caption{Arcs of order $2$}
\label{fig:arc22}
\end{figure}

Here we draw the tube of size $2$ corresponding to the first picture of figure \ref{fig:arc22}.

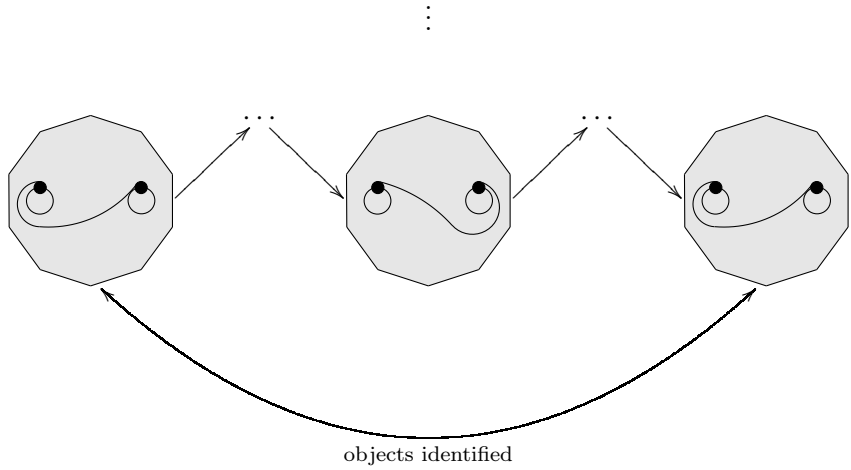
\[ \xymatrix{
& & \vdots & & \\
& \ar[dr] & & \ar[dr] & \\
\begin{tikzpicture}[scale=0.4]
\fill[fill=black,fill opacity=0.1] (0,4) -- (-1.66,3.46) -- (-2.69,2.05) -- (-2.68,0.3) -- (-1.66,-1.11) -- (0,-1.65) -- (1.66,-1.11) -- (2.69,0.3) -- (2.69,2.05) -- (1.66,3.46) -- cycle;
\draw [fill=black,fill opacity=1.0] (1.66,1.61) circle (0.2cm);
\draw [fill=black,fill opacity=1.0] (-1.66,1.61) circle (0.2cm);
\draw (0,4)-- (-1.66,3.46);
\draw (-1.66,3.46)-- (-2.69,2.05);
\draw (-2.69,2.05)-- (-2.68,0.3);
\draw (-2.68,0.3)-- (-1.66,-1.11);
\draw (-1.66,-1.11)-- (0,-1.65);
\draw (0,-1.65)-- (1.66,-1.11);
\draw (1.66,-1.11)-- (2.69,0.3);
\draw (2.69,0.3)-- (2.69,2.05);
\draw (2.69,2.05)-- (1.66,3.46);
\draw (1.66,3.46)-- (0,4);
\draw(1.67,1.18) circle (0.44cm);
\draw(-1.67,1.17) circle (0.44cm);
\draw [shift={(-1.66,1.06)}] plot[domain=1.61:4.66,variable=\t]({1*0.75*cos(\t r)+0*0.75*sin(\t r)},{0*0.75*cos(\t r)+1*0.75*sin(\t r)});
\draw [shift={(-1.39,3.98)}] plot[domain=4.63:5.64,variable=\t]({1*3.68*cos(\t r)+0*3.68*sin(\t r)},{0*3.68*cos(\t r)+1*3.68*sin(\t r)});
\end{tikzpicture} \ar[ur] \ar@/_5pc/[rrrr]_{\text{objects identified}} & & \begin{tikzpicture}[scale=0.4]
\fill[fill=black,fill opacity=0.1] (0,4) -- (-1.66,3.46) -- (-2.69,2.05) -- (-2.68,0.3) -- (-1.66,-1.11) -- (0,-1.65) -- (1.66,-1.11) -- (2.69,0.3) -- (2.69,2.05) -- (1.66,3.46) -- cycle;
\draw [fill=black,fill opacity=1.0] (-1.66,1.61) circle (0.2cm);
\draw [fill=black,fill opacity=1.0] (1.66,1.61) circle (0.2cm);
\draw (0,4)-- (-1.66,3.46);
\draw (-1.66,3.46)-- (-2.69,2.05);
\draw (-2.69,2.05)-- (-2.68,0.3);
\draw (-2.68,0.3)-- (-1.66,-1.11);
\draw (-1.66,-1.11)-- (0,-1.65);
\draw (0,-1.65)-- (1.66,-1.11);
\draw (1.66,-1.11)-- (2.69,0.3);
\draw (2.69,0.3)-- (2.69,2.05);
\draw (2.69,2.05)-- (1.66,3.46);
\draw (1.66,3.46)-- (0,4);
\draw(1.67,1.18) circle (0.44cm);
\draw(-1.67,1.17) circle (0.44cm);
\draw [shift={(1.45,0.95)}] plot[domain=-2.42:1.18,variable=\t]({1*0.89*cos(\t r)+0*0.89*sin(\t r)},{0*0.89*cos(\t r)+1*0.89*sin(\t r)});
\draw [shift={(-2.78,-2.83)}] plot[domain=0.73:1.32,variable=\t]({1*4.78*cos(\t r)+0*4.78*sin(\t r)},{0*4.78*cos(\t r)+1*4.78*sin(\t r)});
\end{tikzpicture} \ar[ur] & & \begin{tikzpicture}[scale=0.4]
\fill[fill=black,fill opacity=0.1] (0,4) -- (-1.66,3.46) -- (-2.69,2.05) -- (-2.68,0.3) -- (-1.66,-1.11) -- (0,-1.65) -- (1.66,-1.11) -- (2.69,0.3) -- (2.69,2.05) -- (1.66,3.46) -- cycle;
\draw [fill=black,fill opacity=1.0] (1.66,1.61) circle (0.2cm);
\draw [fill=black,fill opacity=1.0] (-1.66,1.61) circle (0.2cm);
\draw (0,4)-- (-1.66,3.46);
\draw (-1.66,3.46)-- (-2.69,2.05);
\draw (-2.69,2.05)-- (-2.68,0.3);
\draw (-2.68,0.3)-- (-1.66,-1.11);
\draw (-1.66,-1.11)-- (0,-1.65);
\draw (0,-1.65)-- (1.66,-1.11);
\draw (1.66,-1.11)-- (2.69,0.3);
\draw (2.69,0.3)-- (2.69,2.05);
\draw (2.69,2.05)-- (1.66,3.46);
\draw (1.66,3.46)-- (0,4);
\draw(1.67,1.18) circle (0.44cm);
\draw(-1.67,1.17) circle (0.44cm);
\draw [shift={(-1.66,1.06)}] plot[domain=1.61:4.66,variable=\t]({1*0.75*cos(\t r)+0*0.75*sin(\t r)},{0*0.75*cos(\t r)+1*0.75*sin(\t r)});
\draw [shift={(-1.39,3.98)}] plot[domain=4.63:5.64,variable=\t]({1*3.68*cos(\t r)+0*3.68*sin(\t r)},{0*3.68*cos(\t r)+1*3.68*sin(\t r)});
\end{tikzpicture} \ar@/^5pc/[llll] } \]

The arc appearing at the leftmost second level is the one of figure \ref{fig:selfcross}.

\begin{figure}[!h]
	\centering
\begin{tikzpicture}[scale=0.4]

	\fill[fill=black,fill opacity=0.1] (0,4) -- (-1.66,3.46) -- (-2.685564174424477,2.0474073042320233) -- (-2.684961866287507,0.3017843102195975) -- (-1.6584231368257114,-1.110100329867886) -- (0.0019491100746363715,-1.6489546717109196) -- (1.6619491100746362,-1.1089546717109195) -- (2.687513284499113,0.30363802405705653) -- (2.6869109763621433,2.0492610180694824) -- (1.6603722469003486,3.4611456581569664) -- cycle;

	\draw [fill=black,fill opacity=1] (-1.6593495840332233,1.5749498112556668) circle (0.2cm);

	\draw [fill=black,fill opacity=1] (1.6610226628671236,1.5760954694126326) circle (0.2cm);

	\draw (0,4)-- (-1.66,3.46);

	\draw (-1.66,3.46)-- (-2.685564174424477,2.0474073042320233);

	\draw (-2.685564174424477,2.0474073042320233)-- (-2.684961866287507,0.3017843102195975);

	\draw (-2.684961866287507,0.3017843102195975)-- (-1.6584231368257114,-1.110100329867886);

	\draw (-1.6584231368257114,-1.110100329867886)-- (0.0019491100746363715,-1.6489546717109196);

	\draw (0.0019491100746363715,-1.6489546717109196)-- (1.6619491100746362,-1.1089546717109195);

	\draw (1.6619491100746362,-1.1089546717109195)-- (2.687513284499113,0.30363802405705653);

	\draw (2.687513284499113,0.30363802405705653)-- (2.6869109763621433,2.0492610180694824);

	\draw (2.6869109763621433,2.0492610180694824)-- (1.6603722469003486,3.4611456581569664);

	\draw (1.6603722469003486,3.4611456581569664)-- (0,4);

	\draw (-1.6592115684128554,1.174949835066057) circle (0.4cm);

	\draw (1.6611606784874913,1.1760954932230228) circle (0.4cm);

	\draw [shift={(-1.2464497165009731,0.9998737653768779)}]  plot[domain=-0.9077464429357116:3.5830214883250533,variable=\t]({-0.9303765628680883*1.072662228151896*cos(\t r)+-0.3666053071980307*0.7266871247086502*sin(\t r)},{0.3666053071980307*1.072662228151896*cos(\t r)+-0.9303765628680883*0.7266871247086502*sin(\t r)});

	\draw [shift={(1.139315552340147,1.1955137978439103)}]  plot[domain=-4.81123466701167:0.10108853883732456,variable=\t]({-0.9904461688879015*1.3998698833477312*cos(\t r)+0.13789991492121553*0.7031702256089931*sin(\t r)},{-0.13789991492121553*1.3998698833477312*cos(\t r)+-0.9904461688879015*0.7031702256089931*sin(\t r)});

	\draw [shift={(1.1490477674527368,3.982022013993753)}]  plot[domain=3.8478790053651766:4.6980801404512755,variable=\t]({1*3.4989694028365297*cos(\t r)+0*3.4989694028365297*sin(\t r)},{0*3.4989694028365297*cos(\t r)+1*3.4989694028365297*sin(\t r)});

\end{tikzpicture}
\caption{Second level: the arc does not correspond to an $m$-rigid object since it has some self-crossing}
\label{fig:selfcross}
\end{figure}

\vspace{20pt}

Note that in this tube, only the first line corresponds to rigid objects. The upper arcs cross themselves and thus no not correspond to rigid objects anymore.

The following picture gives the tube of size $5$, with the corresponding arcs. We can notice that in the first four ranks, the arcs do not cross since they match to a rigid object, and this is not the case anymore from rank $5$.

\cor{NB: In order to have clear pictures, we have replaced the inner boundary component inside of $P$, which should be a monogon, with a simple disk.}

\begin{landscape}
\begin{figure}[!h]
\[ \scalebox{0.7}{
 \xymatrix{
 & \ar[dr] & & \ar[dr] & & \ar[dr] & & \ar[dr] & & \\
 \begin{tikzpicture}[scale=0.3]
 \fill[fill=black,fill opacity=0.1] (0,4) -- (-1.78,3.42) -- (-2.88,1.9) -- (-2.88,0.03) -- (-1.78,-1.48) -- (0.01,-2.06) -- (1.79,-1.48) -- (2.88,0.04) -- (2.88,1.91) -- (1.78,3.42) -- cycle;
 \draw [fill=black,fill opacity=1.0] (-1.36,0.97) circle (0.4cm);
 \draw [fill=black,fill opacity=1.0] (1.36,0.97) circle (0.4cm);
 \draw (0,4)-- (-1.78,3.42);
 \draw (-1.78,3.42)-- (-2.88,1.9);
 \draw (-2.88,1.9)-- (-2.88,0.03);
 \draw (-2.88,0.03)-- (-1.78,-1.48);
 \draw (-1.78,-1.48)-- (0.01,-2.06);
 \draw (0.01,-2.06)-- (1.79,-1.48);
 \draw (1.79,-1.48)-- (2.88,0.04);
 \draw (2.88,0.04)-- (2.88,1.91);
 \draw (2.88,1.91)-- (1.78,3.42);
 \draw (1.78,3.42)-- (0,4);
 \draw [shift={(0.65,0.53)}] plot[domain=2.46:4.47,variable=\t]({1*2.67*cos(\t r)+0*2.67*sin(\t r)},{0*2.67*cos(\t r)+1*2.67*sin(\t r)});
 \draw [shift={(0.17,1.06)}] plot[domain=0.09:2.51,variable=\t]({1*1.97*cos(\t r)+0*1.97*sin(\t r)},{0*1.97*cos(\t r)+1*1.97*sin(\t r)});
 \draw [shift={(-0.82,1.33)}] plot[domain=4.39:6.26,variable=\t]({1*2.96*cos(\t r)+0*2.96*sin(\t r)},{0*2.96*cos(\t r)+1*2.96*sin(\t r)});
 \end{tikzpicture} \ar[ur] \ar[dr] & & \begin{tikzpicture}[scale=0.3]
 \fill[fill=black,fill opacity=0.1] (0,4) -- (-1.78,3.42) -- (-2.88,1.9) -- (-2.88,0.03) -- (-1.78,-1.48) -- (0.01,-2.06) -- (1.79,-1.48) -- (2.88,0.04) -- (2.88,1.91) -- (1.78,3.42) -- cycle;
 \draw [fill=black,fill opacity=1.0] (-1.36,0.97) circle (0.4cm);
 \draw [fill=black,fill opacity=1.0] (1.36,0.97) circle (0.4cm);
 \draw (0,4)-- (-1.78,3.42);
 \draw (-1.78,3.42)-- (-2.88,1.9);
 \draw (-2.88,1.9)-- (-2.88,0.03);
 \draw (-2.88,0.03)-- (-1.78,-1.48);
 \draw (-1.78,-1.48)-- (0.01,-2.06);
 \draw (0.01,-2.06)-- (1.79,-1.48);
 \draw (1.79,-1.48)-- (2.88,0.04);
 \draw (2.88,0.04)-- (2.88,1.91);
 \draw (2.88,1.91)-- (1.78,3.42);
 \draw (1.78,3.42)-- (0,4);
 \draw [shift={(-0.15,-0.01)}] plot[domain=1.03:3.13,variable=\t]({1*2.73*cos(\t r)+0*2.73*sin(\t r)},{0*2.73*cos(\t r)+1*2.73*sin(\t r)});
 \draw [shift={(0.34,0.92)}] plot[domain=-1.26:1,variable=\t]({1*1.68*cos(\t r)+0*1.68*sin(\t r)},{0*1.68*cos(\t r)+1*1.68*sin(\t r)});
 \draw [shift={(0.03,2.12)}] plot[domain=3.22:5,variable=\t]({1*2.92*cos(\t r)+0*2.92*sin(\t r)},{0*2.92*cos(\t r)+1*2.92*sin(\t r)});
 \end{tikzpicture} \ar[ur] \ar[dr] & & \begin{tikzpicture}[scale=0.3]
 \fill[fill=black,fill opacity=0.1] (0,4) -- (-1.78,3.42) -- (-2.88,1.9) -- (-2.88,0.03) -- (-1.78,-1.48) -- (0.01,-2.06) -- (1.79,-1.48) -- (2.88,0.04) -- (2.88,1.91) -- (1.78,3.42) -- cycle;
 \draw [fill=black,fill opacity=1.0] (-1.36,0.97) circle (0.4cm);
 \draw [fill=black,fill opacity=1.0] (1.36,0.97) circle (0.4cm);
 \draw (0,4)-- (-1.78,3.42);
 \draw (-1.78,3.42)-- (-2.88,1.9);
 \draw (-2.88,1.9)-- (-2.88,0.03);
 \draw (-2.88,0.03)-- (-1.78,-1.48);
 \draw (-1.78,-1.48)-- (0.01,-2.06);
 \draw (0.01,-2.06)-- (1.79,-1.48);
 \draw (1.79,-1.48)-- (2.88,0.04);
 \draw (2.88,0.04)-- (2.88,1.91);
 \draw (2.88,1.91)-- (1.78,3.42);
 \draw (1.78,3.42)-- (0,4);
 \draw [shift={(-0.65,0.65)}] plot[domain=0.31:1.96,variable=\t]({1*2.99*cos(\t r)+0*2.99*sin(\t r)},{0*2.99*cos(\t r)+1*2.99*sin(\t r)});
 \draw [shift={(0.15,0.95)}] plot[domain=-2.37:0.3,variable=\t]({1*2.15*cos(\t r)+0*2.15*sin(\t r)},{0*2.15*cos(\t r)+1*2.15*sin(\t r)});
 \draw [shift={(0.77,1.28)}] plot[domain=1.85:3.84,variable=\t]({1*2.83*cos(\t r)+0*2.83*sin(\t r)},{0*2.83*cos(\t r)+1*2.83*sin(\t r)});
 \end{tikzpicture} \ar[ur] \ar[dr] & & \begin{tikzpicture}[scale=0.3]
 \fill[fill=black,fill opacity=0.1] (0,4) -- (-1.78,3.42) -- (-2.88,1.9) -- (-2.88,0.03) -- (-1.78,-1.48) -- (0.01,-2.06) -- (1.79,-1.48) -- (2.88,0.04) -- (2.88,1.91) -- (1.78,3.42) -- cycle;
 \draw [fill=black,fill opacity=1.0] (-1.36,0.97) circle (0.4cm);
 \draw [fill=black,fill opacity=1.0] (1.36,0.97) circle (0.4cm);
 \draw (0,4)-- (-1.78,3.42);
 \draw (-1.78,3.42)-- (-2.88,1.9);
 \draw (-2.88,1.9)-- (-2.88,0.03);
 \draw (-2.88,0.03)-- (-1.78,-1.48);
 \draw (-1.78,-1.48)-- (0.01,-2.06);
 \draw (0.01,-2.06)-- (1.79,-1.48);
 \draw (1.79,-1.48)-- (2.88,0.04);
 \draw (2.88,0.04)-- (2.88,1.91);
 \draw (2.88,1.91)-- (1.78,3.42);
 \draw (1.78,3.42)-- (0,4);
 \draw [shift={(-0.46,1.64)}] plot[domain=-1.45:0.67,variable=\t]({1*2.86*cos(\t r)+0*2.86*sin(\t r)},{0*2.86*cos(\t r)+1*2.86*sin(\t r)});
 \draw [shift={(-0.45,0.89)}] plot[domain=1.82:4.88,variable=\t]({1*2.13*cos(\t r)+0*2.13*sin(\t r)},{0*2.13*cos(\t r)+1*2.13*sin(\t r)});
 \draw [shift={(-0.07,-1.34)}] plot[domain=0.83:1.78,variable=\t]({1*4.39*cos(\t r)+0*4.39*sin(\t r)},{0*4.39*cos(\t r)+1*4.39*sin(\t r)});
 \end{tikzpicture} \ar[ur] \ar[dr] & & \begin{tikzpicture}[scale=0.3]
 \fill[fill=black,fill opacity=0.1] (0,4) -- (-1.78,3.42) -- (-2.88,1.9) -- (-2.88,0.03) -- (-1.78,-1.48) -- (0.01,-2.06) -- (1.79,-1.48) -- (2.88,0.04) -- (2.88,1.91) -- (1.78,3.42) -- cycle;
 \draw [fill=black,fill opacity=1.0] (-1.36,0.97) circle (0.4cm);
 \draw [fill=black,fill opacity=1.0] (1.36,0.97) circle (0.4cm);
 \draw (0,4)-- (-1.78,3.42);
 \draw (-1.78,3.42)-- (-2.88,1.9);
 \draw (-2.88,1.9)-- (-2.88,0.03);
 \draw (-2.88,0.03)-- (-1.78,-1.48);
 \draw (-1.78,-1.48)-- (0.01,-2.06);
 \draw (0.01,-2.06)-- (1.79,-1.48);
 \draw (1.79,-1.48)-- (2.88,0.04);
 \draw (2.88,0.04)-- (2.88,1.91);
 \draw (2.88,1.91)-- (1.78,3.42);
 \draw (1.78,3.42)-- (0,4);
 \draw [shift={(0.28,1.98)}] plot[domain=3.88:5.64,variable=\t]({1*3.25*cos(\t r)+0*3.25*sin(\t r)},{0*3.25*cos(\t r)+1*3.25*sin(\t r)});
 \draw [shift={(-0.48,1.18)}] plot[domain=1.56:3.85,variable=\t]({1*2.14*cos(\t r)+0*2.14*sin(\t r)},{0*2.14*cos(\t r)+1*2.14*sin(\t r)});
 \draw [shift={(-0.58,0.34)}] plot[domain=-0.65:1.53,variable=\t]({1*2.98*cos(\t r)+0*2.98*sin(\t r)},{0*2.98*cos(\t r)+1*2.98*sin(\t r)});
 \end{tikzpicture} \ar[dr] \ar[ur] &  \\
 & \begin{tikzpicture}[scale=0.3]
 \fill[fill=black,fill opacity=0.1] (0,4) -- (-1.78,3.42) -- (-2.88,1.9) -- (-2.88,0.03) -- (-1.78,-1.48) -- (0.01,-2.06) -- (1.79,-1.48) -- (2.88,0.04) -- (2.88,1.91) -- (1.78,3.42) -- cycle;
 \draw [fill=black,fill opacity=1.0] (-1.36,0.97) circle (0.4cm);
 \draw [fill=black,fill opacity=1.0] (1.36,0.97) circle (0.4cm);
 \draw (0,4)-- (-1.78,3.42);
 \draw (-1.78,3.42)-- (-2.88,1.9);
 \draw (-2.88,1.9)-- (-2.88,0.03);
 \draw (-2.88,0.03)-- (-1.78,-1.48);
 \draw (-1.78,-1.48)-- (0.01,-2.06);
 \draw (0.01,-2.06)-- (1.79,-1.48);
 \draw (1.79,-1.48)-- (2.88,0.04);
 \draw (2.88,0.04)-- (2.88,1.91);
 \draw (2.88,1.91)-- (1.78,3.42);
 \draw (1.78,3.42)-- (0,4);
 \draw [shift={(0.22,-0.07)}] plot[domain=1.4:3.11,variable=\t]({1*3.1*cos(\t r)+0*3.1*sin(\t r)},{0*3.1*cos(\t r)+1*3.1*sin(\t r)});
 \draw [shift={(0.33,0.99)}] plot[domain=-0.68:1.36,variable=\t]({1*2.04*cos(\t r)+0*2.04*sin(\t r)},{0*2.04*cos(\t r)+1*2.04*sin(\t r)});
 \draw [shift={(-0.92,2.17)}] plot[domain=4.48:5.57,variable=\t]({1*3.75*cos(\t r)+0*3.75*sin(\t r)},{0*3.75*cos(\t r)+1*3.75*sin(\t r)});
 \end{tikzpicture} \ar[ur] \ar[dr] & & \begin{tikzpicture}[scale=0.3]
 \fill[fill=black,fill opacity=0.1] (0,4) -- (-1.78,3.42) -- (-2.88,1.9) -- (-2.88,0.03) -- (-1.78,-1.48) -- (0.01,-2.06) -- (1.79,-1.48) -- (2.88,0.04) -- (2.88,1.91) -- (1.78,3.42) -- cycle;
 \draw [fill=black,fill opacity=1.0] (-1.36,0.97) circle (0.4cm);
 \draw [fill=black,fill opacity=1.0] (1.36,0.97) circle (0.4cm);
 \draw (0,4)-- (-1.78,3.42);
 \draw (-1.78,3.42)-- (-2.88,1.9);
 \draw (-2.88,1.9)-- (-2.88,0.03);
 \draw (-2.88,0.03)-- (-1.78,-1.48);
 \draw (-1.78,-1.48)-- (0.01,-2.06);
 \draw (0.01,-2.06)-- (1.79,-1.48);
 \draw (1.79,-1.48)-- (2.88,0.04);
 \draw (2.88,0.04)-- (2.88,1.91);
 \draw (2.88,1.91)-- (1.78,3.42);
 \draw (1.78,3.42)-- (0,4);
 \draw [shift={(-0.73,0.44)}] plot[domain=0.41:1.91,variable=\t]({1*3.16*cos(\t r)+0*3.16*sin(\t r)},{0*3.16*cos(\t r)+1*3.16*sin(\t r)});
 \draw [shift={(0.08,0.96)}] plot[domain=-1.79:0.34,variable=\t]({1*2.22*cos(\t r)+0*2.22*sin(\t r)},{0*2.22*cos(\t r)+1*2.22*sin(\t r)});
 \draw [shift={(1.06,2.5)}] plot[domain=3.29:4.34,variable=\t]({1*3.98*cos(\t r)+0*3.98*sin(\t r)},{0*3.98*cos(\t r)+1*3.98*sin(\t r)});
 \end{tikzpicture} \ar[ur] \ar[dr] & & \begin{tikzpicture}[scale=0.3]
 \fill[fill=black,fill opacity=0.1] (0,4) -- (-1.78,3.42) -- (-2.88,1.9) -- (-2.88,0.03) -- (-1.78,-1.48) -- (0.01,-2.06) -- (1.79,-1.48) -- (2.88,0.04) -- (2.88,1.91) -- (1.78,3.42) -- cycle;
 \draw [fill=black,fill opacity=1.0] (-1.36,0.97) circle (0.4cm);
 \draw [fill=black,fill opacity=1.0] (1.36,0.97) circle (0.4cm);
 \draw (0,4)-- (-1.78,3.42);
 \draw (-1.78,3.42)-- (-2.88,1.9);
 \draw (-2.88,1.9)-- (-2.88,0.03);
 \draw (-2.88,0.03)-- (-1.78,-1.48);
 \draw (-1.78,-1.48)-- (0.01,-2.06);
 \draw (0.01,-2.06)-- (1.79,-1.48);
 \draw (1.79,-1.48)-- (2.88,0.04);
 \draw (2.88,0.04)-- (2.88,1.91);
 \draw (2.88,1.91)-- (1.78,3.42);
 \draw (1.78,3.42)-- (0,4);
 \draw [shift={(-0.23,1.64)}] plot[domain=-1.13:0.73,variable=\t]({1*2.69*cos(\t r)+0*2.69*sin(\t r)},{0*2.69*cos(\t r)+1*2.69*sin(\t r)});
 \draw [shift={(0.2,1.51)}] plot[domain=3.39:5.01,variable=\t]({1*2.41*cos(\t r)+0*2.41*sin(\t r)},{0*2.41*cos(\t r)+1*2.41*sin(\t r)});
 \draw [shift={(0.61,1.29)}] plot[domain=1.79:3.28,variable=\t]({1*2.78*cos(\t r)+0*2.78*sin(\t r)},{0*2.78*cos(\t r)+1*2.78*sin(\t r)});
 \end{tikzpicture} \ar[ur] \ar[dr] & & \begin{tikzpicture}[scale=0.3]
 \fill[fill=black,fill opacity=0.1] (0,4) -- (-1.78,3.42) -- (-2.88,1.9) -- (-2.88,0.03) -- (-1.78,-1.48) -- (0.01,-2.06) -- (1.79,-1.48) -- (2.88,0.04) -- (2.88,1.91) -- (1.78,3.42) -- cycle;
 \draw [fill=black,fill opacity=1.0] (-1.36,0.97) circle (0.4cm);
 \draw [fill=black,fill opacity=1.0] (1.36,0.97) circle (0.4cm);
 \draw (0,4)-- (-1.78,3.42);
 \draw (-1.78,3.42)-- (-2.88,1.9);
 \draw (-2.88,1.9)-- (-2.88,0.03);
 \draw (-2.88,0.03)-- (-1.78,-1.48);
 \draw (-1.78,-1.48)-- (0.01,-2.06);
 \draw (0.01,-2.06)-- (1.79,-1.48);
 \draw (1.79,-1.48)-- (2.88,0.04);
 \draw (2.88,0.04)-- (2.88,1.91);
 \draw (2.88,1.91)-- (1.78,3.42);
 \draw (1.78,3.42)-- (0,4);
 \draw [shift={(0.27,1.87)}] plot[domain=3.95:5.67,variable=\t]({1*3.19*cos(\t r)+0*3.19*sin(\t r)},{0*3.19*cos(\t r)+1*3.19*sin(\t r)});
 \draw [shift={(-0.56,0.91)}] plot[domain=1.75:3.91,variable=\t]({1*1.93*cos(\t r)+0*1.93*sin(\t r)},{0*1.93*cos(\t r)+1*1.93*sin(\t r)});
 \draw [shift={(-0.13,-2.33)}] plot[domain=0.95:1.72,variable=\t]({1*5.2*cos(\t r)+0*5.2*sin(\t r)},{0*5.2*cos(\t r)+1*5.2*sin(\t r)});
 \end{tikzpicture} \ar[ur] \ar[dr] & & \begin{tikzpicture}[scale=0.3]
 \fill[fill=black,fill opacity=0.1] (0,4) -- (-1.78,3.42) -- (-2.88,1.9) -- (-2.88,0.03) -- (-1.78,-1.48) -- (0.01,-2.06) -- (1.79,-1.48) -- (2.88,0.04) -- (2.88,1.91) -- (1.78,3.42) -- cycle;
 \draw [fill=black,fill opacity=1.0] (-1.36,0.97) circle (0.4cm);
 \draw [fill=black,fill opacity=1.0] (1.36,0.97) circle (0.4cm);
 \draw (0,4)-- (-1.78,3.42);
 \draw (-1.78,3.42)-- (-2.88,1.9);
 \draw (-2.88,1.9)-- (-2.88,0.03);
 \draw (-2.88,0.03)-- (-1.78,-1.48);
 \draw (-1.78,-1.48)-- (0.01,-2.06);
 \draw (0.01,-2.06)-- (1.79,-1.48);
 \draw (1.79,-1.48)-- (2.88,0.04);
 \draw (2.88,0.04)-- (2.88,1.91);
 \draw (2.88,1.91)-- (1.78,3.42);
 \draw (1.78,3.42)-- (0,4);
 \draw [shift={(0.74,0.86)}] plot[domain=2.73:4.47,variable=\t]({1*3.01*cos(\t r)+0*3.01*sin(\t r)},{0*3.01*cos(\t r)+1*3.01*sin(\t r)});
 \draw [shift={(-0.08,0.82)}] plot[domain=0.78:2.57,variable=\t]({1*2.32*cos(\t r)+0*2.32*sin(\t r)},{0*2.32*cos(\t r)+1*2.32*sin(\t r)});
 \draw [shift={(-1,0.33)}] plot[domain=-0.58:0.69,variable=\t]({1*3.33*cos(\t r)+0*3.33*sin(\t r)},{0*3.33*cos(\t r)+1*3.33*sin(\t r)});
 \end{tikzpicture} \\
 \begin{tikzpicture}[scale=0.3]
 \fill[fill=black,fill opacity=0.1] (0,4) -- (-1.78,3.42) -- (-2.88,1.9) -- (-2.88,0.03) -- (-1.78,-1.48) -- (0.01,-2.06) -- (1.79,-1.48) -- (2.88,0.04) -- (2.88,1.91) -- (1 .78,3.42) -- cycle;
 \draw [fill=black,fill opacity=1.0] (-1.36,0.97) circle (0.4cm);
 \draw [fill=black,fill opacity=1.0] (1.36,0.97) circle (0.4cm);
 \draw (0,4)-- (-1.78,3.42);
 \draw (-1.78,3.42)-- (-2.88,1.9);
 \draw (-2.88,1.9)-- (-2.88,0.03);
 \draw (-2.88,0.03)-- (-1.78,-1.48);
 \draw (-1.78,-1.48)-- (0.01,-2.06);
 \draw (0.01,-2.06)-- (1.79,-1.48);
 \draw (1.79,-1.48)-- (2.88,0.04);
 \draw (2.88,0.04)-- (2.88,1.91);
 \draw (2.88,1.91)-- (1.78,3.42);
 \draw (1.78,3.42)-- (0,4);
 \draw [shift={(-0.12,0.07)}] plot[domain=0.74:3.16,variable=\t]({1*2.75*cos(\t r)+0*2.75*sin(\t r)},{0*2.75*cos(\t r)+1*2.75*sin(\t r)});
 \draw [shift={(-0.36,0.31)}] plot[domain=-0.69:0.62,variable=\t]({1*2.79*cos(\t r)+0*2.79*sin(\t r)},{0*2.79*cos(\t r)+1*2.79*sin(\t r)});
 \end{tikzpicture} \ar[ur] \ar[dr] & & \begin{tikzpicture}[scale=0.3]
 \fill[fill=black,fill opacity=0.1] (0,4) -- (-1.78,3.42) -- (-2.88,1.9) -- (-2.88,0.03) -- (-1.78,-1.48) -- (0.01,-2.06) -- (1.79,-1.48) -- (2.88,0.04) -- (2.88,1.91) -- (1.78,3.42) -- cycle;
 \draw [fill=black,fill opacity=1.0] (-1.36,0.97) circle (0.4cm);
 \draw [fill=black,fill opacity=1.0] (1.36,0.97) circle (0.4cm);
 \draw (0,4)-- (-1.78,3.42);
 \draw (-1.78,3.42)-- (-2.88,1.9);
 \draw (-2.88,1.9)-- (-2.88,0.03);
 \draw (-2.88,0.03)-- (-1.78,-1.48);
 \draw (-1.78,-1.48)-- (0.01,-2.06);
 \draw (0.01,-2.06)-- (1.79,-1.48);
 \draw (1.79,-1.48)-- (2.88,0.04);
 \draw (2.88,0.04)-- (2.88,1.91);
 \draw (2.88,1.91)-- (1.78,3.42);
 \draw (1.78,3.42)-- (0,4);
 \draw [shift={(-0.71,0.88)}] plot[domain=-0.36:1.97,variable=\t]({1*2.75*cos(\t r)+0*2.75*sin(\t r)},{0*2.75*cos(\t r)+1*2.75*sin(\t r)});
 \draw [shift={(-0.68,1.12)}] plot[domain=4.31:5.84,variable=\t]({1*2.83*cos(\t r)+0*2.83*sin(\t r)},{0*2.83*cos(\t r)+1*2.83*sin(\t r)});
 \end{tikzpicture} \ar[ur] \ar[dr] & & \begin{tikzpicture}[scale=0.3]
 \fill[fill=black,fill opacity=0.1] (0,4) -- (-1.78,3.42) -- (-2.88,1.9) -- (-2.88,0.03) -- (-1.78,-1.48) -- (0.01,-2.06) -- (1.79,-1.48) -- (2.88,0.04) -- (2.88,1.91) -- (1.78,3.42) -- cycle;
 \draw [fill=black,fill opacity=1.0] (-1.36,0.97) circle (0.4cm);
 \draw [fill=black,fill opacity=1.0] (1.36,0.97) circle (0.4cm);
 \draw (0,4)-- (-1.78,3.42);
 \draw (-1.78,3.42)-- (-2.88,1.9);
 \draw (-2.88,1.9)-- (-2.88,0.03);
 \draw (-2.88,0.03)-- (-1.78,-1.48);
 \draw (-1.78,-1.48)-- (0.01,-2.06);
 \draw (0.01,-2.06)-- (1.79,-1.48);
 \draw (1.79,-1.48)-- (2.88,0.04);
 \draw (2.88,0.04)-- (2.88,1.91);
 \draw (2.88,1.91)-- (1.78,3.42);
 \draw (1.78,3.42)-- (0,4);
 \draw [shift={(-0.16,1.69)}] plot[domain=-1.22:0.73,variable=\t]({1*2.6*cos(\t r)+0*2.6*sin(\t r)},{0*2.6*cos(\t r)+1*2.6*sin(\t r)});
 \draw [shift={(0,2.03)}] plot[domain=3.19:4.97,variable=\t]({1*2.88*cos(\t r)+0*2.88*sin(\t r)},{0*2.88*cos(\t r)+1*2.88*sin(\t r)});
 \end{tikzpicture} \ar[ur] \ar[dr] & & \begin{tikzpicture}[scale=0.3]
 \fill[fill=black,fill opacity=0.1] (0,4) -- (-1.78,3.42) -- (-2.88,1.9) -- (-2.88,0.03) -- (-1.78,-1.48) -- (0.01,-2.06) -- (1.79,-1.48) -- (2.88,0.04) -- (2.88,1.91) -- (1.78,3.42) -- cycle;
 \draw [fill=black,fill opacity=1.0] (-1.36,0.97) circle (0.4cm);
 \draw [fill=black,fill opacity=1.0] (1.36,0.97) circle (0.4cm);
 \draw (0,4)-- (-1.78,3.42);
 \draw (-1.78,3.42)-- (-2.88,1.9);
 \draw (-2.88,1.9)-- (-2.88,0.03);
 \draw (-2.88,0.03)-- (-1.78,-1.48);
 \draw (-1.78,-1.48)-- (0.01,-2.06);
 \draw (0.01,-2.06)-- (1.79,-1.48);
 \draw (1.79,-1.48)-- (2.88,0.04);
 \draw (2.88,0.04)-- (2.88,1.91);
 \draw (2.88,1.91)-- (1.78,3.42);
 \draw (1.78,3.42)-- (0,4);
 \draw [shift={(0.49,1.72)}] plot[domain=3.85:5.66,variable=\t]({1*2.93*cos(\t r)+0*2.93*sin(\t r)},{0*2.93*cos(\t r)+1*2.93*sin(\t r)});
 \draw [shift={(0.72,1.25)}] plot[domain=1.83:3.67,variable=\t]({1*2.84*cos(\t r)+0*2.84*sin(\t r)},{0*2.84*cos(\t r)+1*2.84*sin(\t r)});
 \end{tikzpicture} \ar[ur] \ar[dr] & & \begin{tikzpicture}[scale=0.3]
 \fill[fill=black,fill opacity=0.1] (0,4) -- (-1.78,3.42) -- (-2.88,1.9) -- (-2.88,0.03) -- (-1.78,-1.48) -- (0.01,-2.06) -- (1.79,-1.48) -- (2.88,0.04) -- (2.88,1.91) -- (1.78,3.42) -- cycle;
 \draw [fill=black,fill opacity=1.0] (-1.36,0.97) circle (0.4cm);
 \draw [fill=black,fill opacity=1.0] (1.36,0.97) circle (0.4cm);
 \draw (0,4)-- (-1.78,3.42);
 \draw (-1.78,3.42)-- (-2.88,1.9);
 \draw (-2.88,1.9)-- (-2.88,0.03);
 \draw (-2.88,0.03)-- (-1.78,-1.48);
 \draw (-1.78,-1.48)-- (0.01,-2.06);
 \draw (0.01,-2.06)-- (1.79,-1.48);
 \draw (1.79,-1.48)-- (2.88,0.04);
 \draw (2.88,0.04)-- (2.88,1.91);
 \draw (2.88,1.91)-- (1.78,3.42);
 \draw (1.78,3.42)-- (0,4);
 \draw [shift={(0.65,0.59)}] plot[domain=2.58:4.47,variable=\t]({1*2.73*cos(\t r)+0*2.73*sin(\t r)},{0*2.73*cos(\t r)+1*2.73*sin(\t r)});
 \draw [shift={(0.56,0.13)}] plot[domain=0.65:2.43,variable=\t]({1*2.93*cos(\t r)+0*2.93*sin(\t r)},{0*2.93*cos(\t r)+1*2.93*sin(\t r)});
 \end{tikzpicture} \ar[ur] \ar[dr] &  \\
 & \begin{tikzpicture}[scale=0.3]
 \fill[fill=black,fill opacity=0.1] (0,4) -- (-1.66,3.46) -- (-2.69,2.05) -- (-2.68,0.3) -- (-1.66,-1.11) -- (0,-1.65) -- (1.66,-1.11) -- (2.69,0.3) -- (2.69,2.05) -- (1.66,3.46) -- cycle;
 \draw [fill=black,fill opacity=1.0] (-1.66,1.17) circle (0.4cm);
 \draw [fill=black,fill opacity=1.0] (1.66,1.18) circle (0.4cm);
 \draw (0,4)-- (-1.66,3.46);
 \draw (-1.66,3.46)-- (-2.69,2.05);
 \draw (-2.69,2.05)-- (-2.68,0.3);
 \draw (-2.68,0.3)-- (-1.66,-1.11);
 \draw (-1.66,-1.11)-- (0,-1.65);
 \draw (0,-1.65)-- (1.66,-1.11);
 \draw (1.66,-1.11)-- (2.69,0.3);
 \draw (2.69,0.3)-- (2.69,2.05);
 \draw (2.69,2.05)-- (1.66,3.46);
 \draw (1.66,3.46)-- (0,4);
 \draw [shift={(-0.64,0.71)}] plot[domain=-0.67:1.93,variable=\t]({1*2.93*cos(\t r)+0*2.93*sin(\t r)},{0*2.93*cos(\t r)+1*2.93*sin(\t r)});
 \end{tikzpicture} \ar[ur] \ar[dr] & & \begin{tikzpicture}[scale=0.3]
 \fill[fill=black,fill opacity=0.1] (0,4) -- (-1.66,3.46) -- (-2.69,2.05) -- (-2.68,0.3) -- (-1.66,-1.11) -- (0,-1.65) -- (1.66,-1.11) -- (2.69,0.3) -- (2.69,2.05) -- (1.66,3.46) -- cycle;
 \draw [fill=black,fill opacity=1.0] (-1.66,1.17) circle (0.4cm);
 \draw [fill=black,fill opacity=1.0] (1.66,1.18) circle (0.4cm);
 \draw (0,4)-- (-1.66,3.46);
 \draw (-1.66,3.46)-- (-2.69,2.05);
 \draw (-2.69,2.05)-- (-2.68,0.3);
 \draw (-2.68,0.3)-- (-1.66,-1.11);
 \draw (-1.66,-1.11)-- (0,-1.65);
 \draw (0,-1.65)-- (1.66,-1.11);
 \draw (1.66,-1.11)-- (2.69,0.3);
 \draw (2.69,0.3)-- (2.69,2.05);
 \draw (2.69,2.05)-- (1.66,3.46);
 \draw (1.66,3.46)-- (0,4);
 \draw [shift={(-0.56,1.58)}] plot[domain=-1.96:0.7,variable=\t]({1*2.91*cos(\t r)+0*2.91*sin(\t r)},{0*2.91*cos(\t r)+1*2.91*sin(\t r)});
 \end{tikzpicture} \ar[ur] \ar[dr] & & \begin{tikzpicture}[scale=0.3]
 \fill[fill=black,fill opacity=0.1] (0,4) -- (-1.66,3.46) -- (-2.69,2.05) -- (-2.68,0.3) -- (-1.66,-1.11) -- (0,-1.65) -- (1.66,-1.11) -- (2.69,0.3) -- (2.69,2.05) -- (1.66,3.46) -- cycle;
 \draw [fill=black,fill opacity=1.0] (-1.66,1.17) circle (0.4cm);
 \draw [fill=black,fill opacity=1.0] (1.66,1.18) circle (0.4cm);
 \draw (0,4)-- (-1.66,3.46);
 \draw (-1.66,3.46)-- (-2.69,2.05);
 \draw (-2.69,2.05)-- (-2.68,0.3);
 \draw (-2.68,0.3)-- (-1.66,-1.11);
 \draw (-1.66,-1.11)-- (0,-1.65);
 \draw (0,-1.65)-- (1.66,-1.11);
 \draw (1.66,-1.11)-- (2.69,0.3);
 \draw (2.69,0.3)-- (2.69,2.05);
 \draw (2.69,2.05)-- (1.66,3.46);
 \draw (1.66,3.46)-- (0,4);
 \draw [shift={(0.27,1.99)}] plot[domain=3.12:5.67,variable=\t]({1*2.95*cos(\t r)+0*2.95*sin(\t r)},{0*2.95*cos(\t r)+1*2.95*sin(\t r)});
 \end{tikzpicture} \ar[ur] \ar[dr] & & \begin{tikzpicture}[scale=0.3]
 \fill[fill=black,fill opacity=0.1] (0,4) -- (-1.66,3.46) -- (-2.69,2.05) -- (-2.68,0.3) -- (-1.66,-1.11) -- (0,-1.65) -- (1.66,-1.11) -- (2.69,0.3) -- (2.69,2.05) -- (1.66,3.46) -- cycle;
 \draw [fill=black,fill opacity=1.0] (-1.66,1.17) circle (0.4cm);
 \draw [fill=black,fill opacity=1.0] (1.66,1.18) circle (0.4cm);
 \draw (0,4)-- (-1.66,3.46);
 \draw (-1.66,3.46)-- (-2.69,2.05);
 \draw (-2.69,2.05)-- (-2.68,0.3);
 \draw (-2.68,0.3)-- (-1.66,-1.11);
 \draw (-1.66,-1.11)-- (0,-1.65);
 \draw (0,-1.65)-- (1.66,-1.11);
 \draw (1.66,-1.11)-- (2.69,0.3);
 \draw (2.69,0.3)-- (2.69,2.05);
 \draw (2.69,2.05)-- (1.66,3.46);
 \draw (1.66,3.46)-- (0,4);
 \draw [shift={(0.69,1.18)}] plot[domain=1.81:4.47,variable=\t]({1*2.91*cos(\t r)+0*2.91*sin(\t r)},{0*2.91*cos(\t r)+1*2.91*sin(\t r)});
 \end{tikzpicture} \ar[ur] \ar[dr] & & \begin{tikzpicture}[scale=0.3]
 \fill[fill=black,fill opacity=0.1] (0,4) -- (-1.66,3.46) -- (-2.69,2.05) -- (-2.68,0.3) -- (-1.66,-1.11) -- (0,-1.65) -- (1.66,-1.11) -- (2.69,0.3) -- (2.69,2.05) -- (1.66,3.46) -- cycle;
 \draw [fill=black,fill opacity=1.0] (-1.66,1.17) circle (0.4cm);
 \draw [fill=black,fill opacity=1.0] (1.66,1.18) circle (0.4cm);
 \draw (0,4)-- (-1.66,3.46);
 \draw (-1.66,3.46)-- (-2.69,2.05);
 \draw (-2.69,2.05)-- (-2.68,0.3);
 \draw (-2.68,0.3)-- (-1.66,-1.11);
 \draw (-1.66,-1.11)-- (0,-1.65);
 \draw (0,-1.65)-- (1.66,-1.11);
 \draw (1.66,-1.11)-- (2.69,0.3);
 \draw (2.69,0.3)-- (2.69,2.05);
 \draw (2.69,2.05)-- (1.66,3.46);
 \draw (1.66,3.46)-- (0,4);
 \draw [shift={(0.23,0.47)}] plot[domain=0.57:3.2,variable=\t]({1*2.92*cos(\t r)+0*2.92*sin(\t r)},{0*2.92*cos(\t r)+1*2.92*sin(\t r)});
 \end{tikzpicture} \\
 \begin{tikzpicture}[scale=0.3]
 \fill[fill=black,fill opacity=0.1] (0,4) -- (-1.66,3.46) -- (-2.69,2.05) -- (-2.68,0.3) -- (-1.66,-1.11) -- (0,-1.65) -- (1.66,-1.11) -- (2.69,0.3) -- (2.69,2.05) -- (1.66,3.46) -- cycle;
 \draw [fill=black,fill opacity=1.0] (-1.66,1.17) circle (0.4cm);
 \draw [fill=black,fill opacity=1.0] (1.66,1.18) circle (0.4cm);
 \draw (0,4)-- (-1.66,3.46);
 \draw (-1.66,3.46)-- (-2.69,2.05);
 \draw (-2.69,2.05)-- (-2.68,0.3);
 \draw (-2.68,0.3)-- (-1.66,-1.11);
 \draw (-1.66,-1.11)-- (0,-1.65);
 \draw (0,-1.65)-- (1.66,-1.11);
 \draw (1.66,-1.11)-- (2.69,0.3);
 \draw (2.69,0.3)-- (2.69,2.05);
 \draw (2.69,2.05)-- (1.66,3.46);
 \draw (1.66,3.46)-- (0,4);
 \draw [shift={(-0.57,-0.59)}] plot[domain=0.68:1.83,variable=\t]({1*4.19*cos(\t r)+0*4.19*sin(\t r)},{0*4.19*cos(\t r)+1*4.19*sin(\t r)});
 \end{tikzpicture} \ar[ur] & & \begin{tikzpicture}[scale=0.3]
 \fill[fill=black,fill opacity=0.1] (0,4) -- (-1.66,3.46) -- (-2.69,2.05) -- (-2.68,0.3) -- (-1.66,-1.11) -- (0,-1.65) -- (1.66,-1.11) -- (2.69,0.3) -- (2.69,2.05) -- (1.66,3.46) -- cycle;
 \draw [fill=black,fill opacity=1.0] (-1.66,1.17) circle (0.4cm);
 \draw [fill=black,fill opacity=1.0] (1.66,1.18) circle (0.4cm);
 \draw (0,4)-- (-1.66,3.46);
 \draw (-1.66,3.46)-- (-2.69,2.05);
 \draw (-2.69,2.05)-- (-2.68,0.3);
 \draw (-2.68,0.3)-- (-1.66,-1.11);
 \draw (-1.66,-1.11)-- (0,-1.65);
 \draw (0,-1.65)-- (1.66,-1.11);
 \draw (1.66,-1.11)-- (2.69,0.3);
 \draw (2.69,0.3)-- (2.69,2.05);
 \draw (2.69,2.05)-- (1.66,3.46);
 \draw (1.66,3.46)-- (0,4);
 \draw [shift={(-1.72,1.17)}] plot[domain=-0.59:0.59,variable=\t]({1*4.08*cos(\t r)+0*4.08*sin(\t r)},{0*4.08*cos(\t r)+1*4.08*sin(\t r)});
 \end{tikzpicture} \ar[ur] & & \begin{tikzpicture}[scale=0.3]
 \fill[fill=black,fill opacity=0.1] (0,4) -- (-1.66,3.46) -- (-2.69,2.05) -- (-2.68,0.3) -- (-1.66,-1.11) -- (0,-1.65) -- (1.66,-1.11) -- (2.69,0.3) -- (2.69,2.05) -- (1.66,3.46) -- cycle;
 \draw [fill=black,fill opacity=1.0] (-1.66,1.17) circle (0.4cm);
 \draw [fill=black,fill opacity=1.0] (1.66,1.18) circle (0.4cm);
 \draw (0,4)-- (-1.66,3.46);
 \draw (-1.66,3.46)-- (-2.69,2.05);
 \draw (-2.69,2.05)-- (-2.68,0.3);
 \draw (-2.68,0.3)-- (-1.66,-1.11);
 \draw (-1.66,-1.11)-- (0,-1.65);
 \draw (0,-1.65)-- (1.66,-1.11);
 \draw (1.66,-1.11)-- (2.69,0.3);
 \draw (2.69,0.3)-- (2.69,2.05);
 \draw (2.69,2.05)-- (1.66,3.46);
 \draw (1.66,3.46)-- (0,4);
 \draw [shift={(-0.39,2.39)}] plot[domain=4.39:5.68,variable=\t]({1*3.72*cos(\t r)+0*3.72*sin(\t r)},{0*3.72*cos(\t r)+1*3.72*sin(\t r)});
 \end{tikzpicture} \ar[ur] & & \begin{tikzpicture}[scale=0.3]
 \fill[fill=black,fill opacity=0.1] (0,4) -- (-1.66,3.46) -- (-2.69,2.05) -- (-2.68,0.3) -- (-1.66,-1.11) -- (0,-1.65) -- (1.66,-1.11) -- (2.69,0.3) -- (2.69,2.05) -- (1.66,3.46) -- cycle;
 \draw [fill=black,fill opacity=1.0] (-1.66,1.17) circle (0.4cm);
 \draw [fill=black,fill opacity=1.0] (1.66,1.18) circle (0.4cm);
 \draw (0,4)-- (-1.66,3.46);
 \draw (-1.66,3.46)-- (-2.69,2.05);
 \draw (-2.69,2.05)-- (-2.68,0.3);
 \draw (-2.68,0.3)-- (-1.66,-1.11);
 \draw (-1.66,-1.11)-- (0,-1.65);
 \draw (0,-1.65)-- (1.66,-1.11);
 \draw (1.66,-1.11)-- (2.69,0.3);
 \draw (2.69,0.3)-- (2.69,2.05);
 \draw (2.69,2.05)-- (1.66,3.46);
 \draw (1.66,3.46)-- (0,4);
 \draw [shift={(1.22,2.06)}] plot[domain=3.14:4.4,variable=\t]({1*3.9*cos(\t r)+0*3.9*sin(\t r)},{0*3.9*cos(\t r)+1*3.9*sin(\t r)});
 \end{tikzpicture} \ar[ur] & & \begin{tikzpicture}[scale=0.3]
 \fill[fill=black,fill opacity=0.1] (0,4) -- (-1.66,3.46) -- (-2.69,2.05) -- (-2.68,0.3) -- (-1.66,-1.11) -- (0,-1.65) -- (1.66,-1.11) -- (2.69,0.3) -- (2.69,2.05) -- (1.66,3.46) -- cycle;
 \draw [fill=black,fill opacity=1.0] (-1.66,1.17) circle (0.4cm);
 \draw [fill=black,fill opacity=1.0] (1.66,1.18) circle (0.4cm);
 \draw (0,4)-- (-1.66,3.46);
 \draw (-1.66,3.46)-- (-2.69,2.05);
 \draw (-2.69,2.05)-- (-2.68,0.3);
 \draw (-2.68,0.3)-- (-1.66,-1.11);
 \draw (-1.66,-1.11)-- (0,-1.65);
 \draw (0,-1.65)-- (1.66,-1.11);
 \draw (1.66,-1.11)-- (2.69,0.3);
 \draw (2.69,0.3)-- (2.69,2.05);
 \draw (2.69,2.05)-- (1.66,3.46);
 \draw (1.66,3.46)-- (0,4);
 \draw [shift={(1.06,0.41)}] plot[domain=1.86:3.17,variable=\t]({1*3.75*cos(\t r)+0*3.75*sin(\t r)},{0*3.75*cos(\t r)+1*3.75*sin(\t r)});
 \end{tikzpicture} \ar[ur] & \\
 }
 }
\]
\caption{A tube of size 5. We can see that from the 5th layer, the arcs cross themselves and do not correspond to an $m$-rigid object anymore}
\label{fig:tubu}
\end{figure}
\end{landscape}

\nocite{BZ}
\nocite{Kel04}

\bibliographystyle{alpha}
\bibliography{biblio}

\newcommand{\etalchar}[1]{$^{#1}$}
\def\ocirc#1{\ifmmode\setbox0=\hbox{$#1$}\dimen0=\ht0 \advance\dimen0
  by1pt\rlap{\hbox to\wd0{\hss\raise\dimen0
  \hbox{\hskip.2em$\scriptscriptstyle\circ$}\hss}}#1\else {\accent"17 #1}\fi}
  \def\ocirc#1{\ifmmode\setbox0=\hbox{$#1$}\dimen0=\ht0 \advance\dimen0
  by1pt\rlap{\hbox to\wd0{\hss\raise\dimen0
  \hbox{\hskip.2em$\scriptscriptstyle\circ$}\hss}}#1\else {\accent"17 #1}\fi}
  \def\ocirc#1{\ifmmode\setbox0=\hbox{$#1$}\dimen0=\ht0 \advance\dimen0
  by1pt\rlap{\hbox to\wd0{\hss\raise\dimen0
  \hbox{\hskip.2em$\scriptscriptstyle\circ$}\hss}}#1\else {\accent"17 #1}\fi}
\begin{thebibliography}{BMR{\etalchar{+}}06}

\bibitem[AS81]{AS}
M.~Auslander and O.~Smalo.
\newblock Almost split sequences in subcategories.
\newblock {\em Journal of Algebra}, 69(2):426--454, 1981.

\bibitem[ASS06]{ASS}
Ibrahim Assem, Daniel Simson, and Andrzej Skowro{\'n}ski.
\newblock {\em Elements of the representation theory of associative algebras.
  {V}ol. 1}, volume~65 of {\em London Mathematical Society Student Texts}.
\newblock Cambridge University Press, Cambridge, 2006.
\newblock Techniques of representation theory.

\bibitem[BM07]{BM02}
Karin Baur and Robert~J. Marsh.
\newblock A geometric description of the {$m$}-cluster categories of type
  {$D_n$}.
\newblock {\em Int. Math. Res. Not. IMRN}, (4):Art. ID rnm011, 19, 2007.

\bibitem[BM08]{BM01}
Karin Baur and Robert~J. Marsh.
\newblock A geometric description of {$m$}-cluster categories.
\newblock {\em Trans. Amer. Math. Soc.}, 360(11):5789--5803, 2008.

\bibitem[BMR{\etalchar{+}}06]{BMRRT}
Aslak~Bakke Buan, Robert Marsh, Markus Reineke, Idun Reiten, and Gordana
  Todorov.
\newblock Tilting theory and cluster combinatorics.
\newblock {\em Adv. Math.}, 204(2):572--618, 2006.

\bibitem[BT09]{BT}
Aslak~Bakke Buan and Hugh Thomas.
\newblock Coloured quiver mutation for higher cluster categories.
\newblock {\em Adv. Math.}, 222(3):971--995, 2009.

\bibitem[BT15]{BauTor}
Karin Baur and Edmund~André Torkildsen.
\newblock A geometric realization of tame categories.
\newblock 2015.

\bibitem[Bua11]{Buan}
Aslak~Bakke Buan.
\newblock An introduction to higher cluster categories.
\newblock {\em Bull. Iranian Math. Soc.}, 37(2):137--157, 2011.

\bibitem[BZ11]{BZ}
Thomas Br{\"u}stle and Jie Zhang.
\newblock On the cluster category of a marked surface without punctures.
\newblock {\em Algebra Number Theory}, 5(4):529--566, 2011.

\bibitem[CB]{WCB}
William Crawley-Boevey.
\newblock Lectures on representations of quivers.

\bibitem[CCS06]{CCS}
Philippe. Caldero, Frédéric. Chapoton, and Ralf. Schiffler.
\newblock Quivers with relations arising from clusters ({$A_n$} case).
\newblock {\em Trans. Amer. Math. Soc.}, 358(3):1347--1364, 2006.

\bibitem[CP15a]{CP01}
Cesar Ceballos and Vincent Pilaud.
\newblock Cluster algebras of type {$D$}: pseudotriangulations approach.
\newblock {\em Electron. J. Combin.}, 22(4):Paper 4.44, 27, 2015.

\bibitem[CP15b]{CP02}
Cesar Ceballos and Vincent Pilaud.
\newblock Denominator vectors and compatibility degrees in cluster algebras of
  finite type.
\newblock {\em Trans. Amer. Math. Soc.}, 367(2):1421--1439, 2015.

\bibitem[Fom10]{F}
Sergey Fomin.
\newblock Total positivity and cluster algebras.
\newblock In {\em Proceedings of the {I}nternational {C}ongress of
  {M}athematicians. {V}olume {II}}, pages 125--145. Hindustan Book Agency, New
  Delhi, 2010.

\bibitem[FR05]{FR}
Sergey Fomin and Nathan Reading.
\newblock Generalized cluster complexes and {C}oxeter combinatorics.
\newblock {\em Int. Math. Res. Not.}, (44):2709--2757, 2005.

\bibitem[FST08]{FST}
Sergey Fomin, Michael Shapiro, and Dylan Thurston.
\newblock Cluster algebras and triangulated surfaces. {I}. {C}luster complexes.
\newblock {\em Acta Math.}, 201(1):83--146, 2008.

\bibitem[FZ02]{FZ}
Sergey Fomin and Andrei Zelevinsky.
\newblock Cluster algebras. {I}. {F}oundations.
\newblock {\em J. Amer. Math. Soc.}, 15(2):497--529 (electronic), 2002.

\bibitem[FZ03]{FZ2}
Sergey Fomin and Andrei Zelevinsky.
\newblock Cluster algebras. {II}. {F}inite type classification.
\newblock {\em Invent. Math.}, 154(1):63--121, 2003.

\bibitem[IY08]{IY}
Osamu Iyama and Yuji Yoshino.
\newblock Mutation in triangulated categories and rigid {C}ohen-{M}acaulay
  modules.
\newblock {\em Invent. Math.}, 172(1):117--168, 2008.

\bibitem[JM]{JM}
Lucie Jacquet-Malo.
\newblock A bijection between $m$-cluster-tilting objects and
  $(m+2)$-angulations in $m$-cluster categories. arxiv:1706.06866.

\bibitem[Kel05]{Kel03}
Bernhard Keller.
\newblock On triangulated orbit categories.
\newblock {\em Doc. Math.}, 10:551--581, 2005.

\bibitem[Kel10]{Kel04}
Bernhard Keller.
\newblock Cluster algebras, quiver representations and triangulated categories.
\newblock In {\em Triangulated categories}, volume 375 of {\em London Math.
  Soc. Lecture Note Ser.}, pages 76--160. Cambridge Univ. Press, Cambridge,
  2010.

\bibitem[KR08]{KR}
Bernhard Keller and Idun Reiten.
\newblock Acyclic {C}alabi-{Y}au categories.
\newblock {\em Compos. Math.}, 144(5):1332--1348, 2008.
\newblock With an appendix by Michel Van den Bergh.

\bibitem[Lam15]{L}
Lisa Lamberti.
\newblock Combinatorial model for the cluster categories of type {E}.
\newblock {\em J. Algebraic Combin.}, 41(4):1023--1054, 2015.

\bibitem[Sch08]{Sch}
Ralf Schiffler.
\newblock A geometric model for cluster categories of type {$D_n$}.
\newblock {\em J. Algebraic Combin.}, 27(1):1--21, 2008.

\bibitem[Tho07]{Tho}
Hugh Thomas.
\newblock Defining an {$m$}-cluster category.
\newblock {\em J. Algebra}, 318(1):37--46, 2007.

\bibitem[Tor12]{Tor}
Hermund~André Torkildsen.
\newblock A geometric realization of the $m$-cluster category of type
  $\tilde{A}$, 2012.

\bibitem[Wr{\ocirc{a}}09]{W}
Anette Wr{\ocirc{a}}lsen.
\newblock Rigid objects in higher cluster categories.
\newblock {\em J. Algebra}, 321(2):532--547, 2009.

\bibitem[Zhu08]{Z}
Bin Zhu.
\newblock Generalized cluster complexes via quiver representations.
\newblock {\em J. Algebraic Combin.}, 27(1):35--54, 2008.

\bibitem[ZZ09]{ZZ}
Yu~Zhou and Bin Zhu.
\newblock Cluster combinatorics of {$d$}-cluster categories.
\newblock {\em J. Algebra}, 321(10):2898--2915, 2009.

\end{thebibliography}

\end{document}